\documentclass[12pt,twoside,reqno]{amsart}

\usepackage{amsmath, amssymb, amsthm, enumitem, esint}
\usepackage[hyperindex]{hyperref}

\usepackage{hyperref}
\hypersetup{bookmarksdepth=3}

\newcommand{\HKVarDef}{Definition~3.1}
\newcommand{\HKPropVarTriangle}{Lemma~3.2}
\newcommand{\HKDefConjHKmeasure}{Definition~2.4}
\newcommand{\HKSubsecHOandHK}{Subsection~2.3}
\newcommand{\HKThmGradientPhiEstimate}{Theorem~4.1}
\newcommand{\HKHnConcentration}{Corollary~3.8}
\newcommand{\HKCenterConstantRmBound}{Proposition~9.5}
\newcommand{\HKVarMonotonicityCHF}{Corollary~3.7}
\newcommand{\HKHCenterDef}{Definition~3.10}
\newcommand{\HKHCenterPropExist}{Proposition~3.12}

\newcommand{\HKBoundBallHCenter}{Proposition~3.13}
\newcommand{\HKSecPstarParabolic}{Section~9}
\newcommand{\HKPropPstarBasic}{Proposition~9.4}
\newcommand{\HKLemWoneMonotone}{Lemma~2.7}
\newcommand{\HKThmGaussianintegral}{Theorem~3.14}
\newcommand{\HKLemHEsubsolutiononParNbhd}{Lemma~9.13}

\newcommand{\la}{\lambda}

\newcommand{\de}{\delta}
\newcommand{\tf}{\mathfrak{t}}
\newcommand{\lb}{\linebreak[1]}
\newcommand{\CF}{\mathfrak{C}}
\newcommand{\IR}{\mathbb{R}}

\newcommand{\IN}{\mathbb{N}}

\newcommand{\IF}{\mathbb{F}}

\newcommand{\PP}{\mathcal{P}}
\newcommand{\RR}{\mathcal{R}}
\newcommand{\CC}{\mathcal{C}}
\newcommand{\MM}{\mathcal{M}}
\newcommand{\II}{\mathcal{I}}
\newcommand{\XX}{\mathcal{X}}
\renewcommand{\t}{\mathfrak{t}}
\newcommand{\eps}{\varepsilon}
\newcommand{\ov}[1]{\overline{#1}}
\newcommand{\td}[1]{\widetilde{#1}}

\DeclareMathOperator*{\length}{length}
\DeclareMathOperator*{\Varalone}{Var}
\DeclareMathOperator*{\loc}{loc}
\newcommand{\Var}{{\Varalone}}
\DeclareMathOperator*{\osc}{osc}

\DeclareMathOperator{\Ric}{Ric}
\DeclareMathOperator{\Rm}{Rm}
\DeclareMathOperator{\tr}{tr}
\DeclareMathOperator{\id}{id}

\DeclareMathOperator{\eucl}{eucl}
\DeclareMathOperator{\vol}{vol}

\DeclareMathOperator{\supp}{supp}

\DeclareMathOperator{\proj}{proj}

\newcommand{\EMPTY}[1]{}

\newtheorem{Theorem}[equation]{Theorem}
\newtheorem{Lemma}[equation]{Lemma}
\newtheorem{Corollary}[equation]{Corollary}
\newtheorem{Proposition}[equation]{Proposition}
\newtheorem{Claim}[equation]{Claim}

\theoremstyle{definition}
\newtheorem{Definition}[equation]{Definition}
\theoremstyle{remark}
\newtheorem{Remark}[equation]{Remark}
\newtheorem{Example}[equation]{Example}

\numberwithin{equation}{section}

\title{Compactness theory of the space of Super Ricci flows}
\author{Richard H Bamler}
\address{Department of Mathematics, UC Berkeley, CA 94720, USA}
\email{rbamler@berkeley.edu}
\thanks{This work was supported by NSF grant DMS-1906500.}
\date{\today}

\begin{document}

\begin{abstract}
We develop a compactness theory for super Ricci flows, which lays the foundations for the partial regularity theory in \cite{Bamler_HK_RF_partial_regularity}.
Our results imply that any sequence of super Ricci flows of the same dimension that is pointed in an appropriate sense subsequentially converges to a certain type of synthetic flow, called a metric flow.
We will study the geometric and analytic properties of this limiting flow, as well as the convergence in detail.
We will also see that, under appropriate local curvature bounds, a limit of Ricci flows can be decomposed into a regular and singular part.
The regular part can be endowed with a canonical structure of a Ricci flow spacetime and we have smooth convergence on a certain subset of the regular part.
\end{abstract}

\maketitle

\tableofcontents

\section{Introduction}
\subsection{Introduction}
A super Ricci flow is given by a smooth family of Riemannian metrics $(g_t)_{t \in I}$ on a manifold $M$ that satisfies the inequality
\[ \partial_t g_t \geq - 2 \Ric_{g_t}, \]
meaning that $\partial_t g_t + 2 \Ric_{g_t}$ is non-negative definite.
Super Ricci flows were initially studied by McCann and Topping  \cite{Topping-McCann} and are natural generalizations of Ricci flows.
We will show that the space of super Ricci flows, pointed in an appropriate sense, is compact in a certain topology.
While the main motivation of our theory is to obtain a compactness theory of Ricci flows, most of our results also apply to other geometric settings, such as metrics with lower Ricci curvature bounds in the Bakry-\'Emery sense or Ricci solitons.

In this paper we will introduce and study several new notions, guided by some results in \cite{Bamler_HK_entropy_estimates}.
We will see that super Ricci flows possess similar compactness properties as spaces with lower Ricci curvature bounds.
More specifically, we will:
\begin{enumerate}[label=\arabic*.]
\item Introduce a notion of synthetic flows, called \emph{($H$-concentrated) metric flows,} of which super Ricci flows are a subset.
A metric flow can be regarded as a parabolic analogue of a metric (measure) space.
\item Analyze the geometric properties that follow from the axioms of the definition of a metric flow.
\item Define a distance function on the space metric flow pairs (a.k.a. ``pointed metric flows''), which can be viewed as a parabolic analogue of the Gromov-Hausdorff distance.
\item Show that certain subsets of metric flow pairs, which contain the class of super Ricci flows, are compact with respect to this new distance function.
So any sequence of pointed metric flow pairs taken out of these subsets subsequentially converges to another metric flow pair.
\item Analyze the convergence behavior of metric flow pairs with respect to this distance function and show that certain important properties survive the limit.
\item Devise a notion of smooth convergence (akin to smooth Cheeger-Gromov convergence) in the case in which the original metric flow pairs are locally given by smooth Ricci flows with bounded curvature.
\end{enumerate}

In subsequent work \cite{Bamler_HK_RF_partial_regularity}, we will further analyze limits of Ricci flows under a non-collapsing condition and derive several structural results.

\subsection{History}
Before describing the new notions and results of this paper in more detail, let us make some historical remarks.
Besides Ricci flows, the class of super Ricci flows also contains other flows, which arise from certain interesting classes of metrics via standard constructions.
The most important of these are probably the classes of Riemannian metrics with lower Ricci curvature bounds ($\Ric_g \geq \lambda g$) and Einstein metrics ($\Ric _g= \lambda g$).
More generally, we can also consider the class of metrics whose Ricci curvature is bounded from below in the Bakry-\'Emery sense ($\Ric_{g}^f = \Ric_g + \nabla^2 f \geq \lambda g$) and the class of gradient Ricci solitons ($\Ric_{g}^f = \Ric_g + \nabla^2 f = \lambda g$).\footnote{In fact, any metric satisfying $\Ric_g + \mathcal{L}_X g \geq \lambda g$ for some vector field $X$ can be turned into a super Ricci flow.
In the case of equality, these metrics are called Ricci solitons.}
The theory developed in this paper will imply a compactness theory for all these classes of metrics.

We will now give an overview over existing compactness theories for these and other specific classes of super Ricci flows.

Let us first consider metrics with lower Ricci curvature bounds and Einstein metrics.
By the work of Gromov \cite{Gromov_metric_struc}, any sequence of manifolds with a uniform lower bound on the Ricci curvature and an upper bound on the dimension converges subsequentially to a metric (measure) space in the Gromov-Hausdorff sense.
In fact, the space of isometry classes of metric spaces that satisfy certain weak regularity conditions can be equipped with a distance function --- the Gromov-Hausdorff distance --- and the subset of spaces corresponding to metrics with lower Ricci curvature bounds is precompact with respect to this distance function.
In the non-collapsed case, the Gromov-Hausdorff limits of sequences of spaces with lower Ricci curvature bound can be characterized further by the work of Anderson, Cheeger, Colding, Tian, Naber \cite{MR999661, Anderson_1990, Cheeger-Colding-Cone, Cheeger_Colding_struc_Ric_below_I,Cheeger-Colding-structure-II,Cheeger_Colding_struc_Ric_below_III, Cheeger_Colding_Tian_2002, Cheeger-Naber-quantitative, Cheeger-Naber-Codim4}.
In this case the limiting space is regular (in a certain sense) on the complement of a singular set of codimension 2 (for spaces with lower Ricci curvature bounds) or codimension 4 (for Einstein metrics).
Another avenue of analyzing the limiting space, which also works in the collapsed case, is due to Lott, Villani and Sturm \cite{Lott_Villani_2009, Sturm_convex_functionals, Sturm-2006-I, Sturm-2006-II}, who introduced a synthetic lower Ricci curvature bound, using optimal transport.
This synthetic bound is preserved under measured Gromov-Hausdorff convergence and can therefore be used to characterize the limiting space.
This has led to the notion of $\text{RCD}(K,N)$-spaces, which have been subject of intensive research.

In the setting of lower Ricci curvature bounds in the  Bakry-\'Emery sense, a compactness theory under an additional bound on the potential function $f$, follows from the work of Wei-Wylie \cite{Wei-Wylie-2009}. 
By an observation of Lott \cite{Lott-Bakry-Emery-2003}, such metrics arise as collapsed limits of metrics with lower Ricci curvature bounds, which reduces the compactness theory to that of spaces with lower Ricci curvature bounds.
See also \cite{wang2013structure} for a further structure theory under additional geometric assumptions.

The case of gradient Ricci solitons was analyzed by Cao-Sesum, X.~Zhang, Z.~Zhang, Weber, Haslhofer, M\"uller, Y.~Zhang, H.~Li, Y.~Li, B.~Wang and S.~Huang, Y.~Li, B.~Wang \cite{X_Zhang_compactness_soliton, Cao_Sesum_compactness_KSol,  Z_Zhang_2010, Weber_conv_Ric_sol, Haslhofer_Mueller_compactness_soliton, Halshofer_Mueller_not_4d_RS_cpt, Y_Zhang_compactness_4D_R_sol, Li_Li_Wang_21, Huang_Li_Wang_21}.
This work essentially shows that any sequence of gradient (shrinking) solitons, pointed at the minimum of the potential, converges to a metric space that is smooth on the complement of a subset of codimension $\geq 4$ --- thus mirroring the results for Einstein metrics mentioned before.

In the setting of Ricci flows, or general super Ricci flow, compactness theorems are only known under very restrictive conditions.
Hamilton's original compactness theorem for Ricci flows \cite{Hamilton_RF_compactness} holds under global curvature and injectivity radius bounds, which guarantee that the limit is another Ricci flow.
In dimension $3$, Perelman's work \cite{Perelman1, Perelman2} implies a satisfactory compactness theory in the non-collapsed case; limits are again smooth Ricci flows.
In \cite{Bamler-Zhang-1, Bamler-Zhang-2, Bamler_struc_theory_sing_spaces, Bamler_conv_RF_bounded_scal, Chen-Wang-I, Chen_Wang_space_RF_II_A, Chen-Wang-II}, Zhang, the author, Chen and Wang devised a compactness theory for the space of Ricci flows and K\"ahler-Ricci flows under a pointwise bound on the scalar curvature.
In this theory limits are regular away from a singular set of codimension at least 4; therefore their behavior is similar to that of Einstein metrics.
In \cite{Sturm_super_RF} Sturm introduced a new notion of super Ricci flows for time-dependent metric measure spaces and proved a compactness theorem for these flows assuming a pointwise lower bound on the Ricci curvature.
This theory resembles the approach of Lott-Villani-Sturm in the stationary case.
Unfortunately, the lower Ricci curvature bound is very restrictive; for example it precludes isolated degenerate or non-degenerate neckpinches, which are quite common in dimension 3 \cite{Angenent_Knopf_precise_asymptotics}.
Further related work can be found in \cite{Lu_2001, Topping_L_optimal_trans,Lott_optimal_transport_reduced_vol, Topping_Cabezas_Rivas_canonical_soliton,  Topping_foundations_RF_opt_trans,  Munn_metric_RF, Haslhofer_Naber_2018}.

\subsection{Overview}
We will now provide an overview of the theory developed in this paper.
We will state our main results in a rather vague way, but refer to the corresponding precise statements in the body of the paper.

We will first introduce the notion of a \emph{metric flow} $\XX$ over an interval $I \subset \IR$ in Section~\ref{sec_met_flows} (see Definition~\ref{Def_metric_flow}).
Roughly speaking, a metric flow consists of a collection of metric spaces $(\XX_t, d_t)$, which are viewed as \emph{time-slices} and a collection of probability measures $\nu_{x;t}$ on each time-slice $(\XX_t, d_t)$, where $x \in \XX_{t'}$, $t < t'$.
These measures should be thought of as \emph{conjugate heat kernels based at $x$ (at time $t'$)} on a super Ricci flow background.
The metrics and the probability measures are required to satisfy certain compatibility conditions, which always hold on a super Ricci flow and are independent of its dimension.
Among other things, these compatibility conditions are the standard reproduction formula and a gradient bound for induced heat flows, which is established in \cite[\HKThmGradientPhiEstimate]{Bamler_HK_entropy_estimates}.
The collection of conjugate heat kernels $\nu_{x;s}$ that make up a metric flow $\XX$ allow us to define solutions to the (forward) heat equation and (backward) conjugate heat equation on $\XX$.
In addition, we will often require a metric flow to be \emph{$H$-concentrated} for some $H < \infty$ (see Definition~\ref{Def_H_Concentration}), which means that the conjugate heat kernels satisfy a certain $L^2$-bound.
This bound was established for super Ricci flows in \cite{Bamler_HK_entropy_estimates} and is the only bound in this paper that depends on the dimension.

Given a super Ricci flow $(g_t)_{t \in I}$ on a manifold $M$ we can construct the associated metric flow as follows.
Set $\XX := M \times I$, where the time-slices are of the form $\XX_t := M \times \{ t \}$, let $d_t := d_{g_t}$ be the length metric at any time $t$ and denote by $d\nu_{(x,t);s}  := K(x,t; \cdot, s) dg_s$ the measure corresponding to the conjugate heat kernel based at $(x,t)$.
Then we have:

\begin{Theorem}[Theorem~\ref{Thm_superRF_metric_flow}]
$\XX$ is an $H_n :=( \frac{(n-1)\pi^2}2 + 4)$-concentrated metric flow.
\end{Theorem}

Note that, while a Ricci flow is given by a family of metrics on a fixed manifold, the topology of time-slices $\XX_t$ of a metric flow may be non-constant in $t$.
In addition, when passing from a super Ricci flow to a metric flow, we have dispensed of one essential piece of structure, namely the concept of world lines.
So for example, the fact that $(x,s), (x,t) \in M \times I$ correspond to the same point at two different times $s,t$ gets lost when we consider the associated metric flow.
However, any metric flow has a natural topology (see Subsection~\ref{subsec_natural_topology}), which agrees with the topology on $M \times I$ if $I$ is left-open (see Theorem~\ref{Thm_superRF_metric_flow}).
In addition, the $H$-concentration property allows us to relate points between two time-slices $\XX_s, \XX_t$, $s \leq t$, of a metric flow up to an error of $\sim \sqrt{|t-s|}$.
More specifically, we can view $y \in \XX_s$ to be ``close'' to $x \in \XX_t$ if it lies ``near the concentration center'' of the conjugate heat kernel $\nu_{x;s}$.
A lack of world lines may seem unintuitive at first; it is however natural in the study of most Ricci flow problems.

In Section~\ref{sec_geom_cont_time_slices} we will analyze the dependence of the time-slices $\XX_t$ of an $H$-concentrated metric flow $\XX$ on time.
Among other things, we will see that, away from a countable set of times, this dependence is continuous in the Gromov-$W_1$-topology (see Theorem~\ref{Thm_future_cont_equiv} and Corollary~\ref{Cor_met_flow_cont_ae}).
We will call $\XX$ \emph{future continuous} (see Definition~\ref{Def_flow_continuous}) if this convergence $\XX_t \to \XX_{t_0}$ holds whenever $t \searrow t_0$.
It will turn out that a future continuous flow is uniquely determined by its behavior over a dense set of times (see Theorems~\ref{Thm_existenc_fut_compl}, \ref{Thm_extend_isometry}); the entire flow can be recovered by passing to the so-called \emph{future completion,} which is similar to a completion of metric spaces.

Our goal in Section~\ref{sec_space_met_flow_pairs} is to study the space of metric flows and define what it means that a sequence of metric flows converges to another metric flow.
In order to do this, we need to consider \emph{metric flow pairs} $(\XX, (\mu_t)_{t \in I})$ (see Definition~\ref{Def_metric_flow_pair}), which are metric flows equipped with a conjugate heat flow $(\mu_t)_{t \in I}$.
Here, the conjugate heat flow $(\mu_t)_{t \in I}$ serves as some kind of basepoint that indicates the ``center'' of the flow; this is similar to choosing basepoints when defining Gromov-Hausdorff convergence for unbounded metric spaces.
The conjugate heat flow $(\mu_t)_{t \in I}$ will often be taken to be a conjugate heat kernel based at some point in the final time-slice of $\XX$, however, our theory is not limited to this case.
We consider two metric flow pairs to be the same if they agree at almost every time and denote the space of such equivalence classes by $\IF_I$ (see Definition~\ref{Def_IF}).
By passing to the future completion, we can represent each such equivalence class by a unique flow pair whose metric flow is future continuous.
We will then define a metric $d_{\IF}$ on $\IF_I$ and show that $(\IF_I, d_{\IF})$ is a complete metric space.
A sequence of metric flow pairs is then said to \emph{$\IF$-converge} to another metric flow pair if it converges in $(\IF_I, d_{\IF})$.
Roughly speaking, $\IF$-convergence implies convergence of almost every time-slice\footnote{in most cases even all, but a \emph{countable} set of time-slices} in the Gromov-$W_1$-topology along with almost every conjugate heat kernel measure.

In Section~\ref{Sec_conv_within} we analyze the notion of $\IF$-convergence further and we relate objects on the sequence of metric flow pairs with objects on the limit.
To do this, we first embed an $\IF$-convergent sequence of metric flow pair into a common ``correspondence'' and then study the convergence behavior of points and conjugate heat flows within this correspondence.

In Section~\ref{sec_compact_subsets_IF} we will show that certain subsets of $\IF_I$ are in fact compact.
These subsets include metric flow pairs corresponding to super Ricci flows that are equipped with a conjugate heat kernel.
So as a result, we will obtain:

\begin{Theorem}[Corollary~\ref{Cor_super_RF_compactness}]
Consider a sequence of pointed super Ricci flows $(M_i, (g_{i,t})_{t \in (-T,0]},x_i)$ of the same dimension.
Then, after passing to a subsequence, the corresponding sequence of metric flows $\XX^i$ equipped with the conjugate heat kernels $(\mu^i_t)_{t \in (-T,0)} := (\nu_{(x_i,0); t})_{t \in (-T,0)}$, based at $(x_i,0)$, $\IF$-converge to a metric flow pair of the form $(\XX^\infty,(\mu^\infty_t)_{t \in (-T,0)})$.
\end{Theorem}

In the case in which these super Ricci flows smoothly converge to some limiting super Ricci flow (such as in \cite{Hamilton_RF_compactness} in the case of Ricci flows), the limiting metric flow $\XX^\infty$ is simply the flow associated to the limiting super Ricci flow.
So smooth convergence implies $\IF$-convergence.
The following example shows, however, that, in general, the limit $\XX^\infty$ may be different from the Gromov-Hausdorff limit if smooth convergence does not hold.

\begin{Example}
Consider the Bryant soliton $(M^n, (g_t)_{t \in \IR})$, $n \geq 3$, which is a rotationally symmetric steady gradient soliton on $\IR^n$ with one end \cite{Bryant2005}.
Denote by $x \in M$ its center of rotation.
If we consider the parabolic blow-downs $(M, (g^\lambda_t:= \lambda^2 g_{\lambda^{-2} t})_{t \in \IR})$, which are Ricci flows, then the Gromov-Hausdorff limit
\begin{equation} \label{eq_ex_GH_limit}
 (M, g^\lambda_0,x) \xrightarrow[\lambda \to 0]{\quad GH \quad} ([0, \infty),0) 
\end{equation}
is a Euclidean ray, while we have the following convergence on compact time-intervals:
\begin{equation} \label{eq_ex_IF_limit}
 \big( M^n, (g^\lambda_t)_{t < 0},(\nu_{(x,0);t})_{t < 0} \big) \xrightarrow[\lambda \to 0]{\quad d_{\IF} \quad} \big(S^{n-1} \times \IR, (g_t = -2(n-1) t g_{S^{n-1}}  + g_{\IR})_{t < 0}, (\mu^\infty_t)_{t < 0} \big). 
\end{equation}
Here the limit is a round shrinking cylinder, which is the asymptotic shrinking soliton of the Bryant soliton in the sense of Perelman \cite{Perelman1}.
So while the limit in (\ref{eq_ex_IF_limit}) does not agree with (\ref{eq_ex_GH_limit}), it still captures an important asymptotic behavior of the Bryant soliton.
\end{Example}

In Section~\ref{sec_intrinsic} we show that the condition that almost all time-slices of a metric flow are intrinsic (i.e. length spaces) survives $\IF$-limits.
Moreover, if we choose the limit to be future continuous, which we can always do, then \emph{all} time-slices are intrinsic.

Lastly, in Section~\ref{sec_reg_pts} we analyze the case in which a certain subset of a metric flow is locally isometric to a smooth Ricci flow on some \emph{regular subset} $\RR \subset \XX$.
The regular subset will carry the canonical structure of a Ricci flow spacetime.
If $\XX$ is constructed from a smooth Ricci flow, then we have $\RR = \XX$.
Next, we consider a sequence of metric flow pairs $(\XX^i, (\mu^i_t)_{t \in I^i})$ that $\IF$-converges to a metric flow pair $(\XX^\infty, (\mu^\infty_t)_{t \in I^\infty})$.
We will show that, under local uniform curvature bounds on the regular parts $\RR^i$ of $\XX^i$, the $\IF$-convergence can be upgraded to smooth convergence over a certain subset $\RR^* \subset \RR^\infty$.

\subsection{Acknowledgements}
I thank Gang Tian, Guofang Wei and Bennett Chow for some inspiring conversations.
I also thank Sigurd Angenent, Dan Knopf and John Lott for useful advice on their work \cite{Angenent_Knopf_precise_asymptotics} and \cite{Lott-Bakry-Emery-2003,Lott_optimal_transport_reduced_vol,Lott_Villani_2009}, respectively. 
I also thank Christina Sormani for pointing out some additional related references.
I am grateful to Bennett Chow and Wangjian Jian for pointing out several typos in an early version of the manuscript.

\section{Preliminaries}
\subsection{Probability measures on metric spaces and the Wasserstein distance}
Let $(X, d)$ be a complete, separable metric space and denote by $\mathcal{B}(X)$ the Borel algebra generated by the open subsets of $X$.
A {\bf probability measure on $X$} is a measure $\mu$ on $\mathcal{B}(X)$ of total mass $\mu(X) = 1$.
The set of probability measures on $X$ is denoted by $\mathcal{P} (X)$.
For any $x \in X$ we denote by $\delta_x  \in \mathcal{P}(X)$ the point mass at $x$.

We recall:

\begin{Lemma} \label{Lem_basic_measure}
Let $(X,d)$ be complete and separable and let $\mu \in \mathcal{P}(X)$.
Then the following holds:
\begin{enumerate}[label=(\alph*)]
\item \label{Lem_basic_measure_a} $\mu$ is regular, i.e. for any $A \in \mathcal{B} (X)$, $\eps > 0$ there are compact and open subsets $K \subset A \subset O \subset X$ such that $\mu ( O \setminus K ) \leq \eps$.
\item \label{Lem_basic_measure_b} The set of bounded Lipschitz functions $X \to \IR$ is dense in $L^p (X, \mu)$ for all $p < \infty$.
\item \label{Lem_basic_measure_c} The support
\[ \supp \mu = \{ x \in X \;\; : \;\; \mu( B(x,r) ) > 0 \quad \text{for all} \quad r > 0 \}. \]
is closed and satisfies $\mu ( X \setminus \supp \mu ) = 0$.
\item \label{Lem_basic_measure_d} For any tight sequence $\mu_i \in \mathcal{P} (X)$ (i.e. for any $\eps > 0$ there is a compact subset $K_\eps$ such that $\mu_i (X \setminus K_\eps) \leq \eps$ for all $i$) there is a subsequence such that we have weak convergence $\mu_i \to \mu_\infty \in \mathcal{P} (X)$ (i.e. $\int_X f \, d\mu_i \to \int_X f \, d\mu_\infty$ for all bounded, continuous functions $f : X \to \IR$, or equivalently, for all bounded, Lipschitz functions $f: X \to \IR$).
\item  \label{Lem_basic_measure_e} A sequence $\mu_i \in \mathcal{P} (X)$ is tight if and only if for any $\eps > 0$ there is a compact subset $K'_\eps \subset X$ such that $\mu_i (X \setminus B(K'_\eps, \eps)) \leq \eps$ for large $i$.
\end{enumerate}
\end{Lemma}

\begin{proof}
For Assertion~\ref{Lem_basic_measure_a} see \cite[Theorem~1.3]{Billingsley_convergenc_prob_meas}.
For any subsets $K \subset A \subset O$ as in Assertion~\ref{Lem_basic_measure_a} the function $f_{K, L} := \min \{ L d(\cdot, X \setminus O), 1 \}$ is Lipschitz and satisfies $f_{K, L} \equiv 1$ on $K$ for large $L$.
This shows that $\chi_A$ can be approximated by Lipschitz functions in $L^p(X, \mu)$.
Since characteristic functions span a dense subspace in $L^p(X, \mu)$, this shows Assertion~\ref{Lem_basic_measure_b}.
For Assertion~\ref{Lem_basic_measure_c} suppose by contradiction that $\mu(X \setminus \supp \mu) > 0$ and choose a compact subset $K \subset X \setminus \supp \mu$ with $\mu(K) > 0$.
However, $K$ can be covered by finitely many balls of mass zero.
For Assertion~\ref{Lem_basic_measure_d} see \cite[Theorem~5.1]{Billingsley_convergenc_prob_meas}.
The equivalence statement concerning weak convergence, consider a bounded, uniformly continuous function $f : X \to \IR$ and let $f_L(x) := \inf_{y \in X} \{ L d(x,y) + f(y) \}$.
Then $f_L$ is bounded, Lipschitz and we have $f_L \to f$ uniformly as $L \to \infty$.
The equivalence statement now follows from \cite[Theorem~2.1]{Billingsley_convergenc_prob_meas}.
Lastly, suppose that the sequence $\mu_i$ has the property described in  Assertion~\ref{Lem_basic_measure_e}.
By Assertion~\ref{Lem_basic_measure_a} it follows that for any $\eps > 0$ there is a compact subset $K'_\eps \subset X$ such that $\mu_i (X \setminus \overline{B(K'_\eps, \eps)}) \leq \eps$ for \emph{all} $i$.
Now fix some $\eps > 0$ and let 
\[ K_\eps := \bigcap_{j = 1}^\infty \ov{B(K'_{\eps 2^{-j}}, \eps 2^{-j})}. \]
Then $K_\eps$ is compact and we have for any $i$
\[ \mu_i ( X \setminus K_\eps) \leq \sum_{j=1}^\infty \mu_i ( X \setminus \ov{B(K'_{\eps 2^{-j}}, \eps 2^{-j})} ) \leq \eps. \qedhere \]
\end{proof}
\medskip

If $(X_i, d_i)$, $i = 1,2$, are two complete, separable metric spaces and $\mu_i \in \mathcal{P}(X_i)$, then a {\bf coupling} between $\mu_1, \mu_2$ is a probability measure $q \in \mathcal{P} (X_1 \times X_2)$ with marginals $\mu_1, \mu_2$, i.e. $\mu_1 (S) = q(S \times X_2)$ and $\mu_2 (S) = q(X_1 \times S)$ for all $S \in \mathcal{B} (X_1), \mathcal{B} (X_2)$, respectively.
Note that $q = \mu_1 \otimes \mu_2$ is a coupling between $\mu_1, \mu_2$.
The following lemma allows us to combine two couplings.

\begin{Lemma} \label{Lem_gluing}
Consider three complete, separable metric spaces $(X_i, d_i)$, $i = 1,2,3$, and probability measures $\mu_i \in \mathcal{P} (X_i)$.
Let $q_{12}$ and $q_{23}$ be couplings between $\mu_1, \mu_2$ and $\mu_2, \mu_3$, respectively.
Then there is a probability measure $q_{123} \in \mathcal{P} (X_1 \times X_2 \times X_3)$ whose marginals onto the first and last two factors equal $q_{12}$ and $q_{23}$, respectively.
In particular, the marginal of $q_{123}$ onto the first and last factor is a coupling between $\mu_1, \mu_3$.
\end{Lemma}

\begin{proof}
See \cite[Lemma~7.6]{Villani_topics_in_OT} or \cite[Lemma 3.3]{Chodosh-2012}.
\end{proof}

Fix some complete, separable metric space $(X,d)$.
We recall the definition of the {\bf $W_p$-Wasserstein distance} for $p \geq 1$ between two probability measures $\mu_1, \mu_2 \in \mathcal{P} (X)$:
\[ d_{W_p} (\mu_1, \mu_2) := \inf_{q} \bigg( \int_{X \times X} d^p(x_1, x_2) dq(x_1, x_2) \bigg)^{1/p}, \]
where the infimum is taken over all couplings $q \in \mathcal{P} (X \times X)$ between $\mu_1, \mu_2$.
We have

\begin{Proposition}
$d_{W_p}$ defines a complete metric on $\mathcal{P} (X)$ if we allow it to attain the value $\infty$.
\end{Proposition}

\begin{proof}
See  \cite[Theorem~7.3]{Villani_topics_in_OT}.
If $(X,d)$ is unbounded, then apply this theorem to the metric $d_A := \min \{ d, A \}$ and let $A \to \infty$.
\end{proof}

Let $\PP^p (X) \subset \PP (X)$ be the set of probability measures $\mu \in \PP(X)$ such that $d_{W_p} (\mu, \delta_x) = (\int_X d^p (x, \cdot) \, d\mu)^{1/p} < \infty$ for one (and thus for any) $x \in X$.
Then $(\PP^p (X), d_{W_p} |_{\PP^p (X)})$ is a complete metric space and we have:

\begin{Lemma} \label{Lem_W_p_separable}
$(\PP^p (X), d_{W_p} |_{\PP^p (X)})$ is separable. 
Moreover, for any dense subset $S \subset X$ the set of measures $\mu \in \PP^p(X)$ of finite support and with the property that $\supp \mu \subset X$ and that $\mu(\{x \}) \in \mathbb{Q}$ for all $x \in X$ is dense in $(\PP^p (X), d_{W_p} |_{\PP^p (X)})$.
\end{Lemma}

\begin{proof}
The first part of the lemma follows from the second part.
For the second part observe that if $\supp \mu$ is compact, then it can be approximated by the desired measures.
So it remains to show that any $\mu \in \PP^p(X)$ is the limit of $\mu \in \PP^p(X)$ with compact support.
For this purpose, fix some $x \in X$ and observe that by Lemma~\ref{Lem_basic_measure}\ref{Lem_basic_measure_a} there is an increasing sequence of compact subsets $K_i \subset X$ with $\bigcup_{i=1}^\infty K_i = \supp \mu$.
Let $\mu_i := \mu |_{K_i} + \mu ( X \setminus K_i) \delta_x$.
Then $\supp \mu_i$ is compact and $q_i :=  \mu |_{K_i} \otimes  \mu |_{K_i} +  \mu |_{X \setminus K_i} \otimes \delta_x$ is a coupling between $\mu, \mu_i$.
Therefore, by dominated convergence
\[ d_{W_p}^p (\mu, \mu_i) = \int_{X \setminus K_i} d^p(\cdot, x) d\mu 
= \int_X d^p(\cdot, x) \chi_{X \setminus K_i} d\mu  \to 0, \]
which finishes the proof.
\end{proof}

We will mainly be concerned with the $W_1$-Wasserstein distance and we will frequently use the following equivalent characterization of $d_{W_1}$ (due to the Kantorovich-Rubinstein Theorem \cite[Theorem~1.14]{Villani_topics_in_OT}):

\begin{Proposition} \label{Prop_Kant_Rub}
We have
\[ d_{W_1} (\mu_1, \mu_2) = \sup_{f} \int_X f \, d(\mu_1 - \mu_2), \]
where the supremum is taken over all bounded $1$-Lipschitz functions $f : X \to \IR$.
\end{Proposition}

\subsection{Variances of measures}
We recall the notion of variance from \cite[\HKVarDef]{Bamler_HK_entropy_estimates}, which can be generalized easily to the setting of metric measure spaces.

\begin{Definition}[Variance] \label{Def_Variance}
The {\bf variance} between two probability measures $\mu_1, \mu_2 \in \mathcal{P}(X)$ on a metric space $(X,d)$ is defined as
\[ \Var (\mu_1, \mu_2) := \int_X \int_X d^2 (x_1, x_2) d\mu_1 (x_1) d\mu_2 (x_2). \]
In the case $\mu_1 = \mu_2 = \mu$, we also write
\[ \Var (\mu) = \Var (\mu, \mu) = \int_X \int_X d^2 (x_1, x_2) d\mu (x_1) d\mu (x_2). \]
\end{Definition}

\begin{Remark}
This notion is similar, but slightly different from \cite[(3.1)]{Sturm-2006-I}.
If we write $\Var^{\text{Sturm}}$ for the notion in \cite[(3.1)]{Sturm-2006-I}, then
\[ \Var^{\text{Sturm}} (\mu) = \inf_{x \in X} \Var (\delta_x, \mu). \]
It follows from Lemma~\ref{Lem_Var_triangle_inequ} below that both notions are comparable:
\[ \Var^{\text{Sturm}} (\mu) \leq \Var (\mu) \leq 4 \Var^{\text{Sturm}} (\mu). \]
\end{Remark}
\medskip

We record the following linearity properties.
If $\mu_i, \mu'_j \in \mathcal{P}(X)$, $i = 1, \ldots, n$, $j = 1, \ldots, n'$, are probability measures and $a_i, a'_j \geq 0$ with $\sum_{i=1}^n a_i = \sum_{j=1}^{n'} a'_j = 1$, then
\[ \Var \Big( \sum_{i=1}^n a_i \mu_i, \sum_{j=1}^{n'} a'_j \mu'_j \Big) = \sum_{i=1}^n \sum_{j=1}^{n'} a_i a'_j \Var (\mu_i, \mu_j ). \]
Moreover, if $(Y, \nu), (Y', \nu')$ are probability spaces and $(\mu_s \in \mathcal{P}(X))_{s \in Y}$, $(\mu'_{s'} \in \mathcal{P}(X'))_{s' \in Y'}$ are integrable families of probability measures, then
\[  \Var \bigg( \int_Y \mu_s d\nu(s),  \int_{Y'} \mu'_{s'} d\nu'(s') \bigg) = \int_Y \int_{Y'} \Var ( \mu_s, \mu'_{s'} ) d\nu(s)  d\nu'(s'). \]

We also mention that for any $x \in X$ we have
\[ \Var (\delta_x ,\mu) = \int_X d^2 (x, y) d\mu (y) \]
and for any $x, y \in X$ we have
\[ \Var (\delta_x, \delta_y) = d^2 (x, y). \]

We will frequently use the following triangle inequality and bound relating $\Var(\mu_1, \mu_2)$ with $d_{W_1} (\mu_1, \mu_2)$:

\begin{Lemma} \label{Lem_W_1_vs_Var}  \label{Lem_Var_triangle_inequ}
If $\mu_1, \mu_2, \mu_3 \in \mathcal{P}(X)$, then
\[ \sqrt{ \Var (\mu_1, \mu_3) } \leq \sqrt{ \Var (\mu_1, \mu_2) } + \sqrt{ \Var (\mu_2, \mu_3) }, \]
\begin{equation} \label{eq_dW1_Var}
 d_{W_1} (\mu_1, \mu_2) \leq \sqrt{\Var(\mu_1, \mu_2)} 
\leq d_{W_1} (\mu_1, \mu_2) + \sqrt{\Var(\mu_1)} + \sqrt{\Var(\mu_2)}. 
\end{equation}
\end{Lemma}

\begin{proof}
The lemma follows along the lines of \cite[\HKPropVarTriangle]{Bamler_HK_entropy_estimates}.
We give another proof of (\ref{eq_dW1_Var}) using couplings.
For the first bound, observe that $\mu_1 \otimes \mu_2$ is a coupling between $\mu_1, \mu_2$, so
\begin{multline*}
 d_{W_1} (\mu_1, \mu_2) \leq \int_{X} \int_{X} d(x_1, x_2) \, d\mu_1 (x_1) d\mu_2 (x_2)
\leq \bigg( \int_{X} \int_{X} d^2(x_1, x_2) \, d\mu_1 (x_1) d\mu_2 (x_2) \bigg)^{1/2} \\
= \sqrt{\Var(\mu_1, \mu_2)} . 
\end{multline*}
For the second bound, we use Lemma~\ref{Lem_Var_triangle_inequ} to deduce for any $x_1, x_2 \in X$
\[ d(x_1,x_2) = \sqrt{\Var (\delta_{x_1}, \delta_{x_2})} 
\geq \sqrt{\Var(\mu_1, \mu_2)} - \sqrt{\Var (\delta_{x_1}, \mu_1)} - \sqrt{\Var (\delta_{x_2}, \mu_2)}. \]
So for any coupling $q$ between $\mu_1, \mu_2$ we have
\begin{align*}
 \int_{X \times X} & d(x_1, x_2) \, dq(x_1, x_2) \\
&\geq \sqrt{\Var(\mu_1, \mu_2)} - \int_X \sqrt{\Var (\delta_{x_1}, \mu_1)} \, d\mu_1 (x_1) 
- \int_X \sqrt{\Var (\delta_{x_2}, \mu_2)} \, d\mu_2 (x_2) \\
&\geq  \sqrt{\Var(\mu_1, \mu_2)} - \bigg( \int_X \Var (\delta_{x_1}, \mu_1) \, d\mu_1 (x_1) \bigg)^{1/2} 
-  \bigg( \int_X \Var (\delta_{x_2}, \mu_2) \, d\mu_2 (x_2) \bigg)^{1/2} \\
&=  \sqrt{\Var(\mu_1, \mu_2)}- \sqrt{\Var(\mu_1)} - \sqrt{\Var(\mu_2)}. 
\end{align*}
This finishes the proof.
\end{proof}

We will also need:

\begin{Lemma} \label{Lem_weak_conv_2_W1_conv}
Suppose that $(X,d)$ is complete and separable and consider a sequence $\mu_i \in \PP(X)$ that weakly converges to some $\mu_\infty \in \PP(X)$ and satisfies $\Var (\mu_i) \leq C < \infty$.
Then $\mu_i \to \mu_\infty$ in $d_{W_1}$ and $\Var(\mu_\infty) \leq \liminf_{i \to \infty} \Var (\mu_i)$.
\end{Lemma}

\begin{proof}
See \cite[Theorem~7.12]{Villani_topics_in_OT}.
\end{proof}

\subsection{Metric measure spaces}
A triple $(X,d, \mu)$, consisting of a complete and separable metric space $(X,d)$ and a probability measure $\mu \in \mathcal{P} (X)$, is called a {\bf (normalized) metric measure space}.
If $\supp \mu = X$, then $(X,d, \mu)$ is said to have {\bf full support}.
If $\mu$ is only a measure on $X$, then $(X,d, \mu)$ is often called an (un-normalized) metric measure space.
In this paper we will only be interested in normalized metric measure spaces, and we will often drop the adjective ``normalized''.

A map $\phi : X_1 \to X_2$ between two metric measure spaces $(X_i,d_i, \mu_i)$, $i= 1,2$, is called an {\bf isometry (between metric measure spaces)} if it is a metric isometry between $(X_1, d_1)$ and $(X_2, d_2)$ and $\phi_* \mu_1 = \mu_2$.
If such an isometry exists, then  $(X_i,d_i, \mu_i)$, $i= 1,2$, are called {\bf isometric}.
We say that $(X_i,d_i, \mu_i)$, $i= 1,2$, have {\bf isometric support}, if the spaces restricted to $\supp \mu_i$, $i=1,2$ are isometric to each other as metric measure spaces.

\subsection{Distances between metric measure spaces}
We will frequently use the following distance notion between metric measure spaces.

\begin{Definition}
Consider two metric measure spaces $(X_1, \lb d_1, \lb \mu_1)$, $(X_2, d_2, \mu_2)$.
We define the {\bf Gromov-$W_p$-Wasserstein distance} for any $p \geq 1$ as
\[ d_{GW_p} \big( (X_1, d_1, \mu_1), (X_2, d_2, \mu_2) \big) := \inf d^{Z}_{W_p} ( (\varphi_1)_* \mu_1, ( \varphi_2)_* \mu_2 ), \]
where the infimum is taken over all isometric embeddings $\varphi_i : (X_i,d_i) \to (Z, d_Z)$ into some common metric space $(Z, d_Z)$.
\end{Definition}

This is a natural generalization the $\mathbf{D}$-distance from \cite{Sturm-2006-I}, to all exponents $p$.
See also \cite{Gigli-Mondino-Savare, Greven-Pfaffelhuber-Winter} for similar constructions.
In this paper we will mainly work with the Gromov-$W_1$-Wasserstein distance, as it is best suited for Ricci flows.

\begin{Proposition} \label{Prop_d_GW_p_pseudometric}
$d_{GW_p}$ satisfies all properties of a pseudometric that is allowed to attain the value $\infty$.
Moreover, $d_{GW_p}  ( (X_1, \lb d_1, \lb \mu_1),\lb (X_2, \lb d_2, \lb \mu_2) )=0$ if and only if $(X_1, \lb d_1, \lb \mu_1), (X_2, \lb d_2, \lb \mu_2)$ have isometric support.
\end{Proposition}

For the proof of Proposition~\ref{Prop_d_GW_p_pseudometric}, we will need the following lemma, which will also be useful throughout this paper.

\begin{Lemma}[Combining isometric embeddings] \label{Lem_combining_embeddings}
Let $2 \leq N \leq \infty$ and consider (possibly finite) sequences of metric spaces $(X_i, d_i)$ for $0 < i < N$, and $(Z_{i, i+1}, d_{i, i+1})$ for $0 \leq i < N$, as well as isometric embeddings $\varphi_{i,-} : X_{i} \to Z_{i,i+1}$ and $\varphi_{i, +} : X_{i} \to Z_{i-1,i}$ for $0 < i < N$.
Then there is a complete metric space $(Z, d_{Z})$ and sequences of isometric embeddings $\td\varphi_i : X_i \to Z$ for $0 < i < N$, $\psi_{i,i+1} : Z_{i,i+1} \to Z$ for $0 \leq i < N$, such that $\td\varphi_{i} = \psi_{i,i+1} \circ \varphi_{i+1,-}$ and $\td\varphi_{i+1} = \psi_{i,i+1} \circ \varphi_{i+1,+}$ for all $0 \leq i < N$.
\end{Lemma}

\begin{proof}
Suppose first that the lemma is true for $N = 2$.
Then by successive application of the lemma for $N=2$, we can construct metric spaces and isometric embeddings $Z_2 \to Z_3 \to \ldots$ such that for any $2 \leq N' < N+1$ the space $Z_{N'}$ allows isometric embeddings of the spaces $(X_1, d_1), \ldots, (X_{N'-1}, d_{N'-1})$ as described in the assertion of the lemma.
So the lemma holds for $2 \leq N < \infty$ and by taking a direct limit of the spaces $Z_{N'}$ it also holds for $N = \infty$.

It remains to verify the lemma for $N=2$.
Let
\[ Z' := (Z_{0,1} \sqcup Z_{1,2})/ \sim, \]
where we identify $\varphi_{1,+} (x) \sim \varphi_{1,-} (x)$ for all $x \in X_1$.
Let $\psi'_{0,1} : Z_{0,1} \to Z'$, $\psi'_{1,2} : Z_{1,2} \to Z'$ be the natural embeddings, which are injective, and set 
\[\td\varphi'_1 := \psi'_{0,1} \circ \varphi_{1,+} = \psi'_{1,2} \circ \varphi_{1,-}, . \]
Define $d_{Z'} : Z' \times Z' \to [0, \infty)$ such $\psi'_{0,1}, \psi'_{1,2}$ are isometric embeddings and such that for any $z_{0,1} \in Z_{0,1}$, $z_{1,2} \in Z_{1,2}$ we have
\[ d_{Z'} \big( \psi'_{0,1} (z_{0,1}), \psi'_{1,2} (z_{1,2} ) \big) := d_{Z'} \big(  \psi'_{1,2} (z_{1,2} ), \psi'_{0,1} (z_{0,1}) \big) := \inf_{x \in X_1} \big( d_{0,1} (z_{0,1}, \varphi_{1,+} (x)) + d_{1,2} (z_{1,2}, \varphi_{1,-} (x)) \big). \]
It can be checked easily that this definition is consistent and that $d_{Z'}$ is a pseudometric.
Let $(Z, d_{ Z})$ be the completion of the metric space that arises by identifying points of distance zero and let $\pi : Z' \to  Z$ be map induced by the natural projection.
Then $\td\varphi_1 := \pi \circ \td\varphi'_1$, $\psi_{0,1} := \pi \circ \psi'_{0,1}$, $\psi_{1,2} := \pi \circ \psi'_{1,2}$ satisfy the desired properties.
\end{proof}

\begin{proof}[Proof of Proposition~\ref{Prop_d_GW_p_pseudometric}.]
The proof is similar to \cite[Lemma~3.3, Theorem~3.6]{Sturm-2006-I}.
In order to verify the triangle inequality, consider three metric measure spaces $(X_i, d_i, \mu_i)$, $i =1,2,3$, and let $\eps > 0$.
Choose isometric embeddings $\varphi_{1} : X_1 \to Z_{1,2}$, $\varphi_{2,+} : X_2 \to Z_{1,2}$ into a metric space $(Z_{1,2}, d_{1,2})$, as well as isometric embeddings $\varphi_{2,-} : X_2 \to Z_{2,3}$, $\varphi_{3} : X_3 \to Z_{2,3}$ into a metric space $(Z_{2,3}, d_{1,2})$ such that
\begin{align*}
 d_{GW_p} \big( (X_1, d_1, \mu_1), (X_2, d_2, \mu_2) \big) &\geq d^{Z_{1,2}}_{W_p} ( (\varphi_1)_* \mu_1, ( \varphi_{2,+})_* \mu_2 ) - \eps, \\
 d_{GW_p} \big( (X_2, d_2, \mu_2), (X_3, d_3, \mu_3) \big) &\geq d^{Z_{2,3}}_{W_p} ( (\varphi_{2,-})_* \mu_2, ( \varphi_{3})_* \mu_3 ) - \eps, 
\end{align*}
By Lemma~\ref{Lem_combining_embeddings} we may assume that $Z_{1,2} = Z_{2,3} =: Z$ and $\varphi_{2,+} = \varphi_{2,-} =: \varphi_2$.
Then
\begin{align*}
  d_{GW_p} \big( (X_1, d_1, \mu_1), & (X_3, d_3, \mu_3) \big)
\leq  d^{Z}_{W_p} ( (\varphi_1)_* \mu_1, ( \varphi_{3})_* \mu_3 ) \\
&\leq d^{Z}_{W_p} ( (\varphi_1)_* \mu_1, ( \varphi_{2})_* \mu_2 ) + d^{Z}_{W_p} ( (\varphi_2)_* \mu_2, ( \varphi_{3})_* \mu_3 ) \\
&\leq  d_{GW_p} \big( (X_1, d_1, \mu_1), (X_2, d_2, \mu_2) \big)  + d_{GW_p} \big( (X_2, d_2, \mu_2), (X_3, d_3, \mu_3) \big)  + 2 \eps. 
\end{align*}
This shows that $d_{GW_p}$ is a pseudometric.

For the second statement, consider first a metric measure space $(X, d, \mu)$ and let $X' := \supp \mu$.
Taking $Z := X$ and considering the natural injections $X, X' \to Z$ allows us to conclude that 
\[ d_{GW_p} \big( (X, d, \mu), (\supp \mu, d|_{\supp \mu}, \mu |_{\supp \mu}) \big) = 0. \]
This proves one direction of the second statement.
The other direction is a consequence of the following lemma.
\end{proof}

\begin{Lemma} \label{Lem_couplings_converge_isometry}
Let $(X_i, d_i, \mu_i)$, $i=1,2$, be two metric measure spaces of full support and consider sequences of embeddings $\varphi^k_i : (X_i,d_i) \to (Z_k, d_{Z_k})$, $i=1,2$, $k = 1, 2, \ldots$, into metric spaces $(Z_k, d_{Z_k})$ and couplings $q_k$ between $\mu_1, \mu_2$ such that
\[ \int_{X_1 \times X_2} d^p_{Z_k} (\varphi^k_1(x_1), \varphi^k_2(x_2)) \, dq_k (x_1, x_2) \to 0. \]
Then, after passing to a subsequence, the couplings $q_k$ weakly converge to a coupling $q_\infty$ between $\mu_1, \mu_2$ of the form $q_\infty = (\id_{X_1}, \phi)_* \mu_1$, where $\phi : (X_1, d_1, \mu_1) \to (X_2, d_2, \mu_2)$ is an isometry.
Moreover, for any $x \in X_1$ we have $d_{Z_k} ( \varphi_1^k (x) , \varphi_2^k (\phi(x))) \to 0$.
\end{Lemma}

\begin{proof}
We first show that the sequence $q_k$ is tight.
Let $\eps > 0$ and choose compact subsets $K_{i,\eps} \subset X_i$ such that $\mu_i (X_i \setminus K_{i, \eps}) \leq \eps/2$, see Lemma~\ref{Lem_basic_measure}\ref{Lem_basic_measure_a}.
Then
\begin{multline*}
 q_k \big( X_1 \times X_2 \setminus K_{1,\eps} \times K_{2, \eps} \big)
\leq q_k \big( (X_1 \setminus K_{1, \eps}) \times X_2 \big) + q_k \big( X_1 \times (X_2 \setminus K_{2, \eps})  \big) \\
= \mu_1 (X_1 \setminus K_{1, \eps}) + \mu_2 (X_2 \setminus K_{2, \eps}) \leq \eps. 
\end{multline*}
So by Lemma~\ref{Lem_basic_measure}\ref{Lem_basic_measure_d} we may pass to a subsequence such that $q_k \to q_\infty \in \mathcal{P}(X_1 \times X_2)$ weakly, where $q_\infty$ is also a coupling between $\mu_1, \mu_2$.

Next, consider the sequence of functions
\[ f_k : X_1 \times X_2 \to [0, \infty), \qquad f_k(x_1, x_2) := d_{Z_k} (\varphi^k_1(x_1), \varphi^k_2(x_2)). \]
It follows from the assumption of the lemma, using H\"older's inequality, that $\int_{X_1 \times X_2} f_k \, dq_k \to 0$.
Since the functions $f_k$ are $2$-Lipschitz and since $q_k \to q_\infty$ weakly, we may apply Arzela-Ascoli, and conclude that, after passing to a subsequence, we have $f_k \to f_\infty$ pointwise, where $f_\infty : X_1 \times X_2 \to [0, \infty)$ is still $2$-Lipschitz.
Moreover, it follows that
\[ \int_{X_1 \times X_2} f_\infty \, dq_\infty = 0, \]
which implies that $\supp q_\infty \subset \{ f_\infty = 0 \}$.
By the triangle inequality and the definition of the functions $f_k$, we have for any $x_i, x'_i \in X_i$, $i = 1,2$,
\begin{equation} \label{eq_f_infty_triangle}
 | d_1 (x_1, x'_1) - d_2 (x_2, x'_2 ) | \leq f_\infty (x_1, x_2) + f_\infty (x'_1, x'_2). 
\end{equation}

It follows from (\ref{eq_f_infty_triangle}) that for any $x_1 \in X_1$ there is at most one $x_2 \in X_2$ with $f_\infty (x_1, x_2) = 0$.
Let $S \subset X_1$ be the set of points $x_1 \in X_1$ for which there is such an $x_2$ and define $\phi' : S \to X_2$ such that $f_\infty (x_1, \phi'(x_1)) = 0$.
Since $\supp \mu_1 = X_1$ and $\mu_1$ is the marginal of $q_\infty$, the set $S$ must be dense in $X_1$.
Due to (\ref{eq_f_infty_triangle}) $\phi'$ is an isometric embedding and therefore it admits a unique extension $\phi : X_1 \to X_2$, which is an isometric embedding.
For any $x_1 \in S$ we have
\[ d_{Z_k} ( \varphi_1^k (x_1) , \varphi_2^k (\phi(x_1))) = f_k (x_1, x_2) \to 0, \]
and since $\phi$ is an isometry, the same holds for all $x_1 \in X_1$.

It remains to show $q_\infty = (\id_{X_1}, \phi)_* \mu_1$, which will also imply that $\phi_* \mu_1 = \mu_2$ and that $\phi$ is surjective.
For this purpose choose a bounded $L$-Lipschitz function $h : X_1 \times X_2 \to \IR$, $|h| \leq A$,
and observe that by (\ref{eq_f_infty_triangle})
\[ L^{-1} \big| h(x_1, x_2) - h (x_1, \phi(x_1)) \big| \leq d_2 (x_2, \phi(x_1))
\leq f_k (x_1, x_2) + f_k (x_1, \phi(x_1)). \]
Therefore,
\begin{align*}
 \limsup_{k \to \infty} \int_{X_1 \times X_2}& \big| h(x_1, x_2) - h (x_1, \phi(x_1)) \big| dq_k (x_1,x_2) \\
&\leq L \limsup_{k \to \infty} \int_{X_1 \times X_2} \min \big\{ f_k (x_1, x_2) + f_k (x_1, \phi(x_1)), 2A \big\} dq_k (x_1, x_2) \\
&\leq L \limsup_{k \to \infty} \int_{X_1 } \min \{ f_k (x_1, \phi(x_1)) , 2A \} d\mu_1 (x_1) = 0. 
\end{align*}
It follows that
\begin{multline*}
 \int_{X_1 \times X_2} h \, dq_\infty 
 = \lim_{k \to \infty} \int_{X_1 \times X_2}  h (x_1, x_2) dq_k (x_1, x_2)
 = \lim_{k \to \infty} \int_{X_1 \times X_2}  h (x_1, \phi(x_1)) dq_k (x_1, x_2) \\
= \int_{X_1 } h (x_1, \phi(x_1)) d\mu_1 (x_1) = \int_{X_1 \times X_2} h \, d((\id_{X_1}, \phi)_* \mu_1). 
\end{multline*}
Due to Lemma~\ref{Lem_basic_measure}\ref{Lem_basic_measure_b}, this implies $q_\infty = (\id_{X_1}, \phi)_* \mu_1$, which finishes the proof.
\end{proof}

The property of having isometric support induces an equivalence relation on the space of all normalized metric measure spaces.
Denote by $\mathbb{M}$ the set of equivalence classes.
Equivalently, we could also define $\mathbb{M}$ to be the set of isometry classes of normalized metric measure spaces of full support.
We have

\begin{Theorem} \label{Thm_GW_p_complete}
$(\mathbb{M}, d_{GW_p})$ is a complete metric space if we allow the distance to attain $\infty$.
\end{Theorem}

\begin{proof}
The proof is similar to that of \cite[Proposition 5.6]{Greven-Pfaffelhuber-Winter}, \cite[Theorem~3.6]{Sturm-2006-I}.

The fact that $(\mathbb{M}, d_{GW_p})$ is a metric space follows from Proposition~\ref{Prop_d_GW_p_pseudometric}.
To prove completeness, consider a Cauchy sequence in $\mathbb{M}$, represented by a sequence of metric measure spaces $(X_i, d_i, \mu_i)$ of full support.
After passing to a subsequence, we may assume that
\[ d_{GW_p} \big( (X_i, d_i, \mu_i), (X_{i+1}, d_{i+1}, \mu_{i+1}) \big) \leq 2^{-i}. \]
Choose isometric embeddings $\varphi_{i,-} : X_i \to Z_{i,i+1}$, $\varphi_{i+1,+} : X_{i+1} \to Z_{i, i+1}$ into metric spaces $(Z_{i,i+1}, d_{i,i+1})$ such that
\[ d^{Z_{i,i+1}}_{W_p} ( (\varphi_{i,-} )_* \mu_i, (\varphi_{i+1,+})_* \mu_{i+1} ) \leq 2^{-i+1}. \]
By Lemma~\ref{Lem_combining_embeddings}, we may assume that $Z_{1,2} = Z_{2,3} = \ldots =: Z$ and $\varphi_{i,-} = \varphi_{i,+} =: \varphi_i$.
By passing to the completion of $\bigcup_{i=1}^\infty \varphi_i (X_i)$, we may moreover assume that $(Z, d_Z)$ is complete and separable.
We have
\[ d^{Z}_{W_p} ( (\varphi_{i} )_* \mu_i, (\varphi_{i+1})_* \mu_{i+1} ) \leq 2^{-i+1}, \]
so $(\varphi_i)_* \mu_i \to \mu'_\infty \in \mathcal{P} (Z)$ in $W_p$.
Let $X_\infty := \supp \mu'_\infty$, $d_\infty := d_Z |_{X_\infty}$ and $\mu_\infty := \mu'_\infty |_{X_\infty}$.
Then
\[ d_{GW_p}  \big( (X_i, d_i, \mu_i), (X_{\infty}, d_{\infty}, \mu_{\infty}) \big) 
\leq d^{Z}_{W_p} ( (\varphi_{i} )_* \mu_i, \mu_\infty ) \leq 2^{-i+1}, \]
which implies that $(X_i, d_i, \mu_i)$ converges to $(X_{\infty}, d_{\infty}, \mu_{\infty})$ in $GW_p$.
\end{proof}

Since the Prokhorov distance is bounded by the $W_1$-Wasserstein distance, $GW_1$-convergence implies convergence in the Gromov-Prokhorov sense or the pointed measured Gromov sense \cite[Theorem~3.15]{Gigli-Mondino-Savare}.
Note however, that, in general, $GW_p$-convergence does not imply Gromov-Hausdorff convergence, even if we assume that all spaces in question have full support.
Consider for example the sequence
\[ \big( X_n := \ov{B}(0,1) \subset \IR^n, d_n := d_{\IR^n} |_{X_n}, \mu_n := ( 1 - n^{-1}) \delta_0 + n^{-1} \omega_n^{-1} \mu_{\IR^n} |_{X_n} \big), \]
where $d_{\IR^n}$ and $\mu_{\IR^n}$ denote the standard Euclidean distance and volume measure and $\omega_n := \mu_{\IR^n} ( B(0,1))$.
As $n \to \infty$ this sequence converges to a single point in the $GW_p$-sense, but the corresponding metric spaces $(X_n, d_n)$ don't converge in the Gromov-Hausdorff sense.

\subsection{Compactness} \label{subsec_compactness}
In this subsection, we define useful compact subsets of $(\mathbb{M}, d_{GW_1})$.
For this purpose, we make the following definition, which is similar to \cite[(6.4)]{Greven-Pfaffelhuber-Winter}

\begin{Definition}
We define the {\bf mass distribution function at scale $r > 0$,} $b^{(X,d, \mu)}_r : (0,1] \to (0,1]$, of a metric measure space $(X, d, \mu)$ by
\begin{equation} \label{eq_def_mass_distr}
 b^{(X,d,\mu)}_r (\eps) := \sup \big\{ \delta > 0 \;\; : \;\; \mu ( \{ x \in X \;\; : \;\; \mu (D(x,\eps r)) < \delta \} ) \leq \eps \big\}. 
\end{equation}
Here $D(x,\eps r) := \{  d(x,\cdot) \leq \eps r \}$ denotes the closed ball around $x$.
\end{Definition}

Note that $b^{(X,d,\mu)}_r (\eps) \in (0,1]$, because by Lemma~\ref{Lem_basic_measure}\ref{Lem_basic_measure_c}
\[ \lim_{\delta \to 0} \mu ( \{ x \in X \;\; : \;\; \mu (D(x,\eps r)) < \delta \} )  \leq \mu ( X \setminus \supp \mu) = 0. \]
Moreover, we have:

\begin{Lemma}
$b^{(X,d,\mu)}_r (\eps)$ is non-decreasing and right semi-continuous.
The supremum in (\ref{eq_def_mass_distr}) is attained and for any function $b : (0,1] \to (0,1]$ the condition $b^{(X,d,\mu)}_r  \geq b$ is equivalent to
\[ \mu ( \{ x \in X \;\; : \;\; \mu (D(x,\eps r)) < b(\eps) \} ) \leq \eps \qquad \text{for all} \quad \eps \in (0,1] \]
\end{Lemma}

\begin{proof}
For the monotonicity statement, note that for any $\eps_1 \leq \eps_2$ and any $\delta > 0$ with $\mu ( \{ x \in X \;\; : \;\; \mu (D(x,\eps_1 r)) < \delta \} ) \leq \eps_1$ we have
\[ \mu ( \{ x \in X \;\; : \;\; \mu (D(x,\eps_2 r)) < \delta \} ) 
\leq \mu ( \{ x \in X \;\; : \;\; \mu (D(x,\eps_1 r)) < \delta \} ) \leq \eps_1. \]
For the right semi-continuity, observe that by the continuity of measures for any $\delta > 0$ with the property that $\mu ( \{ x \in X \;\; : \;\; \mu (D(x,\eps' r)) < \delta \} ) \leq \eps$ for all $\eps' > \eps$ we have
\begin{multline*}
 \mu ( \{ x \in X \;\; : \;\; \mu (D(x,\eps r)) < \delta \} )
= \mu ( \{ x \in X \;\; : \;\; \lim_{\eps' \searrow \eps} \mu (D(x,\eps' r)) < \delta \} ) \\
=\lim_{\eps' \searrow \eps}  \mu ( \{ x \in X \;\; : \;\; \mu (D(x,\eps' r)) < \delta \} ) 
\leq \eps . 
\end{multline*}
Similarly, for $b_0 := b^{(X,d,\mu)}_r (\eps)$
\[  \mu ( \{ x \in X \;\; : \;\; \mu (D(x,\eps r)) < b_0 \} )
= \lim_{\delta \nearrow b_0} \mu ( \{ x \in X \;\; : \;\; \mu (D(x,\eps r)) < \delta \} ) \leq \eps, \]
which implies that the supremum in (\ref{eq_def_mass_distr}) is attained and that if $b^{(X,d,\mu)}_r (\eps) \geq b (\eps)$, then
\[ \mu ( \{ x \in X \;\; : \;\; \mu (D(x,\eps r)) < b(\eps) \} ) \leq  \mu ( \{ x \in X \;\; : \;\; \mu (D(x,\eps r)) < b_0 \} ) \leq \eps. \qedhere \]
\end{proof}
\medskip

We can now define a class of metric measure spaces, which will turn out to be compact.

\begin{Definition} \label{Def_MMVb}
For any $r, V > 0$ and any function $b : (0,1] \to (0,1]$, let $\mathbb{M}_r (V, b) \subset \mathbb{M}$ be the set of isometry classes of metric measure spaces $(X,d,\mu)$ of full support that satisfy the following properties:
\begin{enumerate}[label=(\arabic*)]
\item \label{Def_MMVb_1} $\Var (\mu) \leq V r^2$.
\item \label{Def_MMVb_2} $b^{(X,d,\mu)}_r \geq b$.
\end{enumerate}
\end{Definition}

Property~\ref{Def_MMVb_1} is a generalization of a diameter bound and Property~\ref{Def_MMVb_2} will turn out to be necessary since we don't impose any doubling condition.

\begin{Lemma} \label{Lem_MVb_closed}
$\mathbb{M}_r (V, b)$ is closed in $(\mathbb{M}, d_{GW_1})$.
\end{Lemma}

\begin{proof}
Consider a sequence of metric measure spaces of full support $(X_i, d_i, \mu_i)$ representing classes in $\mathbb{M}_r (V, b)$ for some fixed $r, V > 0$, $b : (0,1] \to (0,1]$ and suppose that $(X_i, d_i, \mu_i) \to (X_\infty, d_\infty, \mu_\infty)$ in $GW_1$.
Our goal is to show that the limit $(X_\infty, d_\infty, \mu_\infty)$ also represents a class in $\mathbb{M}_r (V, b)$.
As in the proof of Theorem~\ref{Thm_GW_p_complete}, we may pass to a subsequence and find isometric embeddings $\varphi_i : X_i \to Z$, $i = 1,2, \ldots, \infty$, into a complete and separable metric space $(Z, d_Z)$ such that $(\varphi_i)_* \mu_i \to (\varphi_\infty)_* \mu_\infty$ in $W_1$.
This reduces the lemma to the following lemma.
\end{proof}

\begin{Lemma}
Consider a complete and separable metric space $(X,d)$ and consider probability measures $\mu_i \in \mathcal{P}(X)$, $i=1,2, \ldots, \infty$, with $\mu_i \to \mu_\infty$ in $W_1$.
Then the following holds:
\begin{enumerate}[label=(\alph*)]
\item $\Var (\mu_\infty) \leq \liminf_{i \to \infty} \Var (\mu_i)$.
\item For any $\eps \in (0,1]$, $r > 0$, we have $b^{(X,d,\mu_\infty)}_r (\eps) \geq \limsup_{i \to \infty} b^{(X,d,\mu_i)}_r (\eps)$.
\end{enumerate}
\end{Lemma}

\begin{proof}
Assertion~(a) is clear.
For Assertion~(b) fix some $\eps \in (0,1]$, $r > 0$ and  suppose that the assertion was false.
Then we can find a $b > 0$ such that after passing to a subsequence we have for all $i$
\begin{equation} \label{eq_choice_b_i}
 \mu_i ( \{ x \in X \;\; : \;\; \mu_i (D(x, \eps r)) < b \} ) \leq \eps, 
\end{equation}
but
\begin{equation} \label{eq_choice_b_infty}
 \mu_\infty ( \{ x \in X \;\; : \;\; \mu_\infty (D(x, \eps r)) < b \} ) > \eps. 
\end{equation}
Since
\begin{multline*}
 \lim_{r' \searrow \eps r} \mu_\infty ( \{ x \in X \;\; : \;\; \mu_\infty (D(x, r')) < b \} )
= \mu_\infty ( \{ x \in X \;\; : \;\; \lim_{r' \searrow \eps r}  \mu_\infty (D(x, r')) < b \} ) \\
= \mu_\infty ( \{ x \in X \;\; : \;\; \mu_\infty (D(x, \eps r)) < b \} ) > \eps,
\end{multline*}
we can choose $r' > \eps r$ such that
\[ \mu_\infty ( \{ x \in X \;\; : \;\; \mu_\infty (D(x, r')) < b \} ) > \eps. \]
Similarly, since
\[ \lim_{b' \nearrow b} \mu_\infty ( \{ x \in X \;\; : \;\; \mu_\infty (D(x, r')) < b' \} )
= \mu_\infty ( \{ x \in X \;\; : \;\; \mu_\infty (D(x, r')) < b \} ) > \eps, \]
we can choose $b' < b$ such that
\begin{equation} \label{eq_choice_b_p} 
\mu_\infty ( \{ x \in X \;\; : \;\; \mu_\infty (D(x, r')) < b' \} ) > \eps. 
\end{equation}

Next, we claim that for large $i$
\begin{equation} \label{eq_S_subset_i}
 \{ x \in X \;\; : \;\; \mu_\infty (D(x, r')) < b' \} \subset \{ x \in X \;\; : \;\; \mu_i (D(x, \eps r)) < b \}. 
\end{equation}
To see this, let $\alpha > 0$ be some small constant whose value we will determine later and choose $i$ large enough such that we can find a coupling $q_i$ between $\mu_i, \mu_\infty$ with
\[ \int_{X \times X} d(x,y) \, dq_i (x,y) \leq \alpha. \]
Suppose that for some $x \in X$ we have $\mu_\infty (D(x, r')) < b'$, but $\mu_i (D(x, \eps r)) \geq b$.
Then
\begin{multline*}
 q_i \big( D(x, \eps r) \times (X \setminus D(x,r')) \big)
\geq q_i \big( D(x, \eps r) \times (X \setminus D(x,r')) \big) -  q_i \big( (X \setminus D(x, \eps r)) \times  D(x,r') \big) \\
= q_i ( D(x, \eps r) \times X ) - q_i (X \times D(x,r'))
= \mu_i (D(x, \eps r)) - \mu_\infty ( D(x,r'))
> b - b'. 
\end{multline*}
and therefore
\[ 0 <  (r' - \eps r)( b - b' ) \leq \int_{D(x, \eps r) \times (X \setminus D(x,r'))} d (x, y) \, dq_i (x,y) \leq \alpha. \]
So if $\alpha < (r' - \eps r) (b-b')$, then we obtain the desired contradiction, which proves (\ref{eq_S_subset_i}).

Next note that $S := \{ x \in X \;\; : \;\; \mu_\infty (D(x, r')) < b' \}$ is open.
So for any $A < \infty$ the function $f_A : X \to \IR$ defined by $f_A (x) := \min \{ A d(x, X \setminus S), 1 \}$ is $A$-Lipschitz and $\{ f_A > 0 \} = S$.
It follows that
\[ \liminf_{i \to \infty} \mu_i(S) \geq \liminf_{i \to \infty} \int_X f_A \, d\mu_i
= \int_X f_A \, d\mu_\infty. \]
Letting $A \to \infty$ implies
\[  \liminf_{i \to \infty} \mu_i (S) \geq \mu_\infty (S). \]
Combining this with (\ref{eq_choice_b_i}), (\ref{eq_S_subset_i}), (\ref{eq_choice_b_p}) implies that
\[ \eps \geq \liminf_{i \to \infty} \mu_i ( \{ x \in X \;\; : \;\; \mu_i (D(x, \eps r)) < b \}) \geq \liminf_{i \to \infty} \mu_i (S) \geq \mu_\infty (S) > \eps, \]
which produces the desired contradiction.
\end{proof}

The following theorem will be important throughout this paper.
Compare also with \cite[Proposition 7.1]{Greven-Pfaffelhuber-Winter}, \cite[Theorem~3.16]{Sturm-2006-I}

\begin{Theorem} \label{Thm_M_compact}
$(\mathbb{M}_r (V, b), d_{GW_1})$ is compact.
\end{Theorem}

\begin{proof}
Due to Theorem~\ref{Thm_GW_p_complete} and Lemma~\ref{Lem_MVb_closed}, we only need to establish total boundedness.
This is a consequence of the following lemma.
\end{proof}

\begin{Lemma} \label{Lem_approx_mms_by_finite_mms}
For any $r, V, \alpha > 0$, $b : (0,1] \to (0,1]$ there is an $N (r, V, b, \alpha) < \infty$ such that for any metric measure space $(X, d, \mu)$ representing an isometry class in $\mathbb{M}_r (V, b)$ there is a finite subset $X' \subset \supp X$ and a measure $\mu' \in \PP (X)$ with $\supp \mu' \subset X'$ such that
\[ d_{GW_1} \big( (X, d, \mu), (X', d|_{X'}, \mu') \big) \leq d_{W_1} (  \mu, \mu') \leq \alpha r \]
and such that $(X',d|_{X'})$ has diameter $\leq N r$, $\# X' \leq N$ and $\mu' (\{ x' \})$ is a multiple of $N^{-1}$ for all $x' \in X'$.
\end{Lemma}

\begin{proof}
After rescaling, we may assume without loss of generality that $r = 1$.
Fix $\alpha, V > 0$, $b :(0,1] \to (0,1]$ and let $\eps > 0$, $N < \infty$ be constants whose values we will determine later.
Consider a metric measure space $(X, d, \mu)$ representing an isometry class in $\mathbb{M}_1 (V, b)$.
Choose a maximal set of points $\{ x_1, \ldots, x_m \} \subset X$ with the property that the closed balls $D(x_i, \eps)$ are pairwise disjoint and
\[ \mu (D(x_i, \eps)) \geq b(\eps). \]
Then $m \leq (b(\eps))^{-1}$.
Choose moreover $x_0 \in X$ such that $\Var (\delta_{x_0}, \mu) \leq \Var (\mu) \leq V$.
Set $X' := \{ x_0, x_1, \ldots, x_m \}$.

Consider the subset
\[ Y := \bigcup_{i=1}^m D(x_i, 3\eps). \]
Then for any $x \in X \setminus Y$ we have $\mu (D(x, \eps)) < b(\eps)$.
It follows that
\[ \mu (X \setminus Y) \leq \eps. \]
For any $i,j = 1, \ldots, m$, $i \neq j$, we have
\[ d(x_i, x_j) - 2 \eps \leq ( b (\eps) )^{-2} \int_{D(x_i, \eps) \times D(x_j, \eps)} d(y_i, y_j) \, d\mu(y_i) d\mu(\mu_j) \leq (b (\eps) )^{-2} V^{1/2}, \]
\[ d(x_0, x_i) - \eps \leq (b(\eps))^{-1} \int_{D(x_i, \eps)} \, d(x_0, y_i) d\mu(y_i)  \leq (b(\eps))^{-1} V^{1/2}, \]
which implies that $X'$ has diameter $\leq (b (\eps))^{-2}  V^{1/2} + 2 \eps$.

Next, define
\[ Y_i := D(x_i, 3\eps) \setminus \bigcup_{j=1}^{i-1} D(x_j, 3\eps). \]
Note that
\[ Y = Y_1 \,\, \dot\cup \,\, \ldots \,\, \dot\cup \,\, Y_m. \]
Set
\[ a_i := \begin{cases} \mu (X \setminus Y) & \text{if $i=0$} \\
\mu (Y_i) & \text{if $1 \leq i \leq m$}  \end{cases}. \]
Then $a_0 + \ldots + a_m = 1$.
Choose numbers $b_0, \ldots , b_m  \in [0,1]$ that are multiples of $N^{-1}$ and satisfy $|a_i - b_i | \leq N^{-1}$ and $b_0 + \ldots + b_m = 1$.
We now define
\[ \mu' := b_0 \delta_{x_0} + \ldots + b_m \delta_{x_m} \]
and
\[ \mu'' := a_0 \delta_{x_0} + \ldots + a_m \delta_{x_m}. \]
We have
\begin{multline*}
 d_{GW_1} ( (X, d, \mu), (X', d|_{X'}, \mu'))
\leq d_{W_1} (\mu, \mu')
 \leq d_{W_1} (\mu, \mu'') + d_{W_1} (\mu'', \mu') \\
 \leq d_{W_1} (\mu, \mu'') + ((b (\eps) )^{-2} V + 2 \eps) N^{-1}. 
\end{multline*}
 The last term can be made $\leq \alpha / 2$ if $N \geq \underline{N} (V,b, \eps, \alpha)$.

It remains to derive a bound $d_{W_1} (\mu, \mu'')$.
For this purpose, consider the following coupling $q$ between $\mu, \mu''$:
\[ q :=\mu |_{X \setminus Y} \otimes \delta_{x_0} +  \sum_{i=1}^{m}  \mu |_{Y_i} \otimes \delta_{x_i} . \]
Then
\begin{multline*}
 \int_{X \times X} d(y,z) \, dq(y,z)
= \int_{X \setminus Y} d(y, x_0) \, d\mu(y)  + \sum_{i=1}^m \int_{Y_i}  d(y, x_i) \, d\mu(y)  \\
\leq \mu^{1/2} (X \setminus Y) \bigg(  \int_{X \setminus Y} d^2(y, x_0)\, d\mu(y)\bigg)^{1/2} + 3 \eps 
\leq \eps^{1/2} V^{1/2} + 3 \eps. 
\end{multline*}
So if $\eps \leq \ov\eps ( V, \alpha)$, then $d_{W_1} (\mu, \mu'') \leq \alpha /2$, which finishes the proof.
\end{proof}
\bigskip

\section{Metric Flows} \label{sec_met_flows}
In this section we introduce the notion of a metric flow, which is a synthetic version of a (super) Ricci flow, as well as associated terminology.
We will discuss some basic properties of metric flows and present some examples and basic constructions.
We will also explain how to convert super Ricci flows and singular Ricci flows into metric flows.

For the remainder of this paper, we will denote by $\Phi : \IR \to (0,1)$ the antiderivative with the following properties:
\begin{equation} \label{eq_erf}
 \Phi'(x) = (4\pi)^{-1/2} e^{-x^2/4}, \qquad \lim_{x \to -\infty} \Phi (x) = 0 \qquad \lim_{x \to \infty} \Phi (x) = 1. 
\end{equation}
We recall that $(x,t) \mapsto \Phi ( t^{-1/2} x)$ is a solution to the 1-dimensional heat equation with initial condition $\chi_{[0, \infty)}$.

\subsection{Definition of a metric flow}
Let us first state the definition of a metric flow:

\begin{Definition}[Metric flow] \label{Def_metric_flow}
Let $I \subset \IR$ be a subset.
A {\bf metric flow (over $I$)} is a tuple of the form 
\begin{equation} \label{eq_XX_tuple}
 (\XX , \tf, (d_t)_{t \in I} , (\nu_{x;s})_{x \in \XX, s \in I, s \leq \tf (x)}) 
\end{equation}
with the following properties:
\begin{enumerate}[label=(\arabic*)]
\item \label{Def_metric_flow_1} $\XX$ is a set consisting of {\bf points}.
\item \label{Def_metric_flow_2} $\tf : \XX \to I$ is a map called {\bf time-function}.
Its level sets $\XX_t := \tf^{-1} (t)$ are called {\bf time-slices} and the preimages $\XX_{I'} := \tf^{-1} (I')$, $I' \subset I$, are called {\bf time-slabs}.
\item \label{Def_metric_flow_3} $(\XX_t, d_t)$ is a complete and separable metric space for all $t \in I$.
\item \label{Def_metric_flow_4} $\nu_{x;s} \in \mathcal{P} (\XX_s)$ for all $x \in \XX$, $s \in I$, $s \leq \tf (x)$.
For any $x \in \XX$ the family $(\nu_{x;s})_{s \in I, s \leq \tf (x)}$ is called the {\bf conjugate heat kernel at $x$}.
\item \label{Def_metric_flow_5} $\nu_{x; \tf (x)} = \delta_x$ for all $x \in \XX$. 
\item \label{Def_metric_flow_6} For all $s, t \in I$, $s< t$, $T \geq 0$ and any measurable function $u_s : \XX_s \to [0,1]$ with the property that if $T > 0$, then $u_s = \Phi \circ f_s$ for some  $T^{-1/2}$-Lipschitz function $f_s : \XX_s \to \IR$ (if $T=0$, then there is no additional assumption on $u_s$), the following is true. 
The function
\begin{equation} \label{eq_heat_flow_Phi_f_Def}
 u_t :\XX_t \longrightarrow \IR, \qquad x \longmapsto  \int_{\XX_s} u_s \, d\nu_{x;s} 
\end{equation}
is either constant or of the form $u_t = \Phi \circ f_t$, where $f_t : \XX_t \to \IR$ is $(t-s+T)^{-1/2}$-Lipschitz.
\item \label{Def_metric_flow_7} For any $t_1,t_2,t_3 \in I$, $t_1 \leq t_2 \leq t_3$, $x \in \XX_{t_3}$  we have the {\bf reproduction formula}
\[ \nu_{x; t_1} = \int_{\XX_{t_2}} \nu_{\cdot; t_1} d\nu_{x; t_2}, \]
meaning that for any Borel set $S\subset \XX_{t_1}$
\[ \nu_{x;t_1} (S)  = \int_{\XX_{t_2}} \nu_{y ; t_1} (S) d\nu_{x; t_2}(y). \]
Note that by Properties~\ref{Def_metric_flow_5}, \ref{Def_metric_flow_6} the integrand is continuous if $t_1 < t_2$ and measurable if $t_1 = t_2$.
\end{enumerate}
\end{Definition}

We will often write $\XX$ instead of (\ref{eq_XX_tuple}).
We will also frequently be dealing with a number of different metric flows at once, which will be denoted by $\XX^i, \XX', \XX^*$ etc.
In this case the objects $d_t, \nu_{x;s}$ will inherit the decorations.
So, for example, $d^i_t, \nu^i_{x;s}$ will denote the objects associated with a metric flow denoted by $\XX^i$.
We will often omit decorations on the time-function $\tf$, as there is no chance of confusion.
We will frequently also use the following shorthand notations for time slabs: $\XX_{<t} := \XX_{I \cap (-\infty,t)}$, $\XX_{\leq t} := \XX_{I \cap (-\infty,t]}$, etc.

\begin{Remark}
We don't require that the time-slices $(\XX_t, d_t)$ are length spaces.
For more details see Section~\ref{sec_intrinsic}.
\end{Remark}

\begin{Remark}
We don't require that $I$ is an interval, although this case will be of most interest to us.
We have kept Definition~\ref{Def_metric_flow} more general, as it gives us some more flexibility later.
For example, it allows us to restrict metric flows defined over intervals $I \subset \IR$ to smaller subsets $I' \subset I$.
It will also be helpful to construct certain metric flows first over a countable dense subset $I' \subset I \subset \IR$ and then pass to the future completion, which is defined over $I$; see Subsection~\ref{subsec_future_completion}.
\end{Remark}

\begin{Remark}
In Subsection~\ref{subsec_superRF_met_flow} we will see that every super Ricci flow $(g_t)_{t \in I}$ on a compact manifold $M$ and over some time-interval $I$ gives rise to a metric flow of the form $\XX = M \times I$.
The metric $d_t$ equals the length metric of $g_t$ and the conjugate heat kernels $(\nu_{x,t;s})_{s \in I, s \leq \tf (x)}$ equal the measures $K(x, t; \cdot, s) dg_s$ associated to the conjugate heat kernel at $(x,t)$.
\end{Remark}

Due to Lemma~\ref{Lem_basic_measure}\ref{Lem_basic_measure_b} the case $T=0$ in Definition~\ref{Def_metric_flow}\ref{Def_metric_flow_6} follows from $T > 0$ by a limit argument.

\begin{Lemma} \label{Lem_met_flow_T_positive}
In Definition~\ref{Def_metric_flow}\ref{Def_metric_flow_6}, we may assume that $T > 0$ and that $u$ takes values in $(0,1)$.
In this case, we may omit the option that $u_t$ is constant.
\end{Lemma}

The next lemma states that Definition~\ref{Def_metric_flow} is invariant under {\bf parabolic rescaling} by some $\lambda > 0$ and a {\bf time-shift} by some $t_0 \in \IR$.

\begin{Lemma}
Let $\lambda > 0$, $t_0 \in \IR$.
If (\ref{eq_XX_tuple}) is a metric flow, then so is
\[ (\XX , \la^2 \tf + t_0, (d_{\la^2 \tf + t_0})_{t \in I} , (\nu_{x; \la^2 s + t_0})_{x \in \XX, s \in I, s \leq \tf (x)}) . \]
\end{Lemma}

Next, we define what we mean by a restriction of a metric flow to a subset of times $I' \subset I$.

\begin{Definition}[Restriction of a metric flow]
If $\XX$ is a metric flow over $I \subset \IR$ and $I' \subset I$, then the {\bf restriction of $\XX$ to $I'$} is given by
\begin{equation} \label{eq_restriction_m_flow}
  \big( \XX_{I'} , \tf |_{\XX_{I'} }, (d_t)_{t \in I'} , (\nu_{x;s})_{x \in \XX_{I'} , s \in I', s \leq \tf (x)} \big).  
\end{equation}
We will often write $\XX_{I'}$ instead of (\ref{eq_restriction_m_flow}).
\end{Definition}

Lastly, we consider maps between metric flows.
We introduce the following convention.
If $\XX$ is a metric flow, $U \subset \XX$ and $\phi : U \to Y$ is some map, then we define $\phi_t := \phi |_{U \cap \XX_t} : U_t := U \cap \XX_t \to Y$.
Let $\XX^i$ be two metric flows over $I^i \subset \IR$, $i =1,2$, $U \subset \XX^1$ and consider a map $\phi : U \to \XX^2$.

\begin{Definition}
We say that $\phi$ is:
\begin{enumerate}
\item {\bf time-preserving} if $\tf(\phi(x)) = \tf (x)$ for all $x \in U$,
\item {\bf $a$-time-equivariant} if there is some $t_0 \in \IR$ such that $\tf (\phi(x)) = a \tf(x) + t_0$ for all $x \in U$,
\item {\bf time-slice-preserving} if for every $t_1 \in I^1$ there is some $t_2 \in I^2$ such that for any $x \in U \cap \XX^1_t$ we have $\tf(\phi(x)) \in \XX^2_t$.
\end{enumerate}
\end{Definition}

If $\phi$ is time-preserving, then we will often express it as a family of maps $(\phi_t : U_t := U \cap \XX^1_t \to \XX^2_t)_{t \in I^1}$.
We now define the notion of an isometry between metric flows.

\begin{Definition}[Isometry between metric flows]
Consider two metric flows $\XX^i$ over $I \subset \IR$, $i = 1,2$.
A map $\phi :  \XX^1 \to  \XX^2$ given by a family of maps $(\phi_t := \phi |_{\XX^1_t} : \XX^1_t \to \XX^2_t)_{t \in I}$ is called a {\bf flow isometry over $I$} if:
\begin{enumerate}
\item $\phi_t : (\XX^1_t, d^1_t) \to (\XX^2_t, d^2_t)$ is a metric isometry for all $t \in I$.
\item $(\phi_s)_* \nu^1_{x; s} = \nu^2_{\phi (x); s}$ for all $x \in \XX^1$, $s \in I$, $s \leq \tf (x)$.
\end{enumerate}
If $\XX^i$ are metric flows over $I^i \subset \IR$ and $I' \subset I^1 \cap I^2$, then a flow isometry $\phi : \XX^1_{I'} \to \XX^2_{I'}$ is called a {\bf  flow isometry between $\XX^1, \XX^2$ over $I'$}.
Moreover, if $I^1 \setminus I'$ and $I^2 \setminus I'$ are sets of measure zero, then a flow isometry between $\XX^1, \XX^2$ over $I'$ is called an {\bf almost everywhere flow isometry between $\XX^1, \XX^2$}.
If $I^1 \subset I^2$ and $\phi$ is a flow isometry between $\XX^1, \XX^2_{I^1}$, then we also call $\phi$ a {\bf flow isometric embedding.}
\end{Definition}

\subsection{(Conjugate) Heat flows on a metric flow}
We will now define the analog of solutions to the (forward) heat equation and the (backward) conjugate heat equation on a super Ricci flow background.
For this purpose let $\XX$ be a metric flow defined over some $I \subset \IR$ and let $I' \subset I$ be some subset.

\begin{Definition}[Heat flow] \label{Def_heat_flow}
A function $u : \XX_{I'} \to \IR$, often expressed as a family of functions $(u_t : \XX_t \to \IR)_{t \in I'}$, is called a {\bf heat flow} if for all $x \in \XX$, $s \in I'$, $s \leq \tf (x)$ the function $u_s$ is integrable with respect to $d\nu_{x;s}$ and
\begin{equation} \label{eq_Def_heat_flow}
 u_t (x) = \int_{\XX_s} u_s \, d\nu_{x;s}. 
\end{equation}
\end{Definition}

\begin{Remark}
If $u_s$ is bounded for some $s \in I'$, then by Definition~\ref{Def_metric_flow}\ref{Def_metric_flow_6}, the functions $u_{s'}$ for $s' > s$ are automatically bounded and continuous.
So the function $u_{s'}$ is automatically integrable with respect to $d\nu_{x;s'}$  if $s' > s$.
\end{Remark}

We have the following forward existence and uniqueness result:

\begin{Proposition} \label{Prop_existence_heat_flow}
Assume that $t_0 := \inf I' \in I'$ and consider a bounded measurable function $\td{u}: \XX_{t_0} \to \IR$.
Then there is a unique heat flow $(u_t)_{t \in I'}$ with $u_{t_0} = \td u$.
\end{Proposition}

\begin{proof}
Define
\[ u_t (x) := \int_{\XX_{t_0}} \td u \, d\nu_{x;t_0}. \]
Then (\ref{eq_Def_heat_flow}) follows using the reproduction formula, Definition~\ref{Def_metric_flow}\ref{Def_metric_flow_7}.
\end{proof}

The next result summarizes basic properties of heat flows.

\begin{Proposition} \label{Prop_basic_HF}
If $(u_t)_{t \in I'}$ is a heat flow on $\XX$ and $s < t$, $s, t \in I'$, then the following holds: 
\begin{enumerate}[label=(\alph*)]
\item \label{Prop_basic_HF_a} Any linear combination of finitely many heat flows is again a heat flow.
\item \label{Prop_basic_HF_b} If $u_s \leq a$ for some $a \in \IR$, then $u_t \leq a$ with equality at some point $x \in \XX_t$ if and only if $u_s \equiv a$ on $\supp \nu_{x;s}$.
\item \label{Prop_basic_HF_c} If $u_s \geq a$ for some $a \in \IR$, then $u_t \geq a$ with equality at some point $x \in \XX_t$ if and only if $u_s \equiv a$ on $\supp \nu_{x;s}$.
\end{enumerate}
\end{Proposition}

\begin{proof}
Assertions~\ref{Prop_basic_HF_a}--\ref{Prop_basic_HF_c} are direct consequences of Definitions~\ref{Def_metric_flow}, \ref{Def_heat_flow}.
\end{proof}

We also have the following gradient-type estimates:

\begin{Proposition} \label{Prop_gradient_HF}
If $(u_t)_{t \in I'}$ is a heat flow on $\XX$ and $s < t$, $s, t \in I'$, $T \geq 0$, then the following holds:
\begin{enumerate}[label=(\alph*)]
\item \label{Prop_gradient_HF_a} 
Assume that $T > 0$ and that $u_s = a  (\Phi \circ f_s)$ for some $a \in \IR$ and some $T^{-1/2}$-Lipschitz function $f_s : \XX_s \to \IR$, or that $T=0$ and  $0 \leq u_s \leq 1$.
Then $u_t = a ( \Phi \circ f_t)$ for some $(t-s+T)^{-1}$-Lipschitz function $f_t : \XX_t \to \IR$.
\item \label{Prop_gradient_HF_b} If $u_s$ is $L$-Lipschitz for some $L \geq 0$, then so is $u_t$.
\end{enumerate}
\end{Proposition}

\begin{proof}
Assertion~\ref{Prop_gradient_HF_a} follows from Definitions~\ref{Def_metric_flow}\ref{Def_metric_flow_6}.
Assertion~\ref{Prop_gradient_HF_b} can be reduced to the case in which $u_s$ is bounded.
Then Assertion~\ref{Prop_gradient_HF_b} follows by applying Assertion~\ref{Prop_gradient_HF_a} to $\frac12 + \eps u$ and letting $\eps \to 0$.
\end{proof}

Next we define the equivalent notion of a solution to the conjugate heat equation, which will concern probability measures:

\begin{Definition}[Conjugate heat flow] \label{Def_conj_heat_flow}
A family of probability measures $(\mu_t \in \mathcal{P} (\XX_t))_{t \in I'}$ is called a {\bf conjugate heat flow} if for all $s, t \in I'$, $s \leq t$ we have
\begin{equation} \label{eq_Def_conj_heat_flow}
 \mu_s  = \int_{\XX_t} \nu_{x;s} \, d\mu_t (x). 
\end{equation}
\end{Definition}

Similarly as before we obtain the following backwards existence and uniqueness result.

\begin{Proposition} \label{Prop_conj_HF_exist_unique}
Assume that $t_0 := \sup I' \in I'$ and consider a probability measure $\td{\mu} \in \mathcal{P} (\XX_{t_0} )$.
Then there is a unique conjugate heat flow $(\mu_t)_{t \in I'}$ with $\mu_{t_0} = \td \mu$.
\end{Proposition}

\begin{proof}
Define
\[ \mu_t := \int_{\XX_{t_0}} \nu_{x;t} \, d\td\mu(x). \]
Then (\ref{eq_Def_conj_heat_flow}) follows using the reproduction formula, Definition~\ref{Def_metric_flow}\ref{Def_metric_flow_7}.
\end{proof}

We summarize basic properties of conjugate heat flows:

\begin{Proposition} \label{Prop_properties_CHK}
The following is true:
\begin{enumerate}[label=(\alph*)]
\item \label{Prop_properties_CHK_a} Every finite convex combination of conjugate heat flows is again a conjugate heat flow.
\item \label{Prop_properties_CHK_b} The conjugate heat kernel $(\nu_{x;s})_{s \in I, s \leq \tf (x)}$ based at any $x \in \XX$ is a conjugate heat flow.
\item  \label{Prop_properties_CHK_bb} If $(\mu^1_t)_{t \in I'}, (\mu^2_t)_{t \in I'}$ are conjugate heat flows and $\mu^1_{t_0} \leq A \mu^2_{t_0}$ for some $t_0 \in I'$, $A > 0$, then $\mu^1_{t} \leq A \mu^2_{t}$ for all $t \leq t_0$, $t \in I'$.
\item \label{Prop_properties_CHK_c} 
Consider a heat flow $(u_t)_{t \in I'}$ and a conjugate heat flow $(\mu_t)_{t \in I'}$ over the same $I' \subset I$.
If $u_t$ is integrable with respect to $d\mu_t$ for some $t \in I'$, then the same is true for all $t \in I'$ and the integral $\int_{\XX_t} u_t \, d\mu_t$ is constant in $t \in I'$.
\end{enumerate}
\end{Proposition}

\begin{proof}
Assertions~\ref{Prop_properties_CHK_a}--\ref{Prop_properties_CHK_bb} are clear.
For Assertion~\ref{Prop_properties_CHK_c} let $t_1, t_2 \in I'$, $t_1 < t_2$.
If $u_{t_i}$ is integrable with respect to $d\mu_{t_i}$ for some $i = 1,2$, then
\[ \int_{\XX_{t_1}} u_{t_1} d\mu_{t_1} = \int_{\XX_{t_2}} \int_{\XX_{t_1}} u_{t_1} d\nu_{x; t_1} d\mu_{t_2}(x) = \int_{\XX_{t_2}} u_{t_2} d\mu_{t_2}, \]
which implies integrability of $u_{t_i}$ with respect to $d\mu_{t_i}$ for both $i =1,2$.
\end{proof}

The following proposition allows us to compare two conjugate heat flows.

\begin{Proposition} \label{Prop_compare_CHF}
Consider two conjugate heat flows $(\mu^i_t)_{t \in I'}$, $i=1,2$, defined over the same subset $I' \subset I$.
Then the following is true:
\begin{enumerate}[label=(\alph*)]
\item \label{Prop_compare_CHF_a} For any $t \in I'$ with $t \neq \sup I'$, the measures $\mu^1_t, \mu^2_t$ are absolutely continuous with respect to each other.
\item \label{Prop_compare_CHF_b} The quantity $d_{W_1}^{\XX_t} (\mu^1_t, \mu^2_t)$ is non-decreasing in $t$.
\item \label{Prop_compare_CHF_c} For any $x_1, x_2 \in \XX_t$ the quantity $d_{W_1}^{\XX_s} (\nu_{x_1;s}, \nu_{x_2;s})$ is non-decreasing in $s$ and we have
\[ d_{W_1}^{\XX_s} (\nu_{x_1;s}, \nu_{x_2;s}) \leq d_t (x_1, x_2). \]
\end{enumerate}
\end{Proposition}

\begin{Remark}
By \cite{Topping-McCann} we have monotonicity of the $W_2$-Wasserstein distance between two conjugate heat flows on a super Ricci flow.
It is an interesting question whether the same holds on a metric flow as well.
\end{Remark}

\begin{proof}
For Assertion~\ref{Prop_compare_CHF_a} consider some $t \in I'$ with $t \neq \sup I'$ and choose $t' \in I'$ with $t' > t$.
Consider a measurable $S \subset \XX_t$ with $\mu^1_t (S) = 0$.
Recall that
\begin{equation} \label{eq_mu_i_S_formula}
  \mu^i_t (S) = \int_{\XX_{t'}} \nu_{x;t} (S) \, d\mu^i_{t'} (x). 
\end{equation}
If $\nu_{x;t} (S) > 0$ for some $x \in \XX_{t'}$, then by Definition~\ref{Def_metric_flow}\ref{Def_metric_flow_6} we have $\nu_{x;t} (S) > 0$ for all $x \in \XX_{t'}$, which contradicts (\ref{eq_mu_i_S_formula}) for $i =1$.
So $\nu_{x;t} (S) = 0$ for all $x \in \XX_{t'}$ and therefore $\mu^2_t (S) = 0$.

Assertion~\ref{Prop_compare_CHF_b} follows by combining Propositions~\ref{Prop_Kant_Rub}, \ref{Prop_gradient_HF}\ref{Prop_gradient_HF_b}, \ref{Prop_properties_CHK}\ref{Prop_properties_CHK_c} as in the proof of \cite[\HKLemWoneMonotone]{Bamler_HK_entropy_estimates}.
More specifically, let $t_1, t_2 \in I'$, $t_1 \leq t_2$ and consider a bounded $1$-Lipschitz function $\td u : \XX_{t_1} \to \IR$.
Let $u : \XX_{\geq t_1} \to \IR$ be the heat flow with $u_{t_1} = \td u$.
Then $u_{t_2}$ is also $1$-Lipschitz and we have
\begin{multline*}
\int_{\XX_{t_1}} \td u \, d\mu^1_{t_1} - \int_{\XX_{t_1}} \td u \, d\mu^2_{t_1}
= \int_{\XX_{t_1}}  u_{t_1} \, d\mu^1_{t_1} - \int_{\XX_{t_1}}  u_{t_1} \, d\mu^2_{t_1} \\
= \int_{\XX_{t_2}}  u_{t_2} \, d\mu^1_{t_2} - \int_{\XX_{t_2}}  u_{t_2} \, d\mu^2_{t_2}
\leq d^{\XX_{t_2}}_{W_1} (\mu^1_{t_2}, \mu^2_{t_2}) .
\end{multline*}
Taking the supremum over all such $\td u$ implies Assertion~\ref{Prop_compare_CHF_b}.

Assertion~\ref{Prop_compare_CHF_c} is a direct consequence of Assertion~\ref{Prop_compare_CHF_b}.
\end{proof}

\subsection{Sets of measure zero and the support of a metric flow}
Let $\XX$ be a metric flow over some $I \subset \IR$.
The following is a direct consequence of Proposition~\ref{Prop_compare_CHF}\ref{Prop_compare_CHF_a}:

\begin{Proposition} \label{Prop_supp_XX}
Consider two conjugate heat flows $(\mu^i_t)_{t \in I'}$, $i=1,2$, defined over the same subset $I' \subset I$ and let $t \in I'$, $t < \sup I'$.
\begin{enumerate}[label=(\alph*)]
\item For any subset $S \subset \XX_t$ we have $\mu^1_t (S) = 0$ if and only if $\mu^2_t (S) = 0$.
\item For any subset $S \subset \XX_t$ we have $\mu^1_t (S) = 1$ if and only if $\mu^2_t (S) = 1$.
\item $\supp \mu^1_t = \supp \mu^2_t$.
\end{enumerate}
\end{Proposition}

We can therefore make the following definitions:

\begin{Definition}
Suppose that $t < \sup I$.
We say that $S \subset \XX_t$ is a {\bf subset of measure zero} if $\mu_t(S) = 0$ for one (and therefore any) conjugate heat flow $(\mu_{t'})_{t' \in I'}$, $I' \subset I$ on $\XX$ with $t < \sup I'$.
We say that $S \subset \XX_t$ is a {\bf subset of full measure} if $\XX_t \setminus S$ is a subset of measure zero.
\end{Definition}

\begin{Definition}
The {\bf support} $\supp \XX_t \subset \XX_t$ of $\XX$ at some time $t \in I$ is defined as follows. If $t < \sup I$, then $\supp \XX_t$ is defined as the subset $S \subset \XX_t$ with the property that $S = \supp \mu_t$ for any conjugate heat flow $(\mu_{t'})_{t' \in I'}$, $I' \subset I$ on $\XX$ with $t < \sup I'$.
If $t = \sup I$, then $\supp \XX_t := \XX_t$.
We write $\supp \XX := \bigcup_{t \in I} \supp \XX_t$.
If $\supp \XX = \XX$, then $\XX$ is said to have {\bf full support} and if $\supp \XX_t = \XX_t$ for some $t \in I$, then $\XX$ is said to have {\bf full support at time $t$}.
\end{Definition}

Proposition~\ref{Prop_supp_XX} also implies that for any metric flow $\XX$
\[ \big( \supp \XX, \tf |_{\supp \XX}, (d_t |_{\supp \XX_t})_{t \in I}, (\nu_{x;s}|_{\supp \XX_s})_{x \in \supp \XX, s \in I, s \leq \tf(x) }\big)  \]
is a metric flow of full support, which we will abbreviate by $\supp \XX$.
For any $x \in \XX \setminus \supp \XX$ the restricted conjugate heat kernel $(\nu_{x;s}|_{\supp \XX_s})_{s \in I, s < \tf(x)}$ is still a conjugate heat flow on $\supp \XX$.

\subsection{\texorpdfstring{$H$-Concentration}{H-Concentration}}
We  now introduce a property called $H$-concentration, which will be central to analysis of metric flows, as it ensures reasonable compactness properties of the space metric flows.
It has been shown in \cite{Bamler_HK_entropy_estimates} that it is satisfied by super Ricci flows for $H = H_n$.
It will be the only property in this paper that is sensitive to the dimension.

In the following let $\XX$ be a metric flow over $I \subset \IR$ and recall the definition of the variance $\Var$ from Definition~\ref{Def_Variance}.

\begin{Definition}[$H$-Concentration] \label{Def_H_Concentration}
$\XX$ is called {\bf $H$-concentrated} if for any $s \leq t$, $s,t \in I$, $x_1, x_2 \in \XX_t$
\begin{equation} \label{eq_H_contractivity_condition}
 \Var (\nu_{x_1; s}, \nu_{x_2; s} ) \leq d^2_t (x_1, x_2) + H (t-s). 
\end{equation}
\end{Definition}

\begin{Remark}
If $s = t$, then we have equality in (\ref{eq_H_contractivity_condition}), as $\Var (\delta_{x_1}, \delta_{x_2}) = d^2_t (x_1, x_2)$.
\end{Remark}

\begin{Remark}
(\ref{eq_H_contractivity_condition}) is invariant under parabolic rescaling and time-shifts.
So if $\XX$ is $H$-concentrated, then so is any other metric flow obtained from $\XX$ by parabolic rescaling and time-shifts.
\end{Remark}

We record that $H$-concentration implies the following monotonicity property; compare with \cite[\HKVarMonotonicityCHF]{Bamler_HK_entropy_estimates}.

\begin{Proposition} \label{Prop_H_monotonicity_Var}
If $\XX$ is  $H$-concentrated, then for any two conjugate heat flows $(\mu^1_t)_{t \in I'}$, $(\mu^2_t)_{t \in I'}$, $I' \subset I$, the function
\[ t \longmapsto \Var (\mu^1_t, \mu^2_t) + H t, \qquad t \in I' \]
is non-decreasing.
In particular, if $\Var (\mu^1_{t_0}, \mu^2_{t_0})< \infty$ for some $t_0 \in I'$, then $\Var (\mu^1_{t}, \mu^2_{t})< \infty$ for all $t \leq t_0$, $t \in I'$.
Moreover, for fixed $t \in I$, $x_1, x_2 \in \XX_t$ the following function is non-decreasing
\[ s \longmapsto \Var (\nu_{x_1;s}, \nu_{x_2;s}) + H (t-s), \qquad s \leq t, \quad s \in I'. \]
\end{Proposition}

\begin{proof}
Let $s' \leq s'' \leq t $, $s', s'' \in I'$.
By Definition~\ref{Def_conj_heat_flow} we have
\begin{multline*}
 \Var (\mu_{s'}, \mu_{s'} ) 
= \int_{\XX_{s''}} \int_{\XX_{s''}}  \Var (\nu_{y_1;s'}, \nu_{y_2;s'} ) d\mu_{s''}(y_1) d\mu_{s''}(y_2) \\
\leq \int_{\XX_{s''}} \int_{\XX_{s''}} \big( d^2_{s''} (y_1, y_2) + H(s''-s') \big) d\mu_{s''}(y_1) d\mu_{s''}(y_2)
= \Var (\mu_{s''},\mu_{s''} ) + H(s''-s') . \qedhere
\end{multline*}
\end{proof}
\medskip

As in \cite[\HKHCenterDef]{Bamler_HK_entropy_estimates} we define:

\begin{Definition}[$H$-center]
A point $z \in \XX_s$ is called an {\bf $H$-center} of some point $x \in \XX_t$ if $s \leq t$ and
\[ \Var(\delta_z, \nu_{x;s}) \leq H (t-s). \]
\end{Definition}

We recall that by Lemma~\ref{Lem_W_1_vs_Var} for any $H$-center $z \in \XX_s$ of a point $x \in \XX_t$, we have
\[ d^{\XX_s}_{W_1} (\delta_z, \nu_{x;s}) \leq \sqrt{\Var(\delta_z, \nu_{x;s})} \leq \sqrt{H (t-s)}. \]
The next proposition shows that $H$-centers always exist in an $H$-concentrated flow; compare with \cite[\HKHCenterPropExist]{Bamler_HK_entropy_estimates}.

\begin{Proposition} \label{Prop_H_center_existence}
Suppose that $\XX$ is $H$-concentrated.
Then for every $x \in \XX_t$ and $s \in I$, $s \leq t$, there is an $H$-center $z \in \XX_s$ of $x$.
Furthermore, for any two such $H$-centers $z_1, z_2 \in \XX_s$ we have $d_s (z_1, z_2) \leq 2\sqrt{H (t-s)}$.
\end{Proposition}

\begin{proof}
We have
\[ \int_{\XX_s} \Var (\delta_z , \nu_{x;s} ) d\nu_{x;s}(z) = \Var ( \nu_{x;s} )  \leq H (t-s), \]
which implies the first assertion.
For the second assertion observe that by Lemma~\ref{Lem_Var_triangle_inequ}
\[ d_s(z_1, z_2) = \sqrt{ \Var (\delta_{z_1}, \delta_{z_2}) }
\leq \sqrt{ \Var (\delta_{z_1},\nu_{x;s} ) } + \sqrt{ \Var (\nu_{x;s}, \delta_{z_2}) }
\leq 2 \sqrt{H (t-s)}. \qedhere \]
\end{proof}

We will also use the following bound (compare with \cite[\HKBoundBallHCenter]{Bamler_HK_entropy_estimates}):

\begin{Lemma} \label{Lem_nu_BA_bound}
If $z \in \XX_s$ is an $H$-center of $x \in \XX_t$, then for all $A > 0$
\[ \nu_{x;s} \big( B(z, \sqrt{A H (t-s)} ) \big) \geq 1 - \frac1{A}. \]
\end{Lemma}

\bigskip

\subsection{\texorpdfstring{$P^*$-parabolic neighborhoods}{P*-parabolic neighborhoods}}
We now generalize the concept of $P^*$-parabolic neighborhoods to metric flows; see \cite[\HKSecPstarParabolic]{Bamler_HK_entropy_estimates}.

In the following let $\XX$ be a metric flow over some subset $I \subset \IR$.

\begin{Definition}[$P^*$-parabolic neighborhood]
Consider a point $x \in \XX$ and suppose that $A, \lb T^-, \lb T^+ \lb  \geq 0$ such that $\tf (x) -T^- \in I$.
The {\bf $P^*$-parabolic neighborhood} $P^* (x,A; -T^-, T^+) \subset \XX$ is defined as the set of points $x' \in \XX$ with the property that
\[ \tf (x') \in [\tf(x) - T^-, \tf(x)+T^+], \qquad d^{\XX_{\tf(x) - T^-}}_{W_1} ( \nu_{x; \tf(x) - T^-}, \nu_{x'; \tf(x) - T^-} ) < A. \]
If $T^- = 0$ or $T^+ = 0$, then we will often write $P^* (x;A,  T^+)$ or $P^* (x;A, -T^-)$ instead of $P^* (x;A, -T^-, T^+)$.
\end{Definition}

The following simplified definition will often suffice for our purposes.

\begin{Definition}[$P^*$-parabolic ball]
Consider a point $x \in \XX$ and suppose that $r > 0$ such that $\tf (x) -r^2 \in I$.
The {\bf $P^*$-parabolic ball} at $x$ of radius $r$ is defined as
\[ P^*(x;r) := P^* (x;r, -r^2, r^2). \]
Similarly, we define the {\bf backward ($-$)  and forward ($+$) $P^*$-parabolic balls}
\[ P^{*-} (x;r) := P^* (x;r, -r^2), \qquad 
P^{*+} (x;r) := P^* (x;r, r^2).
\]
\end{Definition}

The following proposition generalizes \cite[\HKPropPstarBasic]{Bamler_HK_entropy_estimates} to metric flows; its proof carries over to the setting of metric flows.

\begin{Proposition} \label{Prop_basic_parab_nbhd}
The following holds for any $x_1 \in \XX_{t_1}, x_2 \in \XX_{t_2}$ as long as the corresponding $P^*$-parabolic neighborhoods or balls are defined:
\begin{enumerate}[label=(\alph*)]
\item \label{Prop_basic_parab_nbhd_a} For any $A \geq 0$ we have
\[ P^* (x_1; A,0,0) = B(x_1, A) . \]
\item \label{Prop_basic_parab_nbhd_b} If $0 \leq A_1 \leq A_2$, $0 \leq T^\pm_1 \leq T^\pm_2$, then 
\[ P^{*} (x_1; A_1, -T_1^-, T_1^+) \subset P^{*} (x_1; A_2, -T_2^-, T_2^+). \]
\item \label{Prop_basic_parab_nbhd_bb} If $A, T^\pm \geq 0$, and $x_1 \in P^*(x_2; A, -T^-, T^+)$, then
\[ x_2 \in P^*(x_1; A, - (T^- + T^+), T^-) \]
and
\[  P^*(x_2; A, -T^-, T^+) \subset P^*(x_1;2 A, - (T^- + T^+), T^-+T^+). \]
Likewise, if $r > 0$ and $x_1 \in P^*(x_2; r)$, then 
\[ x_2 \in P^*(x_1; \sqrt{2} r) \quad \text{and} \quad P^*(x_2; r) \subset P^*(x_1; 2r). \]
\item \label{Prop_basic_parab_nbhd_c} If $A_1, A_2, T_1^\pm, T_2^\pm \geq 0$ and $x_1 \in P^* (x_2; A_2, -T_2^-, T_2^+)$, then
\[ P^*(x_1; A_1, -T_1^-, T_1^+) \subset P^*(x_2; A_1  + A_2, - (T_1^- + T_2^-), T_1^+ + T_2^+). \]
Likewise, if $r_1, r_2 > 0$ and $x_1 \in P^*(x_2; r_2)$, then
\[ P^*(x_1; r_1) \subset P^*(x_2; r_1+r_2). \]
The same containment relationship also holds for the forward or backward parabolic balls, if $t_1 \geq t_2$ or $t_1 \leq t_2$, respectively.
\item \label{Prop_basic_parab_nbhd_d} If $r_1, r_2 > 0$ and $P^* (x_1; r_1) \cap P^* (x_2; r_2) \neq \emptyset$, then $P^* (x_1; r_1) \subset P^* (x_2; 2r_1+r_2)$.
Again, the same containment relationship also holds for the forward or backward parabolic balls, if $t_1 \geq t_2$ or $t_1 \leq t_2$, respectively.
\end{enumerate}
\end{Proposition}

Using $P^*$-parabolic balls, we can define the Hausdorff measure and dimension as usual.
Suppose in the following that $I \subset \IR$ is an interval.

\begin{Definition}[Hausdorff measure and dimension]
For any $S \subset \XX$ and $d \geq 0$ we define its {\bf $d$-dimensional $*$-Hausdorff measure} by
\[ \mathcal{H}^{* d} (S) := \liminf_{r \to \infty} \bigg\{ \sum_{i=1}^\infty r_i^d \;\; : \;\; \text{there are $x_i \in \XX$, $0 < r_i \leq r$ such that $S \subset \bigcup_{i=1}^\infty P^*(x_i,r_i)$} \bigg\}. \]
The {\bf $*$-Hausdorff dimension} of any subset $S \subset \XX$ is defined as
\[ \dim_{\mathcal{H}^*} S := \inf \big\{ d \geq 0 \;\; : \;\; \mathcal{H}^{*d} (S) < \infty \big\}. \]
\end{Definition}

Similarly, we can define the Minkowski dimension:

\begin{Definition}[Minkowski dimension]
For any subset $S \subset \XX$ set
\[ N_{\text{covering}} (S, r) := \min \bigg\{ N \geq 0 \;\; : \;\; \text{there are $x_1, \ldots, x_N \in \XX$ with $S \subset \bigcup_{i=1}^\infty P^*(x_i,r)$} \bigg\}. \]
Then the {\bf $*$-Minkowski dimension} of $S$ is defined as
\[ \dim_{\MM^*} S := \sup_{x_0, A, T^\pm} \limsup_{r \to 0} \frac{\log N_{\text{covering}} (S \cap P^* (x_0;A, T^-, T^+), r)}{\log (1/ r)}, \]
where the first supremum is taken over all $x_0, A, T^\pm$ with the property that $P^* (x_0;A, T^-, T^+) \subset \XX$ is defined.
\end{Definition}

As usual, it follows that
\[ \dim_{\mathcal{H}^*} S \leq \dim_{\MM^*} S. \]

\subsection{The natural topology on a metric flow} \label{subsec_natural_topology}
Let $\XX$ be a metric flow over some subset $I \subset \IR$.
We will define the following topology on $\XX$, which we will call the {\bf natural topology:}

\begin{Definition} \label{Def_topology}
A subset $U \subset \XX$ is called {\bf open} if for any $x \in U$ there is an $r > 0$ such that for all $r' \in (0,r]$ the following is true: If $P^*(x,r')$ exists, then $P^* (x,r') \subset U$.
\end{Definition}

\begin{Remark}
If $x$ has the property that there is a sequence of times $t_i \in I$ with $t_i \nearrow \tf(x)$, then we may simplify Definition~\ref{Def_topology} and only require that $P^*(x,r') \subset U$ for small enough $r'$.
On the other hand, if $t \in I$ and there is no sequence $t_i \in I$ with $t_i \nearrow t$, i.e. $\sup (I \cap (-\infty, t)) < t$, then the time-slice $\XX_t$ consists of isolated points.
So, for example, this is the case if $I$ is a left-closed interval $I \subset \IR$ and $t = t_{\min} := \min I$.
This is somewhat unintuitive and could be fixed by modifying Definition~\ref{Def_topology}.
However, we will mainly be interested in the case in which $I$ is left-open and  in particularly in which $I = (-\infty,0]$.
\end{Remark}

\begin{Proposition} \label{Prop_topology_properties}
Definition~\ref{Def_topology} defines a topology on $\XX$ with the following properties:
\begin{enumerate}[label=(\alph*)]
\item \label{Prop_topology_properties_a} $\tf : \XX \to \IR$ is continuous.
\item \label{Prop_topology_properties_b} 
If $t \in I$ and $\sup (I \cap (-\infty,t)) < t$, then the inclusion map $\XX_t \to \XX$ is continuous, where we equip $\XX_t$ with the topology induced by the metric $d_t$ and $\XX$ with the natural topology.
\item \label{Prop_topology_properties_c} 
Suppose that $x_\infty \in \XX_{t_\infty}$ and $\sup (I \cap (-\infty, t_\infty)) < t_\infty$.
Then for any sequence $x_i \in \XX_{t_i}$ we have $x_i \to x_\infty \in \XX_{t_\infty}$ with respect to the natural topology if and only if $t_i \to t_\infty$ and for any $t' < t_\infty$ we have
\[ d_{W_1}^{\XX_{t'}} (\nu_{x_i; t'}, \nu_{x_\infty; t'}) \longrightarrow 0. \]
\item  \label{Prop_topology_properties_dd} The $P^*$-parabolic neighborhoods $P^* (x,r)$ are neighborhoods of $x$ if they exist.
Moreover, $\{ P^* (x,r) \}_{x \in \XX, r > 0 }$ together with the one-point subsets $\{x\}$ for all points $x \in \XX$ with the property that $\sup (I \cap (-\infty, \tf(x))) < \tf(x)$ form a basis of the natural topology.
\item  \label{Prop_topology_properties_e_new} Consider two points $x_1, x_2 \in \XX$.
The following are equivalent:
\begin{enumerate}[label=(e\arabic*)]
\item There are neighborhoods $x_i \in U_i \subset \XX$, $i=1,2$, such that $U_1 \cap U_2 = \emptyset$.
\item There is a neighborhood $x_1 \in U_1 \subset \XX$ such that $x_2 \not\in U_1$.
\item The conjugate heat kernels $(\nu_{x_i;t})_{t  < \tf (x_i)}$ restricted to $I \cap (-\infty, \tf(x_i))$, $i=1,2$, are not the same.
\end{enumerate}
\item  \label{Prop_topology_properties_f_new} If $I$ is an interval, then any uniformly bounded heat flow $(u_t)_{t \in I'}$ over a left-open subinterval $I' \subset I$, viewed as a function $u : \XX_{I'} \to \IR$ is continuous.
\end{enumerate}
Moreover, if $\XX$ is $H$-concentrated for some $H < \infty$, then the following holds:
\begin{enumerate}[label=(\alph*), start=7]
\item \label{Prop_topology_properties_e} Suppose that $x_\infty \in \XX_{t_\infty}$ and $t_i \nearrow t_\infty$, $t_i \in I$.
Then for any sequence $x_i \in \XX_{t_i}$ of $H$-centers of $x_\infty$ we have $x_i \to x_\infty$ and $d_{W_1}^{\XX_{t_i}} ( \nu_{x_\infty; t_i},\delta_{x_i}) \to 0$.
\item \label{Prop_topology_properties_f} 
Suppose that $x_\infty \in \supp \XX_{t_\infty}$ and $t_i \searrow t_\infty$, $t_i \in I$.
Then there are points $x_i \in \XX_{t_i}$ such that $d_{W_1}^{\XX_{t_\infty}} (\delta_{x_\infty}, \nu_{x_i; t_\infty}) \to 0$.
In particular, if $\sup (I \cap (-\infty, t_\infty)) < t_\infty$, then $x_i \to x_\infty$ with respect to the natural topology.
\item  \label{Prop_topology_properties_g} If $I$ is a left-open interval, then $\XX$ is separable, i.e. there is a countable subset $S \subset \XX$ that is dense with respect to the natural topology.
\end{enumerate}
\end{Proposition}

\begin{Remark}
In general, $\XX$ may not be Hausdorff and the map $\XX_t \to \XX$ may not be open.
See Example~\ref{Ex_not_Hausdorff} in Subsection~\ref{subsec_Examples}.
\end{Remark}

\begin{proof}[Proof of Proposition~\ref{Prop_topology_properties}.]
To see that Definition~\ref{Def_topology} defines a topology on $\XX$, note that $P^*(x,r_1) \subset P^* (x,r_2)$ if $r_1 \leq r_2$.
To see Assertion~\ref{Prop_topology_properties_dd}, we claim that if $P^*(x,r)$ is defined, then $P^* (x,r) \cap \XX_{( \tf(x) - r^2, \tf(x) + r^2)}$ is open.
Let $x' \in P^*(x,r)$ with $|\tf(x') - \tf(x)| < r^2$.
Choose $r' > 0$ small enough such that $|\tf(x') - \tf(x)| + r^{\prime 2} < r^2$ and $d^{\XX_{\tf(x) - r^2}}_{W_1} (\nu_{x; \tf(x) - r^2}, \nu_{x'; \tf(x) - r^2}) + r' < r$.
We claim that $P^*(x',r') \subset P^*(x,r)$.
To see this, note that for any $x'' \in P^*(x', r') \cap \XX_{( \tf(x) - r^2, \tf(x) + r^2)}$ we have $\tf (x'') \in ( \tf(x) - r^2 , \tf(x) - r^2 )$ and by Proposition~\ref{Prop_compare_CHF}\ref{Prop_compare_CHF_b}
\begin{multline*}
 d^{\XX_{\tf(x) -r^2}}_{W_1} (\nu_{x; \tf(x) - r^2}, \nu_{x''; \tf(x) - r^2})
\leq d^{\XX_{\tf(x) - r^2}}_{W_1} (\nu_{x; \tf(x) - r^2}, \nu_{x'; \tf(x) - r^2}) + d^{\XX_{\tf(x) - r^2}}_{W_1} (\nu_{x'; \tf(x) - r^2}, \nu_{x''; \tf(x) - r^2}) \\
\leq d^{\XX_{\tf(x) - r^2}}_{W_1} (\nu_{x; \tf(x) - r^2}, \nu_{x'; \tf(x) - r^2}) + d^{\XX_{\tf(x) - r^{\prime 2}}}_{W_1} (\nu_{x'; \tf(x) - r^{\prime 2}}, \nu_{x''; \tf(x) - r^{\prime 2}}) \\
\leq d^{\XX_{\tf(x) - r^2}}_{W_1} (\nu_{x; \tf(x) - r^2}, \nu_{x'; \tf(x) - r^2}) + r' < r. 
\end{multline*}
So $x'' \in P^*(x, r) \cap \XX_{( \tf(x) - r^2, \tf(x) + r^2)}$.

Assertion~\ref{Prop_topology_properties_a} is clear.
For Assertion~\ref{Prop_topology_properties_b} note that for any $t \in I$ we have $B(x,r) \subset P^*(x,r) \cap \XX_{\tf(x)}$.
Assertions~\ref{Prop_topology_properties_c}, \ref{Prop_topology_properties_e_new} follow using Proposition~\ref{Prop_compare_CHF}\ref{Prop_compare_CHF_b} and Assertion~\ref{Prop_topology_properties_dd}.
For Assertion~\ref{Prop_topology_properties_f_new} observe that after restricting $u$ to a smaller time-interval, we may assume that $u_t$ is $L$-Lipschitz for all $t \in I'$, where $L < \infty$ is uniform.
So if $P(x,r) \subset \XX_{I'}$ exists, then for any $y \in P(x,r)$ we have for $t' := \tf(x) - r^2$
\[ |u(x) - u(y)|
 = \bigg| \int_{\XX_{t'}} u_{t'} \, d\nu_{x;t'} - \int_{\XX_{t'}} u_{t'} \, d\nu_{y;t'}\bigg| 
 \leq L r. \]

Next, assume that $\XX$ is $H$-concentrated.
For Assertion~\ref{Prop_topology_properties_e} note that if $t' \leq t_i$, $t' \in I$, then
\[ d^{\XX_{t'}}_{W_1} ( \nu_{x_i; t'}, \nu_{x_\infty; t'} ) 
\leq d^{\XX_{t_i}}_{W_1} ( \delta_{x_i}, \nu_{x_\infty; t_i} )
\leq \sqrt{H (t_\infty - t_i)} \to 0. \]
So $x_i \to x_\infty$ by Assertion~\ref{Prop_topology_properties_c}.

For Assertion~\ref{Prop_topology_properties_f}, suppose that $x_\infty \in \supp \XX_{t_\infty}$, $t_i \searrow t_\infty$ and fix a conjugate heat flow $(\mu_t)_{t \in I'}$ with $[t_\infty, t_1] \cap I \subset I'$, for example $\mu_t = \nu_{y;t}$ for some $y \in \XX_{t_1}$.
Let $r > 0$.
It suffices to show that for large $i$ there is a point $x_i \in \XX_{t_i}$ with $d^{\XX_{t_\infty}}_{W_1} (\nu_{x_i; t_\infty}, \delta_{x_\infty}) \leq r$.
To see this, observe that since
\[ \int_{\XX_{t_i}} \nu_{x;t_\infty} ( B(x_\infty, r/2) ) d\mu_{t_i} (x) = \mu_{t_\infty} ( B(x_\infty, r/2)) =: c > 0, \]
we can find points $x_i \in \XX_{t_i}$ with $\nu_{x_i; t_\infty} ( B(x_\infty, r/2) ) \geq c > 0$.
For any $i$ let $z_i \in \XX_{t_\infty}$ be an $H$-center of $x_i$.
Then we have $d_{t_\infty} (z_i, x_\infty) < r/2$ or
\[ c (d_{t_\infty} (z_i, x_\infty) - r/2)^2 
\leq \int_{B(x_\infty, r/2)} d^2( z_i, \cdot) d\nu_{x_i; t_\infty} 
\leq \Var ( \delta_{z_i}, \nu_{x_i; t_\infty}  ) \leq H (t_i - t_\infty) \to 0.\]
Therefore,
\begin{multline*}
 \limsup_{i \to \infty} d^{\XX_{t_\infty}}_{W_1} (\nu_{x_i; t_\infty}, \delta_{x_\infty})
\leq \limsup_{i \to \infty} \Big( d^{\XX_{t_\infty}}_{W_1} (\nu_{x_i; t_\infty}, \delta_{z_i}) + d_{t_\infty} (z_i, x_\infty) \Big)  \\
\leq  \limsup_{i \to \infty}  \sqrt{\Var (\nu_{x_i; t_\infty}, \delta_{z_i})} + r/2
\leq \limsup_{i \to \infty} \sqrt{H (t_i - t_\infty)} + r/2  = r/2.
\end{multline*}

For Assertion~\ref{Prop_topology_properties_g}, let $Q \subset I$ be a countable and dense subset.
For any $t \in Q$ choose a countable and dense subset $S_t \subset \XX_t$.
Let $S := \bigcup_{t \in Q} S_t \subset \XX$.
To see that $S$ is dense, consider some point $x \in \XX$ and choose times $t_i \in Q$ with $t_i \nearrow \tf (x)$.
Let $z_i \in \XX_{t_i}$ be $H$-centers of $x$ and choose $x_i \in \XX_{t_i}$ with $d_{t_i} (x_i, z_i) \to 0$.
Then for any fixed $t' \in I$, $t' < \tf (x)$ we have if $t_i \geq t'$,
\[ d^{\XX_{t'}}_{W_1} ( \nu_{x_i; t'}, \nu_{x; t'} ) 
\leq d^{\XX_{t_i}}_{W_1} ( \delta_{x_i}, \nu_{x; t_i} )
\leq d_{t_i} (x_i,z_i) + d^{\XX_{t_i}}_{W_1} ( \delta_{z_i}, \nu_{x; t_i} )
\leq d_{t_i} (x_i,z_i) + \sqrt{H (\tf(x) - t_i)} \to 0. \]
So $x_i \to x_\infty$ by Assertion~\ref{Prop_topology_properties_c}.
\end{proof}

\subsection{Super Ricci flows and singular Ricci flows as metric flows} \label{subsec_superRF_met_flow}
As mentioned before, the most important class of metric flows are Ricci flows, super Ricci flows and --- in dimension 3 --- singular Ricci flows \cite{Kleiner_Lott_singular,bamler_kleiner_uniqueness_stability}.
These metric flows are $H_n$-concentrated, where $ H_n$ only depends on the dimension.
We will explain in the following how these flows can be turned into metric flows.

Let $M$ be an $n$-dimensional compact manifold and consider a super Ricci flow $(g_t)_{t \in I}$ over some interval $I \subset \IR$.
Recall that this means that
\[ \partial_t g_t \geq - 2 \Ric_{g_t}. \]
For any $(y,s) \in M \times I$ consider the heat kernel $K( \cdot, \cdot; y,s)$ of the standard heat equation using $(g_t)$ as a background, i.e. for fixed $(y,s) \in M \times I$
\begin{equation} \label{eq_K_1_HE}
 \partial_t K(\cdot ,t; y,s) = \triangle_{g_t} K(\cdot ,t; y,s), \qquad K(\cdot ,t, y,s) \xrightarrow[t \searrow s]{} \delta_y. 
\end{equation}
Then $K(x,t;y,s)$ is defined and smooth whenever $s < t$ and if we fix $(x,t) \in M \times I$, then $K(x,t; \cdot, \cdot)$ satisfies the conjugate heat equation:
\begin{equation} \label{eq_K_2_CHE}
 - \partial_s K(x,t; \cdot, s) = \triangle_{g_s} K(x,t; \cdot, s) + \tfrac12 ( \tr_{g_s} \partial_s g_s ) K(x,t; \cdot, s), \qquad
K(x ,t, \cdot,s) \xrightarrow[s \nearrow t]{} \delta_y 
\end{equation}
and we have
\[ \int_M K(x,t; \cdot, s) dg_s = 1. \]
For more details see \cite[\HKSubsecHOandHK]{Bamler_HK_entropy_estimates}.

For any $(x,t) \in M \times I$ and $s \leq t$ consider the following probability  measure on $M \times \{ s \}$ (compare also with \cite[\HKDefConjHKmeasure]{Bamler_HK_entropy_estimates}):
\[ d\nu_{x,t ; s} := \begin{cases} K(x,t; \cdot, s) d\mu_{g_s} & \text{if $s < t$} \\
\delta_{(x,s)} & \text{if $s = t$} \end{cases} \]
Set $X_t := M \times \{ t \}$ and let $d_t := d_{g_t}$ be the length metric induced by $g_t$.
Consider
\begin{equation} \label{eq_met_flow_from_SRF}
  \big( \XX:= M \times I , \tf := \proj_{I}, ( d_t)_{t \in I}, (\nu_{x,t;s})_{(x,t) \in M \times I, s \in I, s \leq t} \big). 
\end{equation}
We have:

\begin{Theorem} \label{Thm_superRF_metric_flow}
(\ref{eq_met_flow_from_SRF}) is an $H_n := ( \frac{(n-1)\pi^2}2 + 4)$-concentrated metric flow.
Heat flows on $\XX$ correspond to solutions to the heat equation on $M \times I$ and conjugate heat flows on $\XX$ correspond to measures of the form $v \, dg_t$, where $v$ is a solution to the conjugate heat equation.

If $I$ is left-open, then the natural topology on $\XX$ agrees with the product topology on $M \times I$.
\end{Theorem}

\begin{proof}
Properties~\ref{Def_metric_flow_1}--\ref{Def_metric_flow_5} of Definition~\ref{Def_metric_flow} are clear by definition.
Property~\ref{Def_metric_flow_6} is a consequence of \cite[\HKThmGradientPhiEstimate]{Bamler_HK_entropy_estimates} and Property~\ref{Def_metric_flow_7} follows from the reproduction formula:
\[ K(x,t_3;z,t_1) = \int_M K(x,t_3;y,t_2)K(y,t_2;z,t_1) dg_{t_2}(y). \]
The $( \frac{(n-1)\pi^2}2 + 4)$-concentration is a consequence of \cite[\HKHnConcentration]{Bamler_HK_entropy_estimates}.
The last statement follows from standard parabolic estimates, see for example \cite[\HKCenterConstantRmBound]{Bamler_HK_entropy_estimates}.
\end{proof}
\bigskip

Next, consider a 3-dimensional singular Ricci flow $\mathcal{M} = (\MM, \tf, \partial_\tf, g)$ over some interval $I = [0,T)$; see \cite{Kleiner_Lott_singular, bamler_kleiner_uniqueness_stability} and Subsection~\ref{subsec_RF_spacetimes}.
We recall that a singular Ricci flow is a Ricci flow spacetime $\mathcal{M}$ as in \cite[Definition~5.1]{bamler_kleiner_uniqueness_stability} whose initial time-slice $\mathcal{M}_0$ is compact, that is $0$-complete in the sense of \cite[Definition~5.4]{bamler_kleiner_uniqueness_stability} and that has the property that for every $\eps, T > 0$ there is an $r_{\eps,T} > 0$ such that $\mathcal{M}_{[0,T)}$ satisfies the $\eps$-canonical neighborhood assumption below scale $r_{\eps, T}$ in the sense of \cite[Definition~5.7]{bamler_kleiner_uniqueness_stability}.
We recall that by \cite{bamler_kleiner_uniqueness_stability} the flow $\mathcal{M}$ is uniquely determined by its initial time-slice $(\mathcal{M}_0, g_0)$, so the theorems of \cite{Kleiner_Lott_singular} also apply.

We will sketch how to convert $\mathcal{M}$ into a metric flow; a more rigorous treatment will be available in forthcoming work.
Let $I' = [0, T')$ or $[0,T'] \subset I$ be a subinterval and consider an open subset $\MM' \subset \MM_{I'}$ with the property that for any $t \in I'$ the time-slice $\MM'_t$ is equal to a connected component of $\MM_t$.
If $I' = [0,T']$, then such subsets are uniquely determined by the component $\MM'_{T'}$.
More specifically, given the component $\MM'_{T'} \subset \MM_{T}$, we can choose the unique component $\MM'_t \subset \MM_t$ for any $t \in [0,T')$ with the property that there is a continuous curve $\gamma : [t, T']$ such that $\tf (\gamma(t')) = t'$ and $\gamma(t) \in \MM'_t$, $\gamma(T') \in \MM'_{T'}$.
The subset $\MM' \subset \MM_{I'}$ can be viewed as one ``branch'' of the singular flow $\MM$ --- it roughly corresponds to choosing a component after every neckpinch.

We will now convert $\MM'$ into a metric space $\XX$ over $I'$.
For any $t \in I'$ let $(X_t, d_t)$ be the completion of the length metric of the time-slice $(\MM'_t, g_t)$ and suppose that $X_t \supset \MM'_t$.
It can be shown that there is a heat kernel $K \in C^\infty (U)$ on $\mathcal{M}$, for $U := \{ (x,y) \in \MM \times \MM \;\; : \;\; \tf(y) < \tf (x) \}$, that satisfies (\ref{eq_K_1_HE}), (\ref{eq_K_2_CHE}) if we replace the time-derivative by a Lie-derivative of the $\partial_{\tf}$-vector field.
$K$ still satisfies the reproduction formula and for any $\ov u \in C^0_c (\mathcal{M}_s)$ the function $ u : \mathcal{M}_{[s, \infty)} \to \IR$ given by
\[  u(x) := \begin{cases}
\ov u(x) & \text{if $\tf(x) =s$} \\
 \int_{\mathcal{M}_s} K(x;\cdot ) dg_s  & \text{if $\tf(x) > s$} \end{cases} \]
is a solution to the heat equation with initial condition $\ov u$.
By the choice of $\MM'$ we have for any $x \in \MM'_t$, $s < t$,
\[ K(x;\cdot) = 0 \qquad \text{on} \quad \MM_{s} \setminus \MM'_{s}. \]
So we may still define the conjugate heat kernel measures $\nu_{x;s}$ as the probability measure on $X_s$ with
\[ \nu_{x;s} (X_s \setminus \mathcal{M}'_s) = 0, \qquad 
d\nu_{x;s} \big|_{\mathcal{M}'_s} =
\begin{cases} K(x; \cdot ) dg_s & \text{if $s < t$} \\
\delta_{x} & \text{if $s = t$} \end{cases}. \]
It can be shown that for fixed $0 \leq s \leq t$ the map $(\MM'_t, d_t) \to (\PP (X_s), d^{X_s}_{W_1})$, $x \mapsto \nu_{x;s}$ can be extended uniquely to a continuous map of the form $X_t \to \PP (\MM'_s)$.
Using this extension, we will now define the metric flow $\XX$ over $I'$ by
\[ \bigg( \XX := \bigcup_{t \in I'} X_t, \tf, (d_t)_{t \in I'}, (\nu_{x;s})_{x \in X_t, s \leq t} \bigg), \]
where $\tf : \XX \to I'$ is the obvious map with $\tf (X_t) = t$.
It can be shown that:

\begin{Theorem} \label{Thm_sing_flow_to_met_flow}
$\XX$ is an $H_3$-concentrated metric flow.
If we view $\MM'$ as a subset of $\XX$, then the natural topology of $\XX$ restricted to $\MM'_{I \setminus \{ 0 \}}$ agrees with the standard topology given by the spacetime manifold.
Moreover, $\XX$ is future continuous (see Subsection~\ref{subsec_fut_past_cont} for further details).
\end{Theorem}

\subsection{Special cases and constructions}
In the following we define certain classes and constructions for metric flows, which are the analogs of common constructions of (super) Ricci flows.
These constructions will be needed in \cite{Bamler_HK_RF_partial_regularity}.

We first define the Cartesian product of two metric flows.

\begin{Definition}[Cartesian product]
The Cartesian product of two metric flows $\XX^1, \XX^2$ that are defined over the same subset $I \subset \IR$ is given by the tuple
\[ \Big( \XX^{12} := \bigsqcup_{t \in I} \XX^1_t \times \XX^2_t, d^{12}_t, (\nu^{12}_{(x_1,x_2) ; s} := \nu^1_{x_1;s} \otimes \nu^2_{x_2;s} )_{(x_1,x_2) \in \XX^1_t \times \XX^2_t, s \leq t} \Big), \]
where 
\[ \big(d^{12}_t ( (x_1, x_2), (y_1, y_2) ) \big)^2 = \big(d^1_t(x_1,y_1) \big)^2 + \big( d^2_t (x_2, y_2) \big)^2. \]
\end{Definition}

The following can be checked easily:

\begin{Proposition}
$\XX^{12}$ is a metric flow over $I \subset \IR$ and the following is true:
\begin{enumerate}[label=(\alph*)]
\item If $(u^i_t)_{t \in I'}$ is a heat flow on $\XX^i$, $i=1,2$, then $(u^1_t u^2_t)_{t \in I'}$ is a heat flow on $\XX^{12}$.
\item If $(\mu^i_t)_{t \in I'}$ is a conjugate heat flow on $\XX^i$, $i=1,2$, then $(\mu^1_t \otimes \mu^2_t)_{t \in I'}$ is a conjugate heat flow on $\XX^{12}$.
\item If $\XX^i$ is $H_i$-concentrated for $i=1,2$, then on $\XX^{12}$ is $(H_1 + H_2)$-concentrated.
\end{enumerate}
\end{Proposition}

Next, we define the analog of a steady gradient soliton.

\begin{Definition}[Static metric flows]
A metric flow $\XX$ over some interval $I \subset \IR$ is called {\bf static} if there is a tuple
\begin{equation} \label{eq_Def_static_model}
 \big( X, d, (\nu'_{x;t} )_{x \in X; I \cap (t+I) \neq \emptyset}  \big) 
\end{equation}
and a map $\phi : \XX \to X$ such that the following holds:
\begin{enumerate}[label=(\arabic*)]
\item \label{Def_static_MF_1} $(X,d)$ is a metric space and for any $t \in I$ the map $\phi_t : (\XX_t, d_t) \to (X,d)$ is an isometry.
\item \label{Def_static_MF_2} $(\nu'_{x;t} )_{x \in X; I \cap (t+I) \neq \emptyset}$ is a family of probability measures on $X$ and for any $x \in \XX_t$, $s \in I$, $s \leq t$ we have $(\phi_s)_* \nu_{x;s} = \nu'_{\phi_t (x); t-s} \in \mathcal{P} (X)$.
\end{enumerate}
The tuple (\ref{eq_Def_static_model}) is called a {\bf static model} for $\XX$.
\end{Definition}

\begin{Remark}
The static model and the map $\phi$ may not be uniquely determined by Properties~\ref{Def_static_MF_1}, \ref{Def_static_MF_2}.
For example, if we consider the constant flow $(g_t)_{t \in \IR}$ on $\IR^n$, then we could choose $\XX = \IR^n \times \IR \to \IR^n =: X$ to be the standard projection, or a map of the form $(\vec x, t) \mapsto \vec x + t \vec a$ for some $\vec a \in \IR^n$.
Then $d$ is the Euclidean metric on $\IR^n$, and $(\nu'_{x;t})$ corresponds to the kernels of the heat equation $\partial_t v = \triangle v - \vec a \cdot \nabla v$.
\end{Remark}

Next, we define the analog of a shrinking soliton.

\begin{Definition}[Metric soliton] \label{Def_metric_soliton}
A pair $(\XX, (\mu_t)_{t \in I})$, consisting of metric flow $\XX$ over some interval $I \subset \IR$ with $\sup I = 0$, $0 \not\in I$ and a conjugate heat flow $(\mu_t)_{t \in I}$ is called a {\bf metric soliton} if there is a tuple
\[ \big( X, d, \mu, (\nu'_{x;t} )_{x \in X; t \leq 0}  \big) \]
and a map $\phi : \XX \to X$ such that the following holds:
\begin{enumerate}
\item For any $t \in I$, the map $\phi_t : (\XX_t, d_t, \mu_t) \to (X, \sqrt{t} d, \mu)$ is an isometry between metric measure spaces.
\item For any $x \in \XX_t$, $s \in I$ with $s \leq t$, we have $(\phi_s)_* \nu_{x;s} = \nu'_{\phi_t (x); \log (s/t)}$.
\end{enumerate}
The conjugate heat flow $(\mu_t)_{t \in I}$ is called the {\bf potential flow} of the metric soliton $(\XX, (\mu_t)_{t \in I})$.
\end{Definition}

If $\XX$ is defined over an interval $I$ of the form $(-T,0]$ or $[-T,0]$, then we will often say that a pair of the form $(\XX, (\mu_t)_{t \in I })$ or $(\XX, (\mu_t)_{t \in I \setminus \{ 0 \}})$ is a metric soliton if Definition~\ref{Def_metric_soliton} holds for the restricted pair $(\XX_{t \in I \setminus \{ 0 \}}, (\mu_t)_{t \in I \setminus \{ 0 \}})$.

The following proposition shows that metric flows are selfsimilar.
Moreover, it shows that given a selfsimilar metric flow $\XX$, there is a potential flow $(\mu_t)_{t \in I}$ such that $(\XX, (\mu_t)_{t \in I})$ is a metric flow and this potential flow only depends on $\XX$ and the family of selfsimilar maps.

\begin{Proposition} \label{Prop_metric_soliton_properties}
Consider a metric soliton $(\XX, (\mu_t)_{t \in I})$ and the map $\phi : \XX \to X$ from Definition~\ref{Def_metric_soliton}.
For any $\lambda \in (0,1]$ consider the map $\psi_\lambda : \XX \to \XX$ that maps every $x \in \XX_t$ to $\psi_\lambda (x) \in \XX_{\lambda^2 t}$ with $\phi (\psi_\lambda (x) ) = \phi(x)$.
Then
\begin{enumerate}[label=(\alph*)]
\item \label{Prop_metric_soliton_properties_a} For any $\lambda \in (0,1]$ the map $\psi_\lambda$  is a flow isometry between $\XX$ and $\XX_{\lambda^2 I}$ if we parabolically rescale the domain by $\lambda$.
\item \label{Prop_metric_soliton_properties_b} For any $\lambda_1, \lambda_2 \in (0,1]$ we have $\psi_{\lambda_1} \circ \psi_{\lambda_2} = \psi_{\lambda_1\lambda_2}$
\item \label{Prop_metric_soliton_properties_c} For any $\lambda \in (0,1]$ we have $(\psi_\lambda)_* \mu_t = \mu_{\lambda^2 t}$.
\end{enumerate}
Vice versa, suppose that $\XX$ is a metric flow over some interval $I \subset \IR$ with $\sup I = 0$, $0 \not\in I$ and consider a family of maps $(\psi_{\lambda} : \XX \to \XX)_{\lambda \in (0,1]}$ that satisfies Properties~(a), (b).
If $\XX$ is $H$-concentrated for some $H < \infty$, then there is a unique conjugate heat flow $(\mu_t)_{t \in I}$ such that $(\XX, (\mu_t)_{t \in I})$ is a metric flow, such that Property~(c) holds and such that $\mu_t \in \PP^1(\XX_t)$ for all $t \in I$, where the latter is space of probability measures that have finite $d_{W_1}$-distance to point masses.

Lastly, if $(\XX, (\mu_t)_{t \in I})$ is a metric soliton and $\XX$ is $H$-concentrated, then $\Var (\mu_t) \leq H |t|$ for all $t \in I$.
\end{Proposition}

\begin{proof}
The first direction can be verified easily.
For the reverse direction, consider an $H$-con\-cen\-trat\-ed metric flow $\XX$ and a family of maps $(\psi_{\lambda} : \XX \to \XX)_{\lambda \in (0,1]}$ satisfying Properties~(a), (b).
Fix some $t_0 \in I$ and consider the map $\phi : \XX \to \XX_{t_0} =: X$ mapping each $x \in \XX_t$ to $\psi_{t_0 /t} (x)$ or $\psi^{-1}_{t/t_0} (x)$, depending on whether $t \leq t_0$ or $t > t_0$.
Let $d := d_{t_0}$ and $\nu'_{x; t} := \nu_{\psi_{e^{-2t}} (x); t_0 e^{-2t}}$.
It remains to construct a conjugate heat flow $(\mu_t)_{t \in I}$ such that Property~(c) holds and to show that this flow is unique.
The fact that $\phi$ satisfies Properties~(1), (2) from Definition~\ref{Def_metric_soliton} for $\mu := \mu_{t_0}$ then follows easily.

Recall that $(\mathcal{P}^1(\XX_{t_0}), d_{W_1}^{\XX_{t_0}})$ is a complete metric space and consider the map
\[ F: \mathcal{P}^1 (\XX_{t_0} ) \to \mathcal{P} (\XX_{t_0} ), \qquad \mu' \mapsto \int_{\XX_{t_0/4}} \nu_{x; t_0} \, d((\psi_{1/2})_* \mu')(x). \]
By Proposition~\ref{Prop_compare_CHF}\ref{Prop_compare_CHF_b} we have for any two $\mu' , \mu'' \in \mathcal{P}'(\XX_{t_0})$
\[ d_{W_1}^{\XX_{t_0}} \big( F(\mu'), F(\mu'') \big) \leq d_{W_1}^{\XX_{t_0/4}} \big( (\psi_{1/2})_* \mu', (\psi_{1/2})_* \mu'' \big)
= \tfrac12 d_{W_1}^{\XX_{t_0}} \big( \mu', \mu'' \big). \]
Due to the $H$-concentration property we have $F(\delta_x) = \nu_{\psi_{1/2}(x); t_0} \in \mathcal{P}^1(\XX_{t_0})$ for any $x \in \XX_{t_0}$, so the image of $F$ lies in $\mathcal{P}^1 (\XX_{t_0} )$ and thus $F$ is a $\frac12$-contraction.
Let $\mu' \in \mathcal{P}^1 (\XX_{t_0} )$ be its unique fixed point.
Then the conjugate heat flows with initial condition $(\psi_{2^{-i}})_* \mu'$ agree.
Letting $i \to \infty$ shows the existence of $(\mu_t)_{t \in I}$.
The uniqueness of $(\mu_t)_{t \in I}$ follows from the uniqueness of the fixed point of $F$.
For the last statement of the proof observe that $F^{(i)}(\delta_x) = \nu_{\psi_{2^{-i}}(x); t_0}  \to \mu_{t_0}$ for $i \to \infty$.
\end{proof}

We also have:

\begin{Proposition}
Consider two metric solitons $(\XX^i, (\mu^i_t)_{t \in I})$, $i =1,2$, and let $\XX^{12}$ be the Cartesian product of $\XX^1, \XX^2$.
Then $(\XX^{12}, (\mu^1_t \otimes \mu^2_t)_{t \in I})$ is also a metric soliton.
\end{Proposition}

Lastly, we consider the case in which the static model of a static flow is a cone.
In this case the flow is also a metric soliton for an appropriate potential flow.

\begin{Definition}[Static cone]
A metric flow $\XX$ over some interval $I \subset \IR$ with $\sup I = 0$, $0 \not\in I$ is called a {\bf static cone} if it is static with static model (\ref{eq_Def_static_model}) and if $(X,d)$ is a metric cone over some metric space $(X', d_{X'})$ with vertex $x_0$ such that the following holds for any $\lambda \in (0,1]$.
Denote by $\psi_\lambda : X \to X$ the radial dilation by $\lambda$ with $\psi_\lambda(x_0) = x_0$.
Then for any $x \in X$ we have
\begin{equation} \label{eq_psi_lambda_nu_p}
 (\psi_\lambda)_* \nu'_{x;t} = \nu'_{\psi_\lambda(x); \lambda^2 t}. 
\end{equation}
The point $x_0$ is called a {\bf vertex of the static model}.
\end{Definition}

Note that if $I = \IR_-$, then (\ref{eq_psi_lambda_nu_p}) also holds if $\lambda > 1$, because $\psi_\lambda = \psi^{-1}_{\lambda^{-1}}$.

The following is a consequence of Proposition~\ref{Prop_metric_soliton_properties}:

\begin{Proposition}
Consider a static cone $\XX$ over $I \subset \IR$, with static model (\ref{eq_Def_static_model}) and vertex $x_0 \in X$.
Let $(\mu_t)_{t \in I}$ be the conjugate heat flow corresponding to $(\nu'_{x_0; t})_{t \in I'}$ on $X$.
Then $(\XX, (\mu_t)_{t \in I})$ is a metric soliton.
\end{Proposition}

\subsection{Further examples} \label{subsec_Examples}
In the last subsection, we discuss further examples of metric flows.

\begin{Example} \label{Ex_Rn_Gaussian}
Consider the metric flow corresponding to the constant Ricci flow on $\IR^n$.
Then
\[ \nu_{\vec x,t;s} = (4 \pi (t-s))^{-n/2} \exp \Big( - \frac{|\vec x - \cdot|^2}{4 (t-s)} \Big) d\vol \]
is the standard heat kernel and we can compute that for any $\vec x, \vec x^{\,\prime} \in \IR^n$
\begin{align*}
 \Var (\nu_{\vec x,t;s}, \nu_{\vec x',t;s}) &= \int_{\IR^n} \int_{\IR^n} |\vec y - \vec y^{\,\prime} |^2  (4 \pi (t-s))^{-n} \exp \Big( - \frac{|\vec x - \vec y|^2 + |\vec x^{\,\prime} - \vec y^{\,\prime} |^2}{4 (t-s)} \Big) d\vec y d\vec y^{\,\prime} \\
&= \sum_{i=1}^n \int_{\IR}\int_{\IR} (y_i - y'_i)^2  (4 \pi (t-s))^{-1} \exp \Big( - \frac{(x_i - y_i)^2 + (x'_i - y'_i)^{ 2}}{4 (t-s)} \Big) dy_i dy'_i \\
&= \sum_{i=1}^n \int_{\IR}\int_{\IR} (x_i - x'_i + z_i - z'_i)^2  (4 \pi (t-s))^{-1} \exp \Big( - \frac{z_i^2 + z_i^{\prime 2}}{4 (t-s)} \Big) dz_i dz'_i \\
&= |\vec x - \vec x^{\,\prime}|^2 + \sum_{i=1}^n \int_{\IR}\int_{\IR} ( z^2_i + z^{\prime 2}_i)  (4 \pi (t-s))^{-1} \exp \Big( - \frac{z_i^2 + z_i^{\prime 2}}{4 (t-s)} \Big) dz_i dz'_i \\
&= |\vec x - \vec x^{\,\prime}|^2 + 2n \int_{\IR}  z^2  (4 \pi (t-s))^{-1/2} \exp \Big( - \frac{z^2}{4 (t-s)} \Big) dz \\
&= |\vec x - \vec x^{\,\prime}|^2 + 2n (t-s) \int_{\IR}  z^2  (4 \pi)^{-1/2} \exp \Big( - \frac{z^2}{4 } \Big) dz \\
&= |\vec x - \vec x^{\,\prime}|^2 + 4n (t-s).
\end{align*}
Therefore $\XX$ is $H$-concentrated if and only if $H \geq 4n$.
The example shows that $H$-concentration is dimension dependent and therefore does not follow from the properties of a metric flow.
\end{Example}
\bigskip

\begin{Example} \label{Ex_not_Hausdorff}
Consider an arbitrary metric space $(X^*, d^*)$.
Let $\XX := X^* \times \{ 0 \} \cup \{ 0 \} \times \IR_-$ and let $\tf : \XX \to (-\infty,0]$ be the projection onto the last factor.
Define $d_0  := d^*$ and $\nu_{x;t} := \delta_{(0,t)}$ for $t < 0$, $\nu_{x;0} := \delta_x$.
Then $\XX$ is a metric flow over $(-\infty, 0]$ whose natural topology is not Hausdorff if $\# X^* > 1$.
\end{Example}
\bigskip

\begin{Example}
Fix some constants $C > 0$ such that for all $A \geq 0$
\begin{equation} \label{eq_ex_choice_C}
16 A^2 \leq C \exp \Big( \frac{A^2}{16} \Big)
\end{equation}
Let $D > 0$ be an arbitrary constant and consider the metric flow $\XX= \XX_{C,D} := \{ -\frac12 D,+\frac12D \} \times \IR$ over $\IR$ with $d_t (-\frac12D,\frac12D) = D$ for all $t \in \IR$ and
\[ \nu_{ \pm \frac12 , t; s} ( \{ (\pm \tfrac12 D, s )\} ) = \frac12 + \frac12 e^{-C(t-s)/D^2}, \qquad
 \nu_{\pm \frac12 , t; s} ( \{( \mp \tfrac12 D , s) \} ) = \frac12 - \frac12 e^{-C(t-s)/D^2}. \]
 We verify the properties of a metric flow.
 Properties~\ref{Def_metric_flow_1}--\ref{Def_metric_flow_5} of Definition~\ref{Def_metric_flow} are clear.
 
  For Property~\ref{Def_metric_flow_6}, it suffices to show that if $(u_t = \Phi ( t^{-1/2} h_t ))_{t \geq s}$ is a heat flow for some $s > 0$, where $h : \{ \pm \frac12 D \} \times [s, \infty) \to \IR$,  then the condition $h(\frac12 D, t) - h(-\frac12 D, t) \leq D$ is preserved.
  To see this, we compute that
  \[   \partial_t u (\pm\tfrac12 D,t) = - \frac{C}{2D^2} u (\pm\tfrac12 D,t) + \frac{C}{2D^2} u (\mp\tfrac12 D,t) \]
  and
  \[ \partial_t u = \big( -\tfrac12 t^{-3/2} h_t  + t^{-1/2} \partial_t h_t \big) \Phi' ( t^{-1/2} h_t), \] 
  which implies
\begin{multline} \label{eq_part_h_diff}
  \partial_t h(\tfrac12 D, t) - \partial_t h(-\tfrac12 D,t) 
  =  \frac1{2t}  h(\tfrac12 D, t) - \frac1{2t}  h(-\tfrac12 D, t) \\ - \frac{C}{2D^2} \big( u(\tfrac12 D, t) - u(-\tfrac12 D, t) \big) \Big( \frac{1 }{t^{-1/2} \Phi' (t^{-1/2} h(\tfrac12 D, t))} 
  +  \frac{1}{t^{-1/2} \Phi' (t^{-1/2} h(-\tfrac12 D, t))} \Big) .
\end{multline}
So it suffices to show that if $h(\frac12 D, t) - h(-\frac12 D, t) = D$, then the right-hand side of (\ref{eq_part_h_diff}) is non-positive.
If we set $u := \frac{h(\frac12 D, t) + h(-\frac12 D, t)}{2t^{1/2}}$ and $A := \frac{D}{2t^{1/2}}$, then this is equivalent to
\[  \Big( e^{(u-A)^2/4} + e^{(u+A)^2/4} \Big) \int_{u-A}^{u+A} e^{-x^2/4} dx \geq \frac{8}{C} A^3 . \]
To see that this inequality holds, we may assume without loss of generality that $u \geq 0$ and estimate, using (\ref{eq_ex_choice_C}),
\begin{multline*}
 e^{(u+A)^2/4} \int_{u-A}^{u+A} e^{-x^2/4} dx 
\geq  \int_{u }^{u+\frac12 A} \exp \Big( \frac14 \big( (u+A)^2 - x^2 \big) \Big) dx \\
\geq \frac12 A \exp \Big( \frac14 \big( (u+A)^2 - (u-\tfrac12 A)^2 \big) \Big) 
\geq \frac12 A \exp \Big( \frac{A^2}{16} \Big) \geq \frac{8}{C} A^3.
\end{multline*}
To verify the reproduction formula, Property~\ref{Def_metric_flow_7}, we compute that for $t_1, t_2 > 0$
\[ \Big( \frac12 + \frac12 e^{-Ct_1 /D^2} \big) \Big( \frac12 + \frac12 e^{-Ct_2 /D^2} \Big) + \Big( \frac12 - \frac12 e^{-Ct_1 /D^2} \big) \Big( \frac12 - \frac12 e^{-Ct_2 /D^2} \Big) 
= \frac12 + \frac12 e^{-C(t_1+t_2)/D^2}. \]
So $\XX_{C,D}$ is a metric flow.
Any two flows $\XX_{C,D_1}, \XX_{C,D_2}$ are parabolic rescalings of one another.

We can compute that for any $p > 0$ and $s \leq t$
 \[ \int_{\XX_s} \int_{\XX_s} d^p_s (y_1, y_2) d\nu_{\pm \frac12 D, t;s}(y_1) d\nu_{\pm \frac12 D,t;s} (y_2) = \frac{D^p}2 ( 1 - e^{-C (t-s)/D^2} ) . \]
 So if $p = 2$, then
 \[ \Var ( \nu_{ \pm \frac12 D, t; s} ) = \frac{D^2}2 ( 1 - e^{-C (t-s)/D^2} ) \leq \frac12  C (t-s), \]
 which shows that $\XX_{C,D}$ is $\frac12  C$-concentrated.
However, if $p > 2$, then the following bound is false for any $C' < \infty$
 \[   \int_{\XX_s} \int_{\XX_s} d^p_s (y_1, y_2) d\nu_{\pm \frac12 D,t;s}(y_1) d\nu_{\pm \frac12 D,t;s} (y_2) \leq C' (t-s)^{p}. \]
This bound holds on a super Ricci flow due to Gaussian concentration; see \cite{Hein-Naber-14}, \cite[\HKThmGaussianintegral]{Bamler_HK_entropy_estimates}.
So Gaussian concentration does not follow from the axioms of a metric flow.
\end{Example}

\section{Geometry and continuity of time-slices of metric flows} \label{sec_geom_cont_time_slices}
Our goal in this section is to study how the geometry of time-slices $(\XX_t, d_t)$ of an $H$-concentrated metric flow $\XX$ changes in time.
We recall that a metric flow does not specify any worldlines, i.e. it does not record whether two points $x_i \in \XX_{t_i}$ from different time-slices correspond ``to the same point at different times''.
Instead, given a point $x \in \XX_t$ and an earlier time $s < t$, we will consider the conjugate heat kernel $\nu_{x;s}$ and we will regard $\nu_{x;s}$ as the ``probability distribution of the points corresponding to $x$ at an earlier time $s \leq t$''.
We may also think of the $H$-centers of $x$ at time $s$ to be the points corresponding to $x$.
By Proposition~\ref{Prop_H_center_existence}, these $H$-centers are determined up to an ``error'' of $2 \sqrt{H(t-s)}$.
Note this viewpoint is slightly different from the conventional concept of worldlines. 
If $\XX$ corresponds to a conventional (super) Ricci flow, then $x,z$ may not correspond to the same points; moreover the point $x' \in \XX_s$ that lies on the same worldline of $x \in \XX_t$ may be far from $H$-centers of $x$ and may therefore not --- or to a lesser degree --- correspond to $x$ in the above sense.

Observe that by Proposition~\ref{Prop_compare_CHF}\ref{Prop_compare_CHF_c}, for any two points $x_1, x_2 \in \XX_t$ we have
\[ d_{W_1}^{\XX_s} (\nu_{x_1;s}, \nu_{x_2;s}) \leq d_t (x_1, x_2), \]
which can be regarded as form of distance distortion estimate, i.e. distances only expand in time in this sense.
More specifically, if $x'_1, x'_2 \in \XX_s$ denote $H$-centers of $x_1, x_2$, then we have the following distance shrinking estimate by Lemma~\ref{Lem_W_1_vs_Var}
\begin{multline*}
 d_s (x'_1, x'_2) \leq d_{W_1}^{\XX_s} (\delta_{x'_1}, \nu_{x_1;s}) + d_{W_1}^{\XX_s} (\nu_{x_1;s}, \nu_{x_2;s}) + d_{W_1}^{\XX_s} (\nu_{x_1;s}, \delta_{x'_2}) \\
 \leq \sqrt{\Var (\delta_{x'_1}, \nu_{x_1;s})} + d_t (x_1, x_2) + \sqrt{\Var (\delta_{x'_1}, \nu_{x_1;s})} 
\leq d_t (x_1, x_2)  +2 \sqrt{H(t-s)}.
\end{multline*}
A reverse bound, on the \emph{expansion} of the distance, is in general harder to come by.
This will be one of the main issues addressed in this section.

\subsection{Mass distribution on time-slices}
Recall the mass distribution function $b^{(X,d,\mu)}_r$ for a metric measure space $(X, d, \mu)$ at scale $r$ from Subsection~\ref{subsec_compactness}.
The following proposition gives a lower bound on this function on time-slices $\XX_t$ of a metric flow $\XX$ equipped with a conjugate heat flow $(\mu_t)_{t \in I'}$ of bounded variance.
So we obtain that these time-slices represent classes in certain spaces of the form $\mathbb{M}_r (V, b)$, which are compact by Theorem~\ref{Thm_M_compact}.

\begin{Proposition} \label{Prop_mass_distribution}
Let $\XX$ be an $H$-concentrated metric flow over some subset $I \subset \IR$, $r > 0$ and let $(\mu_t)_{t \in I'}$, $I' \subset I$, be a conjugate heat flow on $\XX$ with $\sup_{t \in I'} \Var(\mu_t) \leq V r^2$.
Suppose that $t, t + \tau r^2 \in I'$ for $\tau  > 0$.
Then
\[ b^{(\XX_t, d_t, \mu_t)}_r (\eps) \geq \frac12 \Phi \bigg({ - \sqrt{ \frac{8V}{\eps \tau} } }\bigg) \qquad \text{if} \quad  \eps \in \big[ 2 ( \tau H)^{1/3} , 1 \big]. \]
In particular, if there is a sequence $\tau_i \searrow 0$ with $t + \tau_i r^2 \in I'$, then there is  a function $b_{H, V, (\tau_i)} : (0,1) \to (0,1)$, depending only on $H, V, (\tau_i)$,  such that $b^{(\XX_t, d_t, \mu_t)}_r \geq b_{H, V, (\tau_i)}$ and therefore $(\XX_t, d_t, \mu_t) \in \mathbb{M}(V, b_{H, V, (\tau_i)})$.
\end{Proposition}

Note that if $\mu_t = \nu_{x;t_0}$ for some $x \in \XX_{t_0}$, $t_0 > t$, then we can choose $V = H (t_0 - t) r^{-2}$.

\begin{proof}
After parabolic rescaling, we may assume that $r = 1$.
Fix some $\eps \in [2 ( \tau H)^{1/3},1]$ and set
\[  D := \sqrt{\frac{2V}{\eps}} , \qquad \delta :=  \frac12 \Phi \bigg({ - \sqrt{ \frac{8V}{\eps \tau} } }\bigg) = \frac12 \Phi (- 2 \tau^{-1/2} D). \]
Let
\[ t' := t + \tau, \qquad Q_{\eps, \delta} := \{ x \in \XX_t \;\; : \;\; \mu_t (D(x, \eps)) \geq \delta \}. \]
Our goal will be to show that 
\begin{equation} \label{eq_mu_complement_Q}
 \mu_t (\XX_t \setminus Q_{\eps, \delta}) \leq \eps. 
\end{equation}
For this purpose choose $p \in \XX_{t'}$ such that
\[ \int_{\XX_{t'}} d^2_{t'} (p, \cdot) \, d\mu_{t'} = \Var (\delta_p, \mu_{t'}) \leq \Var (\mu_{t'}) \leq V. \]
So
\begin{equation} \label{eq_outside_BpD}
 \mu_{t'} ( \XX_{t'} \setminus B(p,D)) \leq V D^{-2} = \eps/2 \leq \tfrac12. 
\end{equation}
We also note that for any $H$-center $z \in \XX_t$ of a point $x \in \XX_{t'}$ we have
\begin{equation} \label{eq_nu_complement_B_z}
 \nu_{x;t} ( \XX_t \setminus B(z,\eps/2) ) \leq (\eps/2)^{-2} \Var (\nu_{x;t}) \leq (\eps/2)^{-2} \tau H \leq \eps/2 \leq \tfrac12.
\end{equation}

Let $Z \subset \XX_t$ be the set of points that are $H$-centers at time $t$ of some point in $B(p, D)$ and denote by $Z_{\eps/2} := B(Z, \eps/2) \subset \XX_t$ its $\eps/2$-neighborhood.
We claim that 
\begin{equation} \label{eq_Z_subset_Q}
 Z_{\eps/2} \subset Q_{\eps, \delta}. 
\end{equation}
To see this, let $z' \in Z_{\eps/2}$ and choose an $H$-center $z \in Z$ of some $x \in B(p,D)$ with $d_t (z,z') < \eps/2$.
By (\ref{eq_nu_complement_B_z}) we have
\[ \nu_{x;t} (B(z', \eps)) \geq \nu_{x;t} (B(z, \eps/2)) \geq  \tfrac12. \]
So by Definition~\ref{Def_metric_flow}\ref{Def_metric_flow_6} we have $\nu_{\cdot; t} (B(z', \eps)) \geq \Phi (- 2\tau^{-1/2} D ) = 2 \delta$ on $B(p,D)$, which implies by the reproduction formula and (\ref{eq_outside_BpD})
\[
 \mu_t (B(z', \eps)) = \int_{\XX_{t'}} \nu_{x,t} (B(z', \eps)) \mu_{t'} (x) 
\geq 2 \delta \mu_{t'} (B(p,D))  \\
\geq   \delta,
\]
and therefore that $z' \in Q_{\eps, \delta}$, as desired.

By (\ref{eq_Z_subset_Q}), it suffices to show that $\mu_t (\XX_t \setminus Z_{\eps/2}) \leq \eps$ in order to prove (\ref{eq_mu_complement_Q}).
To see this, observe that for every $x \in B(p,D)$ and every $H$-center $z \in \XX_t$ of $x$ we have by (\ref{eq_nu_complement_B_z})
\[ \nu_{x; t} ( \XX_t \setminus Z_{\eps /2}) \leq \nu_{x;t} ( \XX_t \setminus B(z, \eps/2) ) \leq \eps/2. \]
So by the reproduction formula and (\ref{eq_outside_BpD}) we have
\[
 \mu_t ( \XX_t \setminus Z_{\eps/2} ) 
= \int_{\XX_{t'}} \nu_{x; t} ( \XX_t \setminus Z_\eps ) d\mu_{t'} (x) \\
\leq \int_{B(p, D)}(\eps/2) \, d\mu_{t'}  + \int_{\XX_{t'} \setminus B(p, D)} 1 \, d\mu_{t'}  \leq \eps.
\]
This finishes the proof.
\end{proof}
\bigskip

\subsection{Geometric closeness of nearby time-slices}
The goal of this subsection will be to establish geometric closeness of nearby time-slices of an $H$-concentrated metric flow $\XX$.
For this purpose, we will consider a conjugate heat flow $(\mu_t)_{t \in I'}$ and compare the metric measure spaces $(\XX_t, \mu_t)$ for $t \in I$.
We will show that for nearby times $s, t \in I'$, $s \leq t$, the distance $d_{GW_1} ( (\XX_s, \lb \mu_s),  \lb (\XX_t, \lb \mu_t))$ between these spaces is small if and only the following difference is small:
\begin{equation} \label{eq_intint_d_diff}
 \int_{\XX_t} \int_{\XX_t} d_t \, d\mu_t d\mu_t -  \int_{\XX_s} \int_{\XX_s} d_s \, d\mu_s d\mu_s . 
\end{equation}
We will also show that this closeness is described by the following coupling between $\mu_s, \mu_t$
\[ q := \int_{\XX_t} (\nu_{y;s} \otimes \delta_y) \, d\mu_t (y). \]
So, essentially, a map that assigns to any point $y \in \XX_t$ one of its $H$-centers in $\XX_s$ can be regarded as some sort of almost isometry between $(\XX_s, \mu_s)$, $(\XX_t, \mu_t)$.

The following lemma, which will be needed later, illustrates the relevance of the difference (\ref{eq_intint_d_diff}).
Namely, it states that (\ref{eq_intint_d_diff}) is small if $t-s$ and $\Var (\mu_t) - \Var (\mu_s)$ are small.
So, due to the monotonicity of $\Var (\mu_t) + Ht$, this will imply smallness of (\ref{eq_intint_d_diff}) for most $s \leq t$ with $t-s \ll 1$.
We also obtain almost monotonicity of $t \mapsto \int_{\XX_t} \int_{\XX_t} d_t \, d\mu_t d\mu_t$.

\begin{Lemma} \label{Lem_intd_diff_Var_diff}
Let $\XX$ be an $H$-concentrated metric flow over $I$, $(\mu_t)_{t \in I'}$, $I' \subset I$, a conjugate heat flow on $\XX$ and let $s, t \in I'$, $s \leq t$, be two times.
Then for any $y_1, y_2 \in \XX_t$
\begin{equation} \label{eq_f_geq_3_sqrtH}
0 \leq  f(y_1, y_2) :=  d_t (y_1, y_2) - d_{W_1}^{\XX_s} (\nu_{y_1;s}, \nu_{y_2;s}) \leq d_t (y_1, y_2) - \int_{\XX_s} \int_{\XX_s}  d_s \, d\nu_{y_1; s} d\nu_{y_2; s} 
+  \sqrt{H(t-s)} 
\end{equation}
and we have the integral bound
\begin{multline} \label{eq_intint_f_bounds}
-  \sqrt{H (t-s)} \leq  \int_{\XX_t} \int_{\XX_t} f \, d\mu_t d\mu_t -  \sqrt{H (t-s)}  \leq  \int_{\XX_t} \int_{\XX_t} d_t \, d\mu_t d\mu_t -  \int_{\XX_s} \int_{\XX_s} d_s \, d\mu_s d\mu_s  \\
\leq  \sqrt{ \Var( \mu_t)  - \Var(\mu_s) + H (t-s) }   + 2 \sqrt{H (t-s)}. 
\end{multline}
\end{Lemma}

\begin{proof}
Let $y_1, y_2 \in \XX_t$.
For any coupling $q$ between $\nu_{y_1;s}, \nu_{y_2;s}$ we have
\begin{align*}
 \int_{\XX_s \times \XX_s} & d_s(x_1, x_2) \, dq(x_1, x_2)
= \int_{\XX_s}\int_{\XX_s \times \XX_s} d_s(x_1, x_2) \,  dq(x_1, x_2) d\nu_{y_2;s}(x'_2) \\
&\geq \int_{\XX_s} \int_{\XX_s \times \XX_s} \big( d_s(x_1, x'_2) - d_s(x'_2, x_2) \big) \,   dq(x_1, x_2) d\nu_{y_2;s}(x'_2) \displaybreak[1] \\
&\geq \int_{\XX_s}\int_{\XX_s} d_s (x_1, x'_2) \, d\nu_{y_1;s}(x_1) d\nu_{y_2;s}(x'_2) - \int_{\XX_s}\int_{\XX_s} d_s (x'_2, x_2) \, d\nu_{y_2;s}(x_2) d\nu_{y_2;s}(x'_2)  \\
&\geq \int_{\XX_s}\int_{\XX_s} d_s \, d\nu_{y_1;s} d\nu_{y_2;s} - \sqrt{\Var( \nu_{y_2;s})}
\geq \int_{\XX_s}\int_{\XX_s} d_s \, d\nu_{y_1;s} d\nu_{y_2;s} - \sqrt{H(t-s)}.
\end{align*}
Together with Proposition~\ref{Prop_compare_CHF}\ref{Prop_compare_CHF_c}, this implies (\ref{eq_f_geq_3_sqrtH}).

For (\ref{eq_intint_f_bounds}), observe that for any $x_1, x_2 \in \XX_s$
\begin{multline*}
 \Big| d_s (x_1, x_2) -  \sqrt{\Var (\nu_{y_1;s}, \nu_{y_2;s})} \Big|
= \Big| \sqrt{\Var (\delta_{x_1}, \delta_{x_2})} - \sqrt{\Var (\nu_{y_1;s}, \nu_{y_2;s})} \Big| \\
\leq  \sqrt{ \Var (\delta_{x_1}, \nu_{y_1; s})} + \sqrt{ \Var( \delta_{x_2}, \nu_{y_2;s})}. 
\end{multline*}
Integration over $x_1, x_2$ implies
\begin{align}
\bigg| \int_{\XX_s} \int_{\XX_s} & d_s \, d\nu_{y_1; s} d\nu_{y_2; s}
- \sqrt{\Var (\nu_{y_1;s}, \nu_{y_2;s})}  \bigg| \notag \\
&\leq \int_{\XX_s} \sqrt{ \Var (\delta_{x_1}, \nu_{y_1; s})} d\nu_{y_1; s} (x_1)
+ \int_{\XX_s} \sqrt{ \Var (\delta_{x_2}, \nu_{y_2; s})} d\nu_{y_2; s} (x_2) \notag \\
&\leq  \bigg({ \int_{\XX_s}  \Var (\delta_{x_1}, \nu_{y_1; s}) d\nu_{y_1; s} (x_1) }\bigg)^{1/2}
+ \bigg({ \int_{\XX_s}  \Var (\delta_{x_2}, \nu_{y_2; s}) d\nu_{y_2; s} (x_2) }\bigg)^{1/2} \notag \\
&\leq \sqrt{ \Var (\nu_{y_1; s}) } + \sqrt{ \Var (\nu_{y_2; s}) } 
\leq  2\sqrt{H(t-s)}. \label{eq_int_sqrt_Var_2H}
\end{align}
Since
\[ d^2_{t} (y_1,y_2) - \Var (\nu_{y_1;s}, \nu_{y_2;s}) + H (t-s) \geq 0, \]
we have
\[ d_t (y_1, y_2) - \sqrt{\Var (\nu_{y_1;s}, \nu_{y_2;s})} \leq \sqrt{ d^2_{t} (y_1,y_2) - \Var (\nu_{y_1;t}, \nu_{y_2;t}) + H (t-s) }. \]
Combining this with (\ref{eq_int_sqrt_Var_2H}) implies that
\begin{equation*} \label{eq_pointwise_f_bound}
 d_t (y_1, y_2) - \int_{\XX_s} \int_{\XX_s}  d_s \, d\nu_{y_1; s} d\nu_{y_2; s} 
\leq \sqrt{ d^2_{t} (y_1,y_2) - \Var (\nu_{y_1;s}, \nu_{y_2;s}) + H (t-s) } + 2 \sqrt{H(t-s)}. 
\end{equation*}
Now (\ref{eq_intint_f_bounds}) follows by integrating this bound over $y_1, y_2$, using (\ref{eq_f_geq_3_sqrtH}) and the bound
\begin{align*}
 \int_{\XX_t} \int_{\XX_t} & \sqrt{ d^2_{t} (y_1,y_2) - \Var (\nu_{y_1;s}, \nu_{y_2;s}) + H (t-s) } d\mu_t (y_1) d\mu_t (y_2) \\
&\leq \bigg({ \int_{\XX_t} \int_{\XX_t} \big(d^2_{t} (y_1,y_2) - \Var (\nu_{y_1;s}, \nu_{y_2;s}) + H (t-s) \big)  d\mu_t (y_1) d\mu_t (y_2) }\bigg)^{1/2} \\
&= \sqrt{ \Var( \mu_t)  - \Var(\mu_s) + H (t-s) } \qedhere
\end{align*}
\end{proof}
\medskip

The next lemma shows the reverse direction of our goal. 
It states that closeness of two metric measure spaces in the $GW_1$-sense implies smallness of (\ref{eq_intint_d_diff}).

\begin{Lemma} \label{Lem_W1_close_intint_d_close}
Let $(X_i, d_i, \mu_i)$, $i =1,2$, be two metric measure spaces.
Then
\[ \bigg| \int_{X_1} \int_{X_1} d_1 \, d\mu_1 d\mu_1 - \int_{X_2} \int_{X_2} d_2 \, d\mu_2 d\mu_2 \bigg| \leq 2 d_{GW_1} \big( (X_1, d_1, \mu_1), (X_2, d_2, \mu_2) \big). \]
\end{Lemma}

\begin{proof}
Let $\eps > 0$ and consider isometric embeddings $\varphi_i : (X_i, d_i) \to (Z, d)$, $i=1,2$, and a coupling $q$ between $\mu_1, \mu_2$ such that
\[ \int_{X_1 \times X_2}  d_Z (\varphi_1 (x_1), \varphi_2 (x_2)) dq(x_1, x_2) \leq  d_{GW_1} \big( (X_1, d_1, \mu_1), (X_2, d_2, \mu_2) \big) + \eps. \]
Then
\begin{align*}
 \bigg| \int_{X_1} \int_{X_1} & d_1 \, d\mu_1 d\mu_1 - \int_{X_2} \int_{X_2} d_2 \, d\mu_2 d\mu_2 \bigg| \\
&= \bigg| \int_{X_1 \times X_2}  \int_{X_1 \times X_2} \big( d_1 (x_1, y_1) -d_2(x_2, y_2) \big) dq (y_1, y_2) dq (x_1, x_2)  \bigg| \\
&\leq \int_{X_1 \times X_2}  \int_{X_1 \times X_2} \big| d_1 (x_1, y_1) -d_2(x_2, y_2) \big| dq (y_1, y_2) dq (x_1, x_2) \\
&\leq \int_{X_1 \times X_2}  \int_{X_1 \times X_2} \big( d_Z (\varphi_1(x_1),\varphi_2 ( x_2)) + d_Z (\varphi_1(y_1),\varphi_2 ( y_2))\big) dq (y_1, y_2) dq (x_1, x_2) \\
&\leq 2d_{GW_1} \big( (X_1, d_1, \mu_1), (X_2, d_2, \mu_2) \big) + 2\eps.
\end{align*}
Letting $\eps \to 0$ finishes the proof.
\end{proof}

Next, we will show that smallness of (\ref{eq_intint_d_diff}) and $t-s$ implies smallness of $d_{GW_1} ( (\XX_s, \lb \mu_s),  \lb (\XX_t, \lb \mu_t))$.
The following lemma will equip us with the necessary distance distortion estimate. 

\begin{Lemma} \label{Lem_dist_distor_2_pts}
Let $\XX$ be an $H$-concentrated metric flow over $I$ and $(\mu_t)_{t \in I'}$, $I' \subset I$ a conjugate heat flow.
Suppose that for two times $s, t \in I'$, $s \leq t$ we have for $\alpha, \beta, \gamma, r > 0$
\[ t - s \leq \alpha r^2, \qquad \int_{\XX_t} \int_{\XX_t} d_t \, d\mu_t d\mu_t -  \int_{\XX_s} \int_{\XX_s} d_s \, d\mu_s d\mu_s \leq \beta r. \]
Then for any $y_1, y_2 \in \XX_t$ for which
\[ \mu_{t} ( B(y_1, r) ) \mu_{t} ( B(y_2, r) ) \geq \gamma > 0 \]
we have
\begin{equation} \label{eq_d_t_m_Var}
0 \leq d_t(y_1,y_2) - d_{W_1}^{\XX_s} (\nu_{y_1;s} , \nu_{y_2;s})  \leq \Big( \frac{\beta + 3 \sqrt{H\alpha}}{\gamma} +4 \Big) r. 
\end{equation}
\end{Lemma}

\begin{proof}
Define $f : \XX_t \times \XX_t \to \IR$ as in (\ref{eq_f_geq_3_sqrtH}) and observe that by Proposition~\ref{Prop_compare_CHF} $f$ is $2$-Lipschitz in each variable.
Moreover, by Lemma~\ref{Lem_intd_diff_Var_diff}
\[ f  \geq 0, 
\qquad \int_{\XX_t} \int_{\XX_t} f \, d\mu_t d\mu_t \leq \beta r + 3 \sqrt{H\alpha}\,r. \]
Then there are $y'_i \in B(y_i, r)$ with 
\[  f(y'_1, y'_2)  \leq \frac{\beta + 3 \sqrt{H\alpha}}{\gamma} \, r . \]
The upper bound in (\ref{eq_d_t_m_Var}) follows by combining this with
\[ f(y_1, y_2) \leq f(y'_1, y_2) + 2r \leq f(y'_1, y'_2) + 4r. \]
The lower bound in (\ref{eq_d_t_m_Var}) is clear.
\end{proof}
\medskip

The following proposition characterizes the closeness of two nearby time-slices under certain conditions.

\begin{Proposition} \label{Prop_time_s_closeness}
For any $\eps > 0$, $H, V < \infty$ and any function $b : (0,1) \to (0,1)$ there is a $\delta ( H, V,  b, \eps) > 0$ such that the following holds.

Let $\XX$ be an $H$-concentrated metric flow over $I$ and $(\mu_t)_{t \in I'}$, $I' \subset I$, a conjugate heat flow on $\XX$.
Suppose that for two times $s,t \in I'$, $s \leq t$ and $r > 0$ we have
\[ b^{(\XX_t, d_t, \mu_t)}_r  \geq b \qquad \text{on} \quad [\delta, 1] \]
and
\[  t - s \leq \delta r^2, \qquad \Var (\mu_{t}) \leq V r^2, \qquad 
 \int_{\XX_t \times \XX_t} d_t \, d\mu_t d\mu_t - \int_{\XX_s \times \XX_s} d_s \, d\mu_s d\mu_s \leq \delta r. \]
Then there is a closed subset $W \subset \XX_{t}$ such that:
\begin{enumerate}[label=(\alph*)]
\item \label{Prop_time_s_closeness_a} $\mu_{t} ( \XX_{t} \setminus W ) \leq \eps$.
\item \label{Prop_time_s_closeness_b} For any $y_1,y_2 \in W$ we have
\begin{equation} \label{eq_d_m_d_d_m_V}
 0\leq d_t(y_1,y_2) - d_{W_1}^{\XX_s} (\nu_{y_1;s} , \nu_{y_2;s}) \leq \eps r. 
\end{equation}
\end{enumerate}
Moreover, there is a metric space $(Z, d_Z)$ and isometric embeddings $\varphi_s : \XX_s \to Z$, $\varphi_t : \XX_t \to Z$ such that:
\begin{enumerate}[label=(\alph*), start=3]
\item \label{Prop_time_s_closeness_c} For every $x \in \XX_s$ and $y \in W$ we have 
\begin{equation} \label{eq_d_Z_bound}
 d_Z (\varphi_s (x), \varphi_t (y))  \leq d_{W_1}^{\XX_s} (\delta_{x}, \nu_{y;s}) + \eps r \leq \sqrt{\Var (\delta_{x}, \nu_{y;s})} + \eps r. 
\end{equation}
\item \label{Prop_time_s_closeness_d} The probability measure
\[ q := \int_{\XX_t} (\nu_{y;s} \otimes \delta_y) \, d\mu_t (y). \]
is a coupling between $\mu_s, \mu_t$ and
\begin{equation} \label{eq_int_Xs_Xt_d_cross}
 \int_{\XX_s \times \XX_t}  d_Z (\varphi_s (x), \varphi_t (y)) dq(x,y) =  \int_{\XX_t} \int_{\XX_s} d_Z (\varphi_s (x), \varphi_t (y)) d\nu_{y;s} (x) d\mu_t (y) \leq \eps r. 
\end{equation}
\item \label{Prop_time_s_closeness_e} We have
\[
 d_{GW_1} \big( (\XX_s, d_s, \mu_s), (\XX_t, d_t, \mu_t) \big) \leq d^Z_{W_1} ( (\varphi_s)_* \mu_s, (\varphi_t)_* \mu_t ) \leq \eps r. 
\]
\end{enumerate}
\end{Proposition}

\begin{proof}
After parabolic rescaling we may assume that $r = 1$.
Fix $V, H, b, \eps$.
We will determine $\delta$ in the course of the proof.
Let $\zeta \in (0,1)$ be a constant whose value we will determine later and choose
\[ W := \big\{ y \in \XX_t \;\; : \;\; \mu_t (D(y, \zeta)) \geq b(\zeta) \big\}. \]
Then
\[ \mu_t (\XX_t \setminus W ) \leq \zeta, \]
which implies Assertion~\ref{Prop_time_s_closeness_a} for $\zeta \leq \eps$.
Applying Lemma~\ref{Lem_dist_distor_2_pts} with $r = \zeta$ implies that if $\delta \leq \ov\delta (H, \zeta, b)$, then
\[ 0 \leq d_t(y_1,y_2) - d_{W_1}^{\XX_s} (\nu_{y_1;s} , \nu_{y_2;s}) \leq  \Big( \frac{\delta + 3\sqrt{H\delta}}{b^2(\zeta)} + 4 \Big) \zeta  \leq 5\zeta. \]
This proves Assertion~\ref{Prop_time_s_closeness_b} if $\zeta \leq \eps/5$.

The fact that $q$ in Assertion~\ref{Prop_time_s_closeness_d} is a coupling between $\mu_s, \mu_t$ and the equality in (\ref{eq_int_Xs_Xt_d_cross}) are clear.
Assertion~\ref{Prop_time_s_closeness_c} and the inequality in (\ref{eq_int_Xs_Xt_d_cross}) follow from Lemma~\ref{Lem_construction_Z} below.
Assertion~\ref{Prop_time_s_closeness_e} is a direct consequence of Assertion~\ref{Prop_time_s_closeness_d}.
\end{proof}
\bigskip

\begin{Lemma} \label{Lem_construction_Z}
Let $0 < \delta \leq \eps$ and $V, H < \infty$.
Let $\XX$ be a metric flow over $I$ and consider two times $s \leq t$, $s, t \in I$.
Suppose that there is a non-empty, measurable subset $W \subset \XX_t$ such that for any $y_1, y_2 \in W$ we have
\begin{equation} \label{eq_dtVar_Z_constr}
0 \leq d_t(y_1,y_2) - d_{W_1}^{\XX_s} (\nu_{y_1;s} , \nu_{y_2;s})  \leq \delta. 
\end{equation}
Then there is a metric space $(Z, d_Z)$ and isometric embeddings $\varphi_s : \XX_s \to Z$, $\varphi_t : \XX_t \to Z$ such that (\ref{eq_d_Z_bound}) in Proposition~\ref{Prop_time_s_closeness} holds for $r = 1$.

Moreover, suppose that $\XX$ is $H$-concentrated and consider a conjugate heat flow $(\mu_t)_{t \in I'}$ on $\XX$ with $s, t \in I'$.
If $\delta \leq \ov\delta (H,V, \eps)$ and
\begin{equation} \label{eq_Lem_Z_additional_bounds}
 t-s \leq \delta, \qquad \Var (\mu_t) \leq V, \qquad \mu_t (\XX_t \setminus W^\delta) \leq \delta, 
\end{equation}
where $W^\delta := B(W, \delta)$, then (\ref{eq_int_Xs_Xt_d_cross}) in Proposition~\ref{Prop_time_s_closeness} holds for $r=1$.
\end{Lemma}

We will apply Lemma~\ref{Lem_construction_Z} again in the proof of Theorem~\ref{Thm_future_cont_equiv}, where we will also make use of the $\delta$-neighborhood $W^\delta$. 

\begin{proof}
Let $Z := \XX_s \sqcup \XX_t$ and denote by $\varphi_s, \varphi_t$ the standard immersions.
Define $d_Z$ to be equal to $d_s, d_t$ on $\XX_s, \XX_t$, respectively, and for $x \in \XX_s$, $y \in \XX_t$ let
\begin{equation} \label{eq_def_d_Z_Var}
 d_Z( \varphi_s (x), \varphi_t (y) ) = d_Z(  \varphi_t (y) , \varphi_s (x))
:= \inf_{w \in W} \big( d_t (y,w) + d_{W_1}^{\XX_s} (\delta_{x}, \nu_{w;s}) \big) + \delta. 
\end{equation}
We need to verify that $d_Z$ satisfies the triangle inequality.
For this purpose choose $x_1, x_2 \in \XX_s$, $y_1, y_2 \in \XX_t$.
Then
\[ d_Z ( \varphi_s (x_1), \varphi_t (y_1) ) 
\leq d_Z (\varphi_s (x_1), \varphi_t (y_2)) + d_Z ( \varphi_t (y_2), \varphi_t (y_1) ) \]
is a direct consequence of (\ref{eq_def_d_Z_Var}) and the triangle inequality on $\XX_t$.
The bound
\[ d_Z ( \varphi_s (x_1), \varphi_t (y_1) ) \leq d_Z (\varphi_s (x_1), \varphi_s (x_2)) + d_Z ( \varphi_s (x_2), \varphi_t (y_1) ). \]
follows using
\[
d_{W_1}^{\XX_s} (\delta_{x_1}, \nu_{w;s})
\leq d_{W_1}^{\XX_s} (\delta_{x_1}, \delta_{x_2})  +  d_{W_1}^{\XX_s}(\delta_{x_2}, \nu_{w;s})
= d_s (x_1, x_2) +  d_{W_1}^{\XX_s} (\delta_{x_2}, \nu_{w;s}). 
\]
Next, we have
\begin{align*}
 d_Z ( \varphi_t (y_1), \varphi_t (y_2) ) 
 &\leq  \inf_{w_1, w_2 \in W} \big( d_t (y_1, w_1) + d_t (y_2, w_2) + d_t (w_1, w_2) \big) \\
 &\leq  \inf_{w_1, w_2 \in W} \big( d_t (y_1, w_1) + d_t (y_2, w_2) + d_{W_1}^{\XX_s} (\nu_{w_1;s}, \nu_{w_2;s})  + \de \big)  \\
 &\leq \inf_{w_1, w_2 \in W} \big( d_t (y_1, w_1) + d_t (y_2, w_2) + d_{W_1}^{\XX_s} (\delta_{x_1}, \nu_{w_1;s}) +  d_{W_1}^{\XX_s} (\nu_{w_2;s}, \delta_{x_1})  \big) + 2\delta \\
&= d_Z ( \varphi_t (y_1), \varphi_s (x_1) ) + d_Z ( \varphi_s (x_1), \varphi_t (y_2) )
\end{align*}
and
\begin{align*}
 d_Z ( \varphi_s (x_1), \varphi_s (x_2) ) 
&= d_{W_1}^{\XX_s} (\delta_{x_1}, \delta_{x_2}) \\
&\leq  \inf_{w_1, w_2 \in W} \big( d_{W_1}^{\XX_s} (\delta_{x_1}, \nu_{w_1;s}) + d_{W_1}^{\XX_s}(\nu_{w_1;s}, \nu_{w_2;s}) + d_{W_1}^{\XX_s} ( \nu_{w_2;s}, \delta_{x_2})  \big) \\
&\leq  \inf_{w_1, w_2 \in W} \big( d_{W_1}^{\XX_s}(\delta_{x_1}, \nu_{w_1;s}) + d_t(w_1,w_2) +  d_{W_1}^{\XX_s} ( \nu_{w_2;s}, \delta_{x_2})  \big)  \\
&\leq \inf_{w_1, w_2 \in W} \big(  d_{W_1}^{\XX_s} (\delta_{x_1}, \nu_{w_1;s})  + d_t (w_1, y_1) + d_t (y_1, w_2) + d_{W_1}^{\XX_s} (\delta_{x_2}, \nu_{w_2;s}) \big) + 2\delta \\
&= d_Z (\varphi_s (x_1), \varphi_t (y_1)) + d_Z ( \varphi_t (y_1), \varphi_s (x_2) ). 
\end{align*}
This shows that $(Z, d_Z)$ is a metric space and (\ref{eq_d_Z_bound}) in Proposition~\ref{Prop_time_s_closeness} holds if $\delta \leq \eps$, because
\[ d_{W_1}^{\XX_s} (\delta_x, \nu_{y;s}) \leq \sqrt{\Var(\delta_x, \nu_{y;s})} . \]

Before continuing, we observe that for any $x \in \XX_s$ and $y \in W^\delta$ there is a $y' \in W$ with $d_t(y,y') < \delta$ and therefore, using Proposition~\ref{Prop_compare_CHF}\ref{Prop_compare_CHF_c},
\begin{multline} \label{eq_dW1_Var_delta}
 d_Z (\varphi_s(x), \varphi_t(y)) \leq d_Z (\varphi_s(x), \varphi_t(y')) + \delta
\leq d_{W_1}^{\XX_s} (\delta_x, \nu_{y';s}) +2 \delta \\
\leq d_{W_1}^{\XX_s} (\delta_x, \nu_{y;s}) + d_{W_1}^{\XX_s} (\nu_{y;s}, \nu_{y';s}) + 2\delta
\leq d_{W_1}^{\XX_s} (\delta_x, \nu_{y;s}) + 3 \delta
\leq \sqrt{\Var(\delta_x, \nu_{y;s})} + 3\delta. 
\end{multline}

Next, assume that (\ref{eq_Lem_Z_additional_bounds}) holds.
Then, using (\ref{eq_dW1_Var_delta}),
\begin{align}
 \int_{W^\delta} \int_{\XX_s} &d_Z (\varphi_s (x), \varphi_t (y)) \, d\nu_{y;s} (x) d\mu_t (y) \leq 
  \int_{W^\delta} \int_{\XX_s}  \sqrt{\Var (\delta_{x}, \nu_{y;s})} d\nu_{y;s} (x) d\mu_t (y) + 3\delta \notag \\
  &\leq \mu_t^{1/2} (W^\delta) \bigg({ \int_{W^\delta} \int_{\XX_s}  \Var (\delta_{x}, \nu_{y;s}) d\nu_{y;s} (x) d\mu_t (y) }\bigg)^{1/2} +3 \delta \notag \\
  &\leq  \bigg({ \int_{\XX_t}   \Var (\nu_{y;s}, \nu_{y;s})  d\mu_t (y) }\bigg)^{1/2} + 3\delta
  \leq \sqrt{H (t-s)} + 3 \delta \leq \sqrt{H \delta} + 3\delta. \label{eq_int_W_bound}
\end{align}
Assuming $\delta \leq \frac12$, we moreover have $\mu_t (W^\delta) \geq \frac12$, which allows us to bound
\begin{align}
\frac12 \int_{\XX_t \setminus W^\delta} \int_{\XX_s} &d_Z (\varphi_s (x), \varphi_t (y)) d\nu_{y;s} (x) d\mu_t (y) \notag \\
&\leq \int_{W^\delta} \int_{\XX_t \setminus W^\delta} \int_{\XX_s} d_Z (\varphi_s (x), \varphi_t (y)) d\nu_{y;s} (x) d\mu_t (y) d\mu_t (w) \notag \\
&\leq \int_{W^\delta} \int_{\XX_t \setminus W^\delta} \int_{\XX_s} \big( d_Z (\varphi_s (x), \varphi_t (w)) + d_Z (\varphi_t (w), \varphi_t (y)) \big) d\nu_{y;s} (x) d\mu_t (y) d\mu_t (w) \notag \\
&\leq \int_{W^\delta} \int_{\XX_t \setminus W^\delta} \int_{\XX_s}  \sqrt{\Var (\delta_{x}, \nu_{w;s})}   d\nu_{y;s} (x) d\mu_t (y) d\mu_t (w) + 3\delta \notag \\
&\qquad
+ \int_{W^\delta} \int_{\XX_t \setminus W^\delta}  d_t (w, y)   d\mu_t (y) d\mu_t (w)  \notag \\
&\leq \big( \mu_t (W^\delta) \mu_t (\XX_t \setminus W^\delta) \big)^{1/2} \bigg({ \int_{W^\delta} \int_{\XX_t \setminus W^\delta} \int_{\XX_s}  \Var (\delta_{x}, \nu_{w;s})   d\nu_{y;s} (x) d\mu_t (y) d\mu_t (w) }\bigg)^{1/2} \notag \\
&\qquad + \big( \mu_t (W^\delta) \mu_t (\XX_t \setminus W^\delta) \big)^{1/2} \bigg({ \int_{W^\delta} \int_{\XX_t \setminus W^\delta}  d^2_t (w, y)   d\mu_t (y) d\mu_t (w) }\bigg)^{1/2} + 3\delta \notag \\
&\leq  \mu_t^{1/2} (\XX_t \setminus W^\delta)  \bigg({ \int_{W^\delta} \int_{\XX_t \setminus W^\delta}   \Var (\nu_{y;s}, \nu_{w;s})   d\mu_t (y) d\mu_t (w) }\bigg)^{1/2} \notag \\
&\qquad + \mu_t^{1/2} (\XX_t \setminus W^\delta)  \sqrt{\Var ( \mu_t )} + 3\delta \notag \\
&\leq  \mu_t^{1/2} (\XX_t \setminus W^\delta) \big(  \sqrt{\Var (\mu_{s})   }
+   \sqrt{\Var ( \mu_t )}  \big) + 3\delta \notag \\
&\leq  \mu_t^{1/2} (\XX_t \setminus W^\delta) \big(  \sqrt{\Var (\mu_{t}) + H (t-s)   }
+   \sqrt{\Var ( \mu_t )}  \big) + 3\delta \notag \\
&\leq 2 \sqrt{ \delta}  \sqrt{V + H \delta} + 3\delta \label{eq_int_comp_W_bound}
\end{align}
Combining (\ref{eq_int_W_bound}), (\ref{eq_int_comp_W_bound}) implies (\ref{eq_int_Xs_Xt_d_cross}) in Proposition~\ref{Prop_time_s_closeness} if $\delta \leq \ov\delta (  H,V, \eps)$.
\end{proof}

\subsection{Future and past continuity} \label{subsec_fut_past_cont}
In this subsection we define a continuity notion for metric flows.
This notion will imply continuity of time-slices in the $GW_1$-sense if we equip the flow with a conjugate heat flow.
It will turn out that an $H$-concentrated flow is continuous on the complement of a countable set of times.

Let $\XX$ be a metric flow over some subset $I \subset \IR$.

\begin{Definition} \label{Def_flow_continuous}
We say that $\XX$ is {\bf continuous at time $t_0 \in I$} if for all conjugate heat flows $(\mu_t)_{t \in I'}$ that satisfy $t_0 \in I'  \subset I$, $ \Var (\mu_{t}) < \infty$ for all $t \in I'$, the function 
\begin{equation} \label{eq_t_mapsto_intint_d}
 t \longmapsto \int_{\XX_t} \int_{\XX_t} d_t \, d\mu_t d\mu_t 
\end{equation}
is continuous at $t_0$.
We say that $\XX$ is {\bf past continuous at time $t_0$} if $\XX_{\leq t_0}$ is continuous at time $t_0$ and {\bf future continuous at time $t_0$} if $\XX_{\geq t_0}$ is continuous at time $t_0$.
The metric flow $\XX$ is called {\bf (past/future) continuous} if the same is true at all times $t_0 \in I$.
\end{Definition}

\begin{Remark}
Past/future continuity are equivalent to  left/right semi-continuity of the function (\ref{eq_t_mapsto_intint_d}) for any $(\mu_t)_{t \in I'}$ with the properties specified in Definition~\ref{Def_flow_continuous}.
\end{Remark}

\begin{Remark}
It follows from the definition that a flow is continuous at time $t_0$ if and only if it is both past and future continuous.
\end{Remark}

\begin{Remark}
By Lemma~\ref{Lem_intd_diff_Var_diff} we have
\[ \limsup_{t \nearrow t_0} \int_{\XX_t} \int_{\XX_t} d_t \, d\mu_t d\mu_t \leq \int_{\XX_{t_0}} \int_{\XX_{t_0}} d_{t_0} \, d\mu_{t_0} d\mu_{t_0} \leq
\liminf_{t \searrow t_0} \int_{\XX_t} \int_{\XX_t} d_t \, d\mu_t d\mu_t,
 \]
so in order to verify continuity at time $t_0$ it suffices to show that
\begin{equation} \label{eq_lim_int_d_t0}
 \lim_{i \to \infty}\int_{\XX_{t_i}} \int_{\XX_{t_i}} d_{t_i} \, d\mu_{t_i} d\mu_{t_i} = \int_{\XX_{t_0}} \int_{\XX_{t_0}} d_{t_0} \, d\mu_{t_0} d\mu_{t_0} 
\end{equation}
for two sequences of the form $t_i \nearrow t_0, t_i \searrow t_0$, if they exist.
Similarly, for past/future continuity, we only need to verify (\ref{eq_lim_int_d_t0}) for one sequence of the appropriate form.
\end{Remark}

In Examples~\ref{Ex_non_Hausdorff_continuity}, \ref{Ex_neckpinch_continuity} below we will discuss some examples of flows that satisfy or violate Definition~\ref{Def_flow_continuous}.

The following theorem states that we only need to require (left/right) semi-continuity of (\ref{eq_t_mapsto_intint_d}) for \emph{one} conjugate heat flow $(\mu_t)$.
Moreover, we obtain that (past/future) continuity implies continuity of the time-slices in the $GW_1$-sense if we equip $\XX$ with a conjugate heat flow.
We also obtain that (left/right) semi-continuity of $t \mapsto \Var(\mu_t)$ is a necessary condition for (past/future) continuity.

\begin{Theorem} \label{Thm_future_cont_equiv}
Let $\XX$ be an $H$-concentrated metric flow over some subset $I \subset \IR$, where $H < \infty$, and let $t_0 \in I$.
Suppose that $\supp \XX_{t_0} = \XX_{t_0}$.
Let $\mathcal{C}_{t_0}$ be the set of conjugate heat flows $(\mu_t)_{t \in I'}$ on $\XX$ with $t_0 \in I'$, $\Var(\mu_t) < \infty$ for all $t \in I'$.
Let $\mathcal{C}^*_{t_0} \subset \mathcal{C}_{t_0}$ be the subset of conjugate heat flows $(\mu_t)_{t \in I'}$ with $(t_0 - \eps, t_0 + \eps )\cap I \subset I'$ for some $\eps > 0$ and $\supp \mu_{t_0} =  \XX_{t_0}$.

Then the following conditions are equivalent:
\begin{enumerate}[label=(\alph*)]
\item \label{Prop_future_cont_equiv_a} $\XX$ is continuous at time $t_0$.
\item \label{Prop_future_cont_equiv_b} There is a conjugate heat flow $(\mu_t)_{t \in I'} \in \mathcal{C}^*_{t_0}$  such that 
\begin{equation} \label{eq_t_mapsto_intint_d_in_Thm}
 t \longmapsto \int_{\XX_t} \int_{\XX_t} d_t \, d\mu_t d\mu_t 
\end{equation}
is continuous at time $t_0$.
\item \label{Prop_future_cont_equiv_c}  For any conjugate heat flow $(\mu_t)_{t \in I'}\in \mathcal{C}_{t_0}$ we have
\[  (\XX_{t}, d_{t}, \mu_{t}) \xrightarrow[t \to t_0]{\quad GW_1 \quad} (\XX_{t_0}, d_{t_0}, \mu_{t_0}).  \]
\item \label{Prop_future_cont_equiv_d}  There is a conjugate heat flow $(\mu_t)_{t \in I'}\in \mathcal{C}^*_{t_0}$ such that
\[  (\XX_{t_i}, d_{t_i}, \mu_{t_i}) \xrightarrow[i \to \infty]{\quad GW_1 \quad} (\XX_{t_0}, d_{t_0}, \mu_{t_0})  \]
for two sequences $t_i \nearrow t_0$ (if $\inf I' < t_0$) and $t_i \searrow t_0$, (if $\sup I' > t_0$).
\item \label{Prop_future_cont_equiv_e} There is a neighborhood $t_0 \in I_0 \subset I$ of $t_0$ and for any $t \in I_0 \setminus \{ t_0 \}$ there are isometric embeddings $\varphi_t : (\XX_{t}, d_{t}) \to (Z_t, d_{t}^Z)$, $\varphi^0_t : (\XX_{t_0}, d_{t_0}) \to (Z_t, d^Z_{t})$ into a metric space $(Z_t, d^Z_t)$ such that the following holds.
For any conjugate heat flow $(\mu_t)_{t \in I'} \in \mathcal{C}_{t_0}$ the probability measures
\[ q_t := \begin{cases}
 \int_{\XX_{t}} (\nu_{y;t_0} \otimes \delta_y ) \, d\mu_{t} (y) & \text{if $t > t_0$} \\
  \int_{\XX_{t_0}} (  \delta_x \otimes \nu_{x;t} ) \, d\mu_{t_0} (x) & \text{if $t < t_0$}
  \end{cases} \]
are couplings between $\mu_{t_0}, \mu_{t}$ and
\[ \lim_{t \to t_0} \int_{\XX_{t_0} \times \XX_{t}} d^Z_t ( \varphi^0_{t} (x), \varphi_{t} (y)) dq_{t} (x,y)  = 0. \]
In particular,
\[ \lim_{t \to t_0} d^{Z_t}_{W_1} ( (\varphi_{t}^0)_* \mu_{t_0}, (\varphi_{t})_* \mu_{t} ) = 0.\]
\end{enumerate}
Moreover, Conditions~\ref{Prop_future_cont_equiv_a}--\ref{Prop_future_cont_equiv_e} are implied by:
\begin{enumerate}[label=(\alph*), start=6]
\item \label{Prop_future_cont_equiv_f} There is a conjugate heat flow $(\mu_t)_{t \in I'} \in \mathcal{C}^*_{t_0}$ such that $t \mapsto \Var (\mu_t)$ is continuous at time $t_0$.
\end{enumerate}

The corresponding equivalences for past/future continuity follow by applying this theorem to the restricted flows $\XX_{\leq t_0}$ and $\XX_{\geq t_0}$.
In the case of future continuity, we can drop the assumption $\supp \XX_{t_0} = \XX_{t_0}$ of the theorem and the condition $\supp \mu_{t_0} =\supp \XX_{t_0}$ from the definition of $\mathcal{C}^*_{t_0}$.
In the case of past continuity, the assumption $\supp \XX_{t_0} = \XX_{t_0}$ from the theorem may also be dropped.

Lastly, $\XX$ is past continuous at time $t_0$, if and only if for any $x_1, x_2 \in \supp \XX_{t_0}$ we have
\begin{equation} \label{eq_d_W_1_converges}
 \lim_{t \nearrow t_0} d_{W_1}^{\XX_t} (\nu_{x_1;t}, \nu_{x_2, t}) = d_{t_0} (x_1, x_2). 
\end{equation}
\end{Theorem}

\begin{Remark}
If $(\mu_t)_{t \in I'}$ satisfies a stronger concentration bound (for example, an integral Gaussian bound), then Condition~\ref{Prop_future_cont_equiv_f} is equivalent to Conditions~\ref{Prop_future_cont_equiv_a}--\ref{Prop_future_cont_equiv_e}.
\end{Remark}

Since $t \mapsto \Var(\mu_t) + Ht$ is non-decreasing (see Proposition~\ref{Prop_H_monotonicity_Var}), we obtain the following important consequence.

\begin{Corollary} \label{Cor_met_flow_cont_ae}
An $H$-concentrated metric flow is continuous everywhere except, possibly, at a countable set of times.
\end{Corollary}

\begin{Example} \label{Ex_non_Hausdorff_continuity}
The metric flow from Example~\ref{Ex_not_Hausdorff} is not past continuous at time $0$ if $\# X^* > 1$.
\end{Example}

\begin{Example} \label{Ex_neckpinch_continuity}
Consider a (possibly rotationally symmetric) singular Ricci flow $\mathcal{M}$ on $S^2 \times S^1$ that develops a non-degenerate neckpinch of finite diameter at some time $t_0 > 0$.
Such a flow can be constructed using the techniques from  \cite{Angenent_Knopf_precise_asymptotics, Angenent_Knopf_private}.
By Theorem~\ref{Thm_sing_flow_to_met_flow}, $\mathcal{M}$ corresponds to a metric flow $\XX$, whose time-slices equal the metric completions of the time-slices of $\mathcal{M}$.
Note that the time-slice $(\XX_{t_0}, d_{t_0})$ is homeomorphic to $S^3$.
$\XX$ is future continuous, which can be verified using Condition~\ref{Prop_future_cont_equiv_b} in Theorem~\ref{Thm_future_cont_equiv}.
However, $\XX$ is not past continuous.
To see this, consider two points $x_1,x_2 \in \XX_{t_0}$ near the neckpinch, but on opposite sides.
These points violate (\ref{eq_d_W_1_converges}).

We may also construct another metric flow $\XX'$ based on $\mathcal{M}$, which is not future continuous, but past continuous, as follows.
Let $\XX'_t := \XX_t$ for all $t \neq t_0$ and define 
\[ d'_{t_0} (x_1,x_2) := 
 \lim_{t \nearrow t_0} d_{W_1}^{\XX_t} (\nu_{x_1;t}, \nu_{x_2, t})
 = \lim_{t \nearrow t_0} d_{g_{t}} (x_1(t), x_2(t)). \]
 It can be shown that the time-slices $(\XX'_t, d'_t)$ can be equipped with the structure of a metric flow such that $\XX_{\neq t_0}$ and $\XX'_{\neq t_0}$ are  flow isometric.
 Note that the time-slice $(\XX'_{t_0}, d'_{t_0})$ is homeomorphic to $S^2 \times S^1$ with one collapsed cross-sectional sphere.
 This flow will be less interesting to us, because the metric of $d'_{t_0}$ restricted to $\MM_{t_0} \subset \XX'_{t_0}$ does not agree with the length metric of $g_{t_0}$.
\end{Example}

\begin{proof}[Proof of Theorem~\ref{Thm_future_cont_equiv}.]
The implication \ref{Prop_future_cont_equiv_f} $\Rightarrow$ \ref{Prop_future_cont_equiv_b} holds due to Lemma~\ref{Lem_intd_diff_Var_diff}.

So it remains to show the equivalence of \ref{Prop_future_cont_equiv_a}--\ref{Prop_future_cont_equiv_e} and the statement involving (\ref{eq_d_W_1_converges}).
The implications \ref{Prop_future_cont_equiv_a} $\Rightarrow$ \ref{Prop_future_cont_equiv_b}, \ref{Prop_future_cont_equiv_c} $\Rightarrow$ \ref{Prop_future_cont_equiv_d} and \ref{Prop_future_cont_equiv_e} $\Rightarrow$ \ref{Prop_future_cont_equiv_c} are obvious. 
The implications \ref{Prop_future_cont_equiv_c} $\Rightarrow$ \ref{Prop_future_cont_equiv_a} and  \ref{Prop_future_cont_equiv_d} $\Rightarrow$ \ref{Prop_future_cont_equiv_b} are consequences of Lemmas~\ref{Lem_W1_close_intint_d_close}, \ref{Lem_intd_diff_Var_diff}.
The implication \ref{Prop_future_cont_equiv_a} $\Rightarrow$  \ref{Prop_future_cont_equiv_c} follows from Propositions~\ref{Prop_mass_distribution}, \ref{Prop_time_s_closeness}; note that in the case in which there is no sequence $t_i \searrow t_0$, $t_i \in I'$, we don't need to apply Proposition~\ref{Prop_mass_distribution} and can instead just set $b := b^{(\XX_{t_0}, d_{t_0}, \mu_{t_0})}_1$ in Proposition~\ref{Prop_time_s_closeness}.
So to see the equivalence of \ref{Prop_future_cont_equiv_a}--\ref{Prop_future_cont_equiv_e} it remains to establish the implication \ref{Prop_future_cont_equiv_b} $\Rightarrow$  \ref{Prop_future_cont_equiv_e}.

Suppose now that Condition~\ref{Prop_future_cont_equiv_b} holds and let $(\td\mu_t)_{t \in \td I'} \in \mathcal{C}^*_{t_0}$ be the conjugate heat flow for which (\ref{eq_t_mapsto_intint_d_in_Thm}) is continuous at time $t_0$.
By Proposition~\ref{Prop_time_s_closeness}, for any $t \in I_0 \setminus \{ t_0 \}$ there are closed subsets $W_t \subset \XX_{t}$ (if $t > t_0$) or $W_t \subset \XX_{t_0}$ (if $t < t_0$), as well as isometric embeddings $\varphi_t : (\XX_{t}, d_{t}) \to (Z_t, d_{t}^Z)$, $\varphi^0_t : (\XX_{t_0}, d_{t_0}) \to (Z_t, d^Z_{t})$ into a common metric space $(Z_t, d^Z_t)$ and numbers $\eps_t > 0$ such that for all $t \in I_0 \setminus \{ t_0 \}$ we have for any $y_1, y_2 \in W_t$:
\[ \begin{cases} 
 0 \leq d_{t} (y_1,y_2) - d_{W_1}^{Z_{t}} (\nu_{y_1;t_0} , \nu_{y_2;t_0})  \leq \eps_t &\text{if $t > t_0$} \\
 0 \leq d_{t_0} (y_1,y_2) - d_{W_1}^{Z_t} (\nu_{y_1;t} , \nu_{y_2;t})  \leq \eps_t &\text{if $t < t_0$}
 \end{cases}
 \]
 and such that
\begin{equation} \label{eq_lim_complement_W}
 \lim_{t \to t_0} \begin{cases} \td\mu_{t} ( \XX_{t} \setminus W_{t} ) & \text{if $t > t_0$} \\ \td\mu_{t_0} ( \XX_{t_0} \setminus W_{t} ) & \text{if $t < t_0$}
 \end{cases} = 0 .
\end{equation}
Consider some possibly different conjugate heat flow $(\mu_t)_{t \in I'} \in \mathcal{C}_{t_0}$.
We claim that there are numbers $\delta_t > 0$, $t \in I_0 \setminus \{ t_0 \}$ such that $\lim_{t \to t_0} \delta_t = 0$ and
\begin{equation} \label{eq_mu_complement_W}
  \begin{cases} \mu_{t} ( \XX_{t} \setminus W_{t} )\leq \delta_t & \text{if $t > t_0$} \\ \mu_{t_0} ( \XX_{t_0} \setminus B(W_{t}, \delta_t) )\leq \delta_t & \text{if $t < t_0$}
 \end{cases}   
\end{equation}
 This will then imply Condition~\ref{Prop_future_cont_equiv_e} using Lemma~\ref{Lem_construction_Z}.
 To see the second bound in (\ref{eq_mu_complement_W}), note that it suffices to show that for all $\delta > 0$
 \[ \lim_{t \nearrow t_0} \mu_{t_0} ( \XX_{t_0} \setminus B(W_{t}, \delta) ) = 0. \]
This follows from the fact that $\supp \mu_{t_0} \subset  \XX_{t_0} = \supp \td\mu_{t_0}$. 
Let us now show the first bound in (\ref{eq_mu_complement_W}).
Fix some $t^* \in I_0 \cap I'$, $t^* > t_0$ and observe that for any $t \in I_0 \cap I'$ with $t_0 < t < t^*$ we have
\[  \td\mu_{t} ( \XX_{t} \setminus W_{t} ) = \int_{\XX_{t^*}} \nu_{x; t} ( \XX_{t} \setminus W_{t} ) d\td\mu_{t^*} (x), \qquad  \mu_{t} ( \XX_{t} \setminus W_{t} ) = \int_{\XX_{t^*}} \nu_{x; t} ( \XX_{t} \setminus W_{t} ) d\mu_{t^*} (x). \]
Since the first integral goes to $0$ as $t \searrow t_0$, we obtain using Definition~\ref{Def_metric_flow}\ref{Def_metric_flow_6} that $\nu_{x; t} ( \XX_t \setminus W_t ) \to 0$ uniformly on bounded subsets.
Therefore, the second integral goes to $0$ as well.

Lastly, we prove the statement involving (\ref{eq_d_W_1_converges}).
Suppose first that $\XX$ is past continuous at time $t_0$ and let $x_1, x_2 \in \XX_{t_0}$.
Then by applying Condition~\ref{Prop_future_cont_equiv_e} to $\mu_{t} = \nu_{x_j;t}$, $j= 1,2$, we obtain that for $t < t_0$ close to $t_0$ 
\[ \lim_{t \nearrow t_0} d_{W_1}^{Z_t} ( (\varphi^0_t)_* \delta_{x_j}, (\varphi_t)_* \nu_{x_j; t} ) = 0. \]
It follows that
\begin{multline*}
 \lim_{t \nearrow t_0} d_{W_1}^{\XX_t} (  \nu_{x_1; t}, \nu_{x_2; t})
= \lim_{t \nearrow t_0} d_{W_1}^{Z_t} ( (\varphi_t)_* \nu_{x_1; t}, (\varphi_t)_* \nu_{x_2; t}) \\
\leq \lim_{t \nearrow t_0} \big( d_{W_1}^{Z_t} ( (\varphi_t)_* \nu_{x_1; t}, (\varphi^0_t)_* \delta_{x_1}) + d_{W_1}^{Z_t} ((\varphi^0_t)_* \delta_{x_1},(\varphi^0_t)_* \delta_{x_2}) + d_{W_1}^{Z_t} ( (\varphi^0_t)_* \delta_{x_2}, (\varphi_t)_* \nu_{x_2; t} )    \big) \\ = d_{t_0} (x_1, x_2). 
\end{multline*}
Conversely, suppose that (\ref{eq_d_W_1_converges}) holds for all $x_1, x_2 \in \XX_{t_0}$ and consider some conjugate heat flow $(\mu_t)_{t \in I'}$ with $(t_0 - \eps, t_0] \cap I \subset I'$ for some $\eps > 0$, $\Var(\mu_{t_0}) < \infty$.
For $t \in I'$ define $f_t : \XX_{t_0} \times \XX_{t_0} \to \IR$ by
\[ f_t (x_1, x_2) := d_{t_0} (x_1, x_2) - d_{W_1}^{\XX_t} (\nu_{x_1;t}, \nu_{x_2;t}) \geq 0. \]
Since $f_t \leq d_{t_0} (x_1, x_2)$ and
\[ \int_{\XX_{t_0}} \int_{\XX_{t_0}} d_{t_0} \, d\mu_{t_0} d\mu_{t_0} 
\leq \sqrt{\Var (\mu_{t_0})} < \infty, \]
we obtain by dominated convergence that
\[ \lim_{t \nearrow t_0} \int_{\XX_{t_0}} \int_{\XX_{t_0}} f_t \, d\mu_{t_0} d\mu_{t_0}  = 0. \]
Since
\[ f_t (x_1, x_2) \geq d_{t_0} (x_1, x_2) -  \int_{\XX_t} \int_{\XX_t} d_t \, d\nu_{x_1;t}d\nu_{x_2;t}, \]
we obtain that
\[ \limsup_{t \nearrow t_0} \bigg( \int_{\XX_{t_0}} \int_{\XX_{t_0}} d_{t_0} \, d\mu_{t_0} d\mu_{t_0} - \int_{\XX_{t}} \int_{\XX_{t}} d_t \, d\mu_{t} d\mu_{t} \bigg) \leq 0. \]
Combining this with Lemma~\ref{Lem_intd_diff_Var_diff} implies that (\ref{eq_t_mapsto_intint_d_in_Thm}) is left semi-continuous at $t_0$, which implies past continuity.
\end{proof}

We will also need the following result:

\begin{Proposition} \label{Prop_fut_cont_two_points_approx}
Let $\XX$ be an $H$-concentrated metric flow over some subset $I \subset \IR$, where $H< \infty$.
Suppose that $\XX$ is future continuous at time $t_0 \in I$ and suppose that there is a sequence of times $t_i \in I$, $t_i \searrow t_0$.
Then for any two points $x_1, x_2 \in \supp \XX_{t_0}$, we can find points $x_{j,i} \in \XX_{t_i}$, $j = 1,2$, $i = 1, 2, \ldots$ such that $x_{j,i} \to x_j$ and
\[ \lim_{i \to \infty} d_{W_1}^{\XX_{t_0}} ( \delta_{x_j}, \nu_{x_{j,i}; t_0} ) = 0, \qquad \lim_{i \to \infty} d_{t_i} (x_{1,i}, x_{2,i}) = d_{t_0} (x_1, x_2). \]
\end{Proposition}

\begin{proof}
Consider the metric spaces $(Z_t, d^Z_t)$ and isometric embeddings $\varphi_t : (\XX_t, d_t) \to (Z_t, d^Z_t)$, $\varphi^0_t : (\XX_{t_0}, d_{t_0}) \to (Z_t, d^Z_t)$ from Theorem~\ref{Thm_future_cont_equiv}\ref{Prop_future_cont_equiv_e}.
The proposition now follows from the following claim.

\begin{Claim}
For any $x \in \supp \XX_{t_0}$ there are $x_{i} \in \XX_{t_i}$ such that
\[ \lim_{i \to \infty} d_{W_1}^{\XX_{t_0}} ( \delta_{x}, \nu_{x_{i}; t_0} ) =
\lim_{i \to \infty} d_{t}^Z (\varphi_{t_i} (x_{i}), \varphi^0_{t_i} (x) ) = 0. \]
\end{Claim}

\begin{proof}
Let $\delta > 0$.
Fix some conjugate heat flow $(\mu_t)_{t \in I'}$ with $t_0 \in I'$ and $t_i \in I'$ for large $i$.
If we denote by $q_t$ the couplings from Theorem~\ref{Thm_future_cont_equiv}\ref{Prop_future_cont_equiv_e}, then we can find some sequence $y_i \in \XX_{t_i}$ with $d_{t}^Z (\varphi_{t_i} (y_{i}), \varphi^0_{t_i} (x) ) \to 0$.
Next note that
\begin{multline*}
 \liminf_{i \to \infty}  q_{t_i} (B(x,2\delta) \times B(y_i, \delta)) \\
= \liminf_{i \to \infty} \big( q_{t_i} (B(x,2\delta) \times \XX_{t_i} ) - q_{t_i}  (B(x,2\delta) \times (\XX_{t_i} \setminus B(y_i,  \delta))) \big) \\
= \liminf_{i \to \infty} q_{t_i} (B(x,2\delta) \times \XX_{t_i} ) 
= \mu_{t_0} (B(x,2\delta)) > 0. 
\end{multline*}
So by the definition of $q_{t_i}$, we obtain that there are points $x_i \in B(y_i, \delta)$ such that
\[ \liminf_{i \to \infty}  \nu_{x_i; t_0} (B(x,2\delta) ) > 0. \]
Thus by Lemma~\ref{Lem_nu_BA_bound}, for large $i$ any $H$-center $z_i \in \XX_{t_0}$ of $x_i$ must be contained in $B(x, 3\delta)$.
Therefore, for large $i$
\[ d_{W_1}^{\XX_{t_0}} ( \delta_{x}, \nu_{x_{i}; t_0} )
\leq 3\delta + d_{W_1}^{\XX_{t_0}} ( \delta_{z_i}, \nu_{x_{i}; t_0} )
\leq 3\delta + \sqrt{H(t_i-t_0)} \leq 4\delta. \]
Letting $\delta \to 0$ implies the claim.
\end{proof}

\end{proof}

\subsection{The future completion} \label{subsec_future_completion}
Consider some subset $I \subset \IR$.
We denote by $\ov{I}^- \subset \IR$ the set of $t_\infty \in \IR$ with the property that $t_i \searrow t_\infty$ for some sequence $\{ t_i \} \subset I$ with $t_i \geq t_\infty$.
So, for example, if $a < b$, then $\ov{(a,b)}^-$ = $[a,b)$.

Our goal will be to extend a metric flow $\XX$ over $I$ to a metric flow $\XX^*$ over $\ov{I}^-$.
This metric flow will be called the {\bf future completion}, because its time slices at any time $t_\infty \in \ov{I}^- \setminus I$ will be obtained by a limit of time slices at times $t_i \in I$ with $t_i \searrow t_\infty$ and the flow will therefore be future continuous at all times $t_\infty \in \ov{I}^- \setminus I$.
We will also show that, under appropriate conditions, the future completion of a metric flow is unique up to flow isometry.

We first make the following definition:

\begin{Definition}
Let $\XX$ be a metric flow over some subset $I \subset \IR$.
A pair $(\XX^*, \phi)$, consisting of a metric flow $\XX^*$ over $\ov{I}^-$ and a flow isometric embedding $\phi: \XX \to \XX^*$ is called a {\bf future completion of $\XX$} if the following holds:
\begin{enumerate}[label=(\arabic*)]
\item \label{Thm_existenc_fut_compl_a}  $\XX^*$ is future continuous at all times $t \in \ov{I}^- \setminus I$.
\item \label{Thm_existenc_fut_compl_b} $\XX^*$ is $H$-concentrated for some $H < \infty$.
\item \label{Thm_existenc_fut_compl_c} $\supp \XX^*_t = \XX^*_t$ for all $t \in \ov{I}^- \setminus I$.
\end{enumerate}
\end{Definition}

We will often view $\XX^* \supset \XX$ as an extension of $\XX$ and call $\XX^*$ \emph{the} future completion of $\XX$.

\begin{Theorem}[Existence of future completion] \label{Thm_existenc_fut_compl}
Let $\XX$ be an $H$-concentrated metric flow over some $I \subset \IR$, where $H < \infty$.
Then $\XX$ has a future completion $(\XX^*, \phi)$ and $\XX^*$ is $H$-concentrated.
\end{Theorem}

The following two results address the uniqueness of the future completion.
The first result concerns the extension of an isometry between two metric flows at a future continuous time. 

\begin{Theorem}[Extension of isometries of metric flows] \label{Thm_extend_isometry}
Let $\XX^j$, $j=1,2$, bet two $H$-concentrated metric flows over some $I \subset \IR$, where $H < \infty$.
Assume that there is a flow isometry $\phi : \XX^1_{I'} \to \XX^2_{I'}$ over some subset $I' \subset I$.
Let $I_0$ be the union of $I'$ with the set of all times $t \in \ov{I'}^- \cap I$ at which  $\XX^1$ and $\XX^2$ are both future continuous and satisfy $\supp \XX^j_t = \XX^j_t$, $j=1,2$.

Then $\phi$ can be extended uniquely to a flow isometry between $\XX^1, \XX^2$ over $I_0$.
\end{Theorem}

As a corollary, we obtain the uniqueness result of the future completion up to flow isometry.

\begin{Corollary}[Uniqueness of future completion]
Let $\XX$ be an $H$-concentrated metric flow over some $I \subset \IR$, where $H < \infty$.
Consider two future completions $(\XX^{*,j}, \phi^j : \XX \to \XX^{*,j})$ of $\XX$.
Then there is a unique flow-isometry $\psi : \XX^{*,1} \to \XX^{*,2}$ such that $\psi \circ \phi^1 = \phi^2$.
\end{Corollary}

\begin{proof}[Proof of Theorem~\ref{Thm_existenc_fut_compl}.]
In the following we will construct $\XX^*$ by extending $\XX$. 
So we will have $\XX^* \supset \XX$ and $\phi_t = \id_{\XX_t}$ for $t \in I$ and $\nu^*_{x;s} = \nu_{x;s}$ for any $s, t \in I$, $s \leq t$, $x \in \XX_t$.

We first explain how to construct the time-slices $(\XX_s, d_s)$ and the measures $\nu_{x;s}$ for any $t \in I$, $x \in \XX_t$ and $s \in  \ov{I}^- \setminus I$, $s < t$.
For this purpose, fix some $s_\infty \in \ov{I}^- \setminus I$ and a sequence $s_i \searrow s_\infty$, $s_i \in I$.
Next, fix a conjugate heat flow $(\mu_s)_{s \in I'}$, $I' \subset I$, with $ \Var( \mu_s) < \infty$ for all $s \in I'$ and with the property that $\mu_{s_i}$ is defined for large $i$, for example a conjugate heat kernel.
By Proposition~\ref{Prop_H_monotonicity_Var} the map $s \mapsto \Var (\mu_s) + Hs$ is non-decreasing and thus the limit $\lim_{i \to \infty} \Var(\mu_{s_i})$ exists.
Therefore, by Propositions~\ref{Prop_mass_distribution}, \ref{Prop_time_s_closeness} we have
\[ \lim_{i,j \to \infty} d_{GW_1} \big( (\XX_{s_i}, d_{s_i}, \mu_{s_i}), (\XX_{s_j}, d_{s_j}, \mu_{s_j}) \big) = 0. \]
This implies convergence of the sequence $(\XX_{s_i}, d_{s_i}, \mu_{s_i})$ in the $d_{GW_1}$-metric.
For future purposes, we will recover its limit in a particular way.

By Proposition~\ref{Prop_time_s_closeness} there is a sequence $\delta_i \to 0$ such that for any $i < j$ there is a closed subset $W_{i,j} \subset \XX_{s_i}$ such that:
\begin{enumerate}
\item $\mu_{s_i} ( \XX_{s_i} \setminus W_{i,j} ) \leq \delta_i$.
\item For any $y_1,y_2 \in W_{i,j}$ we have
\[ \big| d_{s_i} (y_1,y_2) - d_{W_1}^{\XX_{s_j}} (\nu_{y_1;s_j} , \nu_{y_2;s_j}) \big| \leq \delta_i. \]
\end{enumerate}
For any $i < j$ let $(Z_{i,j}, d^{Z_{i,j}})$ and $\varphi^{i,j}_i : \XX_{s_i} \to Z_{i,j}$, $\varphi^{i,j}_j : \XX_{s_j} \to Z_{i,j}$ be the metric space and isometric embeddings from Lemma~\ref{Lem_construction_Z}.
Then Lemma~\ref{Lem_construction_Z} implies that for any conjugate heat flow $(\td\mu_s)_{s \in I'}$, $\td I' \subset I$ that is defined at time $s_i$ for large $i$ and satisfies $\Var (\td\mu_{s_i}) \leq V < \infty$ for large $i$, we have 
\begin{equation} \label{eq_dW1Zij_eps_bound}
 d_{W_1}^{Z_{i,j}} \big( (\varphi^{i,j}_i )_* \td\mu_{s_i}, (\varphi^{i,j}_j )_* \td\mu_{s_j} \big) \leq \eps (V, \delta_i + \td\mu_{s_i} ( \XX_{s_i} \setminus W_{i,j} ) + s_i - s_\infty), 
\end{equation}
where $\eps : \IR_+ \times \IR_+ \to (0, \infty]$ is a function that is non-decreasing in both parameters and satisfies $\lim_{a \to 0} \eps (V, a) = 0$ for any fixed $V > 0$.

\begin{Claim} \label{Cl_eps_summable}
After passing to a subsequence, we may assume that for any conjugate heat flow $(\td\mu_s)_{s \in \td I'}$, $\td I' \subset I$ that is defined at time $s_i$ for large $i$ and satisfies $\Var (\td\mu_{s_i}) \leq V < \infty$ we have for large $i_0$
\[ \sum_{i = i_0}^\infty \eps (V, \delta_i + \td\mu_{s_i} ( \XX_{s_i} \setminus W_{i,j} ) + s_i - s_\infty) < \infty. \]
\end{Claim}

\begin{proof}
Fix some large $k \geq 1$ and observe that for $k < i < j$
\begin{align}
  \mu_{s_i} ( \XX_{s_i} \setminus W_{i,j} ) &= \int_{\XX_{s_k}} \nu_{x; s_i} ( \XX_{s_i} \setminus W_{i,j} ) d\mu_{s_k} (x) \leq \delta_i, \notag \\
    \td\mu_{s_i} ( \XX_{s_i} \setminus W_{i,j} ) &= \int_{\XX_{s_k}} \nu_{x; s_i} ( \XX_{s_i} \setminus W_{i,j} ) d\td\mu_{s_k} (x).  \label{eq_tdmu_XX_m_WW_ij}
\end{align}
Fix some $z \in \XX_{s_k}$ such that $\Var(\delta_z, \td\mu_{s_k} ) \leq V$ and choose $D < \infty$ such that $\mu_{s_k} (B(z, D)) \geq \frac12$.
It follows that
\[ \inf_{x \in B(z,D)}  \nu_{x; s_i} ( \XX_{s_i} \setminus W_{i,j} ) \leq 2 \delta_i. \]
Therefore, by Definition~\ref{Def_metric_flow}(6) for any $D' < \infty$ and for $\tau := s_k - s_{k+1} \leq s_k - s_i$
\begin{equation} \label{eq_sup_xBzDpnuXW}
 \sup_{x \in B(z, D')} \nu_{x; s_i} ( \XX_{s_i} \setminus W_{i,j} ) 
\leq \Phi \big( \Phi^{-1} (2\delta_i) + \tau^{-1/2} (D' + D) \big). 
\end{equation}
It follows using (\ref{eq_tdmu_XX_m_WW_ij}), (\ref{eq_sup_xBzDpnuXW}) that
\begin{multline*}
  \td\mu_{s_i} ( \XX_{s_i} \setminus W_{i,j} ) \leq \Phi \big( \Phi^{-1} (2\delta_i) + \tau^{-1/2} (D' + D) \big) + \td\mu_{s_k} (\XX_k \setminus B(z, D')) \\
\leq \Phi \big( \Phi^{-1} (2\delta_i) + \tau^{-1/2} (D' + D) \big) + \frac{V}{D^{\prime 2}}. 
\end{multline*}
So there is a function $f : \IR_+ \times \IR_+ \times \IR_+ \to \IR_+$ that is non-decreasing in each argument and satisfies $\lim_{\delta \to 0} f(V, D, \delta) = 0$ for fixed $V, D > 0$ such that
\[ \td\mu_{s_i} ( \XX_{s_i} \setminus W_{i,j} ) \leq f(V, D, \delta_i). \]

This reduces the claim to showing that after passing to a subsequence, for any $V, D < \infty$ there we have for large $i_0$:
\[ \sum_{i = i_0}^\infty \eps (V, \delta_i + f(V,D, \delta_i) + s_i - s_\infty) < \infty. \]
To accomplish this, observe that for fixed $V, D < \infty$ we have $\lim_{i \to \infty} \eps (V, \delta_i + f(V,D, \delta_i) + s_i - s_\infty)=0$.
So we may choose sequences $V_i , D_i \to \infty$
\[ \lim_{i \to \infty} \eps (V_i, \delta_i + f(V_i,D_i, \delta_i) + s_i - s_\infty)  = 0. \]
So after passing to a subsequence, we can ensure that
\[ \eps (V_i, \delta_i + f(V_i,D_i, \delta_i) + s_i - s_\infty) \leq 2^{-i}. \]
Then for any fixed $V, D$ we have for large $i$
\[ \eps (V, \delta_i + f(V,D, \delta_i) + s_i - s_\infty)
\leq \eps (V_i, \delta_i + f(V_i,D_i, \delta_i) + s_i - s_\infty) \leq 2^{-i}, \]
which proves the summability.
\end{proof}

Consider again the metric spaces $(Z_{i,j}, d_{Z_{i,j}})$ and isometric embeddings $\varphi^{i,j}_i : \XX_{s_i} \to Z_{i,j}$, $\varphi^{i,j}_j : \XX_{s_j} \to Z_{i,j}$.
By Lemma~\ref{Lem_combining_embeddings}, we may assume that $Z := Z_{1,2} = Z_{2,3} =\ldots$ and $\varphi_i := \varphi^{i-1,i}_i = \varphi^{i,i+1}_i$.
We may furthermore assume that $(Z, d_Z)$ is complete.
Then (\ref{eq_dW1Zij_eps_bound}) and Claim~\ref{Cl_eps_summable} imply the following:

\begin{Claim} \label{Cl_summary_tdmu}
For any conjugate heat flow $(\td\mu_s)_{s \in I'}$, $\td I' \subset I$ that is defined at time $s_i$ for large $i$ and satisfies $\Var (\td\mu_{s_i}) \leq V < \infty$ we have
\[ (\varphi_i)_* \td\mu_{s_i} \xrightarrow[i \to \infty]{\quad W_1 \quad} \td\mu_{s_\infty}\in \mathcal{P} ( Z ). \]
In particular, for any $t \in I$, $t > s_\infty$ and $x \in \XX_t$ we have
\[ (\varphi_i)_* \nu_{x;s_i} \xrightarrow[i \to \infty]{\quad W_1 \quad} \nu_{x;s_\infty}\in \mathcal{P} ( Z ). \]
\end{Claim}

In the following we will write $\td\mu_{s_\infty}, \nu_{x;s_\infty} \in \mathcal{P}(Z)$ for any limit obtained according to Claim~\ref{Cl_summary_tdmu}.
Set $\XX_{s_\infty} := \supp \mu_{s_\infty} \subset Z$ and $d_{s_\infty} := d_Z |_{\XX_{s_\infty}}$.
The following claim will allow us to forget the space $Z$.

\begin{Claim} \label{Cl_properties_new_Xt}
The following is true:
\begin{enumerate}[label=(\alph*)]
\item \label{Cl_properties_new_Xt_a} For any $t \in I$, $t > s_\infty$, $x \in \XX_t$ we have $\supp \nu_{x; s_\infty} = \supp \mu_{s_\infty}  = \XX_{s_\infty}$.
\item \label{Cl_properties_new_Xt_b}  Property~\ref{Def_metric_flow_6} of Definition~\ref{Def_metric_flow} holds between time $s_\infty$ and all $t \in I$, $t > s_\infty$.
In other words, for all $T \geq 0$ and any $T^{-1/2}$-Lipschitz function $f : \XX_{s_\infty} \to \IR$, the function
\begin{equation} \label{eq_XR_x_Phi}
 \XX_t \longrightarrow \IR, \qquad x \longmapsto \Phi^{-1} \bigg( \int_{X_{s_\infty}} \Phi (f) d\nu_{x;s_\infty} \bigg). 
\end{equation}
is $(t-s_\infty+T)^{-1/2}$-Lipschitz.
If $T = 0$, then we only require $f$ to be measurable.
\item \label{Cl_properties_new_Xt_c}  Property~\ref{Def_metric_flow_7} of Definition~\ref{Def_metric_flow} holds between time $s_\infty$ and all $t_1, t_2 \in I$ with $s_\infty < t_1 < t_2$.
In other words, for all $x \in \XX_{t_2}$
\[ \nu_{x;s_\infty} = \int_{\XX_{t_1}} \nu_{\cdot, s_\infty} d\nu_{x;t_1}. \]
\item \label{Cl_properties_new_Xt_d} There is a sequence $s_i \searrow s_\infty$, $s_i \in I$, such that for any $x_\infty, x'_\infty \in \XX_{s_\infty}$ there are sequences $x_i, x'_i \in \XX_{s_i}$ such that we have
\[ \lim_{i \to \infty} d^{\XX_{s_\infty}}_{W_1} (\delta_{x_\infty}, \nu_{x_i; s_\infty})
=\lim_{i \to \infty} d^{\XX_{s_\infty}}_{W_1} (\delta_{x'_\infty}, \nu_{x'_i; s_\infty})
=0, \qquad
\lim_{i \to \infty} d_{s_i} (x_i, x'_i) = d_{s_\infty} (x_\infty, x'_\infty). \]
\item \label{Cl_properties_new_Xt_f} There is a sequence $s_i \searrow s_\infty$, $s_i \in I$ such that for any $t_1, t_2 \in I$ with $t_1, t_2 > s_\infty$ and $x_1 \in \XX_{t_1}$, $x_2 \in \XX_{t_2}$ we have
\[ \qquad \lim_{i \to \infty} \int_{\XX_{s_i}} \int_{\XX_{s_i}} d_{s_i} (y_1, y_2) d\nu_{x_1; s_i} (y_1) d\nu_{x_2;s_i} (y_2)
= \int_{\XX_{s_\infty}} \int_{\XX_{s_\infty}} d_{s_\infty} (y_1, y_2) d\nu_{x_1; s_\infty} (y_1) d\nu_{x_2;s_\infty}(y_2). \]
\item \label{Cl_properties_new_Xt_g} For any $t \in I$ with $t > s_\infty$ and $x_1, x_2 \in \XX_t$ we have
\[ \Var (\nu_{x_1; s_\infty}, \nu_{x_2; s_\infty}) \leq d^2_{t}(x_1, x_2) + H (t - s_\infty). \]
\item \label{Cl_properties_new_Xt_h} For any $s_\pm, t_1, t_2 \in I$ with $s_- < s_\infty < s_+$ and $t_1, t_2 \geq s_\infty$ and $x_j \in \XX_{t_j}$, $j =1,2$, we have
\[ d_{W_1}^{\XX_{s_-}} (\nu_{x_1;s_-}, \nu_{x_2;s_-}) 
\leq d_{W_1}^{\XX_{s_\infty}} (\nu_{x_1;s_\infty}, \nu_{x_2;s_\infty})
\leq d_{W_1}^{\XX_{s_+}} (\nu_{x_1;s_+}, \nu_{x_2;s_+}) . \]
\end{enumerate}
\end{Claim}

\begin{proof}
For Assertion~\ref{Cl_properties_new_Xt_a} choose some time $s' \in I$, $s_\infty < s' < t$.
If $i$ is large enough such that $s_i < s'$, then
\[\mu_{s_i} = \int_{\XX_{s'}} \nu_{\cdot; s_i} d\mu_{s'}, \qquad
 \nu_{x; s_i} = \int_{\XX_{s'}} \nu_{\cdot;s_i} d\nu_{x;s'}. \]
Assume that $z \in \supp \nu_{x;s_\infty} \setminus   \supp \mu_{s_\infty} $.
Then there is some small $r > 0$ such that for large $i$
\[  \lim_{i \to \infty}  ((\varphi_i)_* \mu_{s_i}) (B( z, r)) = 0 , \qquad
\liminf_{i \to \infty}  ((\varphi_i)_* \nu_{x;s_i}) (B( z, r)) > 0,  \]
Let $u_i (y) := \nu_{y; s_i} ( (\varphi_i)^{-1} (B(z,r)))$.
Then
\begin{equation} \label{eq_XX_s_p_u_i}
  \lim_{i \to \infty} \int_{\XX_{s'}} u_i \, d\mu_{s'} =0, \qquad
\liminf_{i \to \infty} \int_{\XX_{s'}} u_i \, d\nu_{x;s'} > 0.
\end{equation}
By Definition~\ref{Def_metric_flow}\ref{Def_metric_flow_6} we know that the functions $\Phi^{-1} (u_i)$ are uniformly Lipschitz.
So by the first identity in (\ref{eq_XX_s_p_u_i}) we must have $u_i \to 0$ pointwise, which contradicts the second identity.
This shows $\supp \nu_{x; s_\infty} \subset \supp \mu_{s_\infty}$; the reverse inclusion follows by reversing the roles of $(\nu_{x; t})$ and $(\mu_t)$.

For Assertion~\ref{Cl_properties_new_Xt_b} consider a $T^{-1}$-Lipschitz function $f : \XX_{s_\infty} \to \IR$ for some $T > 0$.
We can extend $f$ to a $T^{-1}$-Lipschitz function $\td{f} : Z \to \IR$ by setting
\[ \td{f} (z) := \inf_{x_\infty \in \XX_{t_\infty}} ( T^{-1} d_Z (z, x_\infty) + f(x_\infty) ). \]
Now the functions
\[ \XX_t \longrightarrow \IR, \qquad x \longmapsto \Phi^{-1} \bigg( \int_{\XX_{s_i}} \Phi (\td f \circ \varphi_i) d\nu_{x;s_i} \bigg) = \Phi^{-1} \bigg( \int_{Z} \Phi (\td f ) d( (\varphi_i)_* \nu_{x;s_i}) \bigg). \]
are $(T + t - s_i)^{-1}$-Lipschitz and converge pointwise to (\ref{eq_XR_x_Phi}).
This proves Assertion~\ref{Cl_properties_new_Xt_b} if $T > 0$.
The case $T = 0$ follows since the space of bounded Lipschitz functions on $\XX_{s_\infty}$ is dense in $L^1 (\nu_{x; s_\infty})$ for any $x \in \XX_t$; see also Lemmas~\ref{Lem_basic_measure}, \ref{Lem_met_flow_T_positive}.

Assertion~\ref{Cl_properties_new_Xt_c} is a consequence of Claim~\ref{Cl_summary_tdmu} and the reproduction formula of $\XX$.

For Assertion~\ref{Cl_properties_new_Xt_d} it suffices to show that for every $x_\infty \in \XX_{s_\infty} \subset Z$ there is a sequence $x_i \in \XX_{s_i}$ with $\varphi_i (x_i) \to x_\infty$ in $Z$ and $\nu_{x_i;s_\infty} \to \delta_{x_\infty}$ in the $W_1$-sense.
For this purpose, note that the proof of Proposition~\ref{Prop_topology_properties}\ref{Prop_topology_properties_e} applies in this case and allows us to choose a sequence $y_i \in \XX_{s_i}$ with $\nu_{y_i;s_\infty} \to \delta_{x_\infty}$ in $W_1$.
We will use this sequence to choose another sequence $x_i$ with the desired properties.
For this purpose, fix some small $r > 0$ and choose $j \geq 1$ large enough such that
\begin{equation} \label{eq_choice_y_j}
 r^{-2} H (s_j - s_\infty) \leq \tfrac12, \qquad d_{W_1}^Z (\delta_{x_\infty}, \nu_{y_j; s_\infty} ) < \tfrac12 r.  
\end{equation}
Fix $j$ for the moment.
For any $i > j$  let $z_{i,j} \in \XX_{s_i}$ be an $H$-center of $y_j$ at time $s_i$.
Then by Lemma~\ref{Lem_nu_BA_bound}
\begin{equation} \label{eq_nu_yj_si_geq_12}
  \nu_{y_j; s_i} ( B(z_{i,j}, r) ) \geq 1-  r^{-2} H (s_j - s_i) \geq \tfrac12. 
\end{equation}
We claim that 
\begin{equation} \label{eq_limsup_3r}
 \limsup_{i \to \infty} d_Z(\varphi_i (z_{i,j}), x_\infty) < 3 r. 
\end{equation}
If not, then we had
\[ ((\varphi_i)_* \nu_{y_j ; s_i}) ( Z \setminus B(x_\infty, r) ) \geq ((\varphi_i)_* \nu_{y_j; s_i}) ( B( \varphi_i (z_{i,j}), r) ) \geq \tfrac12 \]
for infinitely many $i$, which would imply by Claim~\ref{Cl_summary_tdmu} that
\[ \nu_{y_j; s_\infty} (Z \setminus B(x_\infty, r) ) \geq \tfrac12 , \]
and thus
\[ d_{W_1}^Z (\delta_{x_\infty}, \nu_{y_j; s_\infty} )  \geq \tfrac12 r, \]
contrary to our choice of $j$.
So (\ref{eq_limsup_3r}) holds and we may choose $i > j$ to be large enough that $d_Z(\varphi_j (z_{i,j}), x_\infty) < 3 r$.
Now, using Assertion~\ref{Cl_properties_new_Xt_c} and (\ref{eq_choice_y_j}), we find
\[ \int_{B(z_{i,j}, r)} d_{W_1}^Z (\delta_{x_\infty}, \nu_{x; s_\infty}) d\nu_{y_j; s_i} (x) 
\leq \int_{\XX_{s_i}}  \int_Z  d_Z (x_\infty, x') d\nu_{x;s_\infty}(x') d\nu_{y_j; s_i} (x) 
= d_{W_1}^Z (\delta_{x_\infty}, \nu_{y_j; s_\infty} ) < \tfrac12 r. \]
Combining this with (\ref{eq_nu_yj_si_geq_12}), implies that there is a point $x_i \in B(z_{i,j}, r) \subset \XX_{s_i}$ with
\[ d_{W_1}^Z (\delta_{x_\infty}, \nu_{x_i; s_\infty}) < r, \qquad d_Z(\varphi_i (x_i), x_\infty) < 4 r. \]
Repeating this procedure for smaller and smaller $r$ yields the desired sequence $x_i$, which finishes the proof of Assertion~\ref{Cl_properties_new_Xt_d}.

Assertions~\ref{Cl_properties_new_Xt_f}--\ref{Cl_properties_new_Xt_h} are a direct consequence of Claim~\ref{Cl_summary_tdmu}.
\end{proof}

Our previous construction can be performed for all $s_\infty \in \ov{I}^- \setminus I$.
So we may extend the flow $\XX$ by time-slices $(\XX_s, d_s)$ for $s \in \ov{I}^- \setminus I$ and the conjugate heat kernels $\nu_{x;s}$, based at $x \in \XX_t$ with $t \in I$, to $s \in \ov{I}^-$ such that Claim~\ref{Cl_properties_new_Xt} holds for any $s_\infty \in \ov{I}^- \setminus I$.
Then we have:

\begin{Claim}\label{Cl_W1_monotonicity_survives}
For any two points $x_j \in \XX_{t_j}$ with $t_1, t_2 \in I$ and $s_1, s_2 \in \ov{I}^-$ with $s_1 \leq s_2 \leq \min \{ t_1, t_2 \}$ we have
\[ d_{W_1}^{\XX_{s_1}} (\nu_{x_1;s_1}, \nu_{x_2;s_1}) \leq d_{W_1}^{\XX_{s_2}} (\nu_{x_1;s_2}, \nu_{x_2;s_2}) . \]
\end{Claim}

\begin{proof}
If $s_1 \in I$ or $s_2 \in I$, then the claim follows from Claim~\ref{Cl_properties_new_Xt}\ref{Cl_properties_new_Xt_h}.
So suppose that $s_1, s_2 \not\in I$ and $s_1 < s_2$.
Then we can find some $s' \in [s_1, s_2] \cap I$ and conclude, again using Claim~\ref{Cl_properties_new_Xt}\ref{Cl_properties_new_Xt_h}, that
\[ d_{W_1}^{\XX_{s_1}} (\nu_{x_1;s_1}, \nu_{x_2;s_1})
\leq d_{W_1}^{\XX_{s'}} (\nu_{x_1;s'}, \nu_{x_2;s'})  
\leq d_{W_1}^{\XX_{s_2}} (\nu_{x_1;s_2}, \nu_{x_2;s_2}) . \qedhere \]
\end{proof}
\medskip

Next we construct the conjugate heat kernels $\nu_{x;s}$ based at points $x \in \XX_t$ with $t \in \ov{I}^- \setminus I$.
Fix for a moment times $s < s_\infty$ with $s \in \ov{I}^-$ and $s_\infty \in \ov{I}^- \setminus I$ and a point $x_\infty \in \XX_{s_\infty}$.
By Claim~\ref{Cl_properties_new_Xt}\ref{Cl_properties_new_Xt_d} there is a sequence of times $s_i \searrow s_\infty$ and points $x_i \in \XX_{s_i}$ such that $d^{\XX_{s_\infty}}_{W_1} (\delta_{x_\infty} , \nu_{x_i; s_\infty}) \to 0$.
It follows that for any $i \leq j$, using Claim~\ref{Cl_W1_monotonicity_survives}
\begin{equation*}
  d_{W_1}^{\XX_s} ( \nu_{x_i;s}, \nu_{x_j; s} )  
\leq  d_{W_1}^{\XX_{s_\infty}} ( \nu_{x_i;s_\infty}, \nu_{x_j; s_\infty} ) \\  
\leq  d_{W_1}^{\XX_{s_\infty}} ( \nu_{x_i;s_\infty}, \delta_{x_\infty} ) 
+  d_{W_1}^{\XX_{s_\infty}} (\delta_{x_\infty}, \nu_{x_j; s_\infty} )  \xrightarrow[i \to \infty]{} 0.
\end{equation*}
Therefore, we have
\[ \nu_{x_i;s} \xrightarrow[i \to \infty]{\quad W_1 \quad} \nu_{x_\infty;s} \in \mathcal{P} (\XX_s) \]
and $\nu_{x_\infty;s}$ is independent of the choice of the sequence $x_i$.
Moreover, we have for any $t' > s_\infty$ with $t' \in I$ and $x' \in \XX_{t'}$
\begin{equation} \label{eq_dW1_nu_x_infty}
  d_{W_1}^{\XX_s} ( \nu_{x';s}, \nu_{x_\infty; s} )  \leq   d_{W_1}^{\XX_{s_\infty}} ( \delta_{x_\infty}, \nu_{x'; s_\infty}) . 
\end{equation}

Repeating this procedure for all $s_\infty \in \ov{I}^- \setminus I$, $s \in \ov{I}^-$ and $x_\infty \in \XX_{s_\infty}$ allows us to define $\nu_{x;s}$ for all $x \in \XX_t$, $s,t \in \ov{I}^-$, $s \leq t$.
It remains to verify that the new objects constructed so far define an $H$-concentrated metric flow $\XX^* \supset \XX$ that is a future completion of $\XX$.
Due to a limit argument, the statement of Claim~\ref{Cl_W1_monotonicity_survives} can be generalized to the case in which $t_1, t_2 \in \ov{I}^-$.

Properties~\ref{Def_metric_flow_1}--\ref{Def_metric_flow_5} of Definition~\ref{Def_metric_flow} are clear by construction.
Property~\ref{Def_metric_flow_6} holds by Claim \ref{Cl_properties_new_Xt}\ref{Cl_properties_new_Xt_b} if $t \in I$.
Suppose now that $t \in \ov{I}^- \setminus I$ and choose $T \geq 0$ and $f : \XX_s \to \IR$ according to Definition~\ref{Def_metric_flow}\ref{Def_metric_flow_6}.
By Lemma~\ref{Lem_met_flow_T_positive} it suffices to assume that $T > 0$, so $f$ is Lipschitz.
Let $x_\infty, x'_\infty \in \XX_{t_\infty}$.
By Claim~\ref{Cl_properties_new_Xt}\ref{Cl_properties_new_Xt_d} we can find times $t_i \searrow t_\infty$, $t_i \in I$ and points $x_i, x'_i \in \XX_{t_i}$ with
\begin{equation} \label{eq_lim_i_Var_0}
 \lim_{i \to \infty} d_{W_1}^{\XX_{t_\infty}} (\delta_{x_\infty}, \nu_{x_i; t_\infty}) = \lim_{i \to \infty} d_{W_1}^{\XX_{t_\infty}} (\delta_{x'_\infty}, \nu_{x'_i; t_\infty}) = 0 
\end{equation}
and
\begin{equation} \label{eq_i_d_t_i_d_infty}
 \lim_{i \to \infty} d_{t_i} (x_i, x'_i) = d_{t_\infty} (x_\infty, x'_\infty). 
\end{equation}
By (\ref{eq_dW1_nu_x_infty}), (\ref{eq_lim_i_Var_0}) we have
\begin{align*}
 \lim_{i \to \infty} \Phi^{-1} \bigg( \int_{\XX_s} \Phi (f) d\nu_{x_i;s} \bigg) &= \Phi^{-1} \bigg( \int_{\XX_s} \Phi (f) d\nu_{x_\infty;s} \bigg), \\
 \lim_{i \to \infty} \Phi^{-1} \bigg( \int_{\XX_s} \Phi (f) d\nu_{x'_i;s} \bigg) &= \Phi^{-1} \bigg( \int_{\XX_s} \Phi (f) d\nu_{x'_\infty;s} \bigg).
\end{align*}
It follows, using (\ref{eq_i_d_t_i_d_infty}), that
\begin{multline*}
 \bigg| \Phi^{-1} \bigg( \int_{\XX_s} \Phi (f) d\nu_{x_\infty;s} \bigg) - \Phi^{-1} \bigg( \int_{\XX_s} \Phi (f) d\nu_{x'_\infty;s} \bigg) \bigg|
\leq \lim_{i \to \infty} (t_i - s + T)^{-1/2} d_{t_i} (x_i, x'_i)  \\
= (t_\infty - s + T)^{-1/2} d_{t_\infty} (x_\infty, x'_\infty). 
\end{multline*}

By Claim~\ref{Cl_properties_new_Xt}\ref{Cl_properties_new_Xt_c}, Property~\ref{Def_metric_flow_7} of Definition~\ref{Def_metric_flow} holds whenever $t_1 \in \ov{I}^-$, $t_2, t_3 \in I$.
Assume next that $t_1 \in \ov{I}^-$, $t_2 \in \ov{I}^- \setminus I$, $t_3 \in I$ and $x_3 \in \XX_{t_3}$.
Let $f_{t_1} : \XX_{t_1} \to \IR$ be a bounded $1$-Lipschitz function and consider the corresponding heat flow for $t \geq t_1$, $t \in \ov{I}^-$:
\[ f_t : \XX_t \longrightarrow \IR, \qquad x \longmapsto \int_{\XX_{t_1}} f_{t_1} \, d\nu_{x ;t_1}. \]
Fix a time $t_2 < t'_2 < t_3$, $t'_2 \in I$.
Then
\begin{equation} \label{eq_ft1_fs}
 \int_{\XX_{t_1}} f_{t_1} \, d\nu_{x_3; t_1}
= \int_{\XX_{t'_2}} \int_{\XX_{t_1}} f_{t_1} \, d\nu_{y;t_1} d\nu_{x_3;t'_2} (y)
= \int_{\XX_{t'_2}} f_{t'_2} \, d\nu_{x_3; t'_2} . 
\end{equation}
Our goal is to pass this identity to the limit $t'_2 \searrow t_2$.
For this purpose, observe that by (\ref{eq_dW1_nu_x_infty}) we have for any $y \in \XX_{t'_2}$, $z \in \XX_{t_2}$
\[ |f_{t'_2} (y) - f_{t_2}(z)| 
\leq d^{\XX_{t_1}}_{W_1} (\nu_{y;t_1}, \nu_{z;t_1})
\leq d^{\XX_{t_2}}_{W_1} (\nu_{y;t_2}, \delta_z)
 \leq \sqrt{ \Var ( \nu_{y;t_2}, \delta_z)}. \]
Integration against $d\nu_{y;t_2}(z) d\nu_{x_3;t'_2}(y)$ yields
\begin{align*}
\bigg| \int_{\XX_{t'_2}} & f_{t'_2} (y) d\nu_{x_3;t'_2} (y) - \int_{\XX_{t_2}} f_{t_2}(z) d\nu_{x_3; t_2}(z) \bigg| \\
&\leq   \int_{\XX_{t_2}} \int_{\XX_{t'_2}} | f_{t'_2} (y) - f_{t_2} (z) | d\nu_{y;t_2}(z) d\nu_{x_3;t'_2}(y)   \displaybreak[1]  \\
&\leq \int_{\XX_{t_2}} \int_{\XX_{t'_2}} \sqrt{ \Var (\nu_{y;t_2},\delta_z)} d\nu_{y;t_2} (z) d\nu_{x_3;t'_2}(y)   \displaybreak[1]  \\
&\leq \int_{\XX_{t_2}} \bigg( \int_{\XX_{t'_2}}  \Var (\nu_{y;t_2},\delta_z) d\nu_{y;t_2} (z) \bigg)^{1/2} d\nu_{x_3;t'_2}(y)  \displaybreak[1] \\
&=  \int_{\XX_{t_2}} \sqrt{ \Var (\nu_{y;t_2}) } d\nu_{x_3;t'_2}(y) 
\leq \sqrt{ H (t'_2-t_2) } \xrightarrow[t'_2 \searrow t_2]{} 0.
\end{align*}
Combining this with (\ref{eq_ft1_fs}) implies that
\[  \int_{\XX_{t_2}} f_{t_2} \, d\nu_{x_3; t_2}
= \int_{\XX_{t_1}} f_{t_1} \, d\nu_{x_3; t_1}
= \int_{\XX_{t_2}} \int_{\XX_{t_1}} f_{t_1} \, d\nu_{y;t_1} d\nu_{x_3;t_2} (y)
 , \]
as desired, which proves the reproduction identity if $t_1, t_2 \in \ov{I}^-$ and $t_3 \in I$.
The case $t_1, t_2, t_3 \in \ov{I}^-$ follows from the previous case via a simple limit argument, using (\ref{eq_dW1_nu_x_infty}).

We have shown that $\XX$ is a metric flow.
Next we show that $\XX$ is $H$-concentrated.
By Claim~\ref{Cl_properties_new_Xt}\ref{Cl_properties_new_Xt_g} it suffices to show that for any $s < t_\infty$, $s \in \ov{I}^-$, $t_\infty \in \ov{I}^- \setminus I$ and $x_\infty, x'_\infty \in \XX_{t_\infty}$ we have
\begin{equation} \label{eq_Var_XX_star}
 \Var (\nu_{x_\infty; s}, \nu_{x'_\infty;s}) \leq d^2_{t_\infty} (x_\infty, x'_\infty) + H (t_\infty - s). 
\end{equation}
By Claim~\ref{Cl_properties_new_Xt}\ref{Cl_properties_new_Xt_d} there are sequences $x_i, x'_i \in \XX_{t_i}$, $t_i \searrow t_\infty$, with 
\[
\lim_{i\to\infty} d^{\XX_{t_\infty}}_{W_1} (\delta_{x_\infty}, \nu_{x_i; t_\infty}) = \lim_{i\to\infty} d^{\XX_{t_\infty}}_{W_1} (\delta_{x'_\infty}, \nu_{x'_i; t_\infty}) = 0, \qquad \lim_{i \to \infty} d_{t_i} (x_i, x'_i) = d_{t_\infty} (x_\infty, x'_\infty). \]
Passing Claim~\ref{Cl_properties_new_Xt}\ref{Cl_properties_new_Xt_g} to the limit and using
\[ \Var (\nu_{x_\infty; s}, \nu_{x'_\infty;s}) \leq \liminf_{i \to \infty} \Var (\nu_{x_i; s}, \nu_{x'_i;s}), \]
which holds due to (\ref{eq_dW1_nu_x_infty}), implies (\ref{eq_Var_XX_star}).

Lastly note that $\XX^*$ is future complete at any time $t \in \ov{I}^- \setminus I$ due to Claim~\ref{Cl_properties_new_Xt}\ref{Cl_properties_new_Xt_f} and Theorem~\ref{Thm_future_cont_equiv}.
Moreover, for any such $t$ we have $\supp \XX_t = \XX_t$ due to Claim~\ref{Cl_properties_new_Xt}\ref{Cl_properties_new_Xt_a}.
This shows that $\XX^*$ is a future completion of $\XX$.
\end{proof}
\bigskip

\begin{proof}[Proof of Theorem~\ref{Thm_extend_isometry}.]
Let $t_\infty \in I_0 \setminus I'$ and choose times $t_i \in I'$ with $t_i \searrow t_\infty$.
By Proposition~\ref{Thm_future_cont_equiv}\ref{Prop_future_cont_equiv_e} and Lemma~\ref{Lem_combining_embeddings}, and after passing to a subsequence, we can find sequences of metric spaces $(Z_{i}, d_{Z_i})$ and isometric embeddings $\varphi^{i,j}_{t_\infty} : \XX^j_{t_\infty}  \to Z_i$, $\varphi^{i,j}_{t_i} : \XX^j_{t_i} \to Z_i$, $j=1,2$, that satisfy $ \varphi^{i,1}_{t_i} = \varphi^{i,2}_{t_i} \circ \phi_{t_i}$ and such that the following holds.
Consider two conjugate heat flows $(\mu^j_t)_{t \in I''}$ on $\XX^j$, $j=1,2$, with $(\phi_{t_i} )_* \mu^1_{t_i} = \mu^2_{t_i}$, $\Var(\mu^j_t) < \infty$ for all $t \in I''$ and $t_i, t_\infty \in I''$ for large $i$.
Then
\[  d^{Z_i}_{W_1} ( (\varphi^{i,j}_{t_\infty} )_* \mu^j_{t_\infty}, (\varphi^{i,j}_{t_i} )_* \mu^j_{t_i} ) \to 0. \]
 It follows that
\[ d^{Z_{i}}_{W_1} \big( (\varphi^{i,1}_{t_\infty})_* \mu^1_{t_\infty}, (\varphi^{i,2}_{t_\infty})_* \mu^2_{t_\infty} \big) 
\leq d^{Z_i}_{W_1} \big( (\varphi^{i,1}_{t_\infty})_* \mu^1_{t_\infty}, (\varphi^{i,1}_{t_i})_* \mu^1_{t_i} \big) 
+ d^{Z_i}_{W_1} \big(  (\varphi^{i,2}_{t_i})_* \mu^2_{t_i} , (\varphi^{i,2}_{t_\infty})_* \mu^2_{t_\infty} \big)
 \to 0. \]
Choose couplings $q_i$ between $\mu^1_{t_\infty}, \mu^2_{t_\infty}$ such that
\[ \int_{\XX^1_{t_\infty} \times \XX^2_{t_\infty}} d_{Z_i} ( \varphi^{i,1}_{t_\infty}(x_1), \varphi^{i,2}_{t_\infty}(x_2) ) dq_i (x_1, x_2) \to 0 .\]
By Lemma~\ref{Lem_couplings_converge_isometry}, after passing to a subsequence, we can find an isometry $\phi_{t_\infty} : (\XX^1_{t_\infty}, d^1_{t_\infty} ) \to (\XX^2_{t_\infty}, d^2_{t_\infty})$ with the property that
\begin{equation} \label{eq_phi_properties_isom}
d_{Z_i} (\varphi^{i,1}_{t_\infty}(x), \varphi^{i,2}_{t_\infty}(\phi_{t_\infty}(x))) \to 0 \qquad \text{for all} \quad x \in \XX^1_{t_\infty}
\end{equation}
and such that $q_i$ weakly converges to a coupling  of the form $q_\infty = (\id_{\XX^1_{t_\infty}}, \phi_{t_\infty})_* \mu^1_{t_\infty}$ between $\mu^1_{t_\infty}, \mu^2_{t_\infty}$, which implies that 
\begin{equation} \label{eq_phi_mu1mu2_t_infty}
(\phi_{t_\infty})_* \mu^1_{t_\infty} = \mu^2_{t_\infty}. 
\end{equation}
Since (\ref{eq_phi_properties_isom}) characterizes $\phi_{t_\infty}$ uniquely, the same argument implies that (\ref{eq_phi_mu1mu2_t_infty}) holds for \emph{any} two conjugate heat flows $(\mu^j_t)_{t \in I''}$, $j =1,2$ with the same properties.
In the special case of conjugate heat kernels we obtain that for any $t \in I'$, $t \geq t_\infty$, $x \in \XX^1_t$ we have
\begin{equation} \label{eq_phi_t_infty_characterization}
  ( \phi_{t_\infty})_* \nu^1_{x;t_\infty} = \nu^2_{\phi_t(x);t_\infty}. 
\end{equation}
 
Let $x_\infty \in \XX^1_{t_\infty}$.
By Proposition~\ref{Prop_topology_properties}\ref{Prop_topology_properties_e} there is a sequence of points $x_i \in \XX^1_{t_i}$ with $\nu^1_{x_i;t_\infty} \to \delta_{x_\infty}$.
Then by (\ref{eq_phi_t_infty_characterization}) we also have $\nu^2_{\phi_{t_i} (x_i);t_\infty} = (\phi_{t_i})_* \nu^1_{x_i;t_\infty}   \to (\phi_{t_i})_* \delta_{x_\infty} = \delta_{\phi_{t_\infty} (x_\infty)}$, which shows that (\ref{eq_phi_t_infty_characterization}) characterizes $\phi_{t_\infty}$ uniquely.

Next, fix some $s \in I'$, $s < t_\infty$ and $x_\infty \in \XX^1_{t_\infty}$.
We claim that
\[ (\phi_{s} )_* \nu^1_{x_\infty;s} = \nu^2_{\phi_{t_\infty}(x_\infty); s}. \]
To see this, consider the sequence $x_i \to x_\infty$ from the last paragraph and observe that by Proposition~\ref{Prop_topology_properties}\ref{Prop_topology_properties_c} we have
\[ (\phi_{s} )_* \nu^1_{x_\infty; s} \xleftarrow[i \to \infty]{\quad W_1 \quad} (\phi_{s} )_* \nu^1_{x_i;s}
= \nu^2_{\phi_{t_i}(x_i);s} \xrightarrow[i \to \infty]{\quad W_1 \quad} \nu^2_{\phi_{t_\infty}(x_\infty);s}. \]

By repeating the construction above we can extend $(\phi_t)_{t \in I'}$ to $(\phi_t)_{t \in I_0}$ such that $(\phi_{t})_{t \in I' \cup \{ t_\infty \}}$ is a flow isometry for any $t_\infty \in I$.
It remains to show that $(\phi_t)_{t \in I_0}$ is flow isometry.
To see this, let $t_1 < t_3$, $t_1, t_3 \in I_0$. Then we can find some $t_2 \in I'$ with $t_1 < t_2 < t_3$.
By the reproduction formula we have for any $x \in \XX_{t_3}^1$
\begin{multline*}
 (\phi_{t_1})_* \nu^1_{x;t_1} = (\phi_{t_1})_* \int_{\XX^1_{t_2}} \,\nu_{y;t_1}^1 d\nu^1_{x;t_2}(y)
= \int_{\XX^1_{t_2}} ( (\phi_{t_1})_* \nu_{y;t_1}^1 ) \, d\nu^1_{x;t_2}(y) \\
= \int_{\XX^1_{t_2}}  \nu^2_{\phi_{t_2} (y);t_1} \, d\nu^1_{x;t_2}(y) 
= \int_{\XX^2_{t_2}}  \nu^2_{y;t_1} \,  d\nu^2_{\phi_{t_3} (x);t_2}(y)
= \nu^2_{\phi_{t_3} (x);t_1}. 
\end{multline*}
This finishes the proof.
\end{proof}
\bigskip

\section{The space of metric flow pairs} \label{sec_space_met_flow_pairs}
In the following we will consider metric flows $\XX$ equipped with a conjugate heat flow $(\mu_t)$, called \emph{metric flow pairs}.
We will define a distance function $d_{\IF}^J$ on the space of metric flow pairs, which will turn out to be complete.
Convergence with respect to $d_{\IF}^J$ will roughly be equivalent to convergence in the Gromov-$W_1$-Wasserstein distance at almost every time.
In Section~\ref{sec_compact_subsets_IF} we will see that many important families of metric flow pairs, such as those arising from super Ricci flows, are in fact precompact with respect to the $d_{\IF}^J$ distance.

The conjugate heat flow $(\mu_t)$ on $\XX$ will serve as a way of specifying a rough center of each time-slice.
This will be particularly important in the case in which time-slices are not compact.
So $(\mu_t)$ serves as some kind of ``basepoint'' and convergence with respect to $d_{\IF}^J$ may be compared to \emph{pointed} Gromov-Hausdorff convergence.
In most cases we may choose $(\mu_t)$ to be a conjugate heat kernel of the form $(\nu_{x;t})$, where $x$ is a point in the final time-slice of $\XX$.
\subsection{\texorpdfstring{The $\IF$-distance}{The F-distance}}
For the remainder of this subsection suppose that $I \subset \IR$ is an interval.

\begin{Definition}[Metric flow pairs and isometries] \label{Def_metric_flow_pair}
A pair $(\XX, (\mu_t)_{t \in I'})$ is called a {\bf metric flow pair over $I \subset \IR$} if:
\begin{enumerate}
\item $I' \subset I$ with $|I \setminus I'| = 0$.
\item $\XX$ is a metric flow over $I'$.
\item $(\mu_t)_{t \in I'}$ is a conjugate heat flow on $\XX$ with $\supp \mu_t = \XX_t$ for all $t \in I'$.
\end{enumerate}
If $J \subset I'$, then we say that $(\XX, (\mu_t)_{t \in I})$ {\bf is fully defined over $J$.}

If $\td{I} \subset I$ is some subinterval, then the pair $(\XX_{\td{I} \cap I'}, (\mu_t)_{t \in \td I \cap I'})$ is called the {\bf restriction of $(\XX, (\mu_t)_{t \in I})$ to $\td I$.}
If $(\XX^i, (\mu^i_t)_{t \in I^{\prime, i}})$, $i = 1,2$, are two metric flow pairs and $I' \subset I^{\prime, 1} \cap I^{\prime, 2}$, then an isometry $\phi : \XX^1_{I' } \to \XX^2_{ I' }$ between $\XX^1, \XX^2$ over $I'$ is called an {\bf almost always isometry between the metric flow pairs} if $|I^{\prime, 1}  \setminus I'| = |I^{\prime, 2}  \setminus I'| = 0$ and if $(\phi_t )_* \mu^1_t = \mu^2_t$ for all $t \in I'$.
If $J \subset I'$, then we say that $\phi$ {\bf is fully defined over $J$.}
\end{Definition}

Next, let $J \subset I \subset \IR$.
We remark that we will later mainly be interested in the cases $J = \emptyset$ and $J = \{t_{\max} \}$ if $t_{\max} := \max I$ exists.

\begin{Definition}[Spaces of metric flow pairs, $\IF_I^J, \IF^*_I$] \label{Def_IF}
We denote by $\IF_I^J$ the set of equivalence classes of metric flow pairs over   $I$ that are fully defined over $J$, where we call two metric flow pairs equivalent if there is an almost always isometry between them that is fully defined over $J$.

If $J = \emptyset$, then we also write $\IF_I := \IF_I^J$ and if $J = \{ t_0 \}$, then we also write $\IF_I^{t_0} := \IF_I^{\{ t_0 \}}$.
If $t_{\max} := \max I$ exists and $t_{\max} \in J$, then we denote by $\IF_I^{J,*} \subset \IF_I^{ J }$ the subset of equivalence classes of all metric flow pairs $(\XX, (\mu_t)_{t \in I})$ with the property that $\XX_{t_{\max}}$ consists of a single point.
We also write $\IF^{*}_I := \IF^{\{ t_{\max} \}, *}_I$.
\end{Definition}

\begin{Remark}
For any representative $(\XX, (\mu_t)_{t \in I'})$ of an element of $\IF^{J,*}_I$ the measure $\mu_{t_{\max}}$ must be a point mass.
Therefore, if $\XX_{t_{\max}} = \{ x_{\max} \}$, then $\mu_{t} = \nu_{x_{\max}; t}$.
\end{Remark}

If there is no chance of confusion, then we will often conflate isometry classes of metric flow pairs with their representatives.
So we will often write $(\XX, (\mu_t)_{t \in I'}) \in \IF_I^J$ instead of $[(\XX, (\mu_t)_{t \in I'})] \in \IF_I^J$.

The following definition allows us to compare two or more different metric flow pairs and also characterize their convergence.

\begin{Definition}[Correspondence]
Let $\XX^i$ be metric flows over $I^{\prime, i} \subset \IR$, indexed by some $i \in \mathcal{I}$.
A {\bf correspondence} between these metric flows {\bf over some subset $I'' \subset \IR$}  is a pair of the form
\begin{equation} \label{eq_correspondence_form}
 \CF := \big(  (Z_t, d^Z_t)_{t \in I''},(\varphi^i_t)_{t \in I^{\prime\prime, i}, i \in \mathcal{I}} \big), 
\end{equation}
where:
\begin{enumerate}
\item $(Z_t, d^Z_t)$ is a metric space for any $t \in I''$.
\item $I^{\prime\prime, i} \subset I^{\prime, i} \cap I''$ for any $i \in \mathcal{I}$.
\item $\varphi^i_t : (\XX^i_t, d^i_t) \to (Z_t, d^Z_t)$ is an isometric embedding for any $i \in \mathcal{I}$ and $t \in I^{\prime\prime, i}$.
\end{enumerate}
If $J \subset I^{\prime\prime, i}$ for all $i \in \II$, then we say that $\CF$ is {\bf fully defined over $J$.}
If $\td I'' \subset I''$ is some subset, then the pair 
\[ \CF |_{\td I''} := \big(  (Z_t, d^Z_t)_{t \in \td I''},(\varphi^i_t)_{t \in I^{\prime\prime, i} \cap \td I'', i \in \mathcal{I}} \big) \]
is called the {\bf restriction of $\CF$ to $\td I''$.}
If $\td{\mathcal{I}} \subset {\mathcal{I}}$, then the pair $(  (Z_t, d^Z_t)_{t \in  I''},(\varphi^i_t)_{t \in I^{\prime\prime, i} \cap  I'', i \in \td{\mathcal{I}}})$ is called the {\bf restriction of $\CF$ to the index set $\td{\mathcal{I}}$.}
\end{Definition}

The idea behind the subsets $I^{\prime\prime,i}$ is that we want to allow the possibility that the embeddings $\varphi^i_t$ are undefined at certain times.
For example if $\II = \IN \cup \{ \infty \}$, then $\CF$ may describe the convergence behavior of metric flows $\XX^1, \XX^2, \ldots$ to some metric flow $\XX^\infty$; we will provide more details in Section~\ref{Sec_conv_within}.
In the case in which $I = (-\infty, 0]$, we may want to allow this convergence to occur on compact time-intervals of the form $[-T, 0]$.
So we may only require $\varphi^i_t$ to be defined on a time-interval of the form $[-T_i,0]$ for some $T_i \to \infty$.

In the following, we will use correspondences to define a notion of distance between metric flow pairs.
Let first $(\XX^i, (\mu^i_t)_{t \in I^{\prime, i}})$, $i =1,2$, be two metric flow pairs defined over some intervals $I^i \subset \IR$ that are fully defined over some common $J \subset \IR$ and consider a correspondence $\CF = (  (Z_t, d^Z_t)_{t \in I''},(\varphi^i_t)_{t \in I^{\prime\prime, i}, i =1,2} )$ between $\XX^1, \XX^2$ over $I''$ that is also fully defined over $J$.
We will first define an \emph{extrinsic} notion, measuring the closeness of the metric flow pairs $(\XX^i, (\mu^i_t)_{t \in I^{\prime, i}})$, $i =1,2$, within $\CF$.

\begin{Definition}[$\IF$-distance within correspondence] \label{Def_IF_dist_within_CF}
We define the {\bf $\IF$-distance between two metric flow pairs within $\CF$ (uniform over $J$),}
\[ d_{\IF}^{\,\CF, J} \big( (\XX^1, (\mu^1_t)_{t \in I^{\prime,1}}), (\XX^2, (\mu^2_t)_{t \in I^{\prime,2}}) \big), \] 
to be the infimum over all $r > 0$ with the property that there is a measurable subset $E \subset I''$ with
\[ J \subset I'' \setminus E \subset  I^{\prime\prime, 1} \cap I^{\prime\prime,2} \]
and a family of couplings $(q_t)_{t \in I'' \setminus E}$ between $\mu^1_t, \mu^2_t$ such that:
\begin{enumerate}[label=(\arabic*)]
\item \label{Def_IF_dist_within_CF_1} $|E| \leq r^2$.
\item \label{Def_IF_dist_within_CF_3} For all $s, t \in I'' \setminus E$, $s \leq t$, we have
\[ \int_{\XX^1_t \times \XX_t^2} d_{W_1}^{Z_s} ( (\varphi^1_s)_* \nu^1_{x^1; s}, (\varphi^2_s)_* \nu^2_{x^2; s} ) dq_t (x^1, x^2) \leq r. \]
\end{enumerate}
If $J = \emptyset$, then we also write $d_{\IF}^{\,\CF} := d_{\IF}^{\,\CF, J}$ and if $J = \{ t_0 \}$, then we also write $d_{\IF}^{\,\CF, t_0} := d_{\IF}^{\,\CF, J}$.
\end{Definition}

\begin{Remark} \label{Rmk_F_dist_GW1_dist}
By setting $s=t$ in Property~\ref{Def_IF_dist_within_CF_3}, we obtain the following bound for all $t \in I'' \setminus E$:
\begin{multline*}
 d_{GW_1} \big( (\XX^1_t, d^1_t, \mu^1_t), (\XX^2_t, d^2_t, \mu^2_t) \big)
\leq d_{W_1}^{Z_t} ( (\varphi^1_t)_* \mu^1_t , (\varphi^2_t)_* \mu^2_t) \\
\leq \int_{\XX^1_t \times \XX^2_t} d^Z_t ( \varphi^1_t (x^1), \varphi^2_t (x^2) ) dq_t (x^1, x^2) \leq r. 
\end{multline*}
\end{Remark}

Note that the definition of $d_{\IF}^{\, \CF, J}$ depends on the subset $I''$ over which $\CF$ is defined.
So $d_{\IF}^{\, \CF, J}$ measures the closeness of two metric flow pairs restricted to $I''$.
Note also that we allow $d_{\IF}^{\, \CF, J}$ to attain the value $\infty$.

Next, suppose that $(\XX^i, (\mu^i_t)_{t \in I^{\prime, i}})$, $i =1,2$, are metric flow pairs over a common interval $I \subset \IR$ that are both fully defined over some $J \subset I$.
We define the $\IF$-distance between these metric flow pairs by taking the infimum of the $\IF$-distances within all possible correspondences.

\begin{Definition}[$\IF$-distance] \label{Def_IF_dist}
The {\bf $\IF$-distance between two metric flow pairs (uniform over $J$),} 
\[ d_{\IF}^{ J} \big( (\XX^1, (\mu^1_t)_{t \in I^{\prime,1}}), (\XX^2, (\mu^2_t)_{t \in I^{\prime,2}}) \big), \] 
is defined as the infimum of
\[ d_{\IF}^{\,\CF, J} \big( (\XX^1, (\mu^1_t)_{t \in I^{\prime,1}}), (\XX^2, (\mu^2_t)_{t \in I^{\prime,2}}) \big), \] 
over all correspondences $\CF$ between $\XX^1, \XX^2$ over $I$ that are fully defined over $J$.
If $J = \emptyset$, then we also write $d_{\IF} := d_{\IF}^{ J}$ and if $J = \{ t_0 \}$, then we also write $d_{\IF}^{ t_0} := d_{\IF}^{ J}$.
\end{Definition}

\begin{Lemma} \label{Lem_d_IF_star_remove_t_max}
If $t_{\max} := \max I$ exists, then we have $d^{ J}_{\IF} = d^{J \cup \{ t_{\max} \}}_{\IF}$ between any two metric flow pairs representing classes in $\IF^{*, J \cup \{ t_{\max} \}}_I$.
In particular, $d_{\IF} = d^{t_{\max}}_{\IF}$ between any two metric flow pairs representing classes in $\IF^*$.
Moreover, if $(\XX^i, (\mu^i_t)_{t \in I^{\prime,i}})$, $i=1,2$, represent classes in $\IF^*_I$, then
\[ d^{t_{\max}}_{\IF} \big( (\XX^1, (\mu^1_t)_{t \in I^{\prime,1}}), (\XX^2, (\mu^2_t)_{t \in I^{\prime,2}}) \big) = d_{\IF} \big(  (\XX^1_{I \setminus \{ t_{\max} \}}, (\mu^1_t)_{t \in I^{\prime,1} \setminus \{ t_{\max} \}}), (\XX^2_{I \setminus \{ t_{\max} \}}, (\mu^2_t)_{t \in I^{\prime,2} \setminus \{ t_{\max} \}}) \big) . \]
\end{Lemma}

\begin{proof}
If $\CF$ is a correspondence between $\XX^1, \XX^2$ over $I$ that is fully defined over $J$ and $d^{\,\CF, J}_{\IF} \leq r$ and $J \subset I''$ is as in Definition~\ref{Def_IF_dist_within_CF}, then by Remark~\ref{Rmk_F_dist_GW1_dist} we have $d_{W_1}^{Z_s} ( (\varphi^1_s)_* \nu^1_{x_{\max}^1; s}, (\varphi^2_s)_* \nu^2_{x^2_{\max}; s} ) = d_{W_1}^{Z_s} ( (\varphi^1_s)_* \mu^1_{ s}, (\varphi^2_s)_* \mu^2_{s} ) \leq r$ for all $s \in I'' \setminus E$.
So we may replace $\CF$ by a correspondence $\CF'$ in which $\varphi^1_{t_{\max}},\varphi^1_{t_{\max}}$ map to the same point and we still have $d^{\,\CF, J}_{\IF} \leq r$.
\end{proof}

\begin{Remark}
Suppose that $t_{\max} := \max I$ exists.
If $(\XX, (\mu_t)_{t \in I'})$ is a metric flow pair representing a class in $\IF^*_I$, then its restriction to $I \setminus \{ t_{\max} \}$ is a metric flow pair with 
\begin{equation} \label{eq_lim_to_t_max_Var_0}
 \lim_{t \nearrow t_{\max}, t \in I'} \Var (\mu_t) = 0. 
\end{equation}
Vice versa, any metric flow pair over $I \setminus \{ t_{\max} \}$ satisfying (\ref{eq_lim_to_t_max_Var_0}) can be extended to a metric flow pair over $I$ that represents a class in $\IF^*_I$.
By Lemma~\ref{Lem_d_IF_star_remove_t_max} the $\IF$-distance does not change if we restrict metric flow pairs to $I \setminus \{ t_{\max} \}$.
This is why we may sometime conflate the representatives of $\IF^*_I$ with the representatives of $\IF_{I \setminus \{ t_{\max} \}}$ satisfying (\ref{eq_lim_to_t_max_Var_0}).
\end{Remark}

The following is a direct consequence of Definitions~\ref{Def_IF_dist_within_CF}, \ref{Def_IF_dist}.

\begin{Lemma}
$d_{\IF}^{ J}$ is invariant under almost always isometries between metric flow pairs that are fully defined over $J$.
So it descends to a symmetric function
\[ d_{\IF}^{ J} : \IF^J_I \times \IF^J_I \longrightarrow [0, \infty]. \]
\end{Lemma}

\subsection{\texorpdfstring{$(\IF^J_I, d^J_{\IF})$ is a metric space}{(F{\textasciicircum}J\_I, d{\textasciicircum}J\_F) is a metric space}}
Let $I \subset \IR$ be an interval and $J \subset I$ be a subset.
The main result of this subsection is:

\begin{Theorem} \label{Thm_IF_metric_space}
$(\mathbb{F}^J_I, d^J_{\IF})$ is a metric space if we allow infinite distances.
\end{Theorem}

We also obtain the analogous statement for the $d^{\, \CF, J}_{\IF}$-distance.

\begin{Proposition} \label{Prop_triangle_ineq_CC}
Let $(\XX^i, (\mu^i_t)_{t \in I^{\prime, i}})$, $i =1,2,3$, be three metric flow pairs over $I$ that are each fully defined over $J$.
Consider a correspondence $\CF = (  (Z_t, d^Z_t)_{t \in I''},(\varphi^i_t)_{t \in I^{\prime\prime, i}, i =1,2,3} )$ between $\XX^1, \XX^2, \XX^3$ over $I'' \subset \IR$ that is fully defined over $J$.
Then
\begin{multline*}
 d_{\IF}^{\,\CF, J} \big( (\XX^1, (\mu^1_t)_{t \in I^{\prime,1}}), (\XX^3, (\mu^3_t)_{t \in I^{\prime,3}}) \big) \\
\leq d_{\IF}^{\,\CF, J} \big( (\XX^1, (\mu^1_t)_{t \in I^{\prime,1}}), (\XX^2, (\mu^2_t)_{t \in I^{\prime,2}}) \big) 
+ d_{\IF}^{\,\CF, J} \big( (\XX^2, (\mu^2_t)_{t \in I^{\prime,2}}), (\XX^3, (\mu^3_t)_{t \in I^{\prime,3}}) \big).
\end{multline*}
\end{Proposition}

\begin{proof}[Proof of Proposition~\ref{Prop_triangle_ineq_CC}.]
Choose $r^{12}, r^{23} > 0$ such that there are subsets $E^{12}, E^{23} \subset I$ with 
\[ J \subset I'' \setminus E^{12} \subset I^{\prime\prime,1} \cap I^{\prime\prime,2}, \qquad 
J \subset I'' \setminus E^{23} \subset I^{\prime\prime,2} \cap I^{\prime\prime,3}, \]
and families of couplings $(q^{12}_t)_{t \in I \setminus E^{12}}$, $(q^{23}_t)_{t \in I \setminus E^{23}}$ between $\mu_t^1, \mu_t^2$ and $\mu_t^2, \mu_t^3$, respectively, such that Properties~\ref{Def_IF_dist_within_CF_1}, \ref{Def_IF_dist_within_CF_3} of Definition~\ref{Def_IF_dist_within_CF} hold for the flow pairs $(\XX^1, (\mu^1_t)_{t \in I^{\prime, 1}}), (\XX^2, (\mu^2_t)_{t \in I^{\prime, 2}})$ and $(\XX^2, (\mu^2_t)_{t \in I^{\prime, 2}}), (\XX^3, (\mu^3_t)_{t \in I^{\prime, 3}})$, respectively.
Let $r^{13} := r^{12} + r^{23}$ and $E^{13} := E^{12} \cup E^{23}$.
By Lemma~\ref{Lem_gluing}, for any $t \in I'' \setminus E^{13}$ there is a probability measure $q_t^{123}$ on $\XX^1_t \times \XX^2_t \times \XX^3_t$ whose marginals onto the first and last two factors equal $q_t^{12}, q_t^{23}$, respectively.
Let $q_t^{13}$ be the marginal of $q_t^{123}$ onto the first and third factor.
We now verify Properties~\ref{Def_IF_dist_within_CF_1}, \ref{Def_IF_dist_within_CF_3} of Definition~\ref{Def_IF_dist_within_CF}.
Property~\ref{Def_IF_dist_within_CF_1} holds since 
\[ |E^{13}| \leq |E^{12}| + |E^{23}| \leq (r^{12})^2 + (r^{23})^2 \leq (r^{13})^2. \]
For Property~\ref{Def_IF_dist_within_CF_3} we have for any $s,t \in I'' \setminus E^{13}$, $s \leq t$,
\begin{align*}
& \int_{\XX^1_t \times \XX^3_t} d_{W_1}^{Z_s} ( (\varphi^1_s)_* \nu_{x^1; s}, (\varphi^3_s)_* \nu_{x^3; s} ) dq_t^{13} (x^1, x^3) \\
 &\qquad = \int_{\XX^1_t \times \XX^2_t \times X^3_t} d_{W_1}^{Z_s} ( (\varphi^1_s)_* \nu_{x^1; s}, (\varphi^3_s)_* \nu_{x^3; s} ) dq_t^{123} (x^1, x^2, x^3) \displaybreak[1] \\
 &\qquad\leq \int_{\XX^1_t \times \XX^2_t \times X^3_t} \big( d_{W_1}^{Z_s} ( (\varphi^1_s)_* \nu_{x^1; s}, (\varphi^2_s)_* \nu_{x^2; s} ) + d_{W_1}^{Z_s} ( (\varphi^2_s)_* \nu_{x^2; s}, (\varphi^3_s)_* \nu_{x^3; s} ) \big) dq_t^{123} (x^1, x^2, x^3)  \displaybreak[1] \\
 &\qquad = \int_{\XX^1_t \times \XX^2_t}  d_{W_1}^{Z_s} ( (\varphi^1_s)_* \nu_{x^1; s}, (\varphi^2_s)_* \nu_{x^2; s} )  dq_t^{12} (x^1, x^2) \\
 &\qquad\qquad\qquad\qquad
  +  \int_{\XX^2_t \times \XX^3_t}  d_{W_1}^{Z_s} ( (\varphi^2_s)_* \nu_{x^2; s}, (\varphi^3_s)_* \nu_{x^3; s} ) dq_t^{23} (x^2, x^3) \\
  &\qquad \leq r^{12} + r^{23} = r^{13}.
 \end{align*}
 This finishes the proof.
\end{proof}
\bigskip

For the proof of Theorem~\ref{Thm_IF_metric_space}, we will need the following lemma, which states that we can combine correspondences between two pairs of metric flows of the form $\XX^1, \XX^2$ and $\XX^2, \XX^3$.

\begin{Lemma} \label{Lem_combining_correspondences}
Let $\XX^i$, $i =1,2,3$, be three metric flows and consider correspondences
\[  \CF^{12} := \big(  (Z^{12}_t, d^{Z^{12}}_t)_{t \in I^{\prime\prime, 12}},(\varphi^{12,i}_t)_{t \in I^{\prime\prime, 12, i}, i =1,2} \big), \qquad
 \CF^{23} := \big(  (Z^{23}_t, d^{Z^{23}}_t)_{t \in I^{\prime\prime, 23}},(\varphi^{23,i}_t)_{t \in I^{\prime\prime, 23, i}, i =2,3} \big) \]
between $\XX^1, \XX^2$ and $\XX^2, \XX^3$ over $I^{\prime\prime, 12}$ and $I^{\prime\prime, 23}$, respectively.
Set 
\begin{equation*}
 I^{\prime\prime, 123} := I^{\prime\prime, 12} \cup I^{\prime\prime, 23}, \qquad
 I^{\prime\prime, 123,1} := I^{\prime\prime, 12,1} , \qquad
 I^{\prime\prime, 123,2} :=  I^{\prime\prime, 12, 2} \cup  I^{\prime\prime, 23, 2}, \qquad 
  I^{\prime\prime, 123,3} := I^{\prime\prime, 23,3} . 
\end{equation*}
Then there is a correspondence of the form
\[ \CF^{123} := \big(  (Z^{123}_t, d^{Z^{123}}_t)_{t \in I^{\prime\prime, 123}},(\varphi^{123,i}_t)_{t \in I^{\prime\prime, 123, i}, i =1,2,3} \big) \]
between $\XX^1, \XX^2, \XX^3$ over $I^{\prime\prime, 123}$  and families of isometric embeddings
\[ \big( \iota^{12}_t : (Z^{12}_t, d^{Z^{12}}_t) \longrightarrow (Z^{123}_t, d^{Z^{123}}_t) \big)_{t \in I^{\prime\prime, 12}}, \qquad
\big( \iota^{23}_t : (Z^{23}_t, d^{Z^{23}}_t) \longrightarrow (Z^{123}_t, d^{Z^{123}}_t) \big)_{t \in I^{\prime\prime, 23}}, \]
such that the following identities hold for all $t$ for which they are defined
\begin{equation} \label{eq_iota123_varphi}
 \iota^{12}_t \circ \varphi^{12,1}_t   = \varphi^{123, 1}_t, \qquad
 \iota^{12}_t \circ \varphi^{12,2}_t = \iota^{23}_t \circ \varphi^{23,2}_t = \varphi^{123,2}_t, \qquad
 \iota^{23}_t \circ \varphi^{23,3}_t   = \varphi^{123, 3}_t. 
\end{equation}
\end{Lemma}

\begin{proof}
If $t \in I^{\prime\prime, 12, 2} \cap I^{\prime\prime, 23, 2}$, then we may define $(Z^{123}_t, d^{Z^{123}}_t)$, $\iota^{12}_t$, $\iota^{23}_t$ using Lemma~\ref{Lem_combining_embeddings} and choose $\varphi^{123,i}_t$ such that (\ref{eq_iota123_varphi}) holds.
If $t \in (I^{\prime\prime, 12} \cup I^{\prime\prime, 23}) \setminus (I^{\prime\prime, 12, 2} \cap I^{\prime\prime, 23, 2})$, then we may set $(Z^{123}_t, d^{Z^{123}}_t) := (Z^{12}_t, d^{Z^{12}}_t) \times (Z^{23}_t, d^{Z^{23}}_t)$, $ (Z^{12}_t, d^{Z^{12}}_t)$ or $(Z^{23}_t, d^{Z^{23}}_t)$, let $\iota^{12}_t, \iota^{23}_t$ be embeddings of each factor and again choose $\varphi^{123,i}_t$ such that (\ref{eq_iota123_varphi}) holds. 
\end{proof}
\bigskip

\begin{proof}[Proof of Theorem~\ref{Thm_IF_metric_space}.]
The triangle inequality follows by combining Proposition~\ref{Prop_triangle_ineq_CC} and Lemma \ref{Lem_combining_correspondences}.
 It remains to show definiteness.
 For this purpose, assume that representatives of two classes in $\IF_I^J$ satisfy
\[ d^J_{\mathbb{F}} \big( (\XX^1, (\mu^1_t)_{t \in I^{\prime,1}}), (\XX^2, (\mu^2_t)_{t \in I^{\prime,2}}) \big) =0. \]
So there is a sequence of correspondences $\CF^j$ over $I$ that are fully defined over $J$ such that
\[ d^{\,\CF^j, J}_{\mathbb{F}} \big( (\XX^1, (\mu^1_t)_{t \in I^{\prime,1}}), (\XX^2, (\mu^2_t)_{t \in I^{\prime,2}}) \big)  \to 0. \]
For each $j$ choose $E^j \subset I$ and $(q^j_t)_{t \in I \setminus E^j}$ such that the properties of Definition~\ref{Def_IF_dist_within_CF} hold for $\CF = \CF^j$ and $r = r_j$.
Since $|E^j| \to 0$, the set
\[ E := \bigcap_{j=1}^\infty \bigcup_{k=j}^\infty E^k. \]
has measure zero and $J \subset I \setminus E \subset I^{\prime,1} \cap I^{\prime,2}$.
By Remark~\ref{Rmk_F_dist_GW1_dist} we have $d_{GW_1} ( (\XX^1_t, d^1_t, \mu^1_t), (\XX^2_t, d^2_t, \mu^2_t)) = 0$ and therefore
\[ (\XX^1_t, \lb d^1_t, \lb \mu^1_t) \cong (\XX^2_t, \lb d^2_t, \lb \mu^2_t). \]
It remains to specify an appropriate family of isometries for each $t \in I \setminus E$.

\begin{Claim}
Consider an arbitrary subsequence of the sequence $(q^j_t)_{t \in I \setminus E^j}$ and let $s,t \in I \setminus E$, $s \leq t$.
Then we can pass to a further subsequence such that $q^j_s, q^j_t$ converge in the weak topology to couplings of the form $q^\infty_s = (\id_{\XX^1_s}, \phi_s )_* \mu^1_s$, $q^\infty_t = (\id_{\XX^1_t}, \phi_t )_* \mu^1_t$, where $\phi_s : (\XX^1_s, d^1_s, \mu^1_s) \to (\XX^2_s, d^2_s, \mu^2_s)$, $\phi_t : (\XX^1_t, d^1_t, \mu^1_t) \to (\XX^2_t, d^2_t, \mu^2_t)$ are isometries and $(\phi_s)_* \nu^1_{x;s} = \nu^2_{\phi_t (x);s}$ for all $x \in \XX^1_t$.
\end{Claim}

\begin{proof}
Apply Property~\ref{Def_IF_dist_within_CF_3} of Definition~\ref{Def_IF_dist_within_CF} for $s,t$ replaced with $s,s$ and then with $t,t$ and apply Lemma~\ref{Lem_couplings_converge_isometry}.
This produces the maps $\phi_s, \phi_t$.
Next, apply Property~\ref{Def_IF_dist_within_CF_3} of Definition~\ref{Def_IF_dist_within_CF} for $s,t$ and observe that the integrand in this property is 2-Lipschitz to obtain the last statement.
\end{proof}

Choose a countable, dense subset $Q \subset I \setminus E$.
Due to the Claim and after passing to a diagonal subsequence, we may assume that there is a flow isometry $\phi : \XX^1_Q \to \XX^2_Q$ over $Q$ such that for all $t \in Q$ we have $q^j_t \to q^\infty_t = (\id_{\XX^1_t}, \phi_t )_* \mu^1_t$ in the weak sense.
Let us now extend $\phi$ to a flow isometry over $I \setminus E$.
For this purpose, consider a time $t \in I \setminus E$ with $t \not\in Q$.
Again, using the Claim and after passing to a diagonal subsequence, we can find an isometry  $\phi_t : (\XX^1_t, d^1_t, \mu^1_t) \to (\XX^2_t, d^2_t, \mu^2_t)$ such that the following holds for any $s \in Q$:
\begin{enumerate}
\item If $s < t$, then $(\phi_s)_* \nu^1_{x;s} = \nu^2_{\phi_t (x);s}$ for all $x \in \XX^1_t$.
\item If $s > t$, then $(\phi_t)_* \nu^1_{x;t} = \nu^2_{\phi_s (x);t}$ for all $x \in \XX^1_s$.
\end{enumerate}
Repeating this procedure for any $t \in I \setminus E$, $t \not\in Q$, produces family of isometries $(\phi_t)_{t \in I \setminus E}$ such that for any $t \in I \setminus E$ and $s \in Q$ Properties~(1), (2) hold.
(Note that for every single $t$, we can use the initial subsequence.
So we don't have to pass to successive subsequences.)

We claim that $(\phi_t)_{t \in I \setminus E}$ is a flow isometry between $\XX^1, \XX^2$ over $I \setminus E$, which shows the equivalence of the metric flow pairs $(\XX^1, (\mu^1_t)_{t \in I^{\prime,1}}), (\XX^2, (\mu^2_t)_{t \in I^{\prime,2}})$.
For this purpose, we need to show that for any $s, t \in I \setminus E$, $s < t$ and $x \in \XX^1_t$ we have
\[ (\phi_s)_* \nu^1_{x;s} = \nu^2_{\phi_t (x);s}. \]
Choose some time $t' \in (s,t) \cap Q$.
Then by Properties~(1), (2) above and the reproduction formula, Definition~\ref{Def_metric_flow}\ref{Def_metric_flow_7},
\begin{multline} \label{eq_reprod_form_isometry}
 (\phi_s)_* \nu^1_{x;s}
= (\phi_s)_* \int_{\XX^1_{t'}} \nu^1_{y; s} \, d\nu^1_{x; t'}(y)
= \int_{\XX^1_{t'}}  (\phi_s)_* \nu^1_{y; s} \, d\nu^1_{x; t'}(y)
= \int_{\XX^1_{t'}}   \nu^2_{\phi_{t'} (y); s} \, d\nu^1_{x; t'}(y) \\
= \int_{\XX^2_{t'}}   \nu^2_{y; s} \, d(\phi_{t'})_*\nu^1_{x; t'}(y)
= \int_{\XX^2_{t'}}   \nu^2_{y; s} \, d\nu^2_{\phi_t(x); t'}(y)
= \nu^2_{\phi_t(x); s}.
\end{multline}
This finishes the proof.
\end{proof}

\subsection{Useful lemmas}
Before discussing further details, we first establish some useful lemmas addressing the definition $d^{\,\CF, J}_{\IF}$.

The first lemma concerns the case in which two metric flow pairs have isometric metric flows.

\begin{Lemma} \label{Lem_F_dist_same_X}
Let $\XX$ be a metric flow over some subset $I' \subset \IR$ and consider two conjugate heat flows $(\mu^i_t)_{t \in I^{\prime, i}}$ on $\XX$, $I^{\prime, i} \subset I'$, $i = 1,2$, on $\XX$ such that $(\XX_{I^{\prime, i}}, (\mu^i_t)_{t \in I^{\prime, i}})$ are metric flow pairs over some interval $I$ that are fully defined over some subset $J \subset I$.
Consider a correspondence $\CF = (  (Z_t, d^Z_t)_{t \in I''},(\varphi^i_t)_{t \in I^{\prime\prime, i}, i =1,2} )$ over $I'' \subset \IR$ between $\XX$ and itself with $\varphi^1_t = \varphi^2_t$ that is fully defined over $J$.
Then $d_{\IF}^{\,\CF, J} \big( (\XX^1_{I^{\prime,1}}, (\mu^1_t)_{t \in I^{\prime,1}}), (\XX^2_{I^{\prime,2}}, (\mu^2_t)_{t \in I^{\prime,2}}) \big)$ equals the infimum over all $r > 0$ with the property that there is a measurable subset $E \subset I''$ with $J \subset I'' \setminus E \subset I^{\prime\prime, 1} \cap I^{\prime\prime, 2}$ such that
\[ |E| \leq r^2 \qquad \text{and} \qquad d^{\XX_t}_{W_1} ( \mu^1_t, \mu^2_t) \leq r \qquad \text{for all} \quad t \in I'' \setminus E. \]
\end{Lemma}

\begin{proof}
This is a direct consequence of Proposition~\ref{Prop_compare_CHF}\ref{Prop_compare_CHF_c}.
\end{proof}

In the next lemma we derive a bound on the $W_1$-distance between conjugate heat kernels based at nearby points with respect to a correspondence.

\begin{Lemma} \label{Lem_XX12_close_HK_close}
Let $(\XX^i, (\mu^i_t)_{t \in I^{\prime, i}})$, $i =1,2$, be a metric flow pairs over an interval $I \subset \IR$ that are fully defined over some $J \subset I$ and let $\CF = (  (Z_t, d^Z_t)_{t \in I''},(\varphi^i_t)_{t \in I^{\prime\prime, i}, i =1,2} )$ be a correspondence between $\XX^1, \XX^2$ over $I''$ that is also fully defined over $J$.
Let $\delta, r > 0$.
Suppose that
\begin{equation} \label{eq_XX12_close_HK_close_asspt}
 d_{\IF}^{\,\CF, J} \big( (\XX^1, (\mu^1_t)_{t \in I^{\prime,1}}), (\XX^2, (\mu^2_t)_{t \in I^{\prime,2}}) \big) \leq \delta r 
\end{equation}
Consider times $s, t \in J$, $s \leq t$ and points $x^i \in \XX^i_t$ with
\[ d^Z_t ( \varphi^1_t (x^1), \varphi^2_t (x^2)) \leq r, \qquad |B(x^1, r)| \geq 2\delta. \]
Then
\begin{equation} \label{eq_XX12_close_HK_close_assertion}
 d_{W_1}^{Z_s} ( (\varphi^1_s)_* \nu^1_{x^1; s}, (\varphi^2_s)_* \nu^2_{x^2; s}) ) \leq 7r. 
\end{equation}
\end{Lemma}

\begin{proof}
Due to a limit argument, we may assume that we have strict inequality in (\ref{eq_XX12_close_HK_close_asspt}).
Choose $E$ and $(q_t)_{t \in I \setminus E}$ so that Properties~\ref{Def_IF_dist_within_CF_1}, \ref{Def_IF_dist_within_CF_3} of Definition~\ref{Def_IF_dist_within_CF} hold for $r$ replaced with $\delta r$.
Set
\[ d := d_{W_1}^{Z_s} ( (\varphi^1_s)_* \nu^1_{x^1; s}, (\varphi^2_s)_* \nu^2_{x^2; s}) ). \]
By Proposition~\ref{Prop_compare_CHF}\ref{Prop_compare_CHF_c} we have for $k = 1,2$
\[ d_{W_1}^{Z_s} \big( (\varphi^k_s)_* \nu^k_{x^k;s}, (\varphi^k_s)_* \nu^k_{y;s} \big) 
= d_{W_1}^{\XX^k_s} \big(  \nu^k_{x^k;s},  \nu^k_{y;s} \big) 
\leq 3r \qquad \text{for all} \quad y \in B(x^k, 3r). \]
Therefore, by Definition~\ref{Def_IF_dist_within_CF}\ref{Def_IF_dist_within_CF_3}
\begin{align}
 d &= \frac1{q_t (B(x^1, 3r) \times B(x^{2}, 3r))} \int_{B(x^1,3 r) \times B(x^{2}, 3r)} d \, dq_t \notag \\
&\leq \frac1{q_t (B(x^1, 3r) \times B(x^{2}, 3r))}  \int_{B(x^1,3 r) \times B(x^{2}, 3r)} \big( d_{W_1}^{ Z_s} \big( (\varphi^1_s)_* \nu^1_{y^1;s}, (\varphi^{2}_s)_* \nu^2_{y^{2};s} \big) + 6r \big) dq_t (y^1, y^2) \notag \\ & 
\leq \frac{\delta r}{q_t (B(x^1, 3r) \times B(x^{2}, 3r))} + 6r. \label{eq_d_bound_integral}
\end{align}
For any $y^1 \in B(x^1, r)$ and $y^2 \in \XX^2_t \setminus B(x^{2}, 3r)$ we have
\begin{multline*}
 d^{Z}_t ( \varphi^1_t (y^1), \varphi^{2}_t (y^2) ) 
\geq d^{ Z}_t ( \varphi^{2}_t (x^{2}), \varphi^{2}_t (y^2) ) - 
d^{ Z}_t ( \varphi^2_t (x^{2}), \varphi^{1}_t (x^1) ) - d^{ Z}_t ( \varphi^1_t (x^{1}), \varphi^{1}_t (y^{1}) ) \\
\geq 3r - r - r = r.
\end{multline*}
So by Definition~\ref{Def_IF_dist_within_CF}\ref{Def_IF_dist_within_CF_3} for $s=t$ we have
\[ q_t \big(   B(x^1, r) \times (\XX^{2}_t \setminus B(x^{2},3r)) \big) 
\leq \frac1{r} \int_{ B(x^1, r) \times (\XX^{2}_t \setminus B(x^{2},3r)) } d^Z_t ( \varphi^1_t (y^1), \varphi^2_t (y^2) ) dq_t (y^1, y^2) 
\leq \delta . \]
which implies
\[ q_t \big( B(x^1, 3r) \times B(x^{2},3r)) \geq \mu^1_t \big(B(x^1, r) \big)  - q_t \big(   B(x^1, r) \times (\XX^{2}_t \setminus B(x^{2},3r)) \big) \geq 2\delta - \delta = \delta. \]
Combining this with (\ref{eq_d_bound_integral}) implies (\ref{eq_XX12_close_HK_close_assertion}).
\end{proof}

Next, we prove a lemma that allows us to compare two different couplings between probability measures in two metric flows.

\begin{Lemma} \label{Lem_X1X2_correspondence_mu112}
Let $\CF = (  (Z_t, d^Z_t)_{t \in I''},(\varphi^i_t)_{t \in I^{\prime\prime, i}, i =1,2} )$ be a correspondence between two metric flows $\XX^1, \XX^2$ and consider times $s, t \in I^{\prime\prime, 1} \cap I^{\prime\prime,2}$, $s \leq t$.
Let $\mu_1, \mu'_1 \in \PP (\XX^1_t)$, $\mu_2 \in \PP (\XX^2_t)$ and consider couplings $q, q'$ between $\mu_1, \mu_2$ and $\mu'_1, \mu_2$, respectively.
Then
\begin{align*}
 \int_{\XX^1_t \times \XX_t^2} & d_{W_1}^{Z_s} ( (\varphi^1_s)_* \nu^1_{x^1; s}, (\varphi^2_s)_* \nu^2_{x^2; s} ) dq' (x^1, x^2)  \\
&\leq  \int_{\XX^1_t \times \XX_t^2} \big( d_t^Z ( \varphi^1_t (x^1), \varphi^2_t (x^2)) + d_{W_1}^{Z_s} ( (\varphi^1_s)_* \nu^1_{x^1; s}, (\varphi^2_s)_* \nu^2_{x^2; s} ) \big) dq (x^1, x^2) \\
&\qquad +  \int_{\XX^1_t \times \XX_t^2} d_t^Z ( \varphi^1_t (x^1), \varphi^2_t (x^2)) dq' (x^1, x^2)  .
\end{align*}
\end{Lemma}

\begin{proof}
By Lemma~\ref{Lem_gluing} we can find a probability measure $\ov q \in \PP (\XX^1_t \times \XX^1_t \times \XX^2_t)$ whose marginals onto the first and last two factors equal $q, q'$, respectively.
Then, using Proposition~\ref{Prop_compare_CHF}\ref{Prop_compare_CHF_c},
\begin{align*}
 \int_{\XX^1_t \times \XX_t^2} & d_{W_1}^{Z_s} ( (\varphi^1_s)_* \nu^1_{x^1; s}, (\varphi^2_s)_* \nu^2_{x^2; s} ) dq' (x^1, x^2)  \\
&= \int_{\XX^1_t \times \XX^1_t \times \XX_t^2}  d_{W_1}^{Z_s} ( (\varphi^1_s)_* \nu^1_{y^1; s}, (\varphi^2_s)_* \nu^2_{x^2; s} ) d\ov q (x^1, y^1, x^2) \displaybreak[1] \\
&\leq \int_{\XX^1_t \times \XX^1_t \times \XX_t^2} \big(  d_{W_1}^{Z_s} ( (\varphi^1_s)_* \nu^1_{y^1; s}, (\varphi^1_s)_* \nu^1_{x^1; s} ) +  d_{W_1}^{Z_s} ( (\varphi^1_s)_* \nu^1_{x^1; s}, (\varphi^2_s)_* \nu^2_{x^2; s} )  \big) d\ov q (x^1, y^1, x^2) \displaybreak[1] \\
&\leq \int_{\XX^1_t \times \XX^1_t \times \XX_t^2} d_{t}^{Z} ( \varphi^1_t (y^1), \varphi^1_t(x^1) ) d\ov q (x^1, y^1, x^2) \\
&\qquad + \int_{\XX^1_t \times \XX_t^2} d_{W_1}^{Z_s} ( (\varphi^1_s)_* \nu^1_{x^1; s}, (\varphi^2_s)_* \nu^2_{x^2; s} )  dq (x^1, x^2) \displaybreak[1] \\
&\leq \int_{\XX^1_t \times \XX^1_t \times \XX_t^2} \big(  d_{t}^{Z} ( \varphi^1_t (y^1), \varphi^2_t(x^2) ) + d_{t}^{Z} ( \varphi^2_t (x^2), \varphi^1_t(x^1) ) \big) d\ov q (x^1, y^1, x^2) \\
&\qquad 
+ \int_{\XX^1_t \times \XX_t^2} d_{W_1}^{Z_s} ( (\varphi^1_s)_* \nu^1_{x^1; s}, (\varphi^2_s)_* \nu^2_{x^2; s} )  dq (x^1, x^2) \displaybreak[1] \\
&=  \int_{\XX^1_t \times \XX_t^2} \big( d_t^Z ( \varphi^1_t (x^1), \varphi^2_t (x^2)) + d_{W_1}^{Z_s} ( (\varphi^1_s)_* \nu^1_{x^1; s}, (\varphi^2_s)_* \nu^2_{x^2; s} ) \big) dq (x^1, x^2) \\
&\qquad +  \int_{\XX^1_t \times \XX_t^2} d_t^Z ( \varphi^1_t (x^1), \varphi^2_t (x^2)) dq' (x^1, x^2)  . \qedhere
\end{align*}
\end{proof}
\medskip

The following lemma shows that we are quite flexible in the choice of the couplings $(q_t)_{t \in I'' \setminus E}$ in Definition~\ref{Def_IF_dist_within_CF}.
In fact, we can replace these couplings by other couplings $(q'_t)_{t \in I'' \setminus E}$ as long as we can ensure bounds on $\int_{\XX^1_t \times \XX^2_t} d^Z_t ( \varphi^1_t (x^1), \varphi^2_t (x^2) ) dq'_t (x^1, x^2)$.

\begin{Lemma} \label{Lem_F_dist_td_mu}
Let $(\XX^i, (\mu^i_t)_{t \in I^{\prime, i}})$, $i =1,2$, be two metric flow pairs and consider a correspondence $\CF = (  (Z_t, d^Z_t)_{t \in I''},(\varphi^i_t)_{t \in I^{\prime\prime, i}, i =1,2} )$ between $\XX^1, \XX^2$ over $I''$ that is fully defined over two times $\{ s, t \}$, $s \leq t$.
Then for any coupling $q'$ between $\mu^1_t, \mu^2_t$ we have
\begin{multline*}
 \int_{\XX^1_t \times \XX_t^2} d_{W_1}^{Z_s} ( (\varphi^1_s)_* \nu^1_{x^1; s}, (\varphi^2_s)_* \nu^2_{x^2; s} ) dq' (x^1, x^2) 
 \leq 
  \int_{\XX^1_t \times \XX^2_t} d^Z_t ( \varphi^1_t (x^1), \varphi^2_t (x^2) ) dq' (x^1, x^2) \\
+ 2 d^{\, \CF, \{ s, t \}}_{\IF} \big((\XX^1, (\mu^1_t)_{t \in I^{\prime, 1}}), (\XX^2, (\mu^2_t)_{t \in I^{\prime, 2}}) \big).
\end{multline*}
\end{Lemma}

\begin{proof}
This is a direct consequence of Lemma~\ref{Lem_X1X2_correspondence_mu112}.
\end{proof}

\subsection{Completeness}
For the remainder of this subsection fix again some $J \subset I \subset \IR$, where $I$ is an interval.
The main result of this subsection is:

\begin{Theorem} \label{Thm_F_complete}
$(\mathbb{F}^J_I , d^J_\mathbb{F})$ is complete.
\end{Theorem}

Theorem~\ref{Thm_F_complete} will be a consequence of the following lemma, which establishes the existence of a limit within a given correspondence.

\begin{Lemma} \label{Lem_F_completemess_within_CF}
Let $(\XX^i, (\mu^i_t)_{t \in I^{\prime,i}})$, $i \in \IN$, be a sequence of metric flows over $I$ that are fully defined over $J$.
Consider a correspondence $\CF := (  (Z_t, d^Z_t)_{t \in I},(\varphi^i_t)_{t \in I^{\prime\prime, i}, i \in \IN} )$ between the metric flows $\XX^i$ over $I$ that is also fully defined over $J$ and suppose that the metric spaces $(Z_t, d^Z_t)_{t \in I}$ are complete.
Suppose that the metric flow pairs $(\XX^i, (\mu^i_t)_{t \in I^{\prime,i}})$ form a Cauchy sequence within $\CF$ that is uniform over $J$, in the sense that for any $\eps >0$ there is an $\underline{i} \geq 1$ such that for all $i, j \geq \underline{i}$
\[ d^{\, \CF, J}_{\IF} \big( (\XX^i, (\mu^i_t)_{t \in I^{\prime,i}}), (\XX^j, (\mu^j_t)_{t \in I^{\prime,j}}) \big) \leq \eps. \]
Then there is a metric flow pair $(\XX^\infty, (\mu^\infty_t)_{t \in I^{\prime,\infty}})$ over $I$ that is fully defined over $J$ and a family of isometric embeddings $(\varphi^\infty_t : \XX^\infty_t \to Z_t)_{t \in I^{\prime\prime, \infty}}$, $I^{\prime\prime, \infty} \subset  I$ such that
\begin{equation} \label{eq_CFp_def_lemma}
 \CF' := (  (Z_t, d^Z_t)_{t \in I},(\varphi^i_t)_{t \in I^{\prime\prime, i}, i \in \IN \cup \{ \infty \}}  \big) 
\end{equation}
is a correspondence between all metric flows $\XX^i$, $i \in \IN \cup \{ \infty \}$, and such that we have convergence
\begin{equation} \label{eq_dFCFp_conv_assertion}
 d^{\, \CF', J}_{\IF} \big( (\XX^i, (\mu^i_t)_{t \in I^{\prime,i}}), (\XX^\infty, (\mu^\infty_t)_{t \in I^{\prime,\infty}}) \big) \to 0. 
\end{equation}
\end{Lemma}

\begin{proof}
After replacing each $I^{\prime, i}$ with $\bigcap_{i=1}^\infty I^{\prime,i}$, we may assume that $I^{\prime, i} = I'$ for all $i$.
As we are allowed to pass to a subsequence, we may further assume that
\[ d^{\, \CF, J}_{\IF} \big( (\XX^i, (\mu^i_t)_{t \in I'}), (\XX^{i+1}, (\mu^{i+1}_t)_{t \in I'}) \big) \leq 2^{-i-2}. \]
For each $i$ choose $E^{i,i+1} \subset I$ with $J \subset I \setminus E^{i, i+1} \subset I^{\prime\prime, i} \cap I^{\prime\prime, i+1}$ and $(q_t^{i, i+1})_{t \in I \setminus E^{i, i+1}}$ such that Properties~\ref{Def_IF_dist_within_CF_1}, \ref{Def_IF_dist_within_CF_3} of Definition~\ref{Def_IF_dist_within_CF} hold for $r = 2^{-i-1}$.
Set
\[ E^i := E^{i, i+1} \cup E^{i+1, i+2} \cup \ldots \]
Then
\[ |E^i| \leq 4^{-i}, \qquad E^1 \supset E^{2} \supset \ldots. \]
So $E^\infty := \bigcap_{i=1}^\infty E^i$ is a set of measure zero.
For any $t \in I \setminus E^\infty$, the probability measures $(\varphi^i_t)_* \mu^i_t \in \mathcal{P} (Z_t)$ are defined for large $i$ and form a Cauchy sequence in $(\PP(Z_t), d_{W_1}^{Z_t})$ (see Remark~\ref{Rmk_F_dist_GW1_dist}).
So they converge to a probability measure $\mu^\infty_t \in \mathcal{P} (Z_t)$ and
\[ d_{W_1}^{Z_t} ( (\varphi^i_t)_* \mu^i_t, \mu^\infty_t ) \leq 2^{-i}. \]
Let $X^\infty_t := \supp \mu^\infty_t$ and $d^\infty_t := d^{Z}_t |_{X^\infty_t}$.
Then $(X^\infty_t, d^\infty_t, \mu^\infty_t)$ is a complete, separable metric measure space of full support.

Let us now analyze the conjugate heat kernels $(\nu^i_{x;s})$.

\begin{Claim} \label{Cl_convergence_nu}
For every $s, t \in I \setminus E^\infty$, $s\leq t$ and $x^\infty \in X^\infty_t$ and every sequence $x^i \in \XX^i_t$ with $\varphi^i_t (x^i) \to x^\infty$ we have
 \[ (\varphi^i_t)_* \nu^i_{x^i;s} \xrightarrow{\quad W_1 \quad} \nu^\infty_{x^\infty;s}, \]
for some probability measure $\nu^\infty_{x^\infty;s} \in \mathcal{P} ( Z_s )$ with $\supp \nu^\infty_{x^\infty;s} \subset X^\infty_s$.
Moreover, the limit does not depend on the choice of the sequence $x^i$.
\end{Claim}

\begin{proof}
Consider a sequence $x^i \in \XX^i_t$ with $\varphi^i_t (x^i) \to x^\infty$ and let $r > 0$.
Since for large $i$
\[ \mu^i_t (B(x^i, r)) = ((\varphi^i_t)_* \mu^i_t )(B^{Z_t} (\varphi^i_t (x^i), r) )
\geq ((\varphi^i_t)_* \mu^i_t )(B^{Z_t} (x^\infty, r/2) ), \]
we have $\liminf_{i \to \infty} \mu^i_t (B(x^i, r)) > 0$.
So the claim, except for the statement concerning the support of $\nu^\infty_{x^\infty;s}$, follows from Lemma~\ref{Lem_XX12_close_HK_close}.

If the statement concerning the support were false, then there would be a ball $B(y^\infty, 3r) \cap  \supp \mu^\infty_s = \emptyset$ with $y^\infty \in \supp \nu^\infty_{x^\infty;s}$.
Choose a sequence $y^i \in \XX^i_s$ with $\varphi^i_s (y^i) \to y^\infty$.
Consider the function $f := ( r - d^{Z}_s (y^\infty, \cdot))_+: Z_s \to \IR$.
We obtain that 
\[ \liminf_{i \to \infty} \int_{\XX^i_s}  (f \circ \varphi^i_s) d\nu^i_{x^i;s} >2 \alpha> 0, \qquad 
\lim_{i \to \infty} \int_{\XX^i_s}  (f \circ \varphi^i_s) d\mu^i_s = 0. \] 
By Proposition~\ref{Prop_compare_CHF}\ref{Prop_compare_CHF_c}, since $f_i := f \circ \varphi^i_s$ is 1-Lipschitz, so is $z \mapsto \int_{\XX^i_s}  f_i \, d\nu^i_{z;s}$.
Therefore, for large $i$
\[ \int_{\XX^i_s}  f_i \, d\mu^i_s 
= \int_{\XX^i_t} \int_{\XX^i_s}  f_i \, d\nu^i_{y;s} d\mu^i_t (y)
\geq \int_{B(x^i, \alpha)} \int_{\XX^i_s}  f_i \, d\nu^i_{y;s} d\mu^i_t (y) 
\geq \alpha \mu^i_t (B(x^i, \alpha)), \]
which contradicts the fact that $\varphi^i_t(x^i) \to x^\infty \in \supp \mu^\infty_t$.
This shows that $\nu^\infty_{x^\infty; s} \in \mathcal{P}(X^\infty_s)$.
\end{proof}

Applying Claim~\ref{Cl_convergence_nu} for any $s,t \in I \setminus E^\infty$, $s \leq t$ and $x^\infty \in X^\infty_t$ produces a families of probability measures $\nu^\infty_{x^\infty;s} \in \mathcal{P} (X^\infty_s)$, which we will fix henceforth.
Set $I^{\prime, \infty} := I^{\prime\prime, \infty} :=  I \setminus E^\infty$.
By abuse of notation, we will denote the tuple
\[ \Big( \XX^\infty := \bigsqcup_{t \in I \setminus E^\infty } X^\infty_t,  \tf^\infty, (d^\infty_t)_{t \in I \setminus E^\infty}, (\nu^\infty_{x;s})_{x \in X^\infty_t, s,t \in I \setminus E^\infty, s \leq t} \Big), \]
where $\tf^\infty : \XX^\infty \to \IR$ is the natural map, by $\XX^\infty$.
Let moreover $(\varphi^\infty_t : \XX^\infty_t = X^\infty_t \to Z_t )_{t \in I \setminus E^\infty }$ be the family of inclusion maps.

\begin{Claim}
$(\XX^\infty, (\mu^\infty_t)_{t \in I \setminus E^\infty})$ is a metric flow pair and $\CF'$ in (\ref{eq_CFp_def_lemma}) is a correspondence.
\end{Claim}

\begin{proof}
Properties~\ref{Def_metric_flow_1}--\ref{Def_metric_flow_5} of Definition~\ref{Def_metric_flow} hold trivially.

To see Property~\ref{Def_metric_flow_6} let $s, t \in I \setminus E^\infty$, $s < t$, $T > 0$ and consider a $T^{-1}$-Lipschitz function $f : X_s^\infty \to \IR$.
Then the function $\widehat{f} : Z_s \to \IR$, defined by
\[ \widehat{f} (x) := \inf_{z \in X^\infty_s} \big(  f(z) + T^{-1} d(x,z) \big) \]
is also $T^{-1}$-Lipschitz.
It follows that the functions 
\[ h^i : \XX^i_t \to \IR, \qquad x \mapsto \Phi^{-1} \bigg( \int_{\XX^i_s} (\Phi \circ \widehat{f} \circ \varphi^i_s) d\nu_{x;s}^i \bigg) \]
are $(t-s+T)^{-1}$-Lipschitz.
By Claim~\ref{Cl_convergence_nu} for any $x^i \in \XX^i_t$ with $\varphi^i_t (x^i) \to x^\infty \in X^\infty_t$ we have $h^i(x^i) \to \Phi^{-1} ( \int_{X^\infty_s} \Phi (f) d\nu_{x^\infty;s})$.
This shows that $X^\infty_t \to \IR$, $x \mapsto \Phi^{-1} ( \int_{X^\infty_s} \Phi (f) d\nu^\infty_{x^\infty;s})$ is $(t-s+T)^{-1}$-Lipschitz, and therefore Property~\ref{Def_metric_flow_6} if $T > 0$.
By Lemma~\ref{Lem_met_flow_T_positive} this implies the case $T=0$.

For Property~\ref{Def_metric_flow_7} fix $t_1, t_2, t_3 \in I \setminus E^\infty$, $t_1 < t_2 < t_3$, $x^\infty \in X^\infty_{t_3}$.
It suffices to show that for every bounded Lipschitz function $f : X^\infty_{t_3} \to \IR$
\[ \int_{X^\infty_{t_1}} f \, d\nu^\infty_{x^\infty;t_1} = \int_{X^\infty_{t_2}} \int_{X^\infty_{t_1}} f \, d\nu^\infty_{y; t_1} d\nu^\infty_{x^\infty; t_2}(y). \]
As explained in the last paragraph, we may extend $f$ to a bounded Lipschitz function $\widehat{f} : Z_{t_1} \to \IR$.
Fix a sequence $x^i \in \XX^i_{t_3}$ such that $\varphi^i_{t_3} (x^i) \to x^\infty$.
Then by Claim~\ref{Cl_convergence_nu}
\[ \int_{\XX^i_{t_1}} \widehat{f} \circ \varphi^i_{t_1} \, d\nu^i_{x^i; t_1} \longrightarrow \int_{X^\infty_{t_1}} f \, d\nu^\infty_{x^\infty; t_1}. \]
Similarly, the functions
\[ h^i : \XX^i_{t_2} \to \IR, \qquad y \mapsto \int_{\XX^i_{t_1}} \widehat{f} \circ \varphi^i_{t_1} d\nu^i_{y; t_1} \]
are uniformly Lipschitz and for any sequence $y^i \in \XX^i_{t_2}$ with $\varphi^i_{t_2} (y^i) \to y^\infty \in X^\infty_{t_2}$ we have $h^i (y^i) \to \int_{X^\infty_{t_1}} f \, d\nu^\infty_{y^\infty; t_1}$.
It follows, again using Claim~\ref{Cl_convergence_nu}, that
\[   \int_{X^\infty_{t_1}} f \, d\nu^\infty_{x^\infty; t_1} \longleftarrow \int_{\XX^i_{t_1}} \widehat{f} \circ \varphi^i_{t_1}\, d\nu^i_{x^i;t_1} 
= \int_{\XX^i_{t_2}} h^i \, d\nu^i_{x^i; t_2} \longrightarrow  \int_{X^\infty_{t_2}} \int_{X^\infty_{t_1}} f \, d\nu^\infty_{y; t_1} d\nu^\infty_{x^\infty; t_2}(y). \]
The proof that $(\mu^\infty_t)_{t \in I \setminus E^\infty}$ is a conjugate heat flow on $\XX^\infty$ is almost the same.
\end{proof}

It remains to show (\ref{eq_dFCFp_conv_assertion}), i.e. that we have convergence within $\CF'$.
For this purpose, choose for any $i \geq 1$ and $t \in I \setminus E^i$ a coupling $q^{i, \infty}_t$ between $\mu^i_t, \mu^\infty_t$ with the property that
\[ \int_{\XX^i_t \times \XX^\infty_t} d^Z_t ( \varphi^i_t (x^i), \varphi^\infty_t (x^\infty) ) dq^{i, \infty}_t (x^i, x^\infty) \leq 2^{-i+1}. \]
For any $i \leq j$ and $t \in I \setminus E^i$, use Lemma~\ref{Lem_gluing} to find a probability measure $q^{i,j, \infty}$ on $\XX^i_t \times \XX^j_t  \times \XX^\infty_t$ whose projection onto the first and last two factors equals $q^{i,\infty}_t, q^{j, \infty}_t$, respectively.
Its projection $\td{q}^{i,j}$ onto the first two factors is a coupling between $\mu^i_t, \mu^j_t$ and as in the proof of Proposition~\ref{Prop_triangle_ineq_CC} we find
\begin{align*}
\int_{\XX^i_t \times \XX^j_t} & d^Z_t ( \varphi^i_t (x^i), \varphi^j_t (x^j) ) d\td q^{i, j}_t (x^i, x^j)
= \int_{\XX^i_t \times \XX^j_t \times \XX^\infty_t} d^Z_t ( \varphi^i_t (x^i), \varphi^j_t (x^j) ) d q^{i, j, \infty}_t (x^i, x^j, x^\infty) \\
&\leq \int_{\XX^i_t \times \XX^j_t \times \XX^\infty_t} \big( d^Z_t ( \varphi^i_t (x^i), \varphi^\infty_t (x^\infty) ) + d^Z_t ( \varphi^\infty_t (x^\infty), \varphi^j_t (x^j) ) \big) d q^{i, j, \infty}_t (x^i, x^j, x^\infty) \\
&\leq \int_{\XX^i_t \times \XX^\infty_t} d^Z_t ( \varphi^i_t (x^i), \varphi^\infty_t (x^\infty) ) d q^{i,  \infty}_t (x^i,  x^\infty)
+ \int_{\XX^j_t \times \XX^\infty_t} d^Z_t ( \varphi^\infty_t (x^\infty), \varphi^j_t (x^j) )  d q^{ j, \infty}_t (x^j, x^\infty) \\
& \leq 2^{-i+1} + 2^{-j+1} \leq 2^{-i+2}. 
\end{align*}
So by Lemma~\ref{Lem_F_dist_td_mu}
\[ \int_{\XX^i_t \times \XX^j_t} d_{W_1}^{Z_s} ( (\varphi^i_s)_* \nu^i_{x^i; s}, (\varphi^j_s)_* \nu^j_{x^j; s} ) d\td q^{i, j}_t (x^i, x^j) \leq  2^{-i+2} + 2 d^{\, \CF, J}_{\IF} \big( (\XX^i, (\mu^i_t)_{t \in I^{\prime,i}}), (\XX^j, (\mu^j_t)_{t \in I^{\prime,j}}) \big)  \leq 2^{-i + 3}. \]
It follows that for any measurable subset $Y \subset \XX^\infty_t$
\begin{align}
 \int_{\XX^i_t \times Y} &  d_{W_1}^{Z_s} ( (\varphi^i_s)_* \nu^i_{x^i; s}, (\varphi^\infty_s)_* \nu^\infty_{x^\infty; s} ) d q^{i, \infty}_t (x^i, x^\infty) \notag \\
&= \int_{\XX^i_t \times \XX^j_t \times Y}  d_{W_1}^{Z_s} ( (\varphi^i_s)_* \nu^i_{x^i; s}, (\varphi^\infty_s)_* \nu^\infty_{x^\infty; s} ) d q^{i, j, \infty}_t (x^i, x^j, x^\infty) \notag \\
&\leq \int_{\XX^i_t \times \XX^j_t \times Y} \big(    d_{W_1}^{Z_s} ( (\varphi^i_s)_* \nu^i_{x^i; s}, (\varphi^j_s)_* \nu^j_{x^j; s} ) 
+    d_{W_1}^{Z_s} ( (\varphi^j_s)_* \nu^j_{x^j; s}, (\varphi^\infty_s)_* \nu^\infty_{x^\infty; s} )  \big) d q^{i, j, \infty}_t (x^i, x^j, x^\infty) \notag \\
&\leq 2^{-i + 3} + \int_{\XX^j_t \times Y}   d_{W_1}^{Z_s} ( (\varphi^j_s)_* \nu^j_{x^j; s}, (\varphi^\infty_s)_* \nu^\infty_{x^\infty; s} )  dq^{j, \infty}_t (x^j, x^\infty). \label{eq_mindAtdqijinfty}
\end{align}

Next, suppose that $Y$ is compact and fix some $\eps > 0$.
We claim that for large $j$ we have for any $x^j \in \XX^j_t$, $x^\infty \in Y$
\begin{equation} \label{eq_dW1Zs_vs_dZt}
 d_{W_1}^{Z_s} ( (\varphi^j_s)_* \nu^j_{x^j; s}, (\varphi^\infty_s)_* \nu^\infty_{x^\infty; s} )
\leq d^Z_t (\varphi^j_t (x^j), \varphi^\infty_t(x^\infty)) + \eps. 
\end{equation}
Suppose not, so, after passing to a subsequence, we can find sequences of points $x^j \in \XX^j_t$, $x^{\infty,j} \in Y$ that violate (\ref{eq_dW1Zs_vs_dZt}).
After passing to another subsequence, we may assume that $x^{\infty, j} \to x^{\prime, \infty} \in Y$.
Fix a sequence $x^{\prime, j} \in \XX^j_t$ with $\varphi^j_t (x^{\prime, j} ) \to \varphi^\infty_t (x^{\prime, \infty})$.
Then
\begin{align*}
 d_{W_1}^{Z_s}  ( (\varphi^j_s)_* \nu^j_{x^j; s}, (\varphi^\infty_s)_* \nu^\infty_{x^{\infty,j}; s} )& \leq d_{W_1}^{Z_s} ( (\varphi^j_s)_* \nu^j_{x^j; s},  (\varphi^j_s)_* \nu^j_{x^{\prime, j}; s} )
+ d_{W_1}^{Z_s} ( (\varphi^j_s)_* \nu^j_{x^{\prime,j}; s}, (\varphi^\infty_s)_* \nu^\infty_{x^{\prime, \infty}; s} ) \\
&\qquad + d_{W_1}^{Z_s} ( (\varphi^\infty_s)_* \nu^\infty_{x^{\prime, \infty}; s}, (\varphi^\infty_s)_* \nu^\infty_{x^{\infty,j}; s} ) \\
&\leq d^j_t ( x^j, x^{\prime,j}) + d_{W_1}^{Z_s} ( (\varphi^j_s)_* \nu^j_{x^{\prime,j}; s}, (\varphi^\infty_s)_* \nu^\infty_{x^{\prime, \infty}; s} )
+  d^\infty_t (x^{\prime, \infty}, x^{\infty,j}) \\
&\leq  d^Z_t (\varphi^j_t (x^j), \varphi^\infty_t(x^{\infty,j})) + d^Z_t (\varphi^\infty_t(x^{\prime,\infty}),\varphi^j_t (x^{\prime, j})) \\
&\qquad + d_{W_1}^{Z_s} ( (\varphi^j_s)_* \nu^j_{x^{\prime,j}; s}, (\varphi^\infty_s)_* \nu^\infty_{x^{\prime,\infty}; s} )
+ 2 d^\infty_t (x^{\prime,\infty}, x^{\infty,j}).
\end{align*}
By Claim~\ref{Cl_convergence_nu} the last three terms converge to $0$ as $j \to \infty$, which yields the desired contradiction.

Combining (\ref{eq_mindAtdqijinfty}), (\ref{eq_dW1Zs_vs_dZt}) implies that for large $j$
\begin{multline*}
  \int_{\XX^i_t \times Y}   d_{W_1}^{Z_s} ( (\varphi^i_s)_* \nu^i_{x^i; s}, (\varphi^\infty_s)_* \nu^\infty_{x^\infty; s} ) d q^{i, \infty}_t (x^i, x^\infty) \\
\leq 2^{-i+3} + \int_{\XX^j_t \times Y}   d^Z_t (\varphi^j_t (x^j), \varphi^\infty_t(x^\infty))  dq^{j, \infty}_t (x^j, x^\infty) + \eps. 
\end{multline*}
So letting $j \to \infty$ and then $\eps \to 0$ and finally $Y \to \XX^\infty_t$ (see Lemma~\ref{Lem_basic_measure}) yields
\[  \int_{\XX^i_t \times \XX^\infty_t}   d_{W_1}^{Z_s} ( (\varphi^i_s)_* \nu^i_{x^i; s}, (\varphi^\infty_s)_* \nu^\infty_{x^\infty; s} ) d q^{i, \infty}_t (x^i, x^\infty)  \leq 2^{-i+3}. \]
This concludes the proof of (\ref{eq_dFCFp_conv_assertion}).
\end{proof}
\bigskip

\begin{proof}[Proof of Theorem~\ref{Thm_F_complete}.]
Consider a sequence of metric flow pairs $(\XX^i, (\mu^i_t)_{t \in I^{\prime,i}})$ representing elements in $\IF^J_I$ that form a Cauchy sequence in $(\mathbb{F}^J_I , d^J_\mathbb{F})$.
After passing to a subsequence, we may assume that
\[ d^J_{\IF} \big( (\XX^i, (\mu^i_t)_{t \in I^{\prime,i}}), (\XX^{i+1}, (\mu^{i+1}_t)_{t \in I^{\prime,i+1}}) \big) < 2^{-i}. \]
So we can find correspondences $\CF^{i, i+1}$ between $\XX^i, \XX^{i+1}$ over $I$ that are fully defined over $J$ such that
\[ d^{\CF^{i, i+1}, J}_{\IF} \big( (\XX^i, (\mu^i_t)_{t \in I^{\prime,i}}), (\XX^{i+1}, (\mu^{i+1}_t)_{t \in I^{\prime,i+1}}) \big) < 2^{-i}. \]
By an iterative application of Lemma~\ref{Lem_combining_correspondences}, we can construct sequences of correspondences $\CC^{1 \ldots k}$ between $\XX^1, \ldots, \XX^k$ such that for any $1 \leq i < k$ 
\[ d^{\CF^{1 \ldots k}, J}_{\IF} \big( (\XX^i, (\mu^i_t)_{t \in I^{\prime,i}}), (\XX^{i+1}, (\mu^{i+1}_t)_{t \in I^{\prime,i+1}}) \big) =
 d^{\CF^{i, i+1}, J}_{\IF} \big( (\XX^i, (\mu^i_t)_{t \in I^{\prime,i}}), (\XX^{i+1}, (\mu^{i+1}_t)_{t \in I^{\prime,i+1}}) \big) . \]
 Using a direct limit construction on the sequence of metric spaces $(Z^{1\ldots k }_t, d^{Z^{1\ldots k }}_t)$, we find a correspondence $\CF$ between $\XX^1, \XX^2, \ldots$ such that for any $i \in \IN$
 \[ d^{\CF, J}_{\IF} \big( (\XX^i, (\mu^i_t)_{t \in I^{\prime,i}}), (\XX^{i+1}, (\mu^{i+1}_t)_{t \in I^{\prime,i+1}}) \big) =
 d^{\CF^{i, i+1}, J}_{\IF} \big( (\XX^i, (\mu^i_t)_{t \in I^{\prime,i}}), (\XX^{i+1}, (\mu^{i+1}_t)_{t \in I^{\prime,i+1}}) \big) < 2^{-i} . \]
After passing to a their completions, we may assume that the metric spaces $(Z, d^Z_t)$ of the correspondence $\CF$ are complete.
By Lemma~\ref{Lem_F_completemess_within_CF} there is a metric flow pair $(\XX^\infty, (\mu^\infty_t)_{t \in I^{\prime,\infty}})$ over $I$ that is fully defined over $J$ such that for an enlargement $\CF'$ of $\CF$
\[ d^{ J}_{\IF} \big( (\XX^i, (\mu^i_t)_{t \in I^{\prime,i}}), (\XX^\infty, (\mu^\infty_t)_{t \in I^{\prime,\infty}}) \big) \leq  d^{\, \CF', J}_{\IF} \big( (\XX^i, (\mu^i_t)_{t \in I^{\prime,i}}), (\XX^\infty, (\mu^\infty_t)_{t \in I^{\prime,\infty}}) \big) \to 0. \]
This finishes the proof.
\end{proof}
\bigskip

\section{Convergence within a correspondence} \label{Sec_conv_within}
In this section we will study the convergence behavior of metric flow pairs in more detail.
To do this, we will embed an $\IF$-convergent sequence $(\XX^i, (\mu^i_t)_{t \in I^{\prime,i}})$ and its limit $(\XX^\infty, (\mu^\infty_t)_{t \in I^{\prime,\infty}})$ into a common correspondence.
This will allow us to define the notion of \emph{convergence on compact time-intervals} and to relate objects and geometric properties of the sequence with its limit.
For example, we will define what it means that a sequence of conjugate heat flows or points on $\XX^i$ converges to a conjugate heat flow or point on $\XX^\infty$.

One of the main goals of this section will be to show a ``change-of-basepoint'' theorem, which allows us to replace the conjugate heat flows $(\mu^i_t)_{t \in I^{\prime,i}}$ by another convergent sequence of conjugate heat flows and maintain convergence of the metric flow pairs.
This will allow us to show that tangent flows of the limit $\XX^\infty$ occur as $\IF$-limits of parabolic rescalings of the certain metric flow pairs involving the metric flows $\XX^i$.

A number of results presented in this section will be required in Section~\ref{sec_reg_pts} and in subsequent work, but are not required for the compactness theory, as presented in Section~\ref{sec_compact_subsets_IF}.
The reader may decide to skim or skip this section upon first reading.

\subsection{Convergence of metric flow pairs within a correspondence}
In this subsection we define what it means that a sequence of metric flow pairs $\IF$-converges \emph{within} a correspondence.

Let $(\XX^i, (\mu^i_t)_{t \in I^{\prime, i}})$, $i \in \IN \cup \{ \infty \}$, be metric flow pairs over intervals $I^i \subset \IR$.
Suppose that 
\begin{equation} \label{eq_CF_over_N_infty}
  \CF := \big(  (Z_t, d^Z_t)_{t \in I''},(\varphi^i_t)_{t \in I^{\prime\prime, i}, i \in \IN \cup \{ \infty \}} \big),  
\end{equation}
is a correspondence between the metric flows $\XX^i$, $i \in \IN \cup \{ \infty \}$, over some subset $I'' \subset \IR$.
Let $J \subset \IR$ be another subset.

\begin{Definition}[Convergence of metric flow pairs within correspondence] \label{Def_IF_conv_within_CF}
Suppose that $I^{\infty} \setminus I''$ has measure zero, that the metric flow pairs $(\XX^i, (\mu^i_t)_{t \in I^{\prime, i}})$ for large $i$ and the correspondence $\CF$ restricted to some index set of the form $\{ i \geq \underline{i} \}$ are fully defined over $J$ and that
\begin{equation} \label{eq_Def_dIF_CF_to_0}
 d_{\IF}^{\,\CF, J} \big( (\XX^i, (\mu^i_t)_{t \in I^{\prime,i}}), (\XX^\infty, (\mu^\infty_t)_{t \in I^{\prime,\infty}}) \big) \to 0.
\end{equation}
Then we say that the metric flow pairs $(\XX^i, (\mu^i_t)_{t \in I^{\prime, i}})$ {\bf $\IF$-converge to $(\XX^\infty, (\mu^\infty_t)_{t \in I^{\prime, \infty}})$ within $\CF$ and that the convergence is uniform over $J$.}
We write\footnote{We may sometimes omit $\CF$, $J$ above the arrow if there is no chance of confusion.}
\begin{equation} \label{eq_IF_convergence_within_arrow}
 (\XX^i, (\mu^i_t)_{t \in I^{\prime,i}}) \xrightarrow[i \to \infty]{\quad \IF, \CF, J \quad}  (\XX^\infty, (\mu^\infty_t)_{t \in I^{\prime,\infty}}) . 
\end{equation}
If $J = I^\infty$, then we say that the $\IF$-convergence is {\bf uniform}.
If (\ref{eq_Def_dIF_CF_to_0}) holds after replacing $J$ with $J \cup \{ t \}$ for any or some $t \in I^\infty$, then we say that the $\IF$-convergence is {\bf time-wise or time-wise at time $t$.}

Next, suppose that for any compact subinterval $I_0 \subset I^\infty$ we have
\begin{equation} \label{eq_Def_dIF_to_0_finite_ti}
 (\XX^i_{I_0}, (\mu^i_t)_{t \in I^{\prime,i} \cap I_0}) \xrightarrow[i \to \infty]{\quad \IF, \CF|_{I'' \cap I_0}, J \cap I_0 \quad}  (\XX^\infty_{I_0}, (\mu^\infty_t)_{t \in I^{\prime,\infty} \cap I_0}) .  
\end{equation}
Then we say that the metric flow pairs $(\XX^i, (\mu^i_t)_{t \in I^{\prime, i}})$ {\bf $\IF$-converge to $(\XX^\infty, (\mu^\infty_t)_{t \in I^{\prime, \infty}})$ within $\CF$ on compact time-intervals and that the convergence is uniform over $J$ on compact time-intervals.}
If $J = I^\infty$, then we say that the $\IF$-convergence is {\bf uniform on compact time-intervals}.
Similarly as before, if (\ref{eq_Def_dIF_to_0_finite_ti}) holds on compact time-intervals after replacing $J$ with $J \cup \{ t \}$ for any/some $t \in I^\infty$, then we say that the $\IF$-convergence is {\bf time-wise (at time $t$).}
\end{Definition}

\begin{Remark} \label{Rmk_Ipp_extend}
We may always extend $\CF$ such that $I'' = I^\infty$ and $I^{\prime\prime, i} = I^{\prime,i}$ for all $i \in \IN \cup \{ \infty \}$; for example choose $(Z_t, d^Z_t)$ to be the wedge sum of all non-empty $(\XX^i_t, d^i_t)$ and let $\varphi^i_t$ be the natural embeddings.
This does not change the convergence behavior in (\ref{eq_IF_convergence_within_arrow}), so it can be done to simplify the setting.
\end{Remark}

The following lemma shows that in the setting of Definition~\ref{Def_IF_conv_within_CF} we can always pass to a subsequence such that we have time-wise convergence for almost every time.

\begin{Lemma} \label{Lem_pass_to_timewise}
Suppose that (\ref{eq_IF_convergence_within_arrow}) holds (on compact time-intervals).
Then, after passing to a subsequence, there is a subset $E_\infty \subset I^\infty$ such that the convergence (\ref{eq_IF_convergence_within_arrow}) is time-wise at any $t \in I^\infty \setminus E_\infty$ (while converging on compact time-intervals everywhere else).
Moreover, there is a decreasing sequence of subsets $E_1 \supset E_2 \supset \ldots$, $E_i \subset I^\infty$, with $\bigcap_{i=1}^\infty E_i =0$ and $J \subset I^\infty \setminus E_i$ such that for any $j$ the convergence (\ref{eq_IF_convergence_within_arrow}) is even uniform on $I^\infty \setminus E_j$ (on compact time-intervals).
\end{Lemma}

\begin{proof}
We will consider the case in which (\ref{eq_IF_convergence_within_arrow}) holds on $I^\infty$.
The corresponding statement involving convergence on compact time-intervals follows similarly.
By Defintion~\ref{Def_IF_dist_within_CF} there are measurable subsets $E_i \subset I^\infty$ such that $J \subset I'' \setminus E_i \subset I^{\prime\prime, i} \cap I^{\prime\prime, \infty}$ and
\[ |E_i| \to 0, \qquad d_{\IF}^{\,\CF, I'' \setminus E_i} \big( (\XX^i, (\mu^i_t)_{t \in I^{\prime,i}}), (\XX^\infty, (\mu^\infty_t)_{t \in I^{\prime,\infty}}) \big) \to 0.  \]
After passing to a subsequence, we may assume that $|E_i| \leq 2^{-i}$ and after replacing $E_i$ with $E_i \cup E_{i+1} \cup \ldots$, we may assume that $E_1 \supset E_2 \supset \ldots$ and $|E_i| \leq 2^{-i+1}$.
Moreover, after replacing $E_i$ with $E_i \cup (I^\infty \setminus I'')$, we may assume that $I^\infty \setminus I'' \subset E_i$.
Set $E_\infty := \bigcap_{i=1}^\infty E_i$.
Then for any $t \in I^\infty \setminus E_\infty$ we have $J \cup \{ t \} \subset  I'' \setminus E_i$ for large $i$.
\end{proof}

The next lemma shows that the limit of an $\IF$-convergent sequence of metric flow pairs is unique if the convergence only holds on compact time-intervals.

\begin{Lemma}
Suppose that (\ref{eq_IF_convergence_within_arrow}) holds on compact time-intervals for two limiting metric flow pairs $ (\XX^\infty, (\mu^\infty_t)_{t \in I^{\prime,\infty}})$, $ (\XX^{*, \infty}, (\mu^{*, \infty}_t)_{t \in I^{\prime,*,\infty}})$ over the same interval $I^\infty$.
Then there is an almost everywhere isometry between them.
\end{Lemma}

\begin{proof}
By Theorem~\ref{Thm_IF_metric_space} for any compact subinterval $I_0 \subset I^\infty$ there is a set $E_{I_0} \subset I_0$ of measure zero such that $I_0 \setminus E_{I_0} \subset I^{\prime, \infty} \cap I^{ \prime, *,\infty}$ and an almost everywhere isometry 
\[ \phi_{I_0} :  (\XX^\infty_{I_0 \setminus E_{I_0}}, (\mu^\infty_{t})_{t \in I_0 \setminus E_{I_0}}) \to (\XX^{*, \infty}_{I_0 \setminus E_{I_0}}, (\mu^{*, \infty}_t)_{t \in I_0 \setminus E_{I_0}}). \]
Consider an increasing sequence $I_1 \subset I_2 \subset \ldots \subset I^\infty$ of compact subintervals with $\bigcup_{k=1}^\infty I_k = I^\infty$ and let $E_k \subset I_k$, $\phi_k$ be the corresponding sets of measure zero and almost everywhere isometries.
Let $E := \bigcup_{k=1}^\infty E_k$ and for any $t \in I_k \setminus E$ let $q_{k,t} := (\id_{\XX^\infty_t}, \phi_{k,t})_* \mu^\infty_t$ be the coupling between $\mu^\infty_t, \mu^{*, \infty}_t$.
Now the proof of Theorem~\ref{Thm_IF_metric_space} carries over to our setting.
\end{proof}
\bigskip

\subsection{\texorpdfstring{$\IF$-Convergence implies $\IF$-convergence within a correspondence}{F-Convergence implies F-convergence within a correspondence}}
The following theorem states that given an $\IF$- convergent sequence of metric flow pairs, we can construct correspondence within which this sequence of metric flow pairs converges.

\begin{Theorem} \label{Thm_conv_to_conv_within}
Let $(\XX^i, (\mu^i_t)_{t \in I^{\prime, i}})$, $i \in \IN \cup \{ \infty \}$, be metric flow pairs over an interval $I \subset \IR$ that are fully defined over some $J \subset I$.
Suppose that
\begin{equation} \label{eq_condition_F_to_0}
 d_{\IF}^{J} \big( (\XX^i, (\mu^i_t)_{t \in I^{\prime,i}}), (\XX^\infty, (\mu^\infty_t)_{t \in I^{\prime,\infty}}) \big) \to 0. 
\end{equation}
Then there is a correspondence $\CF$ between the metric flows $\XX^i$, $i \in \IN \cup \{ \infty \}$,  over $I$ that is fully defined over $J$ such that
\begin{equation} \label{eq_F_conv_in_CF_conclusion}
 (\XX^i, (\mu^i_t)_{t \in I^{\prime,i}}) \xrightarrow[i \to \infty]{\quad \IF, \CF, J \quad}  (\XX^\infty, (\mu^\infty_t)_{t \in I^{\prime,\infty}}) . 
\end{equation}
Moreover, given an increasing sequence of subsets $J_1 \subset J_2 \subset \ldots \subset I$ with the property that (\ref{eq_condition_F_to_0}) continues to hold after replacing $J$ with $J_k$ for any $k \geq 1$, we can choose $\CF$ such that the $\IF$-convergence in (\ref{eq_F_conv_in_CF_conclusion}) is uniform over any $J_k$.
\end{Theorem}

We also have the corresponding statement for convergence on compact time-intervals.

\begin{Theorem} \label{Thm_conv_to_conv_within_cpt_ti}
Let $(\XX^i, (\mu^i_t)_{t \in I^{\prime, i}})$, $i \in \IN \cup \{ \infty \}$, be metric flow pairs over intervals $I^i \subset \IR$ and $J \subset I^\infty$ such that $(\XX^\infty, (\mu^\infty_t)_{t \in I^{\prime, \infty}})$ is fully defined over $J$.
Suppose that for any compact subinterval $I_0 \subset I^\infty$ and for large $i$ we have $I_0 \subset I^i$ and the metric flow pairs $(\XX^i, (\mu^i_t)_{t \in I^{\prime, i}})$ are fully defined over $J \cap I_0$ and their restrictions to $I_0$ satisfy
\[ d_{\IF}^{J \cap I_0} \big( (\XX^i |_{I_0}, (\mu^i_t)_{t \in I^{\prime,i} \cap I_0}), (\XX^\infty |_{I_0}, (\mu^\infty_t)_{t \in I^{\prime,\infty} \cap I_0}) \big) \to 0. \] 
Then there is a correspondence $\CF$ between the metric flows $\XX^i$, $i \in \IN \cup \{ \infty \}$, over $I^\infty$ such that
\begin{equation} \label{eq_F_conv_conclusion_cpt_interv}
 (\XX^i, (\mu^i_t)_{t \in I^{\prime,i}}) \xrightarrow[i \to \infty]{\quad \IF, \CF, J \quad}  (\XX^\infty, (\mu^\infty_t)_{t \in I^{\prime,\infty}}) . 
\end{equation}
on compact time-intervals.
Moreover, given an increasing sequence of subsets $J_1 \subset J_2 \subset \ldots \subset I^\infty$ with the property that the assumption continues to hold after replacing $J$ with $J_k$ for any $k \geq 1$, then we can choose $\CF$ such that the $\IF$-convergence in (\ref{eq_F_conv_in_CF_conclusion}) is uniform on compact time-intervals over any $J_k$.
\end{Theorem}

\begin{proof}[Proof of Theorem~\ref{Thm_conv_to_conv_within}]
We will prove the last statement, for the first statement set $J_k := J$.
Due to (\ref{eq_condition_F_to_0}) we can find $k_1 \leq k_2 \leq \ldots$ with $k_i \to \infty$ and correspondences $\CF^{i\infty}$ between $\XX^i, \XX^\infty$ over $I$ that are fully defined over $J_{k_i}$ for large $i$ such that
\[  d_{\IF}^{\CF^{i\infty}, J_{k_i}} \big( (\XX^i, (\mu^i_t)_{t \in I^{\prime,i}}), (\XX^\infty, (\mu^\infty_t)_{t \in I^{\prime,\infty}}) \big) \to 0.  \]
By an iterative application of Lemma~\ref{Lem_combining_correspondences}, we can construct sequences of correspondences $\CC^{1 \ldots m \infty}$ between $\XX^1, \ldots, \XX^m, \XX^\infty$ such that for large $i$ and $m \geq i$
\[ d_{\IF}^{\CF^{1 \ldots m \infty}, J_{k_i}} \big( (\XX^i, (\mu^i_t)_{t \in I^{\prime,i}}), (\XX^\infty, (\mu^\infty_t)_{t \in I^{\prime,\infty}}) \big) = d_{\IF}^{\CF^{i\infty}, J_{k_i}} \big( (\XX^i, (\mu^i_t)_{t \in I^{\prime,i}}), (\XX^\infty, (\mu^\infty_t)_{t \in I^{\prime,\infty}}) \big). \]
Using a direct limit construction on the sequence of metric spaces $(Z^{1\ldots m \infty}_t, d^{Z^{1\ldots m \infty}}_t)$, we find a correspondence $\CF$ between $\XX^1, \ldots, \XX^\infty$ such that for large $i \in \IN$
\[ d_{\IF}^{\CF, J_{k_i}} \big( (\XX^i, (\mu^i_t)_{t \in I^{\prime,i}}), (\XX^\infty, (\mu^\infty_t)_{t \in I^{\prime,\infty}}) \big) = d_{\IF}^{\CF^{i\infty}, J_{k_i}} \big( (\XX^i, (\mu^i_t)_{t \in I^{\prime,i}}), (\XX^\infty, (\mu^\infty_t)_{t \in I^{\prime,\infty}}) \big). \]
By Remark~\ref{Rmk_Ipp_extend}, we can extend $\CF$ to a correspondence over $I$.
\end{proof}

\begin{proof}[Proof of Theorem~\ref{Thm_conv_to_conv_within_cpt_ti}.]
We may choose an increasing sequence of (possibly empty) subintervals $I_{0}^i \subset I^\infty$ and $k_1 \leq k_2 \leq \ldots$ with $k_i \to \infty$ such that $\bigcup_{i=1}^\infty I_{0}^i = I^\infty$ and such that for large $i$ we have $I_{0}^i \subset I^i$, the metric flow pairs $(\XX^i, (\mu^i_t)_{t \in I^{\prime, i}})$ are fully defined over $J_{k_i} \cap I_{0}^i$ and their restrictions to $I_{0}^i$ satisfy
\[ d_{\IF}^{J_{k_i} \cap I_{0}^i} \big( (\XX^i_{I_{0}^i}, (\mu^i_t)_{t \in I^{\prime,i} \cap I_{0}^i}), (\XX^\infty_{I_{0}^i}, (\mu^\infty_t)_{t \in I^{\prime,\infty} \cap I_{0}^i}) \big) \to 0. \]  
We can now carry out the same construction as in the previous proof.
\end{proof}
\bigskip

\subsection{Convergence of conjugate heat flows within a correspondence}

Next, we define convergence of conjugate heat flows within a correspondence.
In the following, let $\XX^i$ be metric flows over subsets $I^{\prime, i} \subset \IR$, $i \in \IN \cup \{ \infty \}$, and consider a correspondence $\CF$ as in (\ref{eq_CF_over_N_infty}) between $\XX^i$ over $I''$ and a, possibly empty, subset $J \subset \IR$.
Let $(\mu^i_t)_{t \in I_*^i}$, $i \in \IN \cup \{ \infty \}$, be conjugate heat flows on $\XX^i$, where $I_*^i = I^{\prime, i} \cap (-\infty, T_i)$ or $I^{\prime, i} \cap (-\infty, T_i]$ for some $T_i \in (-\infty, \infty]$.

\begin{Definition} \label{Def_CHF_convergence_within_CF}
We say that the conjugate heat flows $(\mu^i_t)_{t \in I_*^i}$ {\bf converge to $(\mu^\infty_t)_{t \in I_*^\infty}$ within $\CF$ and that the convergence is uniform over $J$} and we write
\begin{equation} \label{eq_IF_CHF_convergence_within_arrow}
  (\mu^i_t)_{t \in I_*^i} \xrightarrow[i \to \infty]{\quad  \CF, J \quad}   (\mu^\infty_t)_{t \in I^\infty_*}  
\end{equation}
if $J \subset I_*^i$ for large $i \leq \infty$ and there are measurable subsets $E_i \subset I''$, $i \in \IN$, such that:
\begin{enumerate}
\item $J \cap I^\infty_* \subset (I^i_*  \cap I'') \setminus E_i = (I^\infty_*  \cap I'') \setminus E_i \subset I^{\prime\prime,i} \cap I^{\prime\prime,\infty}$ for large $i$.
\item $|E_i| \to 0$.
\item $\sup_{t \in (I^\infty_*  \cap I'') \setminus E_i} d_{W_1}^{Z_t} ( (\varphi^i_t)_* \mu^i_t, (\varphi^\infty_t)_* \mu^\infty_t)  \to 0$.
\end{enumerate}
We say that (\ref{eq_IF_CHF_convergence_within_arrow}) {\bf holds on compact time-intervals and is uniform over $J$} if for any compact subinterval $I_0 \subset I^{\infty}_*$ we have (\ref{eq_IF_CHF_convergence_within_arrow}) after replacing $\CF, J$ with $\CF |_{I'' \cap I_0}, J \cap I_0$.
We say that the convergence  (\ref{eq_IF_CHF_convergence_within_arrow}) is {\bf time-wise at time $t \in I''$} if (\ref{eq_IF_CHF_convergence_within_arrow}) holds after replacing $J$ with $J \cup \{ t \}$, that is if
\[ (\varphi^i_t)_* \mu^i_t \xrightarrow[i \to \infty]{\quad W_1 \quad} (\varphi^\infty_t)_* \mu^\infty_t. \]
\end{Definition}

\begin{Remark}
Note that (\ref{eq_IF_CHF_convergence_within_arrow}) implies $T_i \to T_\infty$.
\end{Remark}

The following lemma is a direct consequence of Definition~\ref{Def_IF_dist_within_CF}, see also Remark~\ref{Rmk_F_dist_GW1_dist}.

\begin{Lemma} \label{Lem_met_flow_conv_mu_conv}
Let $(\XX^i, (\mu^i_t)_{t \in I^{\prime, i}})$, $i \in \IN \cup \{ \infty \}$, be metric flow pairs over intervals $I^i \subset \IR$ and consider a correspondence between the metric flows $\XX^i$, $i \in \IN \cup \{ \infty \}$.
If for some $J \subset \IR$ the following convergence holds (on compact time-intervals)
\[ (\XX^i, (\mu^i_t)_{t \in I^{\prime,i}}) \xrightarrow[i \to \infty]{\quad \IF, \CF, J \quad}  (\XX^\infty, (\mu^\infty_t)_{t \in I^{\prime,\infty}}), \]
then the following convergence holds (on compact time-intervals)
\[ (\mu^i_t)_{t \in I^{\prime,i}} \xrightarrow[i \to \infty]{\quad  \CF, J \quad}   (\mu^\infty_t)_{t \in I^{\prime,\infty}} . \]
\end{Lemma}

\subsection{Convergence of points and probability measures within a correspondence}
Next, characterize convergence of points and probability measures within a correspondence.
We will see that there are two different approaches: We may simply characterize this convergence as convergence within the metric spaces $(Z_t, d^Z_t)$ of the correspondence.
Alternatively, we may equate convergence of points or probability measures with convergence of the corresponding conjugate heat kernels or conjugate heat flows.
The latter approach, while weaker and less intuitive, will be more useful in the sequel, as it does not require the entire sequence to live in time-slices corresponding to a fixed time.

Let again $\XX^i$ be metric flows over subsets $I^{\prime, i} \subset \IR$, $i \in \IN \cup \{ \infty \}$, and consider a correspondence $\CF$ as in (\ref{eq_CF_over_N_infty}) between $\XX^i$ over $I''$ and a, possibly empty, subset $J \subset \IR$.

We first define the more useful convergence notion:

\begin{Definition} \label{Def_convergence_PP_measure}
Let $T_i \in I^{\prime, i}$ and consider a sequence of probability measures $\mu^i \in \mathcal{P}(\XX_{T_i}^i)$, $i \in \IN \cup \{ \infty \}$.
We say that $\mu^i$ {\bf converge to $\mu^\infty$ within $\CF$ (and uniform over $J$)}, and write
\begin{equation} \label{eq_mui_conv_CF}
 \mu^i \xrightarrow[i \to \infty]{\quad  \CF, J \quad} \mu^\infty,
\end{equation}
if $T_i \to T_\infty$ and if for the conjugate heat flows $(\td\mu^i_t)_{t \in I^{\prime, i} \cap (-\infty, T_i]}$, $i \in \IN \cup \{ \infty \}$, with initial condition $\td\mu^i_{T_i} = \mu^i$ we have the following convergence on compact time-intervals
\[ (\td\mu^i_t)_{t \in I^{\prime, i} \cap (-\infty, T_i)}  \xrightarrow[i \to \infty]{\quad  \CF, J \quad}   (\td\mu^\infty_t)_{t \in I^{\prime, \infty} \cap (-\infty, T_\infty)} . \]
For any sequence of points $x^i \in \XX^i_{T_i}$, $i \in \IN \cup \{ \infty \}$, we say that $x^i$ {\bf converge to $x^\infty$ within $\CF$ (and uniform over $J$)} and write
\begin{equation} \label{eq_xi_conv_CF}
x^i \xrightarrow[i \to \infty]{\quad  \CF, J \quad} x^\infty
\end{equation}
 if
$\delta_{x^i} \xrightarrow[i \to \infty]{\quad  \CF, J \quad} \delta_{x^\infty}$.
This is equivalent to $T_i \to T_\infty$ and the following convergence on compact time-intervals
\[ (\nu^i_{x^i;t})_{t \in I^{\prime, i} \cap (-\infty, T_i)}  \xrightarrow[i \to \infty]{\quad  \CF, J \quad}   (\nu^\infty_{x^\infty; t})_{t \in I^{\prime, \infty} \cap (-\infty, T_\infty)} . \]
\end{Definition}

\begin{Remark}
In general, the limits in (\ref{eq_mui_conv_CF}), (\ref{eq_xi_conv_CF}) may not be unique.
This is the case if the conjugate heat flows or the conjugate heat kernels of the limiting probability measure or points agree at all times except for the final time; see also the example discussed in Remark~\ref{Rmk_reverse_strict_conv} below.
Moreover, if $T_\infty = \inf I^{\prime,i}$, then (\ref{eq_mui_conv_CF}), (\ref{eq_xi_conv_CF}) are vacuous.
\end{Remark}

Next, we define the more restrictive notion:

\begin{Definition} \label{Def_convergence_CHF_within_CF_strict}
Fix some $T \in I''$ and consider a sequence of probability measures $\mu^i \in \PP (\XX^i_T)$, $i \in \IN \cup \{ \infty \}$.
We say that $\mu^i$ {\bf strictly converge to $\mu^\infty$ within $\CF$} if
\[ (\varphi^i_T)_* \mu^i \xrightarrow[i \to \infty]{\quad W_1 \quad} (\varphi^\infty_T)_* \mu^\infty. \]
For any sequence of points $x^i \in\XX^i_T$, $i \in \IN \cup \{ \infty \}$, we say that $x^i$ {\bf strictly converge to $x^\infty$ within $\CF$} if $\delta_{x^i}$ strictly converge to $\delta_{x^\infty}$ within $\CF$ or, equivalently, if
\[ \varphi^i_T(x^i) \xrightarrow[i \to \infty]{\quad  \quad} \varphi^\infty_T(x^\infty). \]
\end{Definition}

We emphasize that Definitions~\ref{Def_convergence_PP_measure}, \ref{Def_convergence_CHF_within_CF_strict} describe two different notions of characterizing the convergence of measures or points within a correspondence.
The notion strict convergence is usually stronger, as we will soon see, but it only works if all points $x^i$ live in the same time-slices $\XX^i_T$ for some uniform time $T$.
In addition, if the metric flows $\XX^i$ belong to metric flow pairs that $\IF$-converges within $\CF$ then strict convergence only useful if this $\IF$-convergence is time-wise at time $t$.

The next theorem states that strict convergence (in the sense of Definition~\ref{Def_convergence_CHF_within_CF_strict}) implies convergence (in the sense of Definition~\ref{Def_convergence_PP_measure}) if the metric flows $\XX^i$ belong to a convergent sequence of metric flow pairs.

\begin{Theorem} \label{Thm_tdmu_convergence}
Let $(\XX^i, (\mu^i_t)_{t \in I^{\prime, i}})$, $i \in \IN \cup \{ \infty \}$, be metric flow pairs over intervals $I^i \subset \IR$ and consider a correspondence $\CF$  between the metric flows $\XX^i$, $i \in \IN \cup \{ \infty \}$.
Let $J \subset \IR$ be some subset and $T \in J$.
Suppose that we have the following convergence (possibly on compact time-intervals)
\[ (\XX^i, (\mu^i_t)_{t \in I^{\prime,i}}) \xrightarrow[i \to \infty]{\quad \IF, \CF, J \quad}  (\XX^\infty, (\mu^\infty_t)_{t \in I^{\prime,\infty}}), \]
\begin{enumerate}[label=(\alph*)]
\item \label{Thm_tdmu_convergence_a} Consider probability measures $\td\mu^i \in \PP(\XX^i_T)$, $i \in \IN \cup \{ \infty \}$, such that $\td\mu^i$ strictly converges to $\td\mu^\infty$ and assume that $\td\mu^\infty \in \PP^1 (\XX^\infty_T)$  (i.e. its $W_1$-Wasserstein distance to point masses is finite).
Then (possibly on compact time-intervals)
\begin{equation*}  \td\mu^i \xrightarrow[i \to \infty]{\quad  \CF, J \quad} \td\mu^\infty.
\end{equation*}
Moreover, we have the following stronger result: If we consider the conjugate heat flows $(\td\mu^i_t)_{t \in I^{\prime, i}_T }$ on $\XX^i$ with $I^{\prime, i}_T := I^{\prime, i} \cap (-\infty, T]$ and initial condition $\td\mu^i_{T} = \td\mu^i$, $i \in \IN \cup \{ \infty \}$, then
\[ (\td\mu^i_t)_{t \in I^{\prime, i}_T}  \xrightarrow[i \to \infty]{\quad  \CF, J \quad}   (\td\mu^\infty_t)_{t \in I^{\prime, \infty}_T} . \]
\item \label{Thm_tdmu_convergence_b} Consider points $x^i \in \XX^i_T$, $i \in \IN \cup \{ \infty \}$, such that $x^i$ strictly converge to $x^\infty$.
Then
\[ x^i \xrightarrow[i \to \infty]{\quad  \CF, J \quad} x^\infty \]
and moreover for $I^{\prime, i}_T := I^{\prime, i} \cap (-\infty, T]$
\[ (\nu^i_{x^i;t})_{t \in I^{\prime, i}_T}  \xrightarrow[i \to \infty]{\quad  \CF, J \quad}   (\nu^\infty_{x^\infty;t})_{t \in I^{\prime, \infty}_T} . \]
\end{enumerate}
\end{Theorem}

\begin{Remark} \label{Rmk_reverse_strict_conv}
The reverse direction is in general false.
To see this, consider a metric flow $\XX$ as in Example~\ref{Ex_not_Hausdorff}, which is defined over $(-\infty, 0]$ and has the property that $\# \XX_0 > 1$ and $\# \XX_t = 1$ for all $t < 0$.
This flow is not past continuous.
Let $\XX^i := \XX$ for $i \in \IN \cup \{ \infty \}$ and let $\CF$ be the trivial correspondence between the flows $\XX^i$.
Obviously, the metric flows belong to metric flow pairs that converge within $\CF$.
However, for any sequence $x^i \in \XX^i_0$, $i \in \IN \cup \{ \infty \}$, we have $x^i \xrightarrow[i \to \infty]{\quad \CF, (-\infty, 0) \quad} x^\infty$, but we only have strict convergence if $x^i = x^\infty$ for large $i$.
This also shows that limits in the sense of Definition~\ref{Def_convergence_PP_measure} may not be unique.
\end{Remark}

The next theorem shows that convergence of conjugate heat flow (in the sense of Definition~\ref{Def_CHF_convergence_within_CF}) implies strict convergence (in the sense of Definition~\ref{Def_convergence_CHF_within_CF_strict}) at almost every time.
Moreover, if the metric flows $\XX^i$ belong to a convergent sequence of metric flow pairs, then strict convergence holds at every time at which we have time-wise convergence of the metric flow pairs, except possibly at the final time of the limiting conjugate heat flow.

\begin{Theorem} \label{Thm_tdmu_convergence_reverse}
Let $\XX^i$ be metric flows over subsets $I^{\prime, i} \subset \IR$, $i \in \IN \cup \{ \infty \}$, and consider a correspondence $\CF$ between $\XX^i$.
Let $(\td\mu^i_t)_{t \in I_*^i}$, $i \in \IN \cup \{ \infty \}$, be conjugate heat flows on $\XX^i$, where $I_*^i = I^{\prime, i} \cap (-\infty, T_i)$ or $I^{\prime, i} \cap (-\infty, T_i]$ for some $T_i \in (-\infty, \infty]$ and suppose that we have the following convergence (possibly on compact time-intervals):
\begin{equation} \label{eq_tdmu_conv_no_J}
 (\td\mu^i_t)_{t \in I_*^i} \xrightarrow[i \to \infty]{\quad  \CF \quad}   (\td\mu^\infty_t)_{t \in I^\infty_*}.  
\end{equation}
Then the following is true:
\begin{enumerate}[label=(\alph*)]
\item \label{Thm_tdmu_convergence_reverse_a} After passing to a subsequence, (\ref{eq_tdmu_conv_no_J}) is time-wise at almost every time in $I_*^\infty$.
This implies that we have strict convergence of $\td\mu^i_t$ to $\td\mu^\infty_t$ for almost every $t \in I_*^\infty$.
\item \label{Thm_tdmu_convergence_reverse_b} Suppose that there are conjugate heat flows $(\mu^i_t)_{t \in I^{\prime, i}}$ such that $(\XX^i, (\mu^i_t)_{t \in I^{\prime, i}})$ are metric flow pairs, $i \in \IN \cup \{ \infty \}$, and such that we have for some $J \subset \IR$ (possibly on compact time-intervals)
\begin{equation} \label{eq_met_flow_conv_no_J}
 (\XX^i, (\mu^i_t)_{t \in I^{\prime,i}}) \xrightarrow[i \to \infty]{\quad \IF, \CF,  J \quad}  (\XX^\infty, (\mu^\infty_t)_{t \in I^{\prime,\infty}}).
\end{equation}
Then we have weak convergence $(\varphi^i_t)_* \td\mu^i_t \to (\varphi^i_t)_* \td\mu^\infty_t$ for all $t \in I_*^\infty \setminus \{ T_\infty \}$ at which (\ref{eq_met_flow_conv_no_J}) is time-wise.

If in addition $\td\mu^\infty_t \in \PP^1 (\XX^\infty_t)$ for all $t \in I^{\prime, \infty}$, then (\ref{eq_tdmu_conv_no_J}) is uniform over $J \cap (-\infty,T']$ (possibly on compact time-intervals) for any $T' < T_\infty$ and time-wise at every time at which (\ref{eq_met_flow_conv_no_J}) is time-wise, except for possibly at time $T_\infty$.
This is equivalent to strict convergence at these times.
\end{enumerate}
\end{Theorem}

The remainder of this subsection is occupied with the proofs of Theorems~\ref{Thm_tdmu_convergence}, \ref{Thm_tdmu_convergence_reverse}.
Lemma~\ref{Lem_mu_T_conv_mf_conv} below will also be used in the proof of Theorem~\ref{Thm_change_basepoints} in Subsection~\ref{subsec_change_basepoints}.

The following lemma shows that if two metric flow pairs are close within a correspondence, then the same is true after replacing the conjugate heat flows with two other conjugate heat flows whose initial conditions are close within the same correspondence.
Note that the following bound only depends on the closeness of $\td\mu^1_T, \td\mu^2_T$ at time $T$ and that the lemma implies a closeness bound of the flows $(\td\mu^1_t)_{t \in I^{\prime, 1}}$, $(\td\mu^2_t)_{t \in I^{\prime, 2}}$ at earlier times.

\begin{Lemma} \label{eq_F_dist_A_factor}
Let $(\XX^i, (\mu^i_t)_{t \in I^{\prime, i}})$, $i =1,2$, be two metric flow pairs and consider a correspondence $\CF = (  (Z_t, d^Z_t)_{t \in I''},(\varphi^i_t)_{t \in I^{\prime\prime, i}, i =1,2} )$ between $\XX^1, \XX^2$ over $I''$ that is fully defined at some time $T$.
Consider conjugate heat flows $(\td\mu^i_t)_{t \in I^{\prime, i}_T}$, where $I^{\prime, i}_T := I^{\prime,i} \cap (-\infty,T]$, and assume that $\td\mu^1_T \leq A \mu^1_T$ for some $A < \infty$.
Then for $I''_T := I'' \cap (-\infty,T]$ and $J_T := J \cap (-\infty,T]$ we have
\begin{multline*}
 d_{\IF}^{\,\CF|_{I''_T}, J_T} \big( (\XX^1_{I^{\prime, 1}_T}, (\td\mu^1_t)_{t \in I^{\prime, 1}_T}), (\XX^2_{I^{\prime, 2}_T}, (\td\mu^2_t)_{t \in I^{\prime, 2}_T}) \big) \\
 \leq 4 A d_{\IF}^{\,\CF, J} \big( (\XX^1, (\mu^1_t)_{t \in I^{\prime,1}}), (\XX^2, (\mu^2_t)_{t \in I^{\prime,2}}) \big) + d^{Z_T}_{W_1} ( (\varphi^1_T)_* \td\mu^1_T, (\varphi^2_T)_* \td\mu^2_T). 
\end{multline*}
\end{Lemma}

\begin{proof}
After performing a time-shift, we may assume that $T = 0$ and after restricting the flows $\XX^i$ to $I^{\prime, i}_T$ and the correspondence $\CF$ to $I''_T$, we may assume that $I^i_T = I^{\prime, i}$ and $I''_T = I''$.
Let 
\[ r > d_{\IF}^{\,\CF, J} \big( (\XX^1, (\mu^1_t)_{t \in I^{\prime,1}}), (\XX^2, (\mu^2_t)_{t \in I^{\prime,2}}) \big), \qquad \td{r} > d^{Z_0}_{W_1} ( (\varphi^1_0)_* \td\mu^1_0, (\varphi^2_0)_* \td\mu^2_0) \]
and choose $E \subset I''$, $(q_t)_{t \in I'' \setminus E}$ with $J \subset I'' \setminus E \subset I^{\prime\prime,1} \cap I^{\prime\prime,2}$ such that Properties~\ref{Def_IF_dist_within_CF_1}, \ref{Def_IF_dist_within_CF_3} of Definition~\ref{Def_IF_dist_within_CF} hold for $r$.
Choose a coupling $\td q_0$ between $\td\mu^1_0, \td\mu^2_0$ such that
\begin{equation} \label{eq_d_dtdq}
 \int_{\XX^1_0 \times \XX^2_0} d^Z_0 (\varphi^1_0(x^1), \varphi^2_0(x^2)) \, d\td q_0 < \td{r}. 
\end{equation}

By Proposition~\ref{Prop_properties_CHK}\ref{Prop_properties_CHK_bb}  we have $\td\mu^1_t \leq A \mu^1_t$  for all $t \in I^{\prime,1}$.
Therefore, we can find measurable functions $(h_t : \XX^1_t \to \IR)_{t \in I^{\prime,1}}$ such that
\[ d\td\mu_t^1 = h_t \, d\mu^1_t , \qquad 0 \leq h_t \leq A .\]
For any $t \in I'' \setminus E$ define $q'_t$ by $dq'_t (x^1, x^2) := h_t (x^1) dq_t (x^1,x^2)$ and let $\mu^{\prime,2}_t$ be the marginal of $q'_t$ onto the second factor.
Then $q'_t$ is a coupling between $\td\mu^1_t, \mu^{\prime,2}_t$  and we have for all $s, t \in I'' \setminus E$, $s \leq t$
\begin{align}
  \int_{\XX^1_t \times \XX^2_t} d^{Z}_t (\varphi^1_t (x^1), \varphi^2_t (x^1)) d q'_t (x^1, x^2) &\leq A r , \label{eq_d_dqp} \\
  \int_{\XX^1_t \times \XX^2_t} d^{Z_s}_{W_1} ( (\varphi^1_s)_* \nu^1_{x^1;s}, (\varphi^2_s)_* \nu^2_{x^2;s}) d q'_t (x^1, x^2) &\leq A r. \label{eq_dnu_dqp}
\end{align}
Applying Lemma~\ref{Lem_X1X2_correspondence_mu112} to $t = 0$, $q'_0, \td q_0$ and using (\ref{eq_d_dtdq}), (\ref{eq_d_dqp}), (\ref{eq_dnu_dqp}) implies that for any $s \in I'' \setminus E$
\[  \int_{\XX^1_0 \times \XX^2_0} d^{Z_s}_{W_1} ( (\varphi^1_s)_* \nu^1_{x^1;s}, (\varphi^2_s)_* \nu^2_{x^2;s}) d \td q_0 (x^1, x^2) < 2 A r+ \td r. \]

\begin{Claim}
For any $t \in I'' \setminus E$
\begin{equation} \label{eq_dW1_tdmu1tdmu2}
  d^{Z_t}_{W_1} ( (\varphi^1_t)_* \td\mu^{1}_t, (\varphi^2_t)_* \td\mu^{2}_t  )< 2 A r+ \td r. 
\end{equation}
and there is a family of couplings $(\td q_t)_{t \in I'' \setminus E}$ between $\td\mu^1_t, \td\mu^2_t$ such that
\begin{equation} \label{eq_d_dtdq_t}
  \int_{\XX^1_t \times \XX^2_t} d^{Z}_{t} (\varphi^1_t( x^1), \varphi^2_t(x^2)) d \td q_t (x^1, x^2) \leq 2 A r+ \td r. 
\end{equation}
\end{Claim}

\begin{proof}
We will verify (\ref{eq_dW1_tdmu1tdmu2}) using Proposition~\ref{Prop_Kant_Rub}.
So let $f : Z_t \to \IR$ be a bounded $1$-Lipschitz function.
Then
\begin{align*}
 \int_{Z_t}  &f \, d ((\varphi^1_t)_* \td\mu^{1}_t - (\varphi^2_t)_* \td\mu^{2}_t) 
= \int_{\XX^1_t} f( \varphi^1_t (x^1))  \, d\td\mu^{1}_t (x^1) - \int_{\XX^2_t} f( \varphi^2_t (x^2))  \, d\td\mu^2_t (x^2) \\
&= \int_{\XX^1_0} \int_{\XX^1_t} f( \varphi^1_t (x^1))  \, d\nu^1_{y^1;t} (x^1) d\td\mu^{1}_0 (y^1) - \int_{\XX^2_0} \int_{\XX^2_t} f( \varphi^2_t (x^2))  \, d\nu^2_{y^2; t} (x^2) d\td\mu^2_0 (y^2) \\
&=  \int_{\XX^1_0 \times \XX^2_0} \int_{\XX^1_t} \int_{\XX^2_t} \big(  f( \varphi^1_t (x^1)) - f( \varphi^2_t (x^2))  \big) d\nu^2_{y^2; t} (x^2)  d\nu^1_{y^1;t} (x^1)   d \td q_0 (y^1, y^2) \\
&\leq \int_{\XX^1_0 \times \XX^2_0} d^{Z_t}_{W_1} ( (\varphi^1_t)_* \nu^1_{x^1;t}, (\varphi^2_t)_* \nu^2_{x^2;t}) d \td q_0 (x^1, x^2) <  2 A r+ \td r - \delta
\end{align*}
for some small $\delta > 0$ that is independent of $f$.
Lastly, note that $\td q_0$ has already been chosen and satisfies (\ref{eq_d_dtdq_t}) due to (\ref{eq_d_dtdq}).
\end{proof}

Applying Lemma~\ref{Lem_X1X2_correspondence_mu112} to $q'_t, \td q_t$ and using (\ref{eq_d_dtdq_t}), (\ref{eq_d_dqp}), (\ref{eq_dnu_dqp}) implies that for any $s,t \in I'' \setminus E$, $s \leq t$
\[  \int_{\XX^1_t \times \XX^2_t} d^{Z_s}_{W_1} ( (\varphi^1_s)_* \nu^1_{x^1;s}, (\varphi^2_s)_* \nu^2_{x^2;s}) d \td q_t (x^1, x^2) \leq 4 A r + \td r, \]
which proves the lemma after letting $r, \td r$ converge to their respective lower bounds.
\end{proof}

Theorem~\ref{Thm_tdmu_convergence} will be a consequence of the following lemma:

\begin{Lemma} \label{Lem_mu_T_conv_mf_conv}
Suppose we are in the setting of Theorem~\ref{Thm_tdmu_convergence}\ref{Thm_tdmu_convergence_a} and that $\CF$ is of the form (\ref{eq_CF_over_N_infty}).
Then 
\[ (\XX^i_{I^{\prime,i}_T}, (\td\mu^i_t)_{t \in I^{\prime,i}_T}) \xrightarrow[i \to \infty]{\quad \IF, \CF|_{I'' \cap (-\infty,T]}, J  \cap (-\infty, T] \quad}  (\XX^\infty_{I^{\prime, \infty}_T}, (\td\mu^\infty_t)_{t \in I^{\prime, \infty}_T}). \]
\end{Lemma}

\begin{proof}
After performing a time-shift, we may assume that $T = 0$ and after restricting the flows $\XX^i$ to $I^{\prime, i}_T$, the correspondence $\CF$ to $I'' \cap (-\infty,T]$ and replacing $J$ with $J \cap (-\infty,T]$, we may assume that $I^{\prime, i}_T = I^{\prime, i}$ and $J \subset I'' \subset (-\infty,0]$.
We need to show that
\[ d_{\IF}^{\,\CF, J} \big( (\XX^\infty, (\td\mu^\infty_t)_{t \in I^{\prime,\infty}}), (\XX^i, (\td\mu^i_t)_{t \in I^{\prime, i}}) \big) \to 0. \]
Fix some $\eps > 0$.

\begin{Claim}
There is a probability measure $\td\mu^{\infty, \eps}_0 \in \PP (\XX^\infty_0)$ and a number $A_\eps < \infty$, which may depend on $\td\mu_0^\infty$, such that
\[ \td\mu^{\infty, \eps}_0 \leq A_\eps \mu^\infty_0, \qquad
d_{W_1}^{\XX^\infty_0} ( \td\mu_0^\infty, \td\mu^{\infty, \eps}_0 ) \leq 2\eps. \]
\end{Claim}

\begin{proof}
By Lemma~\ref{Lem_W_p_separable}, we can find a probability measure $\td\mu^{\prime,\infty}_0 \in \PP (\XX^\infty_0)$ whose support is finite and contained in $\supp \mu_0^\infty$ such that
\[ d_{W_1}^{\XX^\infty_0} ( \td\mu_0^\infty, \td\mu^{\prime,\infty}_0 ) \leq \eps. \]
Choose $0 < r < \eps$ such that the $r$-balls around every point in $\supp \td\mu^{\prime,\infty}$ are pairwise disjoint.
For any $x \in \supp \td\mu^{\prime,\infty}_0 \subset \supp \mu_0^\infty$ set
\[ a_x := \mu_0^\infty ( B(x, r)) > 0, \qquad b_x := \td\mu^{\prime,\infty} (\{ x \}) \]
and let
\[ \td\mu^{\infty, \eps}_0 := \sum_{x \in \supp \td\mu^{\prime,\infty}_0} \frac{b_x}{a_x} \mu_0^\infty |_{B(x,r)}. \]
Then
\[ d_{W_1}^{\XX^\infty_0} ( \td\mu_0^\infty, \td\mu^{\infty, \eps}_0 ) \leq
d_{W_1}^{\XX^\infty_0} ( \td\mu_0^\infty, \td\mu^{\prime,\infty}_0 ) +
d_{W_1}^{\XX^\infty_0} ( \td\mu_0^{\prime,\infty}, \td\mu^{\infty, \eps}_0 ) \leq 2 \eps. \qedhere\]
\end{proof}
\medskip

Let $(\td\mu^{\infty, \eps}_t)_{t \in I^{\prime, \infty}}$ be the conjugate heat flow with initial condition $\td\mu^{\infty, \eps}_0$.
Using Proposition~\ref{Prop_compare_CHF}\ref{Prop_compare_CHF_b} and Lemma~\ref{Lem_F_dist_same_X} we find that
\begin{equation} \label{eq_dFC_tdmutdmueps}
 d_{\IF}^{\,\CF, J} \big( (\XX^\infty, (\td\mu^\infty_t)_{t \in I^{\prime,\infty}}), (\XX^\infty, (\td\mu^{\infty, \eps}_t)_{t \in I^{\prime,\infty}})\big) \leq 2\eps. 
\end{equation}
By Lemma~\ref{eq_F_dist_A_factor} we have
\begin{multline} \label{eq_tdmuepstdmui}
 \limsup_{i \to \infty} d_{\IF}^{\,\CF, J} \big( (\XX^\infty, (\td\mu^{\infty, \eps}_t)_{t \in I^{\prime,\infty}}), (\XX^i, (\td\mu^i_t)_{t \in I^{\prime, i}}) \big) \leq \limsup_{i \to \infty}  d^{Z_0}_{W_1} ( (\varphi^\infty_0)_* \td\mu^{\infty, \eps}_0, (\varphi^i_0)_* \td\mu^{i}_0) \\
\leq d^{\XX^\infty_0}_{W_1} ( \td\mu^{\infty, \eps}_0, \td\mu^\infty_0) 
+ \limsup_{i \to \infty}  d^{Z_0}_{W_1} ( (\varphi^\infty_0)_* \td\mu^{\infty}_0, (\varphi^i_0)_* \td\mu^{i}_0) 
\leq 2\eps. 
\end{multline}
Combining (\ref{eq_dFC_tdmutdmueps}), (\ref{eq_tdmuepstdmui}) shows that
\[ d_{\IF}^{\,\CF, J} \big( (\XX^\infty, (\td\mu^\infty_t)_{t \in I^{\prime,\infty}}), (\XX^i, (\td\mu^i_t)_{t \in I^{\prime, i}}) \big) \leq 4\eps \]
for large $i$, which finishes the proof.
\end{proof}
\bigskip

\begin{proof}[Proof of Theorem~\ref{Thm_tdmu_convergence}.]
Assertion~\ref{Thm_tdmu_convergence_a} is a direct consequence of Lemmas~\ref{Lem_mu_T_conv_mf_conv}, \ref{Lem_met_flow_conv_mu_conv}.
Assertion~\ref{Thm_tdmu_convergence_b} follows from Assertion~\ref{Thm_tdmu_convergence_a} by setting $\td\mu^i := \delta_{x^i}$.
\end{proof}
\bigskip

\begin{proof}[Proof of Theorem~\ref{Thm_tdmu_convergence_reverse}.]
Assertion~\ref{Thm_tdmu_convergence_reverse_a} follows from Definition~\ref{Def_CHF_convergence_within_CF}; see also the proof of Lemma~\ref{Lem_pass_to_timewise}.

To see Assertion~\ref{Thm_tdmu_convergence_reverse_b} under the assumption that $\mu^\infty_t \in \PP^1 (\XX^\infty_t)$ for all $t \in I^{\prime, \infty}$, it suffices to show uniform convergence over $J \cap (-\infty, T']$ for any $T' < T_\infty$.
By replacing $J$ with $J \cap (-\infty, T']$, we may assume that $\sup J < T_\infty$.
Suppose that the convergence (\ref{eq_tdmu_conv_no_J}) was not uniform over $J$.
Then we may pass to a subsequence such that uniform convergence over $J$ is violated for any further subsequence.
By Assertion~\ref{Thm_tdmu_convergence_reverse_a} and Lemma~\ref{Lem_pass_to_timewise}, we can pass to a subsequence and choose $T \in (\sup J , T_\infty]$ arbitrarily close to $T_\infty$ such that (\ref{eq_tdmu_conv_no_J}) and (\ref{eq_met_flow_conv_no_J}) are time-wise at $T$.
Now Theorem~\ref{Thm_tdmu_convergence}\ref{Thm_tdmu_convergence_a} produces the desired contradiction.

Lastly, we show the statement involving weak convergence in Assertion~\ref{Thm_tdmu_convergence_reverse_b}.
Fix some $t_0 \in I_*^\infty \setminus \{ T_\infty \}$ at which the convergence (\ref{eq_met_flow_conv_no_J}) is time-wise and suppose by contradiction that we don't have weak convergence $(\varphi^i_{t_0})_* \td\mu^i_{t_0} \to (\varphi^i_{t_0})_* \td\mu^\infty_{t_0}$.
Then we can pass to a subsequence and find a bounded, continuous function $f : Z_{t_0} \to [0,1]$ such that
\begin{equation} \label{eq_lim_not_equal}
\lim_{i \to \infty} \int_{\XX^i_{t_0}} f \circ \varphi^i_{t_0} \, d\td\mu^i_{t_0}
= \lim_{i \to \infty} \int_{Z_{t_0}} f \, d(\varphi^i_{t_0})_* \td\mu^i_{t_0} 
 \neq \int_{Z_{t_0}} f \, d(\varphi^\infty_{t_0})_* \td\mu^\infty_{t_0}  
= \int_{\XX^\infty_{t_0}} f \circ \varphi^\infty_{t_0} \, d\td\mu^\infty_{t_0}. 
\end{equation}
By Assertion~\ref{Thm_tdmu_convergence_reverse_a} and Lemma~\ref{Lem_pass_to_timewise}, we can pass to a further subsequence and find some $T \in (t_0 , T_\infty]$ such that (\ref{eq_tdmu_conv_no_J}) is time-wise at $T$ and (\ref{eq_met_flow_conv_no_J}) is uniform over $J \cup \{t_0, T \}$.
Let $\alpha , R > 0$ be small/large constants whose values we will determine later.
Choose a basepoint $p \in Z_{T}$  and continuous function $w : Z_{T} \to [0,1]$ with $w \equiv 1$ on $B(p,R)$ and $w \equiv 0$ outside of $B(p, 2R)$.
Set
\[ a_i := \int_{Z_{T}} w \, d(\varphi^i_{T})_* \td\mu^i_{T}  = \int_{\XX^i_{T}} w \circ \varphi^i_{T} \, d\td\mu^i_{T}. \]
Note that $\lim_{i \to \infty} a_i = a_\infty$ and that by choosing $R$ sufficiently large, we can achieve that $a_\infty > 1-\alpha$.
Let $(\td\mu^{\prime, i}_t)_{t \in I^{\prime,i}, t \leq T}$, $i \in \IN \cup \{ \infty \}$, be the conjugate heat flows with initial condition $d\td\mu^{\prime, i}_T = a_i^{-1} (w \circ \varphi^i_T ) \, d\td\mu^i_T$.
Then we have strict convergence $\td\mu^{\prime, i}_T \to \td\mu^{\prime, \infty}_T$ and $\td\mu^{\prime, \infty}_T \in \PP^1( \XX^\infty_T)$.
So by Theorem~\ref{Thm_tdmu_convergence}\ref{Thm_tdmu_convergence_a} we have strict convergence $\td\mu^{\prime, i}_{t_0} \to \td\mu^{\prime, \infty}_{t_0}$.
On the other hand, Proposition~\ref{Prop_properties_CHK}\ref{Prop_properties_CHK_c} implies that for large $i \leq \infty$
\[ (1- \alpha) \td\mu^{\prime, i}_{t_0} \leq \td\mu^{i}_{t_0}. \]
This implies that
\begin{multline*}
 (1-\alpha) \int_{\XX^i_{t_0}} f \circ \varphi^i_{t_0} \, d\td\mu^{\prime, i}_{t_0} 
\leq  \int_{\XX^i_{t_0}} f \circ \varphi^i_{t_0} \, d\td\mu^{ i}_{t_0} 
= 1 - \int_{\XX^i_{t_0}} (1 - (f \circ \varphi^i_{t_0})) \, d\td\mu^{ i}_{t_0} \\
\leq 1 - (1-\alpha) \int_{\XX^i_{t_0}} (1 - (f \circ \varphi^i_{t_0})) \, d\td\mu^{\prime, i}_{t_0} 
= \alpha +  (1-\alpha) \int_{\XX^i_{t_0}} f \circ \varphi^i_{t_0} \, d\td\mu^{\prime, i}_{t_0}.
\end{multline*}
Since both sides of this inequality converge and $\alpha$ can be chosen arbitrarily, we obtain a contradiction to (\ref{eq_lim_not_equal}).
\end{proof}
\bigskip

\subsection{Change of basepoint theorem} \label{subsec_change_basepoints}
The following theorem, which is the main result of this subsection,  shows that we can exchange the sequence of conjugate heat flows in any convergent sequence of metric flow pairs by any other convergent sequence of conjugate heat flows under a technical assumption.
The statement can be seen as a change of basepoint theorem.
It is analogous to the following statement: If we have pointed Gromov-Hausdorff convergence $(X_i, d_i, x_i) \to (X_\infty, d_\infty, x_\infty)$ and $\td x_i \to \td x_\infty$ for some points $\td x_i \in X_i$, then we also have pointed Gromov-Hausdorff convergence $(X_i, d_i, \td x_i) \to (X_\infty, d_\infty, \td x_\infty)$.

\begin{Theorem} \label{Thm_change_basepoints}
Let $(\XX^i, (\mu^i_t)_{t \in I^{\prime, i}})$, $i \in \IN \cup \{ \infty \}$, be metric flow pairs over intervals $I^i \subset \IR$ and consider a correspondence $\CF$ between the metric flows $\XX^i$, $i \in \IN \cup \{ \infty \}$ over $I''$.
Suppose that for some subset $J \subset \IR$ we have the following convergence (possibly on compact time-intervals)
\begin{equation} \label{eq_metflow_conv_change_basepts}
 (\XX^i, (\mu^i_t)_{t \in I^{\prime,i}}) \xrightarrow[i \to \infty]{\quad \IF, \CF, J \quad}  (\XX^\infty, (\mu^\infty_t)_{t \in I^{\prime,\infty}}), 
\end{equation}
Consider conjugate heat flows $(\td\mu^i_t)_{t \in I^i_*}$, $i \in \IN \cup \{ \infty \}$, where $I^i_* = I^{\prime, i} \cap (-\infty, T_i)$ or $I^{\prime, i} \cap (-\infty, T_i]$ for some $T_i \in I^i$, such that
\begin{equation} \label{eq_td_conv_change_basepts}
  (\td\mu^i_t)_{t \in I_*^i} \xrightarrow[i \to \infty]{\quad  \CF \quad}   (\td\mu^\infty_t)_{t \in I^\infty_*}.  
\end{equation}
Assume that one of the following is true:
\begin{enumerate}[label=(\roman*)]
\item \label{Thm_change_basepoints_i} $T_\infty = \sup I_*^\infty \in J$, the convergence (\ref{eq_td_conv_change_basepts}) is time-wise at time $T_\infty$ and $\td\mu^\infty_T \in \PP^1 (\XX^\infty_{T_\infty})$.
\item \label{Thm_change_basepoints_ii} $\sup (J \cap I_*^\infty) < \sup I_*^\infty = T_\infty$ and $\td\mu^\infty_{t} \in \PP^1 (\XX^\infty_{t})$ for $t \in I^\infty_*$ near $T_\infty$.
\end{enumerate}
Then we have the following convergence (possibly on compact time-intervals)
\begin{equation} \label{eq_td_met_flowpairs}
 (\XX^i_{I^i_*}, (\td\mu^i_t)_{t \in I^i_*}) \xrightarrow[i \to \infty]{\quad \IF, \CF|_{I^\infty_* \cap I''}, J \cap I_*^\infty \quad}  (\XX^\infty_{I^\infty_*}, (\td\mu^\infty_t)_{t \in I^\infty_*}). 
\end{equation}
Moreover, (\ref{eq_td_met_flowpairs}) is time-wise at any time $t \in I^\infty_*$, $t < \sup I^\infty_*$, at which the convergence (\ref{eq_metflow_conv_change_basepts}) is time-wise.
\end{Theorem}

\begin{proof}
It suffices to show (\ref{eq_td_met_flowpairs}); the statement about time-wise convergence follows after replacing $J$ with $J \cup \{ t \}$.
We will carry out the proof in the case in which (\ref{eq_metflow_conv_change_basepts}) holds on $I^\infty$; the case in which the convergence holds on compact time-intervals is similar.

If Condition \ref{Thm_change_basepoints_i} holds, then we have strict convergence of $\td\mu^i_{T_\infty}$ to $\td\mu^\infty_{T_\infty}$, so the theorem follows from Lemma~\ref{Lem_mu_T_conv_mf_conv}.

Suppose now that Condition \ref{Thm_change_basepoints_ii} holds.
Consider an arbitrary $\eps > 0$ and arbitrary subsequences of the given sequences.
It suffices to show that for infinitely many $i$ we have
\begin{equation} \label{eq_goal_dF_eps}
 d_{\IF}^{\,\CF |_{I^\infty_* \cap I''}, J \cap I^\infty_*} \big( (\XX^i_{I^i_*}, (\td\mu^i_t)_{t \in I^i_*}), (\XX^\infty_{I^\infty_*}, (\td\mu^\infty_t)_{t \in I^\infty_*}) \big) \leq \eps. 
\end{equation}

By Lemma~\ref{Lem_pass_to_timewise} and Theorem~\ref{Thm_tdmu_convergence_reverse} and after passing to a subsequence, we can find some $T \in ( T_\infty- \tfrac12 \eps^2, T_\infty)$ such that both convergences (\ref{eq_metflow_conv_change_basepts}), (\ref{eq_td_conv_change_basepts}) are time-wise at time $T$.
So by Lemma~\ref{Lem_mu_T_conv_mf_conv} we have for large $i$
\[ d_{\IF}^{\,\CF |_{I^\infty_* \cap I'' \cap (-\infty,T]}, J \cap I^\infty_*} \big( (\XX^i_{I^i_* \cap (-\infty, T]}, (\td\mu^i_t)_{t \in I^i_* \cap (-\infty, T]}), (\XX^\infty_{I^\infty_* \cap (-\infty, T]}, (\td\mu^\infty_t)_{t \in I^\infty_* \cap (-\infty, T]}) \big) \leq \tfrac12 \eps. \]
Since $|( I^\infty_* \cap I'') \setminus (-\infty,T]| \leq \frac12 \eps^2$, this implies (\ref{eq_goal_dF_eps}).
\end{proof}
\bigskip

\subsection{Representing points as limits of sequences}
The following theorem states that points in the limit of an $\IF$-convergent sequence of metric flow pairs can be represented as limits of points in the sequence.

\begin{Theorem} \label{Thm_pts_as_limits_within}
Let $(\XX^i, (\mu^i_t)_{t \in I^{\prime, i}})$, $i \in \IN \cup \{ \infty \}$, be metric flow pairs over intervals $I^i \subset \IR$ and consider a correspondence $\CF$ between the metric flows $\XX^i$, $i \in \IN \cup \{ \infty \}$.
Suppose that for some $J \subset \IR$ we have on compact time-intervals
\begin{equation} \label{eq_all_pts_limit_F_conv}
 (\XX^i, (\mu^i_t)_{t \in I^{\prime,i}}) \xrightarrow[i \to \infty]{\quad \IF, \CF, J \quad}  (\XX^\infty, (\mu^\infty_t)_{t \in I^{\prime,\infty}})
\end{equation}
and that all $\XX^i$, $i \in \IN \cup \{ \infty \}$, are $H$-concentrated for some uniform $H < \infty$.
Consider some point $x_\infty \in \XX^\infty_{t_\infty}$ with $t_\infty > \inf I^\infty$ and a sequence of times $t_i \in I^{\prime, i}$ with $t_i \to t_\infty$.
Then there are points $x_i \in \XX^i_{t_i}$ such that 
\[ x_i \xrightarrow[i \to \infty]{\quad  \CF, J \quad} x_\infty. \]
\end{Theorem}

Note that if $t_\infty \in I^{\prime, i}$ for all $i \in \IN$, then we can apply the theorem to the constant sequence $t_i = t_\infty$ and obtain a sequence of points $x_i \in \XX^i_{t_\infty}$ at the same time.
It is, however, not guaranteed that this sequence of points also converges \emph{strictly} within $\CF$.

\begin{proof}
Let $r, \eps > 0$ be two constants whose values we will determine later.
Consider an arbitrary subsequence of the given sequence.
It suffices to show the theorem after passing to a further subsequence.

By Theorem~\ref{Thm_tdmu_convergence_reverse} we may pass to a subsequence and assume that the convergence (\ref{eq_all_pts_limit_F_conv}) is time-wise at almost every time.
Choose an $H$-center $y_\infty \in \XX_{t'}^\infty$, $t' \in (t_\infty - \eps, t_\infty)$, of $x_\infty$ such that  (\ref{eq_all_pts_limit_F_conv}) is time-wise at time $t'$.
Then we can find points $y_i \in \XX^i_{t'}$ that strictly converge to $y_\infty$ within $\CF$.
It follows that
\begin{equation} \label{eq_liminfB2rBr}
 \liminf_{i \to \infty} \mu^i_{t'} (B(y_i, 2r)) \geq \mu^\infty_{t'} (B(y_\infty, 2r)). 
\end{equation}

For every $H$-center $y' \in \XX^\infty_{t'}$ of any point $x' \in B(x_\infty, r)$  we have
\[ d_{t'}^\infty (y', y_\infty) \leq d^{\XX^\infty_{t'}}_{W_1} (\delta_{y'}, \nu^\infty_{x';t'}) + d^{\XX^\infty_{t'}}_{W_1} (\nu^\infty_{x';t'}, \nu^\infty_{x_\infty;t'} ) + d^{\XX^\infty_{t'}}_{W_1} (\nu^\infty_{x_\infty;t'}, \delta_{y_\infty} ) \leq r + 2 \sqrt{H \eps}. \]
So, assuming $\eps \leq \ov\eps(r, H)$ is chosen small enough such that $r + 2 \sqrt{H \eps} + \sqrt{2H \eps} \leq 2r$, we obtain, using Lemma~\ref{Lem_nu_BA_bound},
\begin{equation} \label{eq_mutpB2rgeq12Br}
 \mu^\infty_{t'} (B(y_\infty, 2r))
\geq \int_{B(x_\infty, r)} \nu^\infty_{x';t'} ( B(y_\infty, 2r)) d\mu^\infty_{t_\infty} (x') \geq \tfrac12 \mu^\infty_{t_\infty} (B(x_\infty, r)). 
\end{equation}

By combining (\ref{eq_liminfB2rBr}), (\ref{eq_mutpB2rgeq12Br}) and using the reproduction formula, we find points $x_i \in \XX^i_{t_i}$ such that for large $i$
\[ \nu^i_{x_i; t'} ( B(y_i, 2r) ) \geq \tfrac14 \mu^\infty_{t_\infty} (B(x_\infty, r)). \]
Let $z_i \in \XX^i_{t'}$ be $H$-centers of $x_i$ for large $i$.
By Lemma~\ref{Lem_nu_BA_bound} we have for large $i$
\[ d^i_{t'} (z_i, y_i) \leq 4r + \bigg( \frac{H (t_i - t')}{ \tfrac14 \mu^\infty_{t_\infty} (B(x_\infty, r)) } \bigg)^{1/2} \leq 4r + \bigg( \frac{8H \eps}{  \mu^\infty_{t_\infty} (B(x_\infty, r)) } \bigg)^{1/2}. \]
So for large $i$ we have for any $t \in I^{\prime, i}$, $t \leq t'$,
\begin{multline*}
 d^{\XX^i_t}_{W_1} (\nu^i_{x_i; t}, \nu^i_{y_i; t} )
 \leq d^{\XX^i_{t'}}_{W_1} (\nu^i_{x_i; t'}, \delta_{z_i} ) + d^{\XX^i_{t'}}_{W_1} (\delta_{z_i}, \delta_{y_i}) 
 \leq \sqrt{H (t_i - t')} + 4r + \bigg( \frac{8H \eps}{  \mu^\infty_{t_\infty} (B(x_\infty, r)) } \bigg)^{1/2} \\
 \leq \sqrt{2H\eps} + 4r + \bigg( \frac{8H \eps}{  \mu^\infty_{t_\infty} (B(x_\infty, r)) } \bigg)^{1/2}.
\end{multline*}
On the other hand,
\[ d^{\XX^\infty_t}_{W_1} ( \nu^\infty_{y_\infty ; t}, \nu^\infty_{x_\infty; t} )
\leq d^{\XX^\infty_{t'}}_{W_1} ( \delta_{y_\infty}, \nu^\infty_{x_\infty; t'} )
\leq \sqrt{H \eps}. \]

Assume that $\CF$ is as in (\ref{eq_CF_over_N_infty}).
By Theorem~\ref{Thm_tdmu_convergence} for any compact $I_0 \subset I^\infty \cap (- \infty, t')$ we can find measurable subsets $E_i \subset I_i$ such that for large $i$ we have $J \cap I_0 \subset (I_0 \cap I'') \setminus E_i \subset I^{\prime\prime, i} \cap I^{\prime\prime, \infty}$ and
\[ |E_i| \to 0, \qquad \sup_{t \in (I_0 \cap I'') \setminus E_i} d^{Z_t}_{W_1} ( (\varphi^i_t)_* \nu^i_{y_i;t}, (\varphi^\infty_t)_* \nu^\infty_{y_\infty;t} ) \to 0 .\]
So for large $i$ we have $|E_i| \leq \eps$ and for all $t \in (I_0 \cap I'') \setminus E_i$
\begin{multline*}
 d^{Z_t}_{W_1} ( (\varphi^i_t)_* \nu^i_{x_i;t}, (\varphi^\infty_t)_* \nu^\infty_{x_\infty;t} ) 
\leq 
d^{\XX^i_t}_{W_1} (\nu^i_{x_i; t}, \nu^i_{y_i; t} ) 
+ d^{Z_t}_{W_1} ( (\varphi^i_t)_* \nu^i_{y_i;t}, (\varphi^\infty_t)_* \nu^\infty_{y_\infty;t} ) 
+ d^{\XX^\infty_t}_{W_1} ( \nu^\infty_{y_\infty ; t}, \nu^\infty_{x_\infty; t} ) \\
\leq \sqrt{H\eps}  + \eps + \sqrt{2H\eps} + 4r + \bigg( \frac{8H \eps}{  \mu^\infty_{t_\infty} (B(x_\infty, r)) } \bigg)^{1/2}. 
\end{multline*}
Letting first $\eps \to 0$ and then $r \to 0$ implies the desired convergence statement.
\end{proof}
\bigskip

\subsection{Compactness of sequences of points}
The following theorem shows that given an $\IF$-convergent sequence of metric flow pairs $(\XX^i, (\mu^i_t)_{t \in I^{\prime, i}})$ and points $x_i \in \XX^i_{t_i}$, $t_i \to t_\infty$, that remain within bounded distance from the ``center of the flow'', we can pass to a subsequence such that the  conjugate heat kernels based at $x_i$ converge to a conjugate heat flow on $\XX^\infty$ that has similar concentration properties as a conjugate heat kernel.
In general, this limit may not be a conjugate heat kernel of some point $x_\infty \in \XX^\infty_{t_\infty}$, but if it is, then we have convergence of $x_i$ to $x_\infty$ within $\CF$.

\begin{Theorem} \label{Thm_compactness_sequence_pts}
Let $(\XX^i, (\mu^i_t)_{t \in I^{\prime, i}})$, $i \in \IN \cup \{ \infty \}$, be metric flow pairs over intervals $I^i \subset \IR$ and consider a correspondence $\CF$ between the metric flows $\XX^i$, $i \in \IN \cup \{ \infty \}$.
Suppose that for some $J \subset \IR$ we have on compact time-intervals
\begin{equation} \label{eq_pts_compactness}
 (\XX^i, (\mu^i_t)_{t \in I^{\prime,i}}) \xrightarrow[i \to \infty]{\quad \IF, \CF, J \quad}  (\XX^\infty, (\mu^\infty_t)_{t \in I^{\prime,\infty}}), 
\end{equation}
and that all $\XX^i$, $i \in \IN \cup \{ \infty \}$, are $H$-concentrated for some uniform $H < \infty$.
Consider a sequence of points $x_i \in \XX^i_{t_i}$ with $t_i \to t_\infty \in I^\infty$, $t_\infty > \inf I^\infty$.
Suppose that $d^{\XX^i_{t_i}}_{W_1} ( \delta_{x_i}, \mu^i_{t_i}) \leq D$ for all $i \in \IN$, where $D < \infty$ is some uniform constant.
Then, after passing to a subsequence, we can find a conjugate heat flow $(\mu^\infty_t)_{t \in I^{\prime, \infty} \cap (-\infty,t_\infty)}$ on $\XX^\infty$, with
\begin{equation} \label{eq_Var_mu_lim_0}
 \lim_{t \nearrow t_\infty, t \in I^{\prime, \infty}} \Var (\mu^\infty_t) = 0, 
\end{equation}
such that on compact time-intervals
\begin{equation} \label{eq_nui_to_mu_in_CF}
 (\nu^i_{x_i;t})_{t \in I^{\prime, i} \cap (-\infty, t_i)}  \xrightarrow[i \to \infty]{\quad  \CF, J \quad}   (\mu^\infty_t)_{t \in I^{\prime, \infty} \cap (-\infty, t_\infty)} . 
\end{equation}
\end{Theorem}

\begin{Remark}
In general, $(\mu^\infty_t)_{t \in I^{\prime,\infty} \cap (-\infty,t_\infty)}$ may not need to be a conjugate heat kernel itself.
Consider for example a singular Ricci flow $\MM$ starting from $S^3$ that develops a non-degenerate, 3-dimensional neckpinch at time $T > 0$ and let $\MM' \subset \MM$ correspond to the choice of one component after the neckpinch; see the discussion preceding Theorem~\ref{Thm_superRF_metric_flow} for more details.
Let $\XX$ be the metric flow corresponding to $\MM'$.
Choose a point $x_\infty \in \MM_{T} \setminus \MM'_T$ within the other component of the neckpinch and let $(\mu_t)_{t \in [0,T)}$ be the conjugate heat flow corresponding to the conjugate heat kernel based at $x_\infty$.
Then  $(\mu_t)_{t \in [0,T)}$ is not a conjugate heat kernel on $\XX$, but it may arise as a limit as in (\ref{eq_nui_to_mu_in_CF}): Consider for example the constant sequence $\XX^i := \XX$ and the trivial correspondence $\CF$ and let $x_i := x_\infty (-\tau_i) \in \XX_{T - \tau_i}$ for some sequence $\tau_i \to 0$.
\end{Remark}

\begin{proof}
By Theorem~\ref{Thm_tdmu_convergence_reverse} we may pass to a subsequence and assume that the convergence (\ref{eq_pts_compactness}) is time-wise at almost every time.

\begin{Claim}
Let $t \in I^\infty$, $t < t_\infty$ be some time at which (\ref{eq_pts_compactness}) is time-wise.
Then, after passing to a subsequence, we have strict convergence of $\nu^i_{x_i;t}$ to some probability measure $\td\mu_t \in \PP (\XX^\infty_t)$.
\end{Claim}

\begin{proof}
Suppose that $\CF$ is as in (\ref{eq_CF_over_N_infty}).
Let $\eps > 0$ be some small constant and use Lemma~\ref{Lem_basic_measure} to choose a compact subset $K_\eps \subset \XX^\infty_t$ such that $\mu^\infty_t (K_\eps) > 1 -  \eps$.
Then for large $i$ we have for $K_{i, \eps} := (\varphi^i_t)^{-1} ( B( \varphi^\infty_t (K_\eps), \eps)) \subset \XX^i_t$
\[ \mu^i_t ( K_{i, \eps} ) = ((\varphi^i_{t})_* \mu^i_{t}) \big(  B( \varphi^\infty_t (K_\eps), \eps) \big) \geq 1 - 2\eps. \]
Since $\mu^i_{t_i} ( B(x_i, 2D) ) \geq \frac12$, we obtain from Definition~\ref{Def_metric_flow}\ref{Def_metric_flow_6} that for large $i$
\begin{align*}
  1 - 2\eps 
&\leq \mu^i_{t}( K_{i, \eps} ) 
\leq \int_{\XX^i_{t_i} } \nu^i_{x';t}( K_{i, \eps} ) d\mu^i_{t_i} (x') \\
&\leq \mu^i_{t_i} (\XX^i_{t_i} \setminus B(x_i, 2D))
+ \mu^i_{t_i} ( B(x_i, 2D)) \Phi \big( \Phi^{-1} ( \nu^i_{x_i;t}( K_{i, \eps} ) ) + 2(t_i - t)^{-1/2} D \big) \\
&\leq \tfrac12 + \tfrac12 \Phi \big( \Phi^{-1} ( \nu^i_{x_i;t}( K_{i, \eps} ) ) + 2(t_i - t)^{-1/2} D \big).
\end{align*}
 It follows that
 \[ \Phi \big( \Phi^{-1} ( \nu^i_{x_i;t}( K_{i, \eps} ) ) + 2(t_i - t)^{-1/2} D \big) \geq 1 - 4\eps. \]
 This implies that for large $i$
 \[ \nu^i_{x_i;t}( K_{i, \eps} ) \geq 1-  \Psi(\eps | t_\infty - t, D ), \]
 where $\Psi(\eps | t_\infty - t, D )$ denotes a function that goes to zero as $\eps \to 0$, while the other arguments are kept fixed.
 So by Lemma~\ref{Lem_basic_measure} the sequence $(\varphi^i_{t})_* \nu^i_{x_i;t} $ is tight and therefore for some subsequence we have weak convergence to some $\mu'_\infty \in \PP (Z_t)$.
 Moreover, $\supp \mu'_\infty \subset \varphi^\infty_t (\XX^\infty_t)$.
 By Lemma~\ref{Lem_weak_conv_2_W1_conv} this convergence implies convergence in the $W_1$-Wasserstein distance.
\end{proof}

Consider times $t'_k \in I^\infty$, $t'_k \nearrow  t_\infty$.
Apply the Claim successively to each $t'_k$ and pass to a diagonal subsequence.
Denote by $\td\mu_{t'_k} \in \PP (\XX^\infty_{t'_k})$ the probability measures obtained this way and let $(\mu^\infty_{k, t})_{t \in I^{\prime, \infty} \cap (-\infty, t'_k]}$ be the conjugate heat flows on $\XX^\infty$ with initial condition $\mu^\infty_{k, t'_k} = \td\mu_{t'_k}$.
By Theorem~\ref{Thm_tdmu_convergence} we have the following convergence on compact time-intervals
\[ (\nu^i_{x_i; t})_{t \in I^{\prime, i} \cap (-\infty, t'_k]}  \xrightarrow[i \to \infty]{\quad  \CF, J \quad}  (\mu^\infty_{k, t})_{t \in I^{\prime, \infty} \cap (-\infty, t'_k]} . \]
It follows that for any $k_1 \leq k_2$ we have $(\mu^\infty_{k_1, t})_{t \in I^{\prime, \infty} \cap (-\infty, t'_{k_1})} = (\mu^\infty_{k_2, t})_{t \in I^{\prime, \infty} \cap (-\infty, t'_{k_1})}$.
So there is a conjugate heat flow $(\mu^\infty_t)_{t \in I^{\prime, \infty} \cap (-\infty,t_\infty)}$ with $(\mu^\infty_t)_{t \in I^{\prime, \infty} \cap (-\infty,t'_k)} = (\mu^\infty_{k, t})_{t \in I^{\prime, \infty} \cap (-\infty, t'_k)}$ for all $k$.
This shows (\ref{eq_nui_to_mu_in_CF}).
The bound (\ref{eq_Var_mu_lim_0}) follows from a simple limit argument.
\end{proof}
\bigskip

\subsection{Tangent flows of the limit}
Next, we show that tangent flows of the limit of an $\IF$-convergent sequence of metric flow pairs can be represented as limits of rescalings of the original sequence.

Let first $\XX$ be a metric flow over some $I' \subset \IR$.
For any $\Delta T \in \IR$ and $\lambda > 0$ we will denote by $\XX^{-\Delta T, \lambda}$ the result of applying a time-shift of $-\Delta T$ to $\XX$ and then a parabolic rescaling by $\lambda$.
So any point $x \in \XX_t$ corresponds to a point in $\XX^{-\Delta T, \lambda}_{\lambda^2 (t - \Delta T)}$.
Similarly, if $(\mu_t)_{t \in I^*}$, $I^* \subset I'$, is a conjugate heat flow on $\XX$, then we denote by $(\mu^{-\Delta T, \lambda}_{\lambda^2 (t - \Delta T)} := \mu_{t} )_{t \in I^*}$ the corresponding conjugate heat flow on $\XX^{-\Delta T, \lambda}$.

\begin{Definition}[Tangent flow] \label{Def_tangent_cone}
Let $\XX$ be a metric flow over some $I' \subset \IR$ and $x_0 \in \XX_{t_0}$ a point.
We say that a metric flow pair $(\XX^\infty, (\nu^\infty_{x_{\max};t})_{t \in I^{\prime, \infty}}) \in \IF^*_{(-\infty, 0]}$ is a {\bf tangent flow of $\XX$ at $x_0$} if there is a sequence of scales $\lambda_1, \lambda_2, \ldots > 0$ with $\lambda_k \to \infty$ such that for any $T > 0$ the parabolic rescalings \[ \big(\XX^{-t_0, \lambda_k}_{[-T,0]}, (\nu_{x_0;t}^{-t_0, \lambda_k})_{\lambda_k^{-2} t + t_0 \in I' , t \in  [-T,0]}\big) \] $\IF$-converge to $(\XX^\infty_{[-T,0]}, (\nu^\infty_{x_{\max};t})_{t \in I^{\prime, \infty} \cap [-T,0]})$.
If $\lambda_k \to 0$ instead of $\lambda_k \to \infty$, then we call $(\XX^\infty, \lb  (\nu^\infty_{x_{\max};t})_{t \in I^{\prime, \infty}}) \in \IF^*_{(-\infty, 0]}$ a {\bf tangent flow of $\XX$ at infinity}.
\end{Definition}

\begin{Remark} \label{Rmk_tangent_cones_conv_within}
By Theorem~\ref{Thm_conv_to_conv_within_cpt_ti} the $\IF$-convergence condition in Definition~\ref{Def_tangent_cone} implies $\IF$-convergence on compact time-intervals within some correspondence.
\end{Remark}

\begin{Remark}
The tangent flow (at infinity) may depend on the sequence of scales $\lambda_k$, so it may not be unique.
\end{Remark}

Consider now a sequence of metric flow pairs $(\XX^i, (\mu^i_t)_{t \in I^{\prime, i}})$, $i \in \IN \cup \{ \infty \}$, over intervals $I^i \subset \IR$ and consider a correspondence $\CF$  between the metric flows $\XX^i$, $i \in \IN \cup \{ \infty \}$, such that
\[ (\XX^i, (\mu^i_t)_{t \in I^{\prime,i}}) \xrightarrow[i \to \infty]{\quad \IF, \CF \quad}  (\XX^\infty, (\mu^\infty_t)_{t \in I^{\prime,\infty}}). \]
Consider a tangent flow (at infinity) $(\XX^{*,\infty}, (\nu^{*,\infty}_{x^*_{\max};t})_{t \in I^{\prime, *, \infty}}) \in \IF^*_{(-\infty, 0]}$ at some point $x_\infty \in \XX^\infty_{t_\infty}$ corresponding to a sequence $\lambda_k \to \infty$ (or $\lambda_k \to 0$, respectively).
Suppose that $t_\infty \in I^{\prime,i}$ for all $i \in \IN$ and that all metric flows $\XX^i$ are $H$-concentrated for some uniform $H < \infty$.
Then Theorem~\ref{Thm_pts_as_limits_within} allows us to choose points $x_i \in \XX^i_{t_\infty}$ at the same time $t_\infty$ such that
\[ x_i  \xrightarrow[i \to \infty]{\quad  \CF \quad}  x_\infty. \]
So by Theorem~\ref{Thm_change_basepoints} we have
\[ (\XX^i_{\leq t_\infty}, (\nu^i_{x_i;t})_{t \in I^{\prime,i} \cap (-\infty, t_\infty]}) \xrightarrow[i \to \infty]{\quad \IF, \CF \quad}  (\XX^\infty_{\leq t_\infty}, (\nu^\infty_{x_\infty;t})_{t \in I^{\prime,\infty} \cap (-\infty, t_\infty]}). \]
For any $\eps, T > 0$ we can choose $k$ large such that
\[ d_{\IF} \big( (\XX^{\infty, - t_\infty, \lambda_k}_{[-T,0]}, (\nu^{\infty, - t_\infty, \lambda_k}_{x_\infty;t})_{ \lambda_k^{-2} t + t_\infty \in I^{\prime,\infty} , t \in [-T, 0]}), (\XX^{*,\infty}_{[-T,0]}, (\nu^{*,\infty}_{x^*_{\max};t})_{t \in I^{\prime, *, \infty} \cap [-T, 0]}) \big) \leq \eps/2. \]
Given $k$, we can choose $i$ large such that
\[ d_{\IF} \big( (\XX^i_{[t_\infty - \lambda_k^{-2} T, t_\infty]}, (\nu^i_{x_i;t})_{t \in I^{\prime,i} \cap [t_\infty - \lambda_k^{-2} T, t_\infty]}), (\XX^{\infty, - t_\infty, \lambda_k}_{[t_\infty - \lambda_k^{-2} T, t_\infty]}, (\nu^\infty_{x_\infty;t})_{t \in I^{\prime,\infty} \cap [t_\infty - \lambda_k^{-2} T, t_\infty]}) \big) \leq \lambda_k^2 \eps/2. \]
This implies that for the parabolic rescalings
\[  d_{\IF} \big(  (\XX^{i,  -t_\infty, \lambda_k}_{[-T,0]}, (\nu^{i,  -t_\infty, \lambda_k}_{x_i;t})_{\lambda_k^{-2} t + t_\infty \in I^{\prime,i},  t \in  [-T,0]}) , (\XX^{*,\infty}_{[-T,0]}, (\nu^{*,\infty}_{x^*_{\max};t})_{t \in I^{\prime, *, \infty} \cap [-T, 0]}) \big) \leq \eps \]
Letting $\eps \to 0$, $T \to \infty$, passing to a diagonal subsequence and applying Theorem~\ref{Thm_conv_to_conv_within_cpt_ti} implies:

\begin{Theorem}
There is a sequence $k_i \to \infty$ and correspondence $\td\CF$ between the parabolic rescalings $\XX^{i, - t_\infty, \lambda_{k_i}}$ and $\XX^{*, \infty}$ such that on compact time-intervals
\[  (\XX^{i,  -t_\infty, \lambda_{k_i}}, (\nu^{i,  -t_\infty, \lambda_k}_{x_i;t})_{\lambda_k^{-2} t + t_\infty \in I^{\prime,i}, t \leq 0})   \xrightarrow[i \to \infty]{\quad \IF, \td\CF \quad} (\XX^{*,\infty}, (\nu^{*,\infty}_{x^*_{\max};t})_{t \in I^{\prime, *, \infty}}) . \]
\end{Theorem}
\bigskip

\section{\texorpdfstring{Compact subsets of $(\IF^J_I, d_{\IF}^J)$}{Compact subsets of (F{\textasciicircum}J\_I, d{\textasciicircum}J\_F)}} \label{sec_compact_subsets_IF}
\subsection{Statement of the main results}
In this section let $I \subset \IR$ be an interval with $t_{\max} := \sup I < \infty$ and $J \subset I$ a subset.
In the following we will define certain subsets of the form $\IF_I^J (H,V,b,r) \subset \IF_I^J$, which will turn out to be compact if $I$ is a finite interval and $J$ is finite.
These subsets contain all metric flows corresponding to super Ricci flows over $I$, so we obtain that the set of super Ricci flows is precompact in $\IF^J_I$.

Let us now define the subsets $\IF_I^J (H,V,b,r)$.
For this purpose let $H, V \geq 0$, $r > 0$ and let $b : (0,1] \to (0,1]$ be a function.

\begin{Definition} \label{Def_FIHVb}
We define $\mathbb{F}_{I}^J (H,V,b,r) \subset \mathbb{F}_{I}^{J}$ to be the set of equivalence classes that are represented by (at least one) metric flow pair $(\XX, (\mu_t)_{t \in I'})$ over $I$ that is fully defined over $J$ and satisfies the following properties:
\begin{enumerate}
\item[(1)] \label{Def_FIHVb_1} $\XX$ is $H$-concentrated,
\item[(2a)] \label{Def_FIHVb_2a} If $t_{\max} \in J$, then we assume $(\XX_{t_{\max}}, d_{t_{\max}}, \mu_{t_{\max}}) \in \mathbb{M}_r (V,b)$.
\item[(2b)] \label{Def_FIHVb_2b} If $t_{\max} \not\in J$, then we assume that $\limsup_{t \nearrow t_{\max} } \Var (\mu_t) \leq Vr^2$.
\end{enumerate}
In  Case (2b) we may omit the function $b$ and write $\mathbb{F}^J_{I} (H,V,r)$ instead of $\mathbb{F}^J_{I} (H,V,b,r)$.
In both cases, if $V = 0$, then we will also write $\mathbb{F}_{I}^* (H) := \IF_I (H,0,b,r) \subset \mathbb{F}^*_{I}$; in this case the function $b$ and the scale $r$ are inessential.
\end{Definition}

\begin{Remark} \label{Rmk_assume_I_Ip}
By passing to the future completion (see Theorem~\ref{Thm_existenc_fut_compl}) every representative $(\XX, \lb (\mu_t)_{t \in I'})$ of an element in $\IF_I^J (H,V,b,r)$ is equivalent to an $H$-concentrated metric flow pair of the form $(\XX^*, (\mu^*_t)_{t \in I'})$, where $\XX^*$ is a metric flow over $I'$ where $I' = I$ if $t_{\max} \in J$ or $I' = I \setminus \{ t_{\max} \}$ if $t_{\max} \not\in J$.
\end{Remark}

The following is a direct consequence of Theorem~\ref{Thm_superRF_metric_flow}:

\begin{Lemma} \label{Lem_superRF_FHVBr}
Suppose that $\XX$ corresponds to a super Ricci flow $(M, (g_t)_{t \in I})$ on an $n$-dimensional compact manifold and $(\mu_t)_{t \in I}$ corresponds to the flow of the form $v \, dg_t$, where $v$ is a solution to the conjugate heat equation.
Assume that $(M, d_{g_{t_{\max}}}, v \, dg_{t_{\max}}) \in \mathbb{M}_r (V,b)$ for some $V \geq 0$, $r > 0$, $b : (0,1] \to (0,1]$ if $t_{\max} \in J$.
Then $( \XX, (\mu_t)_{t \in I}) \in \IF^J_I (H_n, V, b, r)$, where $H_n$ is as in Theorem~\ref{Thm_superRF_metric_flow}.

In particular, if $(\mu_t)_{t \in I}$ corresponds to a conjugate heat kernel measure $\nu_{x_0 ;s} = K(x_0 ,t_{\max}; \lb \cdot, s) dg_s$, then $(\XX, (\mu_t)_{t \in I}) \in \IF^{*,J}_I (H_n)$.
\end{Lemma}

Our main result will be:

\begin{Theorem} \label{Thm_F_compact}
Assume that $I \subset \IR$ is a finite interval and suppose that $J \subset I$ is a finite subset.
Let $H, V \geq 0$, $r > 0$ and $b : (0,1] \to (0,1]$ be a function.
Then $\IF^J_I (H,V,b,r)$ is a compact subset of $(\IF^J_I, d_{\IF}^J)$.
\end{Theorem}

Using Lemma~\ref{Lem_superRF_FHVBr}, this implies subsequential convergence of super Ricci flows:

\begin{Corollary} \label{Cor_super_RF_compactness}
Consider a sequence of super Ricci flows $(M^i, (g^i_t)_{t \in I})$ on compact $n$-dimensional manifolds together with a sequence of solutions to the conjugate heat equation $(v^i_t)_{t \in I}$ on $M^i$.
If $t_{\max} \in J$, then we assume that $(M^i, d_{g^i_{t_{\max}}}, v^i_{t_{\max}} dg^i_{t_{\max}}) \in \mathbb{M}_r (V,b)$ for some uniform $V \geq 0$, $r > 0$, $b : (0,1] \to (0,1]$.
If $I$ is a finite interval and $J$ is finite, then there is a subsequence, such that the corresponding sequence $(\XX^i, (\mu^i_t)_{t \in I})$ of metric flow pairs converges to a class of metric flows in $\IF^J_I (H_n, V, b, r)$ in the $d_{\IF}^J$-sense.
\end{Corollary}

Note that due to the definition of the $d^J_{\IF}$-distance, the limit of any sequence of metric flow pairs is only well defined wherever the limiting metric flow is continuous, so on the complement of a countable subset.
The next two theorems will address this issue.
Under additional assumptions, we will obtain compactness results, in which the limit is uniquely defined at every time.
We will also obtain convergence on compact time-intervals if the metric flow pairs are not defined over a common finite time-interval.

Fix in the following $H, V \geq 0$, $r > 0$ and $b : (0,1] \to (0,1]$.
Let $I^\infty \subset \IR$ be some interval and assume that $t_{\max} := \sup I^\infty < \infty$.
Consider a sequence of intervals $I^i \subset \IR$ with $\sup I^i \leq t_{\max}$ and $I^i \to I^\infty$ in the sense that $t \in I^\infty$ if and only if $t \in I^i$ for large $i$.
In both of the following theorems consider a sequence of metric flow pairs $(\XX^i, (\mu^i_t)_{t \in I^i})$ that are fully defined over $I^i \subset \IR$ and represent classes in $\IF_{I^i} (H)$ if $t_{\max} \not\in I^\infty$ or $\IF^{ t_{\max}}_{I^i} (H,V,b,r)$ if $t_{\max} \in I^\infty$.

The first theorem, which will be the most useful, concerns the case in which we require the limiting flow pair to be future continuous.

\begin{Theorem} \label{Thm_compactness_CF_conv_fut_cont}
After passing to a subsequence, there is, up to isometry, a unique flow pair $(\XX^\infty, (\mu^\infty_t)_{t \in I^\infty})$ representing a class in $\IF_{I^\infty} (H,V,b,r)$ for which $\XX^\infty$ is future continuous such that the following holds.
There is a correspondence $\CF$ between the metric flows $\XX^i$, $i \in \IN \cup \{ \infty \}$, such that on compact time-intervals
\begin{equation} \label{eq_F_CF_conv_fut_cont}
 (\XX^i, (\mu^i_t)_{t \in I^i}) \xrightarrow[i \to \infty]{\quad \IF, \CF \quad}  (\XX^\infty, (\mu^\infty_t)_{t \in I^\infty})  .
\end{equation}
If $I^\infty$ is a finite interval, then we even have normal $\IF$-convergence within $\CF$.

The convergence (\ref{eq_F_CF_conv_fut_cont}) is time-wise at any time at which $\XX^\infty$ is continuous and it is uniform over any compact $J \subset I^\infty$ that only contains times at which $\XX^\infty$ is continuous.
\end{Theorem}

In the next theorem we require that the $\IF$-convergence is time-wise.

\begin{Theorem} \label{Thm_compactness_CF_conv_timewise}
After passing to a subsequence, there is, up to isometry, a unique flow pair $(\XX^\infty, (\mu^\infty_t)_{t \in I^\infty})$ representing a class in $\IF_{I^\infty} (H,V,b,r)$ such that the following holds.
There is a correspondence $\CF$ between the metric flows $\XX^i$, $i \in \IN \cup \{ \infty \}$, such that
\begin{equation} \label{eq_F_conv_timewise_thm}
 (\XX^i, (\mu^i_t)_{t \in I^i}) \xrightarrow[i \to \infty]{\quad \IF, \CF \quad}  (\XX^\infty, (\mu^\infty_t)_{t \in I^\infty})  
\end{equation}
within $\CF$ on finite time-intervals and the convergence is time-wise at any time of $I^\infty$.
If $I^\infty$ is a finite interval, then we even have normal $\IF$-convergence within $\CF$ and the convergence is time-wise at any time of $I^\infty$.

The convergence (\ref{eq_F_conv_timewise_thm}) is uniform over any compact $J \subset I^\infty$ with the property that $\XX^\infty_J$ is continuous.
In particular, this is the case if $\XX^\infty$ is continuous at all times of $J$.
\end{Theorem}

So if $\XX^\infty$ is continuous, then the limits in Theorems~\ref{Thm_compactness_CF_conv_fut_cont}, \ref{Thm_compactness_CF_conv_timewise} agree and the $\IF$-convergence is uniform over every compact time-interval.

\begin{Remark}
The technical issue underlying Theorems~\ref{Thm_F_compact}, \ref{Thm_compactness_CF_conv_fut_cont}, \ref{Thm_compactness_CF_conv_timewise} can be illustrated by the following analogy.
Consider the space $\mathcal{F}$ of all non-decreasing functions $f: I \to [0,1]$ over some finite interval $I$.
Write $f_1 \sim f_2$ if $f_1 = f_2$ almost everywhere.
Then $(\mathcal{F}/\sim, \Vert \cdot \Vert_{L^1})$ is compact (this is comparable to Theorem~\ref{Thm_F_compact}).
More specifically, given any sequence of equivalence classes $[f_1], [f_2], \ldots \in \mathcal{F}$, we may pass to a subsequence such that $[f_i] \to [f_\infty] \in \mathcal{F}$ in $L^1$.
The representative $f_\infty$ is continuous on the complement of a countable subset $Q \subset I$ and for every $t \in I \setminus Q$ the value $f_\infty(t)$ is uniquely determined as a pointwise limit of the values $f_i(t)$.
For any $t \in Q$, we may choose $f_\infty(t)$ to be the right-limit, in which case $f_\infty$ is right semi-continuous (this is comparable to Theorem~\ref{Thm_compactness_CF_conv_fut_cont}).
Alternatively, we may pass to a further subsequence such that we also have pointwise convergence $f_i(t) \to f_\infty(t)$ for every $t \in Q$ (this  is comparable to Theorem~\ref{Thm_compactness_CF_conv_timewise}).
In fact, this analogy is quite fitting since the proofs of Theorems~\ref{Thm_F_compact}, \ref{Thm_compactness_CF_conv_fut_cont}, \ref{Thm_compactness_CF_conv_timewise} are based on the same compactness behavior of monotone functions.
\end{Remark}

\subsection{\texorpdfstring{Extending $\IF$-distance estimates to larger time domains}{Extending F-distance estimates to larger time domains}}
The following lemma allows us to extend closeness of two metric flow pairs within a correspondence over some given set of times to closeness within a correspondence over a larger set of times. 

\begin{Lemma} \label{Lem_criterion_discrete_closeness}
For every $H, V  \geq 0$ and every function $b : (0,1] \to (0,1]$ there is a function $\delta_{H, V, b} : \IR_+ \to \IR_+$ such that the following holds.

Let $r > 0$ be a scale and consider a subset $I_0  \subset I$ of an interval $I \subset \IR$.
Consider two metric flow pairs $(\XX^i, \lb (\mu^i_t)_{t \in I})$, $i = 1,2$, representing classes in $\IF^I_I (H, V, b,r)$ and a correspondence $\CF_0$ between $\XX^1, \XX^2$ over $I_0$.
Then there is a correspondence $\CF$ between $\XX^1, \XX^2$ over $I$ such that $\CF_0 = \CF |_{I_0}$ and such that the following is true:

Assume that $\eps > 0$ and that for $\delta = \delta_{H,V,b} (\eps)$ the following holds
\begin{equation} \label{eq_d_IF_CF0_delta}
 d_{\IF}^{\,\CF_0, I_0} \big( (\XX^1, (\mu^1_t)_{t \in I}), (\XX^2, (\mu^2_t)_{t \in I}) \big) < \delta  r. 
\end{equation}
(Note that since $\CF_0$ is defined over $I_0$, the properties of Definition~\ref{Def_IF_dist_within_CF} are required to hold for $I'' = I_0$.)
Consider a subset $I_0 \subset I_1 \subset I$ with $\sup I_1 \setminus I_0 < \sup I - \eps r^2$.
Suppose that for any $t \in I_1 \setminus I_0$ there is a minimal $t' \in (t, t+\delta r^2] \cap I_0$ and this $t'$ satisfies
\[ \int_{\XX^i_{t'}} \int_{\XX^i_{t'}} d^i_{t'} \, d\mu^i_{t'} d\mu^i_{t'} - \int_{\XX^i_{t}} \int_{\XX^i_{t}} d^i_{t} \, d\mu^i_{t} d\mu^i_{t} \leq \delta r \qquad \text{for} \quad i = 1,2. \]
Then over $I_1$
\[ d_{\IF}^{\,\CF |_{I_1}, I_1} \big( (\XX^1, (\mu^1_t)_{t \in I}), (\XX^2, (\mu^2_t)_{t \in I}) \big) \leq \eps r. \]
\end{Lemma}

\begin{proof}
Fix $H, V < \infty$ and $b : (0,1] \to (0,1]$.
After parabolic rescaling, we may assume that $r = 1$.
In the following we denote by $\Psi(\de)$ a generic function with the property that $\Psi (\de) \to 0$ as $\de \to 0$, which may depend on the choices of $H, V, b$.

Write
\[ \CF_0 = \big(  (Z_t, d^Z_t)_{t \in I_0},(\varphi^i_t)_{t \in I_0, i =1,2} \big). \]
In the following it suffices to construct $\CF$ over some fixed $I_1$ --- we will call the result $\CF_1$ --- and to observe that the construction over two possibly different such subsets agrees over their intersection.
For any $t \in I$ that does not lie in any $I_1$ satisfying the assumptions of the lemma (for some function $\delta_{H,V,b}$, which we will need to determine), we may simply not fully define $\CF$ over $t$, i.e. we will have $t \not\in I^{\prime\prime, i}$.
So assume that $I_1$ is given such that the assumptions of the lemma hold for some $\delta$, whose value we will determine later.

Choose $(q_t)_{t \in I_0}$ such that Property~\ref{Def_IF_dist_within_CF_3} in Definition~\ref{Def_IF_dist_within_CF} holds for all $s,t \in I_0$, $s \leq t$ and for $r$ replaced with $\delta$; note that we have to choose $E = \emptyset$.
We will first construct objects $(Z_t, d^{Z}_t)$, $(\varphi^i_t)_{ i =1,2}$, $q_t$ for $t \in I_1 \setminus I_0$ that will allow us to extend $\CF_0$ to $\CF_1$.
For this purpose, fix a $t \in I_1 \setminus I_0$ for now and choose $t' \in (t, t+\delta] \cap I_0$ as in the statement of the lemma.
We define
\begin{equation} \label{eq_q_t_from_q_tp}
 q_t := \int_{\XX^1_{t'}} \int_{\XX^2_{t'}} (\nu_{x^1;t}^1 \otimes \nu_{x^2;t}^2) dq_{t'} (x^1, x^2). 
\end{equation}
Then $q_t$ is a coupling between $\mu^1_t, \mu^2_t$.
Next note that by Propositions~\ref{Prop_mass_distribution}, \ref{Prop_time_s_closeness} for $i=1,2$ there are metric spaces $(\td Z^{i}, d^{\td Z^{i}})$ and isometric embeddings $\td\varphi^{i}_t : \XX^i_t \to \td Z^{i}$, $\td\varphi^{i}_{t'} :\XX^i_{t'} \to \td Z^{i}$ such that
\[ \int_{\XX_{t'}^i} \int_{\XX_t^i} d^{\td Z^i} (\td\varphi^i_t (x), \td\varphi^i_{t'} (x')) d\nu^i_{x';t}(x) d\mu^i_{t'} (x') \leq \Psi(\delta). \]
Using Lemma~\ref{Lem_combining_embeddings}, we can combine the spaces $\td Z^{1}$, $Z_{t'}$, $\td Z^{2}$ and assume that the isometric embeddings $\td\varphi^1_t, \td\varphi^1_{t'} = \varphi^1_{t'}, \varphi^2_{t'} = \td\varphi^2_{t'},  \varphi^2_t$ map into a single space $Z_t \supset Z_{t'}$; we will write $\varphi^i_t := \td\varphi^i_t$.
We obtain therefore that
\begin{equation} \label{eq_intint_d_Z_i_xpyp}
 \int_{\XX_{t'}^i} \int_{\XX_t^i} d^Z_t (\varphi^i_t (x), \varphi^i_{t'} (x')) d\nu^i_{x';t}(x) d\mu^i_{t'} (x') \leq \Psi(\delta). 
\end{equation}

After repeating the construction above for all $t \in I_1 \setminus I_0$, we can construct objects $(Z_t, d^{Z}_t)$, $(\varphi^i_t)_{ i =1,2}$, $q_t$, which allow us to extend $\CF_0$ to a correspondence $\CF_1$ between $\XX^1, \XX^2$ that is defined over $I_1$.
Moreover, for any $t \in I_1 \setminus I_0$ and any minimal $t' \in (t, t+\delta] \cap I_0$, we may assume that $Z_t \supset Z_{t'}$ and that (\ref{eq_q_t_from_q_tp}), (\ref{eq_intint_d_Z_i_xpyp}) hold.
It remains to show that Property~\ref{Def_IF_dist_within_CF_3} of Definition~\ref{Def_IF_dist_within_CF} holds for the family of couplings $(q_t)_{t \in I_1}$.
Let $s,t \in I_1$, $s \leq t$.
If $s \not\in I_0$, then choose $s' \in (s, s+\delta] \cap I_0$ minimal, otherwise choose $s' := s$.
Similarly, if $t \not\in I_0$, then choose $t' \in (t, t+\delta] \cap I_0$ minimal, otherwise choose $t' := t$.
Note that $s' \leq t'$.

Let $i = 1,2$.
For any $y' \in \XX^i_{t'}$, we obtain by Proposition~\ref{Prop_compare_CHF}\ref{Prop_compare_CHF_b}
\begin{multline} \label{eq_int_nuspnuypsp}
 \int_{\XX^i_{t}} d_{W_1}^{\XX^i_{s}} ( \nu^i_{y; s}, \nu^i_{y';s} ) d\nu^i_{y'; t} (y) 
\leq \int_{\XX_t^i} d_{W_1}^{\XX^i_{t}} ( \delta_y, \nu^i_{y';t} ) d\nu^i_{y'; t} (y) \\ 
= \int_{\XX_t^i}  \int_{\XX_t^i} d^i_t ( y, y^* ) d\nu^i_{y';t} (y^*) d\nu^i_{y'; t} (y) 
\leq \sqrt{\Var (\nu^i_{y';t})} \leq \Psi (\delta). 
\end{multline}
Next, since for any $y' \in \XX^i_{t'}$
\[ q^* := \int_{\XX^i_s}( \nu^i_{x';s} \otimes \delta_{x'} ) d\nu^i_{y';s} (x') \]
is a coupling between $\nu^i_{y';s}, \nu^i_{y';s'}$, we have
\[ d_{W_1}^{Z_s} ((\varphi^i_s)_* \nu^i_{y'; s}, (\varphi^i_{s'})_* \nu^i_{y';s'} ) 
\leq \int_{\XX^i_{s'}} \int_{\XX^i_{s}} d^Z_s (\varphi^i_{s}( x), \varphi^i_{s'}( x')) d\nu^i_{x';s} (x) d\nu^i_{y';s}(x'). \]
Integrating this over $y'$ against $d\mu^i_{t'}$ implies, using (\ref{eq_intint_d_Z_i_xpyp}),
\begin{multline} \label{eq_int_nuysysp}
\int_{\XX^i_{t'}} d_{W_1}^{Z_s} ((\varphi^i_s)_* \nu^i_{y'; s}, (\varphi^i_{s'})_* \nu^i_{y';s'} ) d\mu^i_{t'} (y')
\leq \int_{\XX^i_{t'}} \int_{\XX^i_{s'}} \int_{\XX^i_{s}} d^Z_s (\varphi^i_{s}( x), \varphi^i_{s'}( x')) d\nu^i_{x';s} (x)  d\nu^i_{y';s}(x') d\mu^i_{t'} (y') \\
=  \int_{\XX^i_{s'}} \int_{\XX^i_{s}} d^Z_s (\varphi^i_{s}( x), \varphi^i_{s'}( x'))  d\nu^i_{x';s} (x) d\mu^i_s (x')
\leq \Psi (\delta).
\end{multline}
Using (\ref{eq_q_t_from_q_tp}), (\ref{eq_int_nuspnuypsp}), (\ref{eq_int_nuysysp}), we obtain
\begin{align*}
\int_{\XX^1_t \times \XX^2_t} &d_{W_1}^{Z_s} ( (\varphi^1_s)_* \nu^1_{y^1; s}, (\varphi^2_s)_* \nu^2_{y^2;s} )  dq_t (y^1, y^2) \\
&= \int_{\XX^1_{t'} \times \XX^2_{t'}} \int_{\XX^1_{t}} \int_{\XX^2_{t}} d_{W_1}^{Z_s} ((\varphi^1_{s})_* \nu^1_{y^1; s}, (\varphi^2_{s})_* \nu^2_{y^2;s} ) d\nu^2_{y^{\prime, 2}; t} (y^2) d\nu^1_{y^{\prime, 1}; t} (y^1) dq_{t'} (y^{\prime, 1}, y^{\prime, 2}) \displaybreak[1] \\
&\leq  \int_{\XX^1_{t'} \times \XX^2_{t'}} \int_{\XX^1_{t}} \int_{\XX^2_{t}} 
\big( 
d_{W_1}^{Z_s} ((\varphi^1_{s})_* \nu^1_{y^1; s}, (\varphi^1_{s})_* \nu^1_{y^{\prime, 1};s} ) 
+ d_{W_1}^{Z_s} ((\varphi^1_{s})_* \nu^1_{y^{\prime,1}; s}, (\varphi^2_{s})_* \nu^2_{y^{\prime,2};s} )  \\
&\qquad\qquad\qquad\qquad\qquad + d_{W_1}^{Z_s} ((\varphi^2_{s})_* \nu^2_{y^{\prime,2}; s}, (\varphi^2_{s})_* \nu^2_{y^2;s} ) 
\big) d\nu^2_{y^{\prime, 2}; t} (y^2) d\nu^1_{y^{\prime, 1}; t} (y^1) dq_{t'} (y^{\prime, 1}, y^{\prime, 2}) \displaybreak[1] \\
&\leq \Psi (\delta) + \int_{\XX^1_{t'} \times \XX^2_{t'}} d_{W_1}^{Z_s} ((\varphi^1_{s})_* \nu^1_{y^{\prime,1}; s}, (\varphi^2_{s})_* \nu^2_{y^{\prime,2};s} )  dq_{t'} (y^{\prime, 1}, y^{\prime, 2}) \displaybreak[1] \\
&\leq \Psi (\delta) + \int_{\XX^1_{t'} \times \XX^2_{t'}} \big( 
d_{W_1}^{Z_s} ( (\varphi^1_s)_* \nu^1_{y^1; s}, (\varphi^1_{s'})_* \nu^1_{y^1;s'} )  
+ d_{W_1}^{Z_s} ((\varphi^1_{s'})_* \nu^1_{y^1; s'}, (\varphi^2_{s'})_* \nu^2_{y^2;s'} ) \\
&\qquad\qquad\qquad\qquad\qquad
+ d_{W_1}^{Z_s} ( (\varphi^2_{s'})_* \nu^2_{y^2; s'}, (\varphi^2_s)_* \nu^2_{y^2;s} ) \big) dq_{t'} (y^1, y^2) \displaybreak[1] \\
&\leq \Psi (\delta)
+ \int_{\XX^1_{t'} \times \XX^2_{t'}}  d_{W_1}^{Z_{s'}} ((\varphi^1_{s'})_* \nu^1_{y^1; s'}, (\varphi^2_{s'})_* \nu^2_{y^2;s'} ) dq_{t'} (y^1, y^2) \leq \Psi(\delta) .
\end{align*}
This finishes the proof.
\end{proof}

\subsection{Proofs of the main theorems}
We will need the following lemmas.

\begin{Lemma} \label{Lem_subconv_finite_I0}
Fix some $H, V  \geq 0$, $r > 0$, a function $b : (0,1] \to (0,1]$ and a finite subset $I_0 \subset I \subset \IR$ of an interval.
Consider a sequence of metric flow pairs $(\XX^i, \lb (\mu^i_t)_{t \in I})$ representing classes in $\IF_I^{I} (H, V, b,r)$, $i = 1,2, \ldots$.
Then, after passing to a subsequence, we can find a correspondence $\CF_0$ between the metric flows $\XX^i$ over $I_0$ that is also fully defined over $I_0$ such that the metric flow pairs $(\XX^i, \lb (\mu^i_t)_{t \in I})$ form Cauchy sequence within $\CF_0$ uniformly over $I_0$, i.e. for any $\eps > 0$ we have
\[  d_{\IF}^{\,\CF_0, I_0} \big( (\XX^i, (\mu^i_t)_{t \in I}), (\XX^j, (\mu^j_t)_{t \in I}) \big) \leq \eps \qquad \text{for large} \quad i, j. \]
\end{Lemma}

\begin{proof}
Write $I_0 =: \{ t_1 < \ldots < t_N \}$.
By Proposition~\ref{Prop_mass_distribution} and Theorem~\ref{Thm_M_compact} we may pass to a subsequence such that for all $k = 1, \ldots, N$
\[ ( \XX^i_{t_k}, d^i_{t_k}, \mu^i_{t_k}) \xrightarrow[i \to \infty]{\quad GW_1 \quad} ( X^\infty_{t_k}, d^\infty_{t_k}, \mu^\infty_{t_k}), \]
where we may assume the limiting spaces to be separable, complete and of full support.
By \cite[Lemma~5.7]{Greven-Pfaffelhuber-Winter} (see also Lemma~\ref{Lem_combining_embeddings}), after passing to another subsequence, we can find complete and separable metric spaces $(Z_{t_k}, d^{Z}_{t_k})$ and isometric embeddings $\varphi_{t_k}^i : \XX^i_{t_k} \to Z_{t_k}$, $\varphi_{t_k}^\infty: X^\infty_{t_k} \to Z_{t_k}$ with the property that
\begin{equation} \label{eq_mui_converges}
 ( \varphi_{t_k}^i )_* \mu^i_{t_k} \xrightarrow[i \to \infty]{\quad W_1 \quad} ( \varphi_{t_k}^\infty )_* \mu^\infty_{t_k}. 
\end{equation}
Choose couplings $q^i_{t_k}$, $k = 1, \ldots, N$, between $\mu^i_{t_k}, \mu^\infty_{t_k}$ with
\begin{equation} \label{eq_choice_of_q_i_t_k}
 \int_{\XX^i_{t_k} \times X^\infty_{t_k}} d^Z_{t_k} ( \varphi^i_{t_k} (x^i), \varphi^\infty_{t_k} (x^\infty)) dq^i_{t_k} (x^i, x^\infty)  \to 0. 
\end{equation}

\begin{Claim}
Let $1 \leq l \leq k \leq N$ and $x^\infty \in X^\infty_{t_k}$.
Then, after passing to a subsequence, there is a probability measure $\nu^\infty_{x^\infty; t_l} \in \mathcal{P} (X^\infty_{t_l})$ such that for any sequence $x^i \in \XX^i_{t_k}$ with $\varphi^i_{t_k} (x^i) \to \varphi^\infty_{t_k} (x^\infty)$ we have
\begin{equation} \label{eq_nu_i_converge}
 ( \varphi_{t_l}^i )_* \nu^i_{x^i; t_l}  \xrightarrow[i \to \infty]{\quad W_1 \quad} ( \varphi_{t_l}^\infty )_* \nu^\infty_{x^\infty; t_l}. 
\end{equation}
\end{Claim}

\begin{proof}
We may assume that $l < k$, because the claim is trivial in the case $l = k$.
We first show that we may pass to a subsequence such that the sequence $ ( \varphi_{t_l}^i )_* \nu^i_{x^i; t_l}$ converges to some probability measure $\nu^{\prime} \in \PP (Z_{t_l})$.
Due to Lemmas~\ref{Lem_basic_measure}\ref{Lem_basic_measure_d}, \ref{Lem_weak_conv_2_W1_conv} it suffices to show that the sequence of probability measures $( \varphi_{t_l}^i )_* \nu^i_{x^i; t_l}$ on $Z_{t_l}$ is tight.
So fix some $\eps > 0$.
By Lemma~~\ref{Lem_basic_measure}\ref{Lem_basic_measure_e} it suffices to show that there is a compact subset $K_\eps \subset Z_{t_l}$ such that for large $i$
\[  \nu^i_{x^i; t_l} \big( \XX^i_{t_l} \setminus ( \varphi_{t_l}^i )^{-1} (B(K_\eps, \eps)) \big) =  \big( ( \varphi_{t_l}^i )_* \nu^i_{x^i; t_l} \big) ( Z_{t_l} \setminus B(K_\eps, \eps) ) \leq \eps. \]
Let $\alpha > 0$ be a constant whose value we will determine later.
By Lemma~\ref{Lem_basic_measure}\ref{Lem_basic_measure_a} we can choose a compact subset $K'_\eps \subset X^\infty_{t_l} $ such that
\[ \mu^\infty_{t_l} ( X^\infty_{t_l} \setminus K'_\eps ) < \alpha. \]
Let $K_\eps := \varphi^\infty_{t_l} (K'_\eps)$ and $K_{i, \eps} := (\varphi_{t_l}^i )^{-1} (B(K_\eps, \eps))$.
Then for large $i$ we have by (\ref{eq_mui_converges})
\begin{equation} \label{eq_mui_XX_sm_Bi_alph}
 \mu^i_{t_l} ( \XX^i_{t_l} \setminus K_{i,\eps} ) < \alpha. 
\end{equation}
Again by (\ref{eq_mui_converges}) we can find a $D < \infty$ such that for large $i$ 
\[ \mu^i_{t_k} (  B(x^i, D) ) \geq \tfrac12. \]
So if for some large $i$ we had
\[  \nu^i_{x^i; t_l} (\XX^i_{t_l} \setminus K_{i, \eps}) > \eps,  \]
then by Definition~\ref{Def_metric_flow}\ref{Def_metric_flow_6}, we would have $\nu^i_{\cdot; t_l} (\XX^i_{t_l} \setminus K_{i,\eps}) \geq \Phi^{-1} ( \Phi(\eps) - D (t_k - t_l)^{-1/2})$ on $B(x^i, D)$, which would imply
\[  \mu^i_{t_l} (\XX^i_{t_l} \setminus K_{i,\eps}) = \int_{\XX_{t_k}^i} \nu^i_{y; t_l} (\XX^i_{t_l} \setminus K_{i,\eps})  d\mu^i_{t_k} (y) \geq \tfrac12 \Phi^{-1} ( \Phi(\eps) - D (t_k - t_l)^{-1/2}). \]
This contradicts (\ref{eq_mui_XX_sm_Bi_alph}) for small enough $\alpha$.

So it follows that, after passing to some subsequence we have
\[ ( \varphi_{t_l}^i )_* \nu^i_{x^i; t_l}  \xrightarrow[i \to \infty]{\quad W_1 \quad}  \nu' \in \PP (Z_{t_l}). \]
Since we had $K_\eps \subset \supp (\varphi^\infty_{t_l})_* \mu^\infty_{t_l}$ in the previous argument, we also get $\supp \nu' \subset \supp (\varphi^\infty_{t_l})_* \mu^\infty_{t_l}$, which implies that $\nu' = ( \varphi_{t_l}^\infty )_* \nu^\infty_{x^\infty; t_l}$ for some $\nu^\infty_{x^\infty; t_l} \in \mathcal{P} (X^\infty_{t_l})$.
.\end{proof}

Since the spaces $X^\infty_{t_k}$ are separable and the maps 
\[ (\XX^i_{t_k} , d^i_{t_k} ) \longrightarrow (\mathcal{P} (Z_{t_l}), d^{Z_{t_l}}_{W_1} ), \qquad y \longmapsto (\varphi^i_{t_l})_* \nu^i_{y; t_l} \]
are $1$-Lipschitz, we may pass to a subsequence and assume that (\ref{eq_nu_i_converge}) holds for all $1 \leq k \leq l \leq N$, $x^\infty \in X^\infty_{t_k}$ and any sequence $x^i \in \XX^i_{t_k}$ with $\varphi^i_{t_k} (x^i) \to \varphi^\infty_{t_k} (x^\infty)$.

\begin{Claim}
For any $\eps > 0$ and any compact subset $K \subset X^\infty_{t_k}$ for large $i$ the following bound holds for any $y^i \in \XX^i_{t_k}$, $y^\infty \in X^\infty_{t_k}$:
\[ d^{Z_{t_l}}_{W_1} ( (\varphi^i_{t_l})_* \nu^i_{y^i; t_l},(\varphi^\infty_{t_l})_* \nu^\infty_{y^\infty; t_l} ) \leq  d^Z_{t_k} (\varphi^i_{t_k} (y^i), \varphi^\infty_{t_k} (y^\infty)) + 2d^\infty_{t_k} (y^\infty, K) + \eps . \]
\end{Claim}

\begin{proof}
Fix some $\eps > 0$ and $K \subset X^\infty_{t_k}$.
Let $\{ x^\infty_1, \ldots, x^\infty_N \} \subset X^\infty_{t_k}$ be an $\eps/4$-net for $K$.
For any $m = 1, \ldots, N$ choose a sequence $x^i_m \in \XX^i_{t_k}$ such that $\varphi^i_{t_k} (x^i_m) \to \varphi^\infty_{t_k} (x^\infty_m)$.
Now suppose that $i$ is large enough such that we have
\[ d^{Z}_{t_k} ( \varphi^i_{t_k} (x^i_m), \varphi^\infty_{t_k} (x^\infty_m) ), \; d^{Z_{t_l}}_{W_1} ( (\varphi^i_{t_l})_* \nu^i_{x^i_m; t_l},(\varphi^\infty_{t_l})_* \nu^\infty_{x^\infty_m; t_l} ) \leq \eps/4 \qquad \text{for all} \quad m = 1, \ldots, N. \]
For any $y^i \in \XX^i_{t_k}$, $y^\infty \in X^\infty_{t_k}$ choose $m \in \{ 1, \ldots, N \}$ such that
\[ d^{Z}_{t_k} ( \varphi^\infty_{t_k} (y^\infty), \varphi^\infty_{t_k} (x^\infty_m) ) =d^\infty_{t_k} (y^\infty, x^\infty_m) \leq d^\infty_{t_k} (y^\infty, K) + \eps/ 4. \]
Then, using Proposition~\ref{Prop_compare_CHF}\ref{Prop_compare_CHF_c},
\begin{align*}
 &d^{Z_{t_l}}_{W_1}  ( (\varphi^i_{t_l})_* \nu^i_{y^i; t_l},(\varphi^\infty_{t_l})_* \nu^\infty_{y^\infty; t_l} ) \\
&\leq d^{Z_{t_l}}_{W_1} ( (\varphi^i_{t_l})_* \nu^i_{y^i; t_l},(\varphi^i_{t_l})_* \nu^i_{x^i_m; t_l} ) 
+ d^{Z_{t_l}}_{W_1} ( (\varphi^i_{t_l})_* \nu^i_{x^i_m; t_l}  ,(\varphi^\infty_{t_l})_* \nu^\infty_{x^\infty_m; t_l} ) 
+ d^{Z_{t_l}}_{W_1} ( (\varphi^\infty_{t_l})_* \nu^\infty_{x^\infty_m; t_l}  ,(\varphi^\infty_{t_l})_* \nu^\infty_{y^\infty; t_l} )  \\
&\leq d^{Z}_{t_k} ( \varphi^i_{t_k} (y^i), \varphi^i_{t_k} (x^i_m) ) + \eps/4 + d^{Z}_{t_k} ( \varphi^\infty_{t_k} (x^\infty_m), \varphi^\infty_{t_k} (y^\infty )) \\
&\leq d^{Z}_{t_k} ( \varphi^i_{t_k} (y^i), \varphi^\infty_{t_k} (y^\infty) )  
+ 2d^{Z}_{t_k} ( \varphi^\infty_{t_k} (y^\infty), \varphi^\infty_{t_k} (x^\infty_m) ) 
+ d^{Z}_{t_k} ( \varphi^\infty_{t_k} (x^\infty_m), \varphi^i_{t_k} (x^i_m) ) + \eps/4 \\
& \leq  d^Z_{t_k} (\varphi^i_{t_k} (y^i), \varphi^\infty_{t_k} (y^\infty)) + 2d^\infty_{t_k} (y^\infty, K) + \eps. \qedhere
\end{align*}
\end{proof}
\medskip

So by (\ref{eq_choice_of_q_i_t_k}) the following holds for any $\eps > 0$, compact $K \subset X^\infty_{t_k}$ and large $i$
\begin{multline*}
 \int_{\XX^i_{t_k} \times X^\infty_{t_k}} d^{Z_{t_l}}_{W_1} ( (\varphi^i_{t_l})_* \nu^i_{y^i; t_l},(\varphi^\infty_{t_l})_* \nu^\infty_{y^\infty; t_l} )  dq^i_{t_k} (y^i, y^\infty)  \\
\leq 2 \eps +  2\int_{X^\infty_{t_k}} d^\infty_{t_k} (y^\infty, K) d\mu^\infty_{t_k}  (y^\infty)
\leq 2 \eps + 2 \mu_{t_k}^{1/2} ( X^\infty_{t_k} \setminus K ) \bigg({ \int_{X^\infty_{t_k}} \big(d^\infty_{t_k} (y^\infty, K) \big)^2 d\mu^\infty_{t_k} (y^\infty) }\bigg)^{1/2}. 
\end{multline*}
Since $K$ can be chosen such that the last integral is bounded by $\Var (\mu^\infty_{t_k})$ and $\mu_{t_k} ( X^\infty_{t_k} \setminus K)$ is arbitrarily small, we find that 
\[ \int_{\XX^i_{t_k} \times X^\infty_{t_k}} d^{Z_{t_l}}_{W_1} ( (\varphi^i_{t_l})_* \nu^i_{y^i; t_l},(\varphi^\infty_{t_l})_* \nu^\infty_{y^\infty; t_l} )  dq^i_{t_k} (x^i, x^\infty) \to 0. \]

As in the proof of Proposition~\ref{Prop_triangle_ineq_CC}, for any $k = 1, \ldots, N$, $1 \leq i \leq j$ we can construct a coupling $q^{i,j}_{t_k}$ between $\mu^i_{t_k}, \mu^j_{t_k}$ such that for some $\eps_i \to 0$ and any $1 \leq l \leq k$
\begin{equation*}
 \int_{\XX^i_{t_k} \times \XX_{t_k}^j} d_{W_1}^{Z_{t_l}} ( (\varphi^i_{t_l})_* \nu^i_{x^i; t_l}, (\varphi^j_{t_l})_* \nu^j_{x^j; t_l} ) dq^{i,j}_t (x^i, x^j) \leq \eps_i. \end{equation*}
This finishes the proof of the lemma.
\end{proof}

Combining Lemmas~\ref{Lem_criterion_discrete_closeness}, \ref{Lem_subconv_finite_I0} yields:

\begin{Lemma} \label{Lem_tot_bounded_I0I1}
For every $\eps > 0$, $H, V  \geq 0$ and every function $b : (0,1] \to (0,1]$ there is a $\delta (\eps, H, V, b) > 0$ such that the following holds.

Consider a sequence of metric flow pairs $(\XX^i, \lb (\mu^i_t)_{t \in I})$ representing classes in $\IF^I_I (H, V, b,r)$ for some $r > 0$ over an interval $I \subset \IR$ that are also fully defined over $I$.
Let $I_0 \subset I_1 \subset I$ be subsets such that the following holds:
\begin{enumerate}[label=(\roman*)]
\item $I_0$ is finite.
\item For any $t \in I_1$ there is a minimal $t' \in I_0 \cap [t, t+\delta r^2]$ and for this $t'$ we have
\[  \int_{\XX^i_{t'}} \int_{\XX^i_{t'}} d^i_{t'} \, d\mu^i_{t'} d\mu^i_{t'} - \int_{\XX^i_{t}} \int_{\XX^i_{t}} d^i_{t} \, d\mu^i_{t} d\mu^i_{t} \leq \delta r \qquad \text{for all} \quad i. \]
\item $| I \setminus I_1 | \leq (\eps r)^2$.
\end{enumerate}
Then, after passing to a subsequence, we have
\[ d_{\IF}^{ I_1} \big( (\XX^i, (\mu^i_t)_{t \in I}), (\XX^j, (\mu^j_t)_{t \in I}) \big) \leq \eps r \qquad \text{for all} \quad i, j. \]
\end{Lemma}

The following will be a consequence of Lemma~\ref{Lem_tot_bounded_I0I1}.

\begin{Lemma} \label{Lem_main_F_compactness}
For every $H, V \geq 0$, $r > 0$, and every function $b : (0,1] \to (0,1]$ the following holds.

Consider a sequence of metric flow pairs $(\XX^i, \lb (\mu^i_t)_{t \in I})$ representing classes in $\IF^I_I (H, V, b,r)$ over an interval $I \subset \IR$ that are also fully defined over $I$.
Then there is a subsequence such that for any $t \in I$ the following limit exists
\begin{equation} \label{eq_Dt_lim}
 D(t) := \lim_{i \to \infty} \int_{\XX^i_{t}} \int_{\XX^i_{t}} d^i_{t} \, d\mu^i_{t} d\mu^i_{t} < \infty. 
\end{equation}

Moreover, whenever we are in the situation that the limit (\ref{eq_Dt_lim}) exists for all $t \in I$, then $D(t)$ is continuous on the complement of a countable subset and the following holds.
Let $J \subset I$ be a compact subset such that the restriction $D |_J$ is continuous and $\eps > 0$.
Then there is a subsequence such that 
\[ d_{\IF}^{ J} \big( (\XX^i, (\mu^i_t)_{t \in I}), (\XX^j, (\mu^j_t)_{t \in I}) \big) \leq |I \setminus J |^{1/2} + \eps \qquad \text{for all} \quad i, j. \]
\end{Lemma}

\begin{proof}
After parabolic rescaling, we may assume that $r = 1$.

Denote by $D_i(t)$ the value of the integral in (\ref{eq_Dt_lim}).
By Lemma~\ref{Lem_intd_diff_Var_diff} we have $D_i(t) - D_i(s) \geq - \sqrt{H(t-s)}$ for any $s, t \in I$, $s \leq t$.
Moreover, by H\"older's inequality we have $D_i(t) \leq \sqrt{\Var(\mu_t)} \leq \sqrt{ V + H(\sup I - t)}$.
After passing to a subsequence, we may assume that the limit in (\ref{eq_Dt_lim}) exists for any $t \in I \cap \mathbb{Q}$.
Then we still have $D(t) - D(s) \geq - \sqrt{H(t-s)}$ for any $s, t \in I \cap \mathbb{Q}$, $s \leq t$.
So there is a countable subset $\partial I \subset S \subset I$ such that $\lim_{t' \to t, t' \in I \cap \mathbb{Q}} D(t')$ exists for all $t \in I \setminus S$ and for any such $t$ this limit agrees with the limit in (\ref{eq_Dt_lim}).
We can now pass to another subsequence such that (\ref{eq_Dt_lim}) exists for all $t \in S$.
This proves the first part of the lemma.

For the second part of the lemma, observe again that  $D(t) - D(s) \geq - \sqrt{H(t-s)}$ for any $s, t \in I$, $s \leq t$, which implies that $D(t)$ is continuous on the complement of a countable subset.
Suppose now that $D |_J$ is continuous for some compact $J \subset I$ and fix $\eps > 0$ and some $\delta > 0$, whose value we will choose later.
For any $t \in J$ there is compact interval $t \in I_t \subset I$ that is a neighborhood of $t$ in $I$ and that satisfies
\[ |I_t| \leq \delta, \qquad \osc_{I_t \cap J} D \leq \delta. \]
By compactness, $J$ is covered by a finite number of these intervals.
So there is a finite subset $I_0 \subset I$ such that for any $t \in J$ there are $t'_1, t'_2 \in I_0$ such that 
\[ t'_1 \leq t \leq t'_2, \qquad t'_2 - t'_1 \leq \delta, \qquad D(t'_2) - D(t'_1) \leq \delta. \] 
After passing to a subsequence, we may assume that $|D_i(t') - D(t')| \leq \delta$ for all $t' \in I_0$.
So for any $t \in J$ there are $t'_1, t'_2 \in I_0$ such that $t'_1 \leq t \leq t'_2$, $t'_2 - t'_1 \leq \delta$ and
\[ 
D_i(t'_2) - D_i (t) \leq  D_i (t'_2) - D_i (t'_1) + \sqrt{H\delta} \leq D (t'_2) - D (t'_1) + \sqrt{H\delta} + 2\delta \leq  \sqrt{H\delta} + 3\delta. \]
Let now $t' \in I_0$ minimal with the property that $t' \geq t$.
Then $0 \leq t'_2 - t' \leq t'_2 - t'_1 \leq \delta$ and therefore
\[ D_i(t') - D_i (t) \leq D_i(t'_2) - D_i (t) + \sqrt{H\delta} \leq 2\sqrt{H\delta} + 3\delta. \]
The lemma now follows using Lemma~\ref{Lem_tot_bounded_I0I1} for small enough $\delta$.
\end{proof}
\bigskip

\begin{Lemma} \label{Lem_FHVb_closed}
Suppose that $I \subset \IR$ is an interval with $\sup I < \infty$, $J \subset I$ is a subset and consider $H, V \geq 0$, $r > 0$, $b : (0,1] \to (0,1]$.
Then $\IF_I^J (H,V,b,r) \subset \IF^J_I$ is closed.
\end{Lemma}

\begin{proof}
Consider a sequence of $H$-concentrated metric flow pairs $(\XX^i, (\mu^i_t)_{t \in I}) \in \IF_I^J (H,V,b,r)$ converging to a metric flow pair $(\XX^\infty, (\mu^\infty_t)_{t \in I}) \in  \IF^J_I$.
If $t_{\max} := \sup I \not\in J$, then we may replace $I$ with $I \setminus \{ t_{\max} \}$, since this does not change the $d_{\IF}^J$-distance.
Due to Remark~\ref{Rmk_assume_I_Ip}, we may assume that the metric flow pairs $(\XX^i, (\mu^i_t)_{t \in I})$ are fully defined over $I$.

By Theorem~\ref{Thm_conv_to_conv_within} we may choose a correspondence $\CF =(  (Z_t, d^Z_t)_{t \in I},(\varphi^i_t)_{t \in I^{\prime\prime, i}, i \in \IN \cup \{ \infty \}} )$ between the metric flows $\XX^i$, $i \in \IN \cup \{ \infty \}$, over $I$ such that
\begin{equation} \label{eq_conv_FHVb_closed_proof}
 (\XX^i, (\mu^i_t)_{t \in I}) \xrightarrow[i \to \infty]{\quad \IF, \CF, J \quad}  (\XX^\infty, (\mu^\infty_t)_{t \in I}) . 
\end{equation}
Let $E_i \subset I$ and $(q^i_t)_{t \in I \setminus E_i}$ be the objects from Definition~\ref{Def_IF_dist_within_CF}.
After passing to a subsequence and replacing each $E_i$ with $E_i \cup E_{i+1} \cup \ldots$, we may assume that $E_i$ is decreasing and that $E_\infty := \bigcap_{i=1}^\infty E_i$ has measure zero.
Then (\ref{eq_conv_FHVb_closed_proof}) also holds after replacing $J$ with $J \cup I \setminus E_j$ for any fixed $j \geq 1$.
By Lemma~\ref{Lem_MVb_closed} we know that $(\XX^\infty_{t_{\max}}, d^\infty_{t_{\max}}, \mu_{t_{\max}}) \in \mathbb{M}_r (V,b)$ if $t_{\max} \in J$.
We also obtain that for any $t \in I \setminus E_\infty$
\[ \Var (\mu^\infty_t) \leq \liminf_{i \to \infty} \Var (\mu^i_t) \leq V + H (t_\infty - t). \]
It remains to show that $\XX^\infty_{I \setminus E_\infty}$ is $H$-concentrated.
To see this, let $x^\infty, y^\infty \in \XX^\infty_t$ for some $t \in I \setminus E_\infty$.
Choose sequences $x^i , y^i \in \XX^i_t$ such that $\varphi^i_t (x^i) \to \varphi^\infty_t (x^\infty)$, $\varphi^i_t (y^i) \to \varphi^\infty_t (y^\infty)$ in $Z_t$.
For any $s \in I \setminus E_\infty$, $s \leq t$ we have using, Lemma~\ref{Lem_XX12_close_HK_close},
\[ (\varphi^i_s)_* \nu^i_{x^i;s} \xrightarrow[i \to \infty]{\quad W_1 \quad} (\varphi^\infty_s)_* \nu^\infty_{x^\infty;s} , \qquad (\varphi^i_s)_* \nu^i_{y^i;s} \xrightarrow[i \to \infty]{\quad W_1 \quad} (\varphi^\infty_s)_* \nu^\infty_{y^\infty;s} . \]
It follows that
\begin{multline*}
 \Var(  \nu^\infty_{x^\infty;s},  \nu^\infty_{y^\infty;s}) 
= \Var( (\varphi^\infty_s)_* \nu^\infty_{x^\infty;s}, (\varphi^\infty_s)_* \nu^\infty_{y^\infty;s}  ) 
\leq \limsup_{i \to \infty} \Var ( (\varphi^i_s)_* \nu^i_{x^i;s},  (\varphi^i_s)_* \nu^i_{y^i;s} ) \\
\leq \big( d^\infty_t (x^\infty, y^\infty) \big)^2 + H(t-s). 
\end{multline*}
This finishes the proof.
\end{proof}
\bigskip

\begin{proof}[Proof of Theorem~\ref{Thm_F_compact}.]
As in the proof of Lemma~\ref{Lem_FHVb_closed}, we may in the following only work with $H$-concentrated metric flow pairs that are fully defined over $I$.

By Lemma~\ref{Lem_FHVb_closed} the subset $\IF_I^J (H,V,b,r) \subset \IF_I^J$ is closed, so by Theorem~\ref{Thm_F_complete} it is complete.
To see total boundedness, suppose by contradiction that there is a sequence $(\XX^i, (\mu^i_t)_{t \in I}) \in \mathbb{F}_I^J (H,V,b,r)$ with the property that for some $\eps > 0$
\begin{equation} \label{eq_dF_not_eps}
 d_{\IF} \big( (\XX^i, (\mu^i_t)_{t \in I}), (\XX^j, (\mu^j_t)_{t \in I}) \big) > \eps r \qquad \text{for all} \quad i \neq j. 
\end{equation}
By Lemma~\ref{Lem_main_F_compactness} we may pass to a subsequence such that the limit (\ref{eq_Dt_lim}) exists for all $t \in I$ and it remains to show that there is a compact subset $J' \subset I$ with the property that $D|_{J \cup J'}$ is continuous and $|I \setminus (J \cup J')| < (\eps r)^2$.
Since $D$ is continuous almost everywhere, we can use Lemma~\ref{Lem_basic_measure}\ref{Lem_basic_measure_a} to find a compact subset $J' \subset I$ consisting only of points where $D$ is continuous that satisfies $|I \setminus J'| < (\eps r)^2$.
Since $J$ is finite, this implies that $D|_{J \cup J'}$ is continuous and $|I \setminus (J \cup J')| < \eps^2$.
\end{proof}
\bigskip

Next we establish Theorem~\ref{Thm_compactness_CF_conv_timewise}. Theorem~\ref{Thm_compactness_CF_conv_fut_cont} will be a direct consequence of Theorem~\ref{Thm_compactness_CF_conv_timewise}.

\begin{proof}[Proof of Theorem~\ref{Thm_compactness_CF_conv_timewise}.]
By Lemma~\ref{Lem_main_F_compactness}, we may pass to a subsequence such that the limit (\ref{eq_Dt_lim}) exists for all $t \in I^\infty$.
The function $D$ is continuous on $I \setminus Q$, where $Q = \{ t^*_1, t^*_2, \ldots \} \subset I$ is a countable subset.
Choose an increasing sequence of compact subintervals $I_{0,k} \subset I^\infty$ with $\bigcup_{k=1}^\infty I_{0,k} = I^\infty$.
As in the proof of Theorem~\ref{Thm_F_compact}, we can choose an increasing sequence of compact subsets $K_k \subset I_{0,k} \setminus  Q$ such that $\bigcup_{k=1}^\infty K_k = I \setminus Q$ and $| I_{0,k} \setminus K_k | \to 0$.
Set $J_k := (K_k \cup \{ t^*_1, \ldots, t^*_k \}) \cap I_{0,k}$.
Then $J_k$ is still compact, $D|_{J_k}$ is continuous for all $k$ and $\bigcup_{k =1}^\infty J_k = I^\infty$.
By the second part of Lemma~\ref{Lem_main_F_compactness} and after passing to a diagonal sequence, the sequence of metric flow pairs forms a Cauchy sequence with respect to $d^{J_k}_{\IF}$ over $I_{0,k}$ for any $k$.
So after passing to another subsequence, we can find correspondences $\CF_{i,i+1}$ between $\XX^i, \XX^{i+1}$ over $I_{0,i}$ that are fully defined over $J_i$ such that
\[ d_{\IF}^{\CF_{i, i+1}, J_i} \big( (\XX^i, (\mu^i_t)_{t \in I^i}), (\XX^{i+1}, (\mu^{i+1}_t)_{t \in I^{i+1}}) \big) \leq 2^{-i}. \]
As in the proof of Theorem~\ref{Thm_F_complete} we may find a correspondence $\CF^*$ between all $\XX^i$ such that $\CF_{i, i+1} = \CF|_{I_{0,i}, \{ i, i+1 \}}$.
By Lemma~\ref{Lem_F_completemess_within_CF} we may find a metric flow pair $(\XX^\infty, (\mu^\infty_t)_{t \in I^\infty})$ and an extension $\CF^{**}$ of $\CF^*$ such that for any $k$
\[ (\XX^i, (\mu^i_t)_{t \in I^i}) \xrightarrow[i \to \infty]{\quad \IF, \CF^{**}|_{I_{0,k}}, J_k \quad} (\XX^\infty, (\mu^\infty_t)_{t \in I^\infty}). \]
This shows that on compact time-intervals
\[  (\XX^i, (\mu^i_t)_{t \in I^i}) \xrightarrow[i \to \infty]{\quad \IF, \CF^{**} \quad}  (\XX^\infty, (\mu^\infty_t)_{t \in I^\infty}) ,  \]
which is time-wise at any time of $I^\infty$.

In order to ensure that we have the full desired uniform convergence properties, we have to carry out a more subtle construction of the correspondence $\CF$.
For this purpose, note first that due to the time-wise convergence we have
\[ D(t) = \int_{\XX^\infty_{t}} \int_{\XX^\infty_{t}} d^\infty_{t} \, d\mu^\infty_{t} d\mu^\infty_{t}. \]
So if $J \subset I^\infty$ is compact with the property that $\XX^\infty_J$ is continuous, then $D |_J$ is continuous.

Set
\[ D_i (t):= \int_{\XX^i_{t}} \int_{\XX^i_{t}} d^i_{t} \, d\mu^i_{t} d\mu^i_{t}. \]
Choose a dense set of times $\{ t_1, t_2, \ldots \} \subset I^\infty$ containing $Q \cup (I^\infty \cap \partial I^\infty)$.
Then there is a sequence $k_i \to \infty$ such that for the correspondence $\CF^{**}_i := \CF^{**} |_{\{ t_1, \ldots, t_{k_i} \}, \{i , \infty \}}$ between $\XX^i, \XX^\infty$ we have as $i \to \infty$
\[  d_{\IF}^{\CF^{**}_i, \{ t_1, \ldots, t_{k_i} \}} \big( (\XX^i, (\mu^i_t)_{t \in I^i}), (\XX^\infty, (\mu^\infty_t)_{t \in I^\infty}) \big) \to 0, \] 
\[ \max_{1 \leq l \leq k_i} |D_i(t_l) - D (t_l)| \to 0. \]
Apply Lemma~\ref{Lem_criterion_discrete_closeness} to each correspondence $\CF^{**}_i$ for $I_0 = \{ t_1, \ldots, t_{k_i} \}$ and denote the resulting correspondence between $\XX^i, \XX^\infty$ over $I^i \cap I^\infty$ by $\CF^{***}_i$.

\begin{Claim}
If $J \subset I^\infty$ is compact and $D |_J$ is continuous, then $\CF^{***}_i$ is fully defined over $J$ for large $i$ and
\[  d_{\IF}^{\CF^{***}_i |_J, J} \big( (\XX^i, (\mu^i_t)_{t \in I^i}), (\XX^\infty, (\mu^\infty_t)_{t \in I^\infty}) \big) \to 0. \] 
\end{Claim}

\begin{proof}
By Lemma~\ref{Lem_criterion_discrete_closeness} it suffices to show that there is a sequence $\delta_i \to 0$ such that the following holds for large $i$.
For any $t \in J$ there is an $l \in \{ 1, \ldots, k_i\}$ such that $t_l - t$ is minimal and $D(t_l) - D(t) \leq \delta_i$.
This can be achieved by a covering argument as in the proof of Lemma~\ref{Lem_main_F_compactness}.
\end{proof}

By successive application of Lemma~\ref{Lem_combining_correspondences} and a direct limit argument, we can combine the correspondences $\CF^{***}_i$ to a single correspondence $\CF$ between $\XX^i$, $i \in \IN \cup \{ \infty \}$ over $I^\infty$; see also the proof of Theorem~\ref{Thm_conv_to_conv_within_cpt_ti}.
So we obtain:

\begin{Claim} \label{Cl_CFJJ}
If $J \subset I^\infty$ is compact and $D |_J$ is continuous, then $\CF$ is fully defined over $J$ for large $i$ and
\[  d_{\IF}^{\CF |_J, J} \big( (\XX^i, (\mu^i_t)_{t \in I^i}), (\XX^\infty, (\mu^\infty_t)_{t \in I^\infty}) \big) \to 0. \] 
\end{Claim}

Note for any compact subinterval $I_0 \subset I^\infty$ and any $\eps > 0$ there is a compact $J \subset I_0 \setminus Q$ with $| I_0 \setminus J | \leq \eps$.
So Claim~\ref{Cl_CFJJ} implies (\ref{eq_F_conv_timewise_thm}), the time-wise convergence and the statement concerning the uniform convergence.

It remains to prove the uniqueness statement.
So suppose that within some other correspondence $\CF'$ we have a time-wise limit $(\XX^{\prime, \infty}, ( \mu^{\prime,\infty}_t)_{t \in I^\infty})$.
So all time-slices $(\XX^\infty, d^\infty_t, \mu^{\infty}_t)$, $(\XX^{\prime, \infty}_t, d_t^{\prime, \infty},  \mu^{\prime, \infty}_t)$, $t \in I^\infty$, are isometric as metric measure spaces.
This implies that $\XX^\infty, \XX^{\prime, \infty}$ are continuous at the same times.
By the same argument as in the proof of Proposition~\ref{Prop_triangle_ineq_CC} we can find sequences of correspondences $\CF^i = (  (Z^i_t, d^{Z^i}_t)_{t \in I^\infty},(\varphi^{i,j}_t)_{t \in I^{\prime\prime,i, j}, j =1,2 })$ between $\XX^\infty,\XX^{\prime, \infty}$ such that for any compact subinterval $I_0 \subset I^\infty$ and any $t_0 \in I_0$ we have
\[ d^{\CF^i|_{I_0}, \{ t_0 \} }_{\IF} \big( (\XX^\infty_{I_0}, (\mu^{\infty}_t)_{t \in I_0}), (\XX^{\prime, \infty}_{I_0}, ( \mu^{\prime, \infty}_t)_{t \in I_0}) \big) \to 0. \]
Consider an arbitrary subsequence of the sequence of metric flows.
By the proof of Theorem~\ref{Thm_IF_metric_space} we obtain a set of measure zero $E_{I_0, t_0} \subset I_0$ with $t_0 \not\in E_{I_0, t_0}$ and an almost always isometry $\phi_{I_0, t_0} : \XX^{ \infty}_{I_0 \setminus E_{I_0, t_0}} \to \XX^{\prime, \infty}_{I_0 \setminus E_{I_0, t_0}}$ between both metric flow pairs that is  fully defined over $t_0$ such that after passing to a subsequence we have 
\begin{equation} \label{eq_phiI0t0_characterization}
d^{Z^i}_t (\varphi^{i,1}_t (x), \varphi^{i,2}_t (\phi_{I_0, t_0} (x)) \to 0 \qquad \text{for any} \quad x \in \XX^\infty_t, t \in I_0 \setminus E_{I_0, t_0}.
\end{equation}

Let $Q = \{ t_1, t_2, \ldots \} \subset I^\infty$ be a dense subset containing the set of times where both flows $\XX^\infty, \XX^{\prime, \infty}$ are not  continuous and $\partial I^\infty \cap I^\infty$.
Fix an increasing sequence of compact subintervals $I_0 \subset I_1 \subset \ldots \subset I^\infty$ with $\bigcup_{k=1}^\infty I_k = I^\infty$ and such that $t_k \in I_k$.
We now apply the argument from the previous paragraph successively for $I_k, t_k$ while passing to a subsequence in each step.
Due to the characterization (\ref{eq_phiI0t0_characterization}), the maps $\phi_{I_k, t_k}$ agree on their overlap.
So after passing to a subsequence, we can find a set of measure zero $E \subset I^\infty \setminus Q$ and a map $\phi : \XX^\infty_{I^\infty \setminus E} \to \XX^{\prime, \infty}_{I^\infty \setminus E}$ such that for any $k \geq 1$ there is  a set of measure zero $E_{k} \subset I_{k}$ with $t_k \not\in E_{k}$ such that $\phi$ restricted to $I_{k} \setminus E_{k}$ is an isometry between both metric flow pairs $(\XX^{ \infty}, ( \mu^{\infty}_t)_{t \in I^\infty})$, $(\XX^{\prime, \infty}, ( \mu^{\prime,\infty}_t)_{t \in I^\infty})$.
Let now $k_2 \geq k_1 \geq 1$ and choose two different times $s_1 \in I_{k_1} \setminus E_{k_1}$, $s_2 \in I_{k_1} \setminus E_{k_2}$.
Then there is a time $s' \in I_{k_1} \setminus (E_{k_1} \cup E_{k_2})$ between $s_1, s_2$.
It follows that $\phi$ is an isometry over $\{ s_1, s' \}$ and $\{ s_2, s' \}$.
So by the reproduction formula (see also (\ref{eq_reprod_form_isometry})), $\phi$ is also an isometry over $\{ s_1, s_2 \}$.
Since for any $k \geq 1$ we have $t_k \not\in E_{k'}$ for large $k'$, we obtain that $\phi$ is an isometry over $I^\infty \setminus E \supset Q$ and since $\XX^\infty, \XX^{\prime \infty}$ are continuous over $I \setminus Q$, we can use Theorem~\ref{Thm_extend_isometry}, to extend $\phi$ to an isometry between the metric flow pairs $(\XX^{ \infty}, ( \mu^{\infty}_t)_{t \in I^\infty})$, $(\XX^{\prime, \infty}, ( \mu^{\prime,\infty}_t)_{t \in I^\infty})$ over $I^\infty$.
\end{proof}

\bigskip

\begin{proof}[Proof of Theorem~\ref{Thm_compactness_CF_conv_fut_cont}.]
Theorem~\ref{Thm_compactness_CF_conv_fut_cont} follows from Theorem~\ref{Thm_compactness_CF_conv_timewise} by replacing $\XX^\infty$ with the future completion of $\XX^\infty_{I^\infty \setminus Q}$, where $Q \subset I^{\infty}$ denotes the set of times at which $\XX^\infty$ is not continuous; compare with Theorem~\ref{Thm_existenc_fut_compl}.
The uniqueness statement follows using Theorem~\ref{Thm_extend_isometry}.
\end{proof}
\bigskip

\section{Intrinsic flows} \label{sec_intrinsic}
In this section we analyze under which conditions time-slices of a metric flow are length spaces.
We define:

\begin{Definition}
We call a metric flow $\XX$ over some $I \subset \IR$  {\bf intrinsic at time $t$} if $(\XX_t, d_t)$ is a length space.
We call $\XX$ {\bf intrinsic} if it is intrinsic for all $t \in I$ and {\bf almost always intrinsic} if it is intrinsic for almost all $t \in I$.
\end{Definition}

We have the following result:

\begin{Theorem} \label{Thm_alm_alw_intrinsic}
Suppose that $\XX$ is an $H$-concentrated metric flow of full support over an interval $I \subset \IR$, where $H < \infty$.
Suppose that there is a dense subset $S \subset I$ such that $\XX$ is intrinsic at every $t \in S$.
Then $\XX$ is almost always intrinsic.
Moreover, $\XX$ is intrinsic for all $t \in I \setminus \sup I$ at which $\XX$ is future continuous, which is the case at all but a countable set of times.
In particular, if $\XX$ is future continuous, then $\XX$ is intrinsic at all times of $I \setminus \sup I$.
\end{Theorem}

\begin{proof}
Suppose that $\XX$ is future continuous at time $t \in I \setminus \sup I$.
Let $x_1, x_2 \in \XX_t$ and $\eps > 0$.
Since $\XX$ is complete, it suffices to construct an approximate midpoint, i.e. a point $z \in \XX_t$ with
\begin{equation} \label{eq_approx_midpoint}
 d_t (x_1, z) ,  d_t (x_2, z) \leq  \tfrac12 d(x_1, x_2) + \eps. 
\end{equation}
For this purpose, fix a sequence of times $t_i \in I$ such that $t_i \searrow t$ and such that $(\XX_{t_i}, d_{t_i})$ is a length space.
By Proposition~\ref{Prop_fut_cont_two_points_approx} there are points $x_{1,i}, x_{2,i} \in \XX_{t_i}$ such that for $j =1,2$
\[ \lim_{i \to \infty} d_{W_1}^{\XX_t} ( \delta_{x_j} , \nu_{x_{j,i}; t} ) = 0, \qquad \lim_{i \to \infty} d_{t_i} (x_{1,i}, x_{2,i}) = d_t (x_1, x_2). \]
Since $(\XX_{t_i}, d_{t_i})$ are length spaces, we can find points $y_i \in \XX_{t_i}$ with
\[ \lim_{i \to \infty} d_{t_i} (x_{1,i}, y_i) = \lim_{i \to \infty} d_{t_i} (x_{2,i}, y_i) =  \tfrac12 d_t (x_1, x_2). \]
Let $z_i \in \XX_t$ be $H$-centers of $y_i$.
Then for $j=1,2$
\begin{multline*}
 \limsup_{i \to \infty} d_{t_i} (x_{j} , z_i ) 
\leq \limsup_{i \to \infty} \big( d^{\XX_t}_{W_1} ( \delta_{x_{j}}, \nu_{x_{j,i};t}) + d^{\XX_t}_{W_1} ( \nu_{x_{j,i};t}, \nu_{y_i;t})
+ d^{\XX_t}_{W_1} (\nu_{y_i;t}, \delta_{z_i}) \big) \\
 \leq \limsup_{i \to \infty}  d_{t_i} (x_{j,i}, y_i) = \tfrac12 d_t (x_1, x_2),
\end{multline*}
which implies (\ref{eq_approx_midpoint}) for large $i$.
\end{proof}

The next theorem shows that the almost always intrinsic property is closed under $\IF$-limits.

\begin{Theorem} \label{Thm_intrinsic_limit}
Consider a sequence of metric flow pairs $(\XX^i, (\mu^i_t)_{t \in I^{\prime, i}})$, $i \in \IN \cup \{ \infty \}$ over intervals $I^i \subset \IR$ such that within some correspondence $\CF$
\begin{equation} \label{eq_conv_intrinsic}
 (\XX^i, (\mu^i_t)_{t \in I^{\prime,i}}) \xrightarrow[i \to \infty]{\quad \IF, \CF \quad}  (\XX^\infty, (\mu^\infty_t)_{t \in I^{\prime,\infty}})  
\end{equation}
on compact time-intervals.
Suppose that the flows $\XX^i$ are $H$-concentrated for some uniform $H < \infty$.
If the flows $\XX^i$ are almost always intrinsic for all $i \in \IN$, then so is $\XX^\infty$.
\end{Theorem}

\begin{Remark}
The property of being intrinsic at a fixed time does, in general, not pass to the limit, even if the $\IF$-convergence is uniform at that time.
Consider for example a (possibly rotationally symmetric) singular Ricci flow $\MM$ on $S^2 \times S^1$ that develops two nearby non-degenerate neckpinches of bounded distance distortion at the same time $t_0 > 0$.
Such a flow can be constructed using the techniques from  \cite{Angenent_Knopf_precise_asymptotics, Angenent_Knopf_private}.
Let $\MM' \subset \MM$ be the subset corresponding to the larger remaining component after the neckpinch and let $\XX'$ be the \emph{past continuous} metric flow corresponding to $\MM'$, which is constructed in a similar fashion as $\XX'$ in Example~\ref{Ex_neckpinch_continuity}.
Then $\XX'_{t_0}$ is homeomorphic to two 3-spheres that are attached to each other at their poles and $\supp \XX'_{t_0} \subset \XX'_{t_0}$ corresponds to one of these 3-spheres.
So if the two neckpinches in $\MM$ are located closely enough to one another, then $(\supp \XX'_{t_0}, d'_{t_0} |_{\supp \XX'_{t_0}})$ is not a length metric, so the metric flow $\supp \XX'$ is not intrinsic at time $t_0$.
However, $\supp \XX'$ may arise as a limit of intrinsic metric flows that is uniform at time $t_0$; consider for example a sequence of time-shifts of the future continuous metric flow $\XX$ corresponding to $\MM'$ via Theorem~\ref{Thm_sing_flow_to_met_flow}.
\end{Remark}

\begin{proof}
We can find a subset $E_1 \subset I^\infty$ of measure zero such that for all $i \in \IN$ the flow $\XX^i$ is  intrinsic at every $t \in I^i \setminus E_1$ and such that $I^{\prime,i} \subset I^i \setminus E_1$.
By Corollary~\ref{Cor_met_flow_cont_ae} $\XX^\infty$ is future continuous at every time $t \in I^\infty \setminus E_2$, for some subset $E_2 \subset I^\infty$ of measure zero with $I^{\prime, \infty} \subset I^\infty \setminus E_2$.
Lastly, by Lemma~\ref{Lem_pass_to_timewise} we may pass to a subsequence and assume that the convergence (\ref{Lem_pass_to_timewise}) is time-wise at any time of $I^\infty \setminus E_3$ for some set of measure zero $E_3 \subset I^\infty$.
Set $E:= E_1 \cup E_2 \cup E_3 \cup \{ \sup I^\infty \}$.

Fix a time $t \in I^\infty \setminus E$.
We will show that $\XX^\infty$ is intrinsic at time $t$ by constructing an almost midpoint $z$ between two given points $x_1, x_2 \in \XX^\infty_t$ and for some $\eps > 0$ as in (\ref{eq_approx_midpoint}).
Let $\delta > 0$ be a constant whose value we will determine later and choose $t' \in I^\infty \setminus E$ with $t < t' < t + \delta$.
Since $\XX^\infty$ is future continuous at time $t$, we can use Proposition~\ref{Prop_fut_cont_two_points_approx} to find points $x'_1, x'_2 \in \XX^\infty_{t'}$ such that, assuming $\delta$ is small enough, we have for $j=1,2$
\[ d^{\XX^\infty_t}_{W_1} ( \delta_{x_j}, \nu^\infty_{x'_j; t}) \leq \frac{\eps}{2}, \qquad
d^\infty_{t'} (x'_1, x'_2) \leq d^\infty_t (x_1, x_2) + \frac{\eps}{2}. \]

Next, choose points $x'_{j, i} \in \XX^i_{t'}$, $i \in \IN$, $j = 1,2$, that strictly converge to $x'_j$ within $\CF$.
Then
\[ \lim_{i \to \infty} d^i_{t'} (x'_{1,i}, x'_{2,i}) = d^\infty_{t'} (x'_1, x'_2) \leq d^\infty_t (x_1, x_2) + \frac{\eps}{2}. \]
Moreover, by Theorem~\ref{Thm_tdmu_convergence} (see also Theorem~\ref{Thm_tdmu_convergence_reverse}) we have strict convergence of $\nu^i_{x'_{j,i}; t}$ to $\nu^\infty_{x'_j; t}$ within $\CF$ for $j=1,2$.
Since all $\XX^i$ are intrinsic at time $t'$, we can find points $y_i \in \XX^i_{t'}$ such that for $j =1,2$
\[ \lim_{i \to \infty} d^i_{t'} (x'_{j,i}, y_i) = \tfrac12 d^\infty_{t'} (x'_1, x'_2) \leq \tfrac12 d^\infty_t (x_1, x_2) + \frac{\eps}{4}. \]
By Theorem~\ref{Thm_compactness_sequence_pts} we may pass to a subsequence and find a conjugate heat flow $(\td\mu_{t''})_{t'' \in I^{\prime,\infty} \cap (-\infty,t')}$ on $\XX^\infty$ such that
\[ (\nu^i_{y_i;t''})_{t'' \in I^{\prime,i} \cap (-\infty,t')} \xrightarrow[i \to \infty]{\quad \IF, \CF, \{ t \} \quad} (\td\mu_{t''})_{t'' \in I^{\prime,\infty} \cap (-\infty,t')} \]
and
\[ \lim_{t'' \nearrow t', t'' \in I^{\prime,\infty}} \Var (\td\mu_{t''}) = 0. \]
This implies that we have strict convergence of $\nu^i_{y_i;t}$ to $\td\mu_t$ within $\CF$ and 
\[ \Var(\td\mu_t) \leq H (t'-t) \leq H \delta.\]
It follows that for $j =1,2$
\[ d^{\XX^\infty_t}_{W_1} (\nu^\infty_{x'_{j};t}, \td\mu_t ) 
= \lim_{i \to \infty} d^{\XX^i_t}_{W_1} (\nu^i_{x'_{j,i};t}, \nu^i_{y_i;t} )
\leq \liminf_{i  \to \infty} d^i_{t'} (x'_{j,i}, y_i)
\leq \tfrac12 d^\infty_t (x_1, x_2) + \frac{\eps}{4}. \]
Choose $z \in \XX^\infty$ such that $\Var(\td\mu_t, \delta_z) \leq H \delta$.
Then for small enough $\delta$ we have for $j=1,2$
\begin{multline*}
 d^\infty_t (x_j, z)
\leq d^{\XX^\infty_t}_{W_1} (\delta_{x_j}, \nu^\infty_{x'_j;t}) + d^{\XX^\infty_t}_{W_1} (\nu^\infty_{x'_j;t}, \td\mu_t) + d^{\XX^\infty_t}_{W_1} (\td\mu_t, \delta_z) \\
\leq \frac{\eps}{2} + \tfrac12 d^\infty_t (x_1, x_2) + \frac{\eps}{4} + \sqrt{H\delta}
\leq \tfrac12 d^\infty_t (x_1, x_2) + \eps, 
\end{multline*}
proving (\ref{eq_approx_midpoint}).
\end{proof}
\bigskip

\section{Regular points and smooth convergence}\label{sec_reg_pts}
In this section we analyze the case in which a metric flow $\XX$ can be locally described by a smooth Ricci flow on some open subset $\RR \subset \XX$, which we will call its \emph{regular part}.
The subset $\RR$ can be equipped with a unique structure of a Ricci flow spacetime, as introduced by Kleiner and Lott \cite{Kleiner_Lott_singular}.
In the special case in which $\XX$ is given by a classical, smooth Ricci flow $(M, (g_t)_{t \in I})$ over a left-open time-interval $I$, we have $\RR = \XX$ and $\RR$ corresponds to the Ricci flow spacetime induced by $(M, (g_t)_{t \in I})$.

This section is structured as follows.
We will first review the basic notions involving Ricci flow spacetimes in Subsection~\ref{subsec_RF_spacetimes}.
Then we will introduce the regular part $\RR$ and prove the existence of a Ricci flow spacetime structure on $\RR$ in Subsection~\ref{subsec_reg_part}.
In Subsection~\ref{subsec_properties_reg_part}, we will discuss further properties of the regular part.
In Subsection~\ref{subsec_smooth_convergence} we will consider a sequence of $\IF$-convergent metric flows.
We will see that the $\IF$-convergence can be upgraded to smooth convergence in certain regions of the regular part of the limit.
This notion is similar to smooth Cheeger-Gromov convergence.
In Subsection~\ref{subsec_conv_parab_nbhd_reg_part} we discuss how parabolic neighborhoods on which the curvature is bounded pass to the limit and discuss one peculiar behavior. 

In the following, we will mainly be interested in metric flows that are $H$-concentrated for some $H < \infty$ and almost always intrinsic.
Note that since metric flows corresponding to smooth Ricci flows fall into this category, this case will be of most interest for us.
We will moreover often restrict to metric flows that are defined on left-open time-intervals $I$, because in this case the natural topology of a metric flow $\XX \cong M \times I$ corresponding to a Ricci flow $(M, (g_t)_{t \in I})$ agrees with the topology on $M \times I$.

\subsection{Ricci flow spacetimes} \label{subsec_RF_spacetimes}
In this subsection we recall the notion of a Ricci flow spacetime and associated terminology.
The following definitions are mainly taken out of \cite{Kleiner_Lott_singular, bamler_kleiner_uniqueness_stability}, with minor modifications; the familiar reader may skip this subsection.

We first define the notion of a Ricci flow spacetime.

\begin{Definition}[Ricci flow spacetime] \label{def_RF_spacetime}
A {\bf Ricci flow spacetime over an interval $I \subset \IR$}  is a tuple $(\MM, \lb \mathfrak{t}, \lb \partial_{\mathfrak{t}}, \lb g)$ with the following properties:
\begin{enumerate}[label=(\arabic*)]
\item $\MM$ is a disjoint union of smooth manifolds (of possibly different dimensions) with (smooth) boundary $\partial \MM$
\item $\mathfrak{t} : \MM \to I$ is a smooth function without critical points (called {\bf time function}).
For any $t \in I$ we denote by $\MM_t := \mathfrak{t}^{-1} (t) \subset \MM$ the {\bf time-$t$-slice} of $\MM$.
\item $\tf (\partial \MM) \subset \partial I$.
\item $\partial_{\mathfrak{t}}$ is a smooth vector field (the {\bf time vector field}) on $\MM$ that satisfies $\partial_{\mathfrak{t}} \mathfrak{t} \equiv 1$.
\item $g$ is a smooth inner product on the spatial subbundle $\ker (d \mathfrak{t} ) \subset T \MM$.
For any $t \in I$ we denote by $g_t$ the restriction of $g$ to the time-$t$-slice $\MM_t$ (note that $g_t$ is a Riemannian metric on $\MM_t$).
\item $g$ satisfies the Ricci flow equation: $\mathcal{L}_{\partial_\mathfrak{t}} g = - 2 \Ric (g)$.
Here $\Ric (g)$ denotes the symmetric $(0,2)$-tensor on $\ker (d \mathfrak{t} )$ that restricts to the Ricci tensor of $(\MM_t, g_t)$ for all $t \in I$.
\end{enumerate}
For any subset $I' \subset I$ the preimage $\MM_{I'} = \mathfrak{t}^{-1} (I')$ is called a {\bf time-slab} of $\MM$ and we sometimes write $\MM_{<t} := \MM_{I \cap (-\infty,t)}$, $\MM_{\leq t} := \MM_{I \cap (-\infty, t]}$ etc.
Curvature quantities on $\MM$, such as the Riemannian curvature tensor $\Rm$, the Ricci curvature $\Ric$, or the scalar curvature $R$ will refer to the corresponding quantities with respect to the metric $g_t$ on each time-slice.
Tensorial quantities will be imbedded using the splitting $T\MM = \ker (d\mathfrak{t} ) \oplus \langle \partial_{\mathfrak{t}} \rangle$.

When there is no chance of confusion, we will often abbreviate the tuple $(\MM, \mathfrak{t}, \partial_{\mathfrak{t}}, g)$ by $\MM$.
The objects $\partial_{\tf}, g$ and sometimes also $\tf$ will inherit the decorations of $\MM$, similarly as explained in Definition~\ref{Def_metric_flow}.
\end{Definition}

Any (conventional) Ricci flow of the form $(M,(g_t)_{t \in I})$ can be converted into a Ricci flow spacetime over $I$ by setting $\MM = M \times I$, letting $\mathfrak{t}$ be the projection to the second factor and letting $\partial_{\mathfrak{t}}$ correspond to the unit vector field on $I$.
Vice versa, if $(\MM, \mathfrak{t}, \partial_{\mathfrak{t}}, g)$ is a Ricci flow spacetime over $I$ and the property that every trajectory of $\partial_{\mathfrak{t}}$ is defined on the entire time-interval $I$ (i.e. $\MM$ is a product domain, see Definition~\ref{def_product_domain}), then $\MM$ comes from such a conventional Ricci flow.

If $(\MM, \mathfrak{t}, \partial_{\mathfrak{t}}, g)$ is a Ricci flow spacetime and $U \subset \MM$ is an open subset, then $(U, \mathfrak{t}|_U, \partial_{\mathfrak{t}}|_U, g|_U)$ is again a Ricci flow spacetime.

We now define some basic geometric notions for Ricci flow spacetimes.
Let in the following $(\MM, \mathfrak{t}, \partial_{\mathfrak{t}}, g)$ be a Ricci flow spacetime over some interval $I$.

\begin{Definition}[Length, distance and metric balls in Ricci flow spacetimes]
For any two points $x, y \in \MM_t$ in the same time-slice of $\MM$ we denote by $d_{g_t} (x,y)$ the {\bf distance} between $x, y$ within $(\MM_t, g_t)$.
The distance between points in different time-slices is not defined.
For any $r \geq 0$ we define the {\bf distance ball} $B_{g_t} (x,r) := \{ y \in \MM_t \;\; : \;\; d_{g_t}(x,y) < r \}  \subset \MM_t$.

Similarly, we define the {\bf length} $\length (\gamma)$ or $\length_t (\gamma)$ of a path $\gamma : [0,1] \to \MM_t$ whose image lies in a single time-slice to be the length of this path when viewed as a path inside the Riemannian manifold $(\MM_t, g_t)$.
\end{Definition}

\begin{Definition}[Points in Ricci flow spacetimes] \label{def_points_in_RF_spacetimes}
Let $x \in \MM$ be a point and set $t := \mathfrak{t} (x)$.
Consider the maximal trajectory $\gamma_x : I' \to \MM$, $I' \subset I$ of the time-vector field $\partial_{\mathfrak{t}}$ such that $\gamma_x (t) = x$.
Note that then $\mathfrak{t} (\gamma_x(t')) = t'$ for all $t' \in I'$.
For any $t' \in I'$ we say that $x$ \textbf{survives until time $t'$} and we write 
\[ x(t') := \gamma_x (t'). \]
Similarly, if $S \subset \MM_t$ is a subset in the time-$t$ time-slice, then we say that $S$ \textbf{survives until time $t'$} if this is true for every $x \in S$ and we set $S(t') := \{ x(t') \;\; : \;\; x \in S \}$.
\end{Definition}

\begin{Definition}[Time-slices/slabs of a subset]
If $S \subset \MM$ is a subset and $t \in I$, then we set $S_t := S \cap \MM_t$.
For any subset $I' \subset I$ we write $S_{I'} := S \cap \MM_{I'}$.
\end{Definition}

\begin{Definition}[Product domain]
\label{def_product_domain}
We call a subset $S \subset \MM$ a {\bf product domain over an interval $I' \subset I$} if for any $t \in I'$ any point $x \in S$ survives until time $t$ and $x(t) \in S$.
\end{Definition}

Note that a product domain $S$ over $I'$ can be identified with the product $S_{t_0} \times I'$ for an arbitrary $t_0 \in I'$.
If $S_{t_0}$ is sufficiently regular (e.g. open or a domain with smooth boundary in $\MM_{t_0}$), then the metric $g$ induces a classical Ricci flow $(g_t)_{t \in I'}$ on $S_{t_0}$.
We will often use the metric $g$ and the Ricci flow $(g_t)_{t \in I'}$ synonymously when our analysis is restricted to a product domain.

\begin{Definition}[Parabolic neighborhood]
For any point $y \in \MM$ let $I'_y \subset I$ be the set of all times until which $y$ survives.
Now consider a point $x \in \MM$ and numbers $A, T^-, T^+ \geq 0$.
Set $t := \mathfrak{t} (x)$.
Then we define the \textbf{parabolic neighborhood} $P(x; A, -T^-, T^+) \subset \MM$ as follows:
\[ P(x; A, -T^-, T^+) := \bigcup_{y \in B_{g_t} (x,A)} \bigcup_{t' \in [t - T^-, t+T^+] \cap I'_y} y(t'). \]
We call $P(x; A, -T^-, T^+)$ \textbf{unscathed} if $B_{g_t} (x,A)$ is relatively compact in $\MM_t$ and if $I_y \supset [t - T^-, t+T^+] \cap I$ for all $y \in B_{g_t} (x,A)$.
If $T^- = 0$ or $T^+ = 0$, then we will often write $P(x; A, T^+)$ or $P(x; A, -T^-)$ instead of $P(x; A, -T^-, T^+)$.
For any $r \geq 0$ we define the {\bf parabolic ball}
\[ P(x;r) := P(x;r,-r^2, r^2) \]
and the {\bf backward ($-$) and forward ($+$) parabolic balls}
\[ P^-(x,r) := P(x;r,-r^2), \qquad P^+(x;r) := P(x;r,r^2). \]
\end{Definition}

Note that if $P(x; A, -T^-, T^+)$ is unscathed, then it is a product domain of the form $B_{g_t} (x,A)  \times ([\tf(x) - T^-, \tf(x)+T^+] \cap I)$.
We emphasize that $P(x; A, -T^-, T^+)$ can be unscathed even if $[\tf(x) - T^-, \tf(x)+T^+] \not\subset I$, that is when it hits the initial/final time-slice earlier than expected.

Next, we consider maps $\phi : U \to \MM'$ between a subset $U \subset \MM$ of a Ricci flow spacetime $(\MM,\tf,\partial_\tf,g)$ over an interval $I$ and a Ricci flow spacetime $(\MM',\tf',\partial_{\tf'},g')$ over an interval $I'$.

\begin{Definition}[Time-preserving and time-equivariant maps]
We say that $\phi$ is \textbf{time-preserving} if $\mathfrak{t}' (\phi (x)) = \mathfrak{t} (x)$ for all $x \in U$.
\end{Definition}

\begin{Definition}[Time-slices of a map]
If $\phi :  \MM \supset U \to \MM'$ is time-preserving and $t \in I$, then we denote by
\[ \phi_t := \phi |_{U_t} : U_t \longrightarrow \MM'_{t} \subset \MM' \]
the \textbf{time-$t$-slice} of $\phi$.
\end{Definition}

\begin{Definition}[$\partial_{\t}$-preserving maps]
\label{def_d_t_preserving}
Suppose that $\phi : U \to \MM'$ is a differentiable map defined on a sufficiently regular domain $U \subset \MM$.
If $\phi_* \partial_{\tf} = \partial'_{\tf}$, then we say that $\phi$ is \textbf{$\partial_{\tf}$-preserving.}
\end{Definition}

Lastly, consider a function $u \in C^2 (\MM_{I'})$ on a time-slap of a Ricci flow spacetime $(\MM,\tf,\partial_\tf,g)$, where $I' \subset I$ is a non-trivial interval.

\begin{Definition}[(Conjugate) Heat operator]
We define $\square u, \square^* u \in C^0 (\MM_{I'})$ by
\[ \square u_t := (\partial_{\tf} - \triangle_{g_t})u_t, 
\qquad \square^* u_t := (-\partial_{\tf} - \triangle_{g_t} + R_{g_t} )u_t. \]
If $\square u = 0$ or $\square^* u = 0$ then $u$ is said to satisfy the {\bf heat equation} or {\bf conjugate heat equation,} respectively.
\end{Definition}
\bigskip

\subsection{The regular part of a metric flow} \label{subsec_reg_part}
Let $\XX$ be a metric flow over some left-open interval $I \subset \IR$.
We will now introduce the notion of regular points.

\begin{Definition}[Regular points] \label{Def_regular_pt}
A point $x \in \XX$ is called {\bf regular} if there is a manifold $M'$, a subinterval $I' \subset I$ that is a neighborhood of $\tf (x)$ in $I$, a Ricci flow $(M', (g'_t)_{t \in I'})$ and map $\phi : M' \times I' \to \XX$ such that the following holds:
\begin{enumerate}[label=(\arabic*)]
\item \label{Def_regular_pt_1} $x \in \phi (M' \times I')$.
\item \label{Def_regular_pt_2} $\phi$ is a homeomorphism onto its image and the image $\phi (M' \times I')$ is open in $\XX$.
Here we consider the natural topology on $\XX$ (see Subsection~\ref{subsec_natural_topology}).
\item \label{Def_regular_pt_3} For any $t \in I'$ we have $\phi_t (M') := \phi (M' \times \{ t \}) \subset \XX_t$ and $\phi_t : (M', d_{g'_t}) \to (\XX_t, d_t)$ is a local isometry.
\item \label{Def_regular_pt_4} For any uniformly bounded heat flow $(u_t)_{t \in \td I}$ on $\XX$ over a left-open subinterval $\td I \subset I'$, the family of functions $(u'_t := u_t \circ \phi_t)_{t \in \td I}$ is a smooth solution to the heat equation $\square  u' = (\partial_t - \triangle_{g'_t}) u' = 0$ on $M'$ with background metric $g'_t$.
\item \label{Def_regular_pt_5} For any conjugate heat flow $(\mu_t)_{t \in \td I}$ on $\XX$ over a right-open subinterval $\td I \subset I'$, we have $\phi_t^* \mu_t = v'_t \, dg'_t$ for all $t \in \td I$, where $v' \in C^\infty (M' \times \td I)$ is a smooth solution to the conjugate heat equation $\square^* v' = (-\partial_t - \triangle_{g'_t} + R) v' = 0$  with background metric $g'_t$.
\end{enumerate}
We define the {\bf local dimension at $x$} as follows: $\dim x := \dim M'$.
Note that due to Property~\ref{Def_regular_pt_3}, $\dim x$ is well defined.
\end{Definition}

The following is the main result of this subsection.
In short, it states that the set of regular points $\RR$ is a Ricci flow spacetime and it establishes some basic properties of $\RR$.

\begin{Theorem} \label{Thm_regular_part_XX}
Let $\RR \subset \XX$ be the set of regular points of $\XX$.
Then $\RR$ is open with respect to the natural topology on $\XX$ and there is a smooth structure on $\RR$ that is compatible with the subspace topology and with respect to which $\tf |_{\RR}$ is smooth.
Moreover, there is a smooth vector field $\partial_\tf$ on $\RR$ and a smooth metric $g$ on $\ker d\tf |_{\RR}$ such that:
\begin{enumerate}[label=(\alph*)]
\item \label{Thm_regular_part_XX_a} $(\RR, \tf, \partial_{\tf}, g)$ is a smooth Ricci flow spacetime over $I$.
\item \label{Thm_regular_part_XX_b} For any $t \in I$ the length metric $d_{g_t}$ of $g_t$ is locally equal to the restriction of $d_t$ to $\RR_t = \XX_t \cap \RR$.
I.e. for any $x \in \RR_t$ there is a neighborhood $x \in U \subset \RR_t$ such that $d_{g_t} |_{U \times U} = d_t |_{U \times U}$.
\item \label{Thm_regular_part_XX_c} For any uniformly bounded heat flow $u : \XX_{I'} \to \IR$ on $\XX$ over any left-open subinterval $I' \subset I$ the following is true:
$u|_{\RR_{I'}}$ is smooth and we have $\square u = (\partial_{\tf} - \triangle ) u = 0$ on $\RR_{I'}$.
\item \label{Thm_regular_part_XX_d} For any conjugate heat flow $(\mu_t)_{t \in I'}$ on $\XX$ over any right-open subinterval $I' \subset I$ the following is true:
We have $d\mu_t = v_t \, dg$ on $\RR_t$ for all $t \in I'$, where $v \in C^\infty(\RR_{I'})$ satisfies the conjugate heat equation $\square^* v = (-\partial_{\tf} - \triangle + R ) v = 0$ on $\RR_{I'}$.
\item \label{Thm_regular_part_XX_e} Subsets of $\RR$ that are compact with respect to the subspace topology on $\RR$ are closed in $\XX$.
\item  \label{Thm_regular_part_XX_f} There is a continuous function
\[ K : \{ (x;y) \in \XX \times \RR \;\; : \;\; \tf(x) > \tf(y) \} \longrightarrow \IR_+ \]
such that for any $s,t \in I$, $s < t$, $x \in \XX_t$ we have $d\nu_{x;s} = K(x;\cdot) dg_s$ on $\RR_s$.
Moreover, for any $x \in \XX_t$, the function $K(x; \cdot) : \RR_{< t} \to \IR_+$ is smooth and the map $x \mapsto K(x; \cdot)$ is continuous in the $C^\infty_{\loc}$-topology.
$K$ restricted to $\{ (x;y) \in \RR \times \RR \;\; : \;\; \tf(x) > \tf(y) \}$ is smooth.
For any $y \in \RR_s$ the function $K(\cdot; y)$ is a heat flow on $\XX_{>s}$ and satisfies $\square K(\cdot; y) = 0$ on $\RR_{>s}$.
\end{enumerate}
Assertions~\ref{Thm_regular_part_XX_a}--\ref{Thm_regular_part_XX_c} uniquely determine the smooth structure on $\RR$, as well as the objects $\partial_{\tf}$, $g$.
\end{Theorem}

We can therefore define:

\begin{Definition}[Regular part of a metric flow]
If $\XX$ is a metric flow, then the set of regular points $\RR \subset \XX$ is called the {\bf regular part of $\XX$.}
Moreover, we denote by $\partial_{\tf}, g$ the vector field and metric on $\RR$ that satisfy Assertions~\ref{Thm_regular_part_XX_a}--\ref{Thm_regular_part_XX_f} of Theorem~\ref{Thm_regular_part_XX}.
We denote by $\mathcal{S} := \XX \setminus \RR$ the {\bf singular part} of $\XX$.
\end{Definition}

\begin{Definition}[Heat kernel on a metric flow]
If $\XX$ is a metric flow, then the function $K$ from Assertion~\ref{Thm_regular_part_XX_f} in Theorem~\ref{Thm_regular_part_XX} is called the {\bf heat kernel} on $\XX$.
\end{Definition}

As before, the objects $\RR, \partial_{\tf}, g,K$ will inherit the decorations of $\XX$.
So, for example, the regular part of a metric flow $\XX^{\prime, i}$ will be denoted by $\RR^{\prime, i}$, which is be equipped with $\partial^{\prime, i}_{\tf}, g^{\prime, i}$ and a heat kernel $K^{\prime, i}$.

We will need the following lemma for the proof of Theorem~\ref{Thm_regular_part_XX} and for the remainder of this section.
We will often refer to this lemma as ``standard parabolic estimates''.

\begin{Lemma} \label{Lem_std_par_est}
For any $n,m \in \IN$, $\alpha > 0$ there is constant $C_m (n, \alpha) < \infty$ such that the following holds.
Let $\MM$ be a Ricci flow spacetime whose time-slices have dimension $n$, $x \in \MM_t$ be a point and $r > 0$ a scale such that the ball $B(x,r) \subset \MM_t$ is relatively compact and has volume $|B(x,r)|_{g_t} \geq \alpha r^n$.
Let moreover $\tau \in (0, r^2]$ and $m_1, m_2 \geq 0$.
\begin{enumerate}[label=(\alph*)]
\item \label{Lem_std_par_est_a} Suppose that the forward parabolic neighborhood $P^+ := P(x;r,\tau)\subset \MM$ is unscathed and that $|{\Rm}| \leq r^{-2}$ on $P^+$.
Consider a smooth function $v \in C^\infty (P^+)$ satisfying the conjugate heat equation $\square^* v = 0$ and the bound $\int_{P^+_{t'}} |v_{t'}| dg_{t'} \leq 1$ for all $t' \in [t, t+r^2]$.
Then on $P(x;r/2, r^2/4)$ we have
\[ |\nabla^{m_1} \partial^{m_2}_{\tf} v| \leq C_{m_1 + 2m_2} r^{-n-m_1-2m_2} + C_{m_1 + 2m_2}  \sum_{m=0}^{m_1+2m_2} \sup_{P^-_{t+\tau}} r^{m-m_1-2m_2}|\nabla^m v_{t + \tau}|. \]
If $\tau = r^2$, then the last term can be omitted.
\item \label{Lem_std_par_est_b} Suppose that the backward parabolic neighborhood $P^- := P(x;r,-\tau)\subset \MM$ is unscathed and that $|{\Rm}| \leq r^{-2}$ on $P^-$.
Consider a smooth function $u \in C^\infty (P^-)$ satisfying the  heat equation $\square u = 0$ and the bound $|u| \leq A < \infty$.
Then on $P(x;r/2, - r^2/4)$ we have
\[ |\nabla^{m_1} \partial^{m_2}_{\tf} u| \leq C_{m_1 + 2m_2} A r^{-m_1-2m_2} + C_{m_1 + 2m_2}  \sum_{m=0}^{m_1+2m_2} \sup_{P^+_{t-\tau}} r^{m-m_1-2m_2}|\nabla^m u_{t - \tau}|. \]
If $\tau = r^2$, then the last term can be omitted.
\end{enumerate}
\end{Lemma}

\begin{proof}
Suppose first that in Assertion~\ref{Lem_std_par_est_a} we have $|v| \leq A$ on $P^+ \cap P^+ (x; .9 r)$.
Then both Assertion~\ref{Lem_std_par_est_b} follows using standard parabolic regularity theory (see for example \cite{Krylov_Hoelder}) and in Assertion~\ref{Lem_std_par_est_a} we get
\[ |\nabla^{m_1} \partial^{m_2}_{\tf} v| \leq C_{m_1 + 2m_2} A r^{-m_1-2m_2} + C_{m_1 + 2m_2}  \sum_{m=0}^{m_1+2m_2} \sup_{P^+_{t+\tau} } r^{m-m_1-2m_2}|\nabla^m v_{t + \tau}|, \]
where the last term can be omitted if $\tau = r^2$.

It remains to establish a $C^0$-bound on $v$ over $P^+ \cap P^+ (x; .9 r)$ in Assertion~\ref{Lem_std_par_est_a}.
For this purpose, suppose that $\MM = P^+ = M \times [t, t+\tau]$ and define $b : P^+ \to [0,1]$ by 
\[ b(x',t') := \big(.9 - r^{-1}d_t (x, x') \big)_+ r^{-2} (t+r^2 - t'). \]
We claim that there is a constant $C^* (n,\alpha) < \infty$ such that we have on $P^+ \cap P^+ (x,t; .9 r)$
\[ |v| \leq C^* \max \big\{ b^{-n}, \sup_{M}|v|(\cdot, t+\tau) \big\}, \]
where the supremum can be omitted if $\tau = r^2$.
To see this, consider a point $(x',t') \in M \times [t,t+\tau)$, where $|v| /  \max \big\{ b^{-n}, \sup_{M} |v|(\cdot, t+\tau) \big\}$ attains its maximum $Z$ and set
\[ r' :=  b(x',t'), \qquad r''  := \min \big\{ r', \sqrt{t+\tau -t'} \big\}. \]
Then $|v| \leq 10 |v|(x',t') \cap Z(r')^{-n}$ on $P^+(x',t'; 2cr') \cap P^+$ for some universal constant $c > 0$.
By our previously established derivative estimates, this implies that $|\nabla v | \leq C'(n,\alpha) r^{\prime\prime -1} |v|(x', t')$ on $P^+(x',t';  c r'')$.
It follows that for some $c'(n, \alpha) > 0$ we have $|v| \geq \frac12 |v|(x',t')$ on $B(x', c'r'')$.
On the other hand, we obtain for some $c''(n, \alpha) > 0$
\[ 1 \geq \int_{M} |v|(\cdot, t') dg_{t'} \geq \tfrac12 |v|(x',t') |B(x',c'r'')|_{g'_{t'}} 
\geq c'' |v|(x',t') (r'')^n. \]
Now let $0 < \beta < 1$ be a constant whose value we will determine later.
If $r'' \geq \beta r' = \beta b(x',t')$, then this implies that $1 \geq c'' \beta^n (|v| b^n)(x',t') \geq c''\beta^n  Z$, which implies an upper bound on $Z$ in terms of $n, \alpha$.
Next, assume that $r'' < \beta r'$, which implies $r'' = \sqrt{t+\tau-t'}$.
Let $\td v := v / v(x',t')$.
Then
\[ \td v(x',t') = 1, \qquad |\td v| \leq 1 \quad \text{on} \quad P(x', t' ; \beta^{-1} r'', (r'')^2), \qquad \] 
\[ | \td v | (\cdot, t' + (r'')^2) 
\leq \frac{\sup_M |v|(\cdot, t+\tau)}{v(x',t')}
\leq Z^{-1} \]
By a standard limit argument (after parabolic rescaling by $(r'')^{-2}$), we obtain that if $\beta \leq \ov\beta$, then $Z \leq \ov{Z} (n, \alpha)$.
\end{proof}
\bigskip

\begin{proof}[Proof of Theorem~\ref{Thm_regular_part_XX}.]
The proof relies on the following claim:

\begin{Claim}
Consider two manifolds $M'_i$, subintervals $I'_i \subset I$, Ricci flows $(M'_i, (g'_{i,t})_{t \in I'_i})$ and maps $\phi_i : M'_i \times I'_i \to \XX$, $i=1,2$, that each satisfy Properties~\ref{Def_regular_pt_2}--\ref{Def_regular_pt_4} of Definition~\ref{Def_regular_pt} and whose images intersect in a subset $U := \phi_1 (M'_1 \times I'_1) \cap \phi_2 (M'_2 \times I'_2) \subset \XX$.
Then $\chi := \phi_2^{-1} \circ \phi_1 |_{\phi^{-1}_1 (U)} : \phi_1^{-1} (U) \to \phi_2^{-1} (U)$ is smooth and $\partial_{\tf}$-preserving when viewed as a map between Ricci flow spacetimes.
\end{Claim}

\begin{proof}
Due to Property~\ref{Def_regular_pt_2} of Definition~\ref{Def_regular_pt}, the map $\chi$ is a homeomorphism and by Property~\ref{Def_regular_pt_3}, for any $(x_0, t_0) \in M'_1 \times I'_1$ there is a product neighborhood $(x_0, t_0) \in M''_1 \times I''_1 \subset \phi_1^{-1} (U)$ such that $\chi'' := \chi |_{M''_1 \times I''_1}$ can be expressed as a family of isometries onto their images $(\chi''_t : (M''_1, g''_{1,t}|_{M''_1} ) \to (M'_2 , g'_{2,t}))_{t \in I''_1}$.
So $\chi''_t$ is smooth for all $t \in I''_1$.
It remains to show that $\chi''_t$ is constant in $t$.
To do this, we may assume in the following that $(M'_1, (g'_{1,t})_{t \in I'_1}) = (M''_1, (g''_{1,t})_{t \in I''_1})$ and $\chi = \chi''$.

Let $\td t \in I'_1$ and consider a compactly supported, smooth function $\td u  \in C^\infty_c(M'_2)$.
Let $u : \XX_{\geq \td t} \to \IR$ be the heat flow with initial condition $u_{\td t} = \td u \circ \phi_{2 , \td t}^{-1}$ on $\phi_{2, \td t} ( M'_2 ) \subset \XX_{\td t}$ and $0$ on $\XX_{\td t} \setminus \phi_{2, \td t} ( M'_2 )$.
By Property~\ref{Def_regular_pt_4} of Definition~\ref{Def_regular_pt} the functions $u_{i} := u \circ \phi_{i} : M'_i \times (I'_i \cap  [\td t, \infty)) \to \IR$, $i = 1,2$, are smooth solutions to the heat equation when restricted to the interior of their domains.
By Proposition~\ref{Prop_topology_properties}\ref{Prop_topology_properties_f_new} and Property~\ref{Def_regular_pt_2} they are also continuous on their entire domain.
Moreover, since $\chi_{\td t}$ is smooth, we also know that $u_{1, \td t} = \td u = u_{2, \td t} \circ \chi_{\td t}$ is smooth.
So by local parabolic regularity theory (see also Lemma~\ref{Lem_std_par_est}), $u_1, u_2$ are smooth solutions to the heat equation on their entire domains.
Therefore, in summary, for any $\td t \in I'_1$ and $\td u  \in C^\infty_c(M'_2)$ there are smooth solutions to the heat equation $u_i \in C^\infty (M'_i \times ( I'_i  \cap  [\td t, \infty)))$ such that $u_2 = \td u$ on $M'_2 \times \{ \td t \}$ and $u_1 = u_2 \circ \chi$ on $M'_1 \times (I'_1 \cap  [\td t, \infty))$.
It follows that on $M'_1 \times (I'_1 \cap  [\td t, \infty))$
\[ \partial_t (u_{2,t} \circ \chi_t) = \partial_t u_{1,t} 
= \triangle_{g_{1,t}} u_{1,t}
= (\triangle_{g_{2,t}} u_{2,t}) \circ \chi_t
= (\partial_t u_{2,t}) \circ \chi_t. \]
Applying this identity to the heat flows corresponding to $n  := \dim M'_1 = \dim M'_2$ functions $\td u_1, \ldots, \td u_n \in C^\infty_c ( M''_1)$ that form a coordinate system near a given point implies that for any $(\td x, \td t) \in M'_1 \times I'_1$ there is a $\td t' > \td t$ such that $[\td t, \td t') \to M'_2$, $t \mapsto \chi_t (\td x, t )$ is constant.
By continuity, this implies that $\chi_t$ is constant in $t$, which finishes the proof of the claim.
\end{proof}

Due to the Claim, the inverses of the maps $\phi$ from Definition~\ref{Def_regular_pt} form a smooth atlas on $\RR$ such that $\tf |_{\RR}$ is smooth.
Moreover, the push-forwards via $\phi$ of the vector fields corresponding to the unit vector fields on the intervals $I'$ define a smooth vector field $\partial_{\tf}$ on $\RR$.
Similarly, the push-forwards of the flows $(g'_t)$ via the maps $\phi$ define a smooth metric $g$ on $\ker d\tf |_{\RR}$.
To see that $\RR$ is Hausdorff, note first that any two points in different time-slices can be separated by open subsets since $\tf$ is continuous.
On the other hand, for any $t \in I$, Property~\ref{Def_regular_pt_3} of Definition~\ref{Def_regular_pt} implies that the subspace topology on $\RR_t =  \RR \cap \XX_t \subset \XX$ agrees with the topology induced by $d_t |_{\RR_t}$.
So points in the same time-slice can be separated as well.
To see that $\RR$ is second countable, fix a countable dense subset $S \subset \XX$ using Proposition~\ref{Prop_topology_properties}\ref{Prop_topology_properties_g} and consider the collection of all maps $\phi : M' \times I' \to \RR$ from Definition~\ref{Def_regular_pt} with the additional property that the endpoints of $I'$ are rational or lie in $\partial I$ and for some $t \in I'$ the image $\phi_t (M') \subset \RR_t$ is a distance ball of rational radius around a point in $S$.
Note that any two such maps $\phi_i : M'_i \times I'_i \to \RR$ with $I'_1 = I'_2$ and with the property that $\phi_t (M'_1) = \phi_t (M'_2)$ for some $t \in I'_1$ we have $\phi_1 (M'_1 \times I'_1) = \phi_2 (M'_2 \times I'_2)$ by the Claim.
So the collection of images of these maps is countable and it can be seen using the Claim that they cover $\RR$.
This finishes the proof of Assertion~\ref{Thm_regular_part_XX_a}.

Assertions~\ref{Thm_regular_part_XX_b}--\ref{Thm_regular_part_XX_d} are direct consequences of Definition~\ref{Def_regular_pt}.
Assertion~\ref{Thm_regular_part_XX_e} follows from the openness of images in Definition~\ref{Def_regular_pt}\ref{Def_regular_pt_2} via a covering argument. 
The existence of the function $K$ and the smoothness of $K(x; \cdot)$ for any $x \in \XX$ in Assertion~\ref{Thm_regular_part_XX_f} is a consequence of Assertion~\ref{Thm_regular_part_XX_d}.
By Proposition~\ref{Prop_topology_properties}\ref{Prop_topology_properties_c}, for any sequence $x_i \to x_\infty \in \XX$ and $t < \tf(x_\infty)$ we have $K(x_i; \cdot) dg_t \to K(x_\infty; \cdot) dg_t$ in the $W_1$-sense.
Since by standard local derivative estimates, the functions  $K(x_i; \cdot)$ are locally uniformly bounded in any $C^m$-norm on $\RR$ (see also Lemma~\ref{Lem_std_par_est}), this implies that $K(x_i; \cdot) \to K(x_\infty; \cdot)$ in $C^\infty_{\loc}$ and that $K$ is continuous.
Next, fix some $y \in \RR_s$.
Then for any small $r > 0$ and any two times $t_1, t_2 \in I$, $s < t_1 < t_2$, and $x \in \XX_{t_2}$ we have
\begin{multline*}
 \int_{B(y,r) \subset \RR_{s}} K(x; y') dg_s (y')
= \nu_{x;s} (B(y,r))
= \int_{\XX_{t_1}} \nu_{x';s} (B(y,r)) \, d\nu_{x;t_1} (x') \\
=  \int_{\XX_{t_1}} \int_{B(y,r) \subset \RR_{s}} K(x'; y') \, dg_s (y') d\nu_{x;t_1} (x')
=  \int_{B(y,r) \subset \RR_{s}} \int_{\XX_{t_1}}  K(x'; y') \, d\nu_{x;t_1} (x') dg_s (y') . 
\end{multline*}
Letting $r \to 0$ implies that for almost all $y \in \RR_{s}$
\[ K(x; y) = \int_{\XX_{t_1}}  K(x'; y) \, d\nu_{x;t_1} (x') . \]
Since both sides are smooth in $y$ by local derivative estimates, this implies that we have equality everywhere and therefore $K(\cdot; y)$ is a heat flow.
Next, fix some $y \in \XX$ and $t^* > \tf(y)$.
By Lemma~\ref{Lem_std_par_est} applied to the conjugate heat kernels $(\nu_{x; t})$ near $y$ we have $K(\cdot;y) \leq C(y, t^*)$ on $\XX_{\geq t^*}$.
So by Assertion~\ref{Thm_regular_part_XX_c}, $K(\cdot; y)$ is smooth on $\RR_{> \tf(y)}$ for all $y \in \XX$ and solves the heat equation.
Lastly, consider the restriction $K'$ of $K$ to $\{ (x,y) \in \RR \times \RR \;\; : \;\; \tf(x) > \tf(y) \}$.
We have already shown that $K'$ is smooth in the first and second variable each.
By local derivative estimates we obtain that difference quotients in one variable converge locally uniformly in the other variable.
Since difference quotients in one variable still satisfy the heat equation or conjugate heat equation in the other variable, this local uniform convergence implies local smooth convergence.
This shows that $K'$ is smooth.

For the uniqueness statement, note that any inclusion map of a product domain on $\RR$ satisfies the properties of Definition~\ref{Def_regular_pt}.
So if we could find two different smooth structures or Ricci flow spacetime structures on $\RR$, then we could apply the Claim to two such inclusion maps coming from each structure.
\end{proof}

\subsection{Properties of the regular part} \label{subsec_properties_reg_part}
In this subsection, we establish some properties of the regular parts of $H$-concentrated, almost always intrinsic metric flows, which will become useful later.
We first state all results; the proofs can be found towards the end of this subsection.

The first result concerns the behavior of the conjugate heat kernel near regular points.
It shows that conjugate heat kernels cannot move too fast  in regions where the curvature is bounded.

\begin{Proposition} \label{Prop_HK_regular}
Consider an $H$-concentrated and almost always intrinsic metric flow $\XX$ over a left-open interval $I$ with regular part $\RR \subset \XX$.
Let $r > 0$, $t_1, t_2 \in I$, $t_1 \leq t_2$, $x \in \XX_{t_2}$ and assume that the parabolic neighborhood $P:= P(x; r, - ( t_2 -t_1))$ is unscathed and $|{\Rm}| \leq r^{-2}$ on $P$.
Then:
\begin{enumerate}[label=(\alph*)]
\item \label{Prop_HK_regular_a} If $t_2 - t_1 \leq A r^2$, then we have
\[ d^{\XX_{t_1}}_{W_1} (\delta_{x(t_1)}, \nu_{x; t_1}) \leq C(H,A) \sqrt{t_2 - t_1}. \]
\item \label{Prop_HK_regular_b} Suppose that $t_2 - t_1 \leq c(H) r^2$, and that $x(t_1) \in \XX_{t_1}$ is an $H$-center of some point $y\in \XX_{t_2}$.
Then $y \in P \subset \RR$ and
\[ d_{t_2}(x, y)   \leq C(H) \sqrt{t_2 - t_1}. \]
\end{enumerate}
\end{Proposition}

In the next proposition we study the relation between the length metric $d_{g_t}$ induced by the Riemannian metric $g_t$ on each time-slice $\RR_t$ and the restriction of the metric $d_t$ to $\RR_t \subset \XX_t$.
If $\XX$ is intrinsic at time $t$, then the following properties are trivial.
The key point of the proposition is that these properties even hold if $\XX$ is not intrinsic at time $t$.

\begin{Proposition} \label{Prop_RR_intrinsic}
Let $\XX$ be an $H$-concentrated and almost always intrinsic metric flow over a left-open interval $I$ and consider its regular part $\RR \subset \XX$.
For any $t \in I$, $x, x_1, x_2 \in \RR_t$ and $r > 0$ the following is true:
\begin{enumerate}[label=(\alph*)]
\item \label{Prop_RR_intrinsic_a} $d_t \leq d_{g_t}$. \item \label{Prop_RR_intrinsic_b} If there is a compact subset $K \subset \RR_t$ with the property that any path in $\RR_t$ of (Riemannian) length $\leq d_t (x_1, x_2)$ between $x_1, x_2$ is contained in $K$ and if $\lim_{t'\nearrow t} d_{g_{t'}} (x_1(t'), x_2(t')) = d_{g_t} (x_1, x_2)$, then $d_{g_t} (x_1, x_2) = d_t (x_1, x_2)$.
\item \label{Prop_RR_intrinsic_c} If the ball $B_{g_t} (x, r) \subset \RR_t$ around $x$ of radius $r$ with respect to the Riemannian metric $g_t$ is relatively compact in $\RR_t$, then $B_{g_t} (x, r) = B(x,r)$, where the latter ball is taken with respect to $d_t$.
Moreover $d_{g_t} (x, y) = d_t (x,y)$ for all $y \in B(x,r)$.
\end{enumerate}
\end{Proposition}

We also obtain:

\begin{Proposition} \label{Prop_H_dim_bound}
If $\XX$ is an $H$-concentrated metric flow with regular part $\RR \subset \XX$, then $\dim x \leq H/4$ for all $x \in \RR$.
\end{Proposition}

Let us now present the proofs.
The logical dependence of the following proofs is somewhat convoluted.
We first show:

\begin{Lemma} \label{Lem_dt_leq_dgt}
Assertion~\ref{Prop_RR_intrinsic_a} of Proposition~\ref{Prop_RR_intrinsic} holds.
\end{Lemma}

\begin{proof}
By Theorem~\ref{Thm_regular_part_XX}\ref{Thm_regular_part_XX_b}, the metrics $d_t$ and $d_{g_t}$ are locally equal on $\RR_t$.
Since $d_{g_t}$ is a length metric, we obtain $d_t \leq d_{g_t}$.
\end{proof}
\bigskip

\begin{Lemma} \label{Lem_HK_regular_weak}
Proposition~\ref{Prop_HK_regular} holds if $C$ is also allowed to depend on $\dim x$ and if for Assertion~\ref{Prop_HK_regular_b} we assume that $\XX$ is intrinsic at time $t_1$ or Proposition~\ref{Prop_RR_intrinsic}\ref{Prop_RR_intrinsic_c} holds.
\end{Lemma}

\begin{proof}
Let $n := \dim x$.
To prove Assertion~\ref{Prop_HK_regular_a}, we may shrink $r$ and assume that without loss of generality $t_2 - t_1 \geq r^2$.
After parabolic rescaling and application of a time-shift, we may assume that $r = 1$, $t_1 = 0$, $t_2 := T$, where $1\leq T \leq A$ for Assertion~\ref{Prop_HK_regular_a}.
Write $(\mu_t)_{t \in [0,T)} := (\nu_{x;t})_{t \in [0,T)}$ for Assertion~\ref{Prop_HK_regular_a} or $(\mu_t)_{t \in [0,T)} := (\nu_{y;t})_{t \in [0,T)}$ for Assertion~\ref{Prop_HK_regular_b}.
Then $d\mu_t = v_t \, dg_t$ on $\RR_t$ for some $v \in C^\infty( \RR_{[0, T)})$ with $\square^* v = 0$.

\begin{Claim}
There is a constant $0 < c_1 (n,A) \leq c_0 (n,A) < .1$ such that for any $x' \in B_{g_T}(x,.5)$ and $t \in [0,T)$
\[ \mu_{0} \big( B_{g_{0}} (x'(0), c_0 ) \big) \geq c_1 \mu_{t} \big( B_{g_{t}} (x'(t), c_1 ) \big),
\qquad B_{g_t} (x'(t), c_0 ) \subset P(x'; .1, -T). \]
\end{Claim}

\begin{proof}
Fix $x' \in B_{g_T}(x,.5)$ and let $P' := P(x'; .1, -T)$.
By \cite[\HKLemHEsubsolutiononParNbhd]{Bamler_HK_entropy_estimates} there is a compactly supported function $u \in C_c^0 (P')$ and constants $0 < c_1 (n,A) \leq c_0 (n,A) < .1$ with the following properties:
\begin{enumerate}
\item $\square u \leq 0$ in the barrier sense.
\item $0 \leq u \leq 1$.
\item For all $t \in [0, T]$ we have $\supp u_{t} \subset B_{g_{t}} (x'(t) , c_0 ) \subset P(x'; .1, -T)$ and $u \geq c_1$ on $B(x'(t), c_1 )$.
\end{enumerate}
We can view $u$ as a function on $\XX_{[0,T]}$ by extending it by $0$ outside of $P'$.
Then in the barrier sense
\[ \frac{d}{dt} \int_{\XX_t} u_{t} \, d\mu_t =   \frac{d}{dt} \int_{\RR_t} u_{t} v_t \, dg_t = \int_{\RR_t} \big( (\square u_{ t}) v_t - u_{t} (\square^* v_t) \big)dg_t \leq 0. \]
It follows that for any $t \in [0, T)$
\[ \mu_{0} \big( B_{g_{0}} (x'(0), c_0 ) \big)
 \geq \int_{\XX_0} u_{0} \, d\mu_0
 \geq \int_{\XX_t} u_{t} \, d\mu_t
 \geq c_1 \mu_{t} \big( B_{g_{t}} (x'(t), c_1 ) \big). \qedhere \]
\end{proof}

Using Theorem~\ref{Thm_alm_alw_intrinsic}, we can pick a sequence of times $t^*_i \nearrow T$ at which $\XX$ is intrinsic.

Let us now show Assertion~\ref{Prop_HK_regular_a}.
Choose a sequence of $H$-centers $x^*_i \in \XX_{t^*_i}$ of $x$.
By Proposition~\ref{Prop_topology_properties}\ref{Prop_topology_properties_e} we have $x^*_i \to x$, which implies that  for large $i$
\begin{equation} \label{eq_xstar_in_B_t_star}
 x^*_i \in B_{g_{t^*_i}} (x(t^*_i),c_1/2) = B (x(t^*_i),c_1/2) . 
\end{equation}
So by Lemma~\ref{Lem_nu_BA_bound}, the Claim and the fact that $d_{0} \leq d_{g_{0}}$ on $\RR_0$ we have
\begin{multline*}
 \mu_{0} \big( B (x(0), c_0) \big) \geq
 \mu_{0} \big( B_{g_{0}} (x(0), c_0) \big) \geq c_1 \mu_{t^*_i} \big( B_{g_{t^*_i}} (x(t^*_i),c_1) \big)  \\
 = c_1 \mu_{t^*_i} \big( B (x(t^*_i),c_1) \big)
 \geq  c_1 \mu_{t^*_i} \big( B (x^*_i ,c_1/2) \big) \xrightarrow[i \to \infty]{\qquad} c_1. 
\end{multline*}
So if $x^*_0 \in \XX_{0}$ is an $H$-center of $x$, then we obtain, again using Lemma~\ref{Lem_nu_BA_bound}
\begin{equation*}
 d^{\XX_{0}}_{W_1} (\delta_{x(0)}, \nu_{x; 0})
\leq d_{0}( x(0), x^*_0)
+ d^{\XX_{0}}_{W_1} (\delta_{x^*_0}, \nu_{x; 0}) 
\leq C(H,A,n)  + \sqrt{\Var (\delta_{x^*_0}, \nu_{x; 0})} \\
\leq C(H,A,n).
\end{equation*}

Next, let us now show Assertion~\ref{Prop_HK_regular_b}.
Since $B_{g_0} (x(0), c_0) \subset P(x; .1, -T) \cap \XX_0 \subset \RR_0$ is relatively compact in $\RR_0$, the extra assumption in the lemma allows us to conclude that
\begin{equation} \label{eq_Bx0_intrinsic}
 B(x(0), c_0) = B_{g_0} (x(0), c_0),
\end{equation}
because this is true if $\XX_0$ is intrinsic or it follows from Proposition~\ref{Prop_RR_intrinsic}\ref{Prop_RR_intrinsic_c}.

Our first goal is to show that $y \in B_{g_T}(x,.99)$.
Suppose by contradiction that $d_{g_T} (x, y) \geq .99$.
Choose $H$-centers $y^*_i \in \XX_{t^*_i}$ of $y$.
By Proposition~\ref{Prop_topology_properties}\ref{Prop_topology_properties_e} we have $y^*_i \to y$.
So by Theorem~\ref{Thm_regular_part_XX}\ref{Thm_regular_part_XX_e} we have $y^*_i \not\in P(x; .98, -T)$ for large $i$.
Therefore, by Lemma~\ref{Lem_nu_BA_bound} we have for large $i$
\begin{equation} \label{eq_mut_star_to_096}
 \mu_{t^*_i} \big( B_{g_{t^*_i}} (x(t^*_i), .96 ) \big) 
= \mu_{t^*_i} \big( B(x(t^*_i), .96 ) \big) 
\leq 1 - \mu_{t^*_i} \big( B(y^*_i, .01 ) \big)
\xrightarrow[i \to \infty]{\qquad} 0. 
\end{equation}
Next, let $\tau > 0$ be some small constant whose value we will determine later and suppose that $T \leq \tau$.
Then
\[ \mu_0 \big( \XX_0 \setminus B (x(0), c_0) \big) \leq c_0^{-2} \Var (\mu_0)
 \leq c_0^{-2} H \tau \leq C(H,  n) \tau. \]
Using the Claim and (\ref{eq_Bx0_intrinsic}), this implies that for any $x' \in \RR_T$ with $d_{g_T} (x, x') = .49$ and $t \in [0,T)$
\begin{multline} \label{eq_mu_t_B_xp_t}
 \mu_t \big(B_{g_t} (x'(t), c_1)  \big)
\leq c_1^{-1} \mu_0 \big( B_{g_0} (x'(0), c_0) \big) 
\leq c_1^{-1} \mu_0 \big( ( B_{g_T} (x', .1))(0) \big) \\
\leq c_1^{-1} \mu_0 \big( \XX_0 \setminus ( B_{g_T} (x, .1))(0) \big)
\leq c_1^{-1}\mu_0 \big( \XX_0 \setminus B_{g_0} (x(0), c_0) \big) \\
= c_1^{-1}  \mu_0 \big( \XX_0 \setminus B (x(0), c_0) \big)
\leq C(H,  n) \tau.
\end{multline}
Using standard volume comparison and distance distortion estimates, we can find points $x'_1, \ldots, x'_N \in \RR_T$ with $d_{g_T} (x, x'_j) = .49$ such that for some constants $c_2(n) > 0$, $C_2(n) < \infty$ we have $N \leq C_2$ and
\[ U: =P(x; .49+c_2, -T) \setminus P(x;.49, -T) \subset \bigcup_{j=1}^N \bigcup_{t \in [0,T]} B_{g_t}(x'_j (t), c_1) . \]
Together with (\ref{eq_mu_t_B_xp_t}), this implies that for some $C_3 (H, n) < \infty$ we have for all $t \in [0,T)$
\begin{equation} \label{eq_mut_Ut_C3tau}
 \mu_t (U_t) \leq C_3 \tau. 
\end{equation}

Using \cite[Lemma~5.3]{Cheeger_Gromov_Chopping_RM} and standard distance distortion estimates, we can construct a function $w \in C^\infty_c (P(x; \lb 1,\lb -T))$ such that:
\begin{enumerate}
\item $0 \leq w \leq 1$.
\item $w = 1$ on $P(x; .1, -T)$.
\item $\partial_{\tf} w = 0$.
\item $|\nabla w|, |\nabla^2 w| \leq C(n)$.
\item $\supp |\nabla w|, \supp |\nabla^2 w| \subset U$
\item $\supp w_t \subset B(x(t), .96)$ for $t$ close to $T$.
\end{enumerate}
We may view $w$ as a function on $\XX_{[0,T]}$ by continuing it by $0$ outside of $P(x; 1,-T)$.
Then, using (\ref{eq_mut_Ut_C3tau}),
\[ \frac{d}{dt} \int_{\XX_t} w_t  \, d\mu_t
= \frac{d}{dt} \int_{\RR_t} w_t v_t  \, dg_t
=  \int_{\RR_t} \big( (\square w_t) v_t  - w_t (\square^* v_t) \big) \, dg_t
=-  \int_{\RR_t} \triangle w_t \, v_t \, dg_t
\geq - C(H,n) \tau. \]
So by (\ref{eq_Bx0_intrinsic}) and (\ref{eq_mut_star_to_096})
\[ \mu_0 (B(x(0), c_0)) = \mu_0 (B_{g_0}(x(0), c_0)) \leq \int_{\XX_0} w_0  \, d\mu_0 \leq  C(H,n) \tau^2 + \int_{\XX_{t^*_i}} w_{t^*_i}  \, d\mu_{t^*_i} \xrightarrow[i \to \infty]{\qquad} C(H,n) \tau^2.  \]
However, this implies that
\[ c_0^2 (1- C(H,n) \tau^2) \leq \Var(\mu_0) \leq H \tau, \]
which gives us a contradiction for $\tau \leq \ov\tau(H, n)$.

So $y \in B_{g_T}(x, .99) \subset \RR_T$ if $T \leq \ov\tau(H, n)$.
We can now apply Assertion~\ref{Prop_HK_regular_a} to $y$, which gives
\begin{multline*}
 d_0 (x(0), y(0))
\leq d^{\XX_0}_{W_1} (\delta_{x(0)}, \nu_{y;0}) + d^{\XX_0}_{W_1} (\nu_{y;0}, \delta_{y(0)}) \\
\leq \sqrt{\Var (\delta_{x(0)}, \nu_{y;0}) } + C(H,n) \sqrt{T}
\leq  C(H,n) \sqrt{T}. 
\end{multline*}
Assuming $T \leq c(H,n)$, this implies that $d_0(x(0), y(0)) < .5$, so
\[ d_{T} (x, y) 
\leq  d_{g_T} (x, y) 
\leq C(n) d_{g_0} (x(0),y(0))
= C(n) d_{0}(x(0),y(0)) \leq C(H, n) \sqrt{T}. \]
Note that the equality holds due to the extra assumption in the lemma.
This finishes the proof.
\end{proof}
\bigskip

\begin{proof}[Proof of Proposition~\ref{Prop_RR_intrinsic}.]
Assertion~\ref{Prop_RR_intrinsic_a} is a restatement of Lemma~\ref{Lem_dt_leq_dgt}.

Next, consider two points $x_1, x_2 \in \RR_t$ as in Assertion~\ref{Prop_RR_intrinsic_b}.
Using Theorem~\ref{Thm_alm_alw_intrinsic}, we can find a sequence of times $t^*_i \nearrow t$ at which $\XX$ is intrinsic.
Choose $H$-centers $x^*_{1,i}, x^*_{2,i} \in \XX_{t^*_i}$ of $x_1, x_2$.
Then by Proposition~\ref{Prop_topology_properties}\ref{Prop_topology_properties_e}, $x^*_{1,i} \to x_{1,i}$ and $x^*_{2,i} \to x_{2}$.
Moreover,
\begin{multline} \label{eq_xstar_1_x_star_2}
  d_{t^*_i} (x^*_{1,i}, x^*_{2,i})
\leq d^{\XX_{t^*_i}}_{W_1} ( \delta_{x^*_{1,i}}, \nu_{x_1; t^*_i})
+ d^{\XX_{t^*_i}}_{W_1} (  \nu_{x_1; t^*_i},  \nu_{x_2; t^*_i})
+ d^{\XX_{t^*_i}}_{W_1} (  \  \nu_{x_2; t^*_i}, \delta_{x^*_{2,i}}) \\
\leq \sqrt{ \Var ( \delta_{x^*_{1,i}}, \nu_{x_1; t^*_i})} + d_t (x_1, x_2) + \sqrt{\Var (    \nu_{x_2; t^*_i}, \delta_{x^*_{2,i}})} \\
\leq d_t (x_1, x_2) + 2 \sqrt{H (t - t^*_i)} \xrightarrow[i \to \infty]{\qquad} d_t (x_1, x_2). 
\end{multline}
By our assumption and an openness argument, there is a small $\eps > 0$ such that for large $i$ there is a compact subset $K_i \subset \RR_{t^*_i}$ with the property that any path in $\RR_{t^*_i}$ of length $\leq d_t(x_1, x_2)+\eps$ between $x^*_{1,i}, x^*_{2,i}$ lies in $K_i$.
Since $\XX$ is intrinsic at time $t^*_i$, we obtain using (\ref{eq_xstar_1_x_star_2}) that  for large $i$
\[ d_{g_{t^*_i}}(x^*_{1,i}, x^*_{2,i})= d_{t^*_i} (x^*_{1,i}, x^*_{2,i}). \]
On the other hand, by assumption $d_{g_{t^*_i}} (x^*_{1,i}, x^*_{2,i}) \to d_{g_t} (x_1, x_2)$, which implies $d_{g_t} (x_1, x_2) \leq d_t (x_1, x_2)$.
So by Assertion~\ref{Prop_RR_intrinsic_a} we have $d_{g_t} (x_1, x_2) = d_t (x_1, x_2)$.

Lastly, we prove Assertion~\ref{Prop_RR_intrinsic_c}.
Assertion~\ref{Prop_RR_intrinsic_a} implies that $B_{g_t}(x,r) \subset B(x,r)$, so we need to show the reverse inclusion.
Let $y \in B(x,r)$.
Choose again a sequence of times $t^*_i \nearrow t$ at which $\XX$ is intrinsic and pick $H$-centers $x^*_i, y^*_i \in \XX_{t^*_i}$ of $x, y$.
As in (\ref{eq_xstar_1_x_star_2}), we obtain that
\[ \limsup_{i \to \infty} d_{t^*_i} (x^*_{i}, y^*_{i}) \leq d_t (x, y)  < r. \]
Since $x^*_i \to x$ we obtain that $y^*_i \in \RR$ for large $i$ and that there is some $r' > 0$ such that for large $i$ the parabolic neighborhood $P_i := P(y^*_i(t); r', - (t - t^*_i)) \subset \RR$ exists and is unscathed and $|{\Rm}| \leq (r')^{-2}$ on $P_i$.
So by Proposition~\ref{Prop_HK_regular}\ref{Prop_HK_regular_b} via Lemma~\ref{Lem_HK_regular_weak} we have $y \in \RR_t$.
Note here that the extra assumption in Lemma~\ref{Lem_HK_regular_weak} is satisfied since $\XX$ is intrinsic at time $t^*_i$.
As in the proof of Assertion~\ref{Prop_RR_intrinsic_b} this implies that $d_{g_{t^*_i}} (x^*_{i}, y^*_{i}) = d_{t^*_i} (x^*_{i}, y^*_{i}) < r$ for large $i$ and
\[ d_{g_t} (x, y) = \lim_{i \to \infty} d_{g_{t^*_i}} (x^*_{i}, y^*_{i}) \leq d_t(x,y).\]
So, again by Assertion~\ref{Prop_RR_intrinsic_a} we have $d_{g_t} (x, y) = d_t (x, y)$.
\end{proof}
\bigskip

\begin{proof}[Proof of Proposition~\ref{Prop_H_dim_bound}.]
Fix some $x \in \RR_{t_0}$ and set $n := \dim x$.
After application of a time-shift, we may assume without loss of generality that $t_0 = 0$.
Denote by $\XX^{\lambda}$ the flow that arises by parabolic rescaling $\XX$ by a factor of $\lambda > 0$.
Write $d\nu^\lambda_{x;t} := v^{\lambda}_t \, dg^\lambda_t$ for some $v^\lambda \in C^\infty (\RR^\lambda_{<0})$ with $\square^* v^\lambda = 0$.
Take a local blow-up limit for $\lambda \to \infty$ near $x$.
Then the metrics $g^\lambda$ near $x$ converge to the constant Ricci flow on $\IR^n$ and after passing to a subsequence, $v^\lambda$ converge to a smooth solution $v^\infty \in C^\infty(\IR^n \times \IR_-)$ to the backwards heat equation $(-\partial_t - \triangle ) v^\infty = 0$ due to the bounds in Lemma~\ref{Lem_std_par_est}.
Using Proposition~\ref{Prop_HK_regular}\ref{Prop_HK_regular_a}  via Lemma~\ref{Lem_HK_regular_weak} and Lemma~\ref{Lem_nu_BA_bound} and passing to the limit we have $\int_{\IR^n} v^\infty_t \, dg_{\eucl} = 1$ for all $t < 0$.
Due to the $H$-concentration condition, we also have for all $t < 0$
\[ \Var (v^\infty_t \, dg_{\eucl} ) \leq H |t|. \]
It follows that $v^\infty$ is the standard Gaussian backwards heat kernel and by the computation in Example~\ref{Ex_Rn_Gaussian} we have for all $t < 0$
\[ 4n |t| = \Var (v^\infty_t \, dg_{\eucl} ) \leq H |t|. \]
This finishes the proof.
\end{proof}
\bigskip

\begin{proof}[Proof of Proposition~\ref{Prop_HK_regular}.]
By Proposition~\ref{Prop_H_dim_bound} we have $\dim x \leq H/4$.
Moreover, Proposition \ref{Prop_RR_intrinsic}\ref{Prop_RR_intrinsic_c} holds.
So Proposition~\ref{Prop_HK_regular} follows from Lemma~\ref{Lem_HK_regular_weak}.
\end{proof}
\bigskip

\subsection{Smooth convergence on the regular part} \label{subsec_smooth_convergence}
We will now analyze the convergence of the regular part in an $\IF$-convergent sequence of metric flow pairs.
For this subsection, we fix a sequence of metric flow pairs $(\XX^i, (\mu^i_t)_{t \in I^{ i}})$, $i \in \IN \cup \{ \infty \}$, that are fully defined over intervals $I^i \subset \IR$.
We will assume that $\XX^i$ is almost always intrinsic for all $i \in \IN$ and $H$-concentrated for all $i \in \IN \cup \{ \infty \}$ for some uniform $H < \infty$.
We will also assume that all intervals $I^i$, $i \in \IN \cup \{ \infty \}$, are left-open.
Suppose that there is a correspondence 
\[  \CF := \big(  (Z_t, d^Z_t)_{t \in I^\infty},(\varphi^i_t)_{t \in I^{\prime\prime, i}, i \in \IN \cup \{ \infty \}} \big) \]
between the flows $\XX^i$, $i \in \IN \cup \{ \infty \}$, over $I^\infty$ such that
\begin{equation} \label{eq_XXimui_convergence_in_smooth_sec}
 (\XX^i, (\mu^i_t)_{t \in I^{i}}) \xrightarrow[i \to \infty]{\quad \IF, \CF \quad}  (\XX^\infty, (\mu^\infty_t)_{t \in I^{\infty}})  
\end{equation}
on compact time-intervals.
We also assume that if $t_{\max} := \max I^\infty$ exists, then all flows $\XX^i$, $i \in \IN \cup \{ \infty \}$, are intrinsic at time $t_{\max}$ and the convergence (\ref{eq_XXimui_convergence_in_smooth_sec}) is time-wise at time  $t_{\max}$.
The most interesting case will be the case in which the flows $\XX^i$, $i \in \IN$, are given by smooth Ricci flows of the same dimension, which implies $\RR^i = \XX^i$.

By Theorem~\ref{Thm_intrinsic_limit}, the limit $\XX^\infty$ is also almost always intrinsic.
Let $\RR^i \subset \XX^i$ be the regular part of $\XX^i$, $i \in \IN \cup \{ \infty \}$, and write $d\mu^i = v^i \, dg^i$ on $\RR^i_{I \setminus \{ t_{\max} \}}$, where $v^i \in C^\infty(\RR^i_{I \setminus \{ t_{\max} \}})$.

\begin{Definition} \label{Def_smooth_convergence}
We say that the convergence (\ref{eq_XXimui_convergence_in_smooth_sec}) is {\bf smooth} at some point $x_\infty \in \XX^\infty$ if the following is true.
There is a scale $r > 0$ and points $x_i \in \XX^i$ such that $x_i \to x_\infty$ within $\CF$ (in the sense of Definition~\ref{Def_convergence_PP_measure}\footnote{We will often write ``$x_i \to x_\infty$ within $\CF$'' instead of
\[ x_i \xrightarrow[i \to \infty]{\quad  \CF, J \quad} x_\infty. \]
Recall that this notion of convergence means that we have convergence of the conjugate heat kernels $(\nu_{x_i;t}) \to (\nu_{x_\infty;t})$ within $\CF$; it is weaker than strict convergence in the sense of Definiton~\ref{Def_convergence_CHF_within_CF_strict}.}) and such that for large $i$ we have $x_i \in \RR^i$, the two-sided parabolic ball $P(x_i; r) \subset \RR^i$ is unscathed and we have
\[ \sup_{P(x_i; r)} |{\Rm}| \leq r^{-2}, \qquad \liminf_{i \to \infty}|B(x_i, r)| > 0. \]
If $\tf (x_\infty) = t_{\max} =  \max I^\infty$, then we require in addition that for large $i$ the function $v^i$ can be smoothly extended onto the entire parabolic neighborhood $P(x_i; r)$ such that we still have $d\mu^i = v^i \, dg^i$ and 
 for all $m \geq 1$ we have $\limsup_{i \to \infty} \sup_{B(x_i, r)} |\nabla^m v^i| < \infty$ and there is no $r' \in (0,r)$ and sequence of points $x'_i \in B(x_i, r)$ such that $B(x'_i,r') \subset B(x_i, r)$ and $\liminf_{i \to \infty} \mu_{\tf(x_i)} (B(x'_i, \lb r')) = 0$.
\end{Definition}

Denote by $\RR^* \subset \XX^\infty$ the set of points at which the convergence (\ref{eq_XXimui_convergence_in_smooth_sec}) is smooth.
The following is our main result of this subsection.
It states that $\RR^*$ is an open subset of the regular part $\RR^\infty \subset \XX^\infty$ and that the convergence (\ref{eq_XXimui_convergence_in_smooth_sec}) can be understood via a sequence of diffeomorphisms between an exhaustion of $\RR^*$ and a sequence of open subsets of $\RR^i$.
This is similar to the characterization of smooth Cheeger-Gromov convergence.

\begin{Theorem} \label{Thm_smooth_convergence}
$\RR^*$ is open and we have $\RR^* \subset \RR^\infty$.
Moreover, we can find an increasing sequence $U_1 \subset U_2 \subset \ldots \subset \RR^*$ of open subsets with $\bigcup_{i=1}^\infty U_i = \RR^*$, open subsets $V_i \subset \RR^i$, time-preserving diffeomorphisms $\psi_i : U_i \to V_i$ and a sequence $\eps_i \to 0$ such that the following holds:
\begin{enumerate}[label=(\alph*)]
\item \label{Thm_smooth_convergence_a} We have
\begin{align*}
 \Vert \psi_i^* g^i - g^\infty \Vert_{C^{[\eps_i^{-1}]} ( U_i)} & \leq \eps_i, \\
 \Vert \psi_i^* \partial^i_{\tf} - \partial^\infty_{\tf} \Vert_{C^{[\eps_i^{-1}]} ( U_i)} &\leq \eps_i, \\
  \Vert  v^i \circ \psi_i - v^\infty \Vert_{C^{[\eps_i^{-1}]} ( U_i)} &\leq \eps_i. 
\end{align*}
\item \label{Thm_smooth_convergence_b} For $U^{(2)}_i := \{ (x;y) \in U_i \times U_i \;\; : \;\; \tf(x) > \tf(y) - \eps_i \}$, $V^{(2)}_i := \{ (x;y) \in V_i \times V_i \;\; : \;\; \tf(x) > \tf(y) - \eps_i \}$ and $\psi_i^{(2)} := (\psi_i, \psi_i) :U^{(2)}_i  \to V^{(2)}_i$, we have convergence of the heat kernels
\[  \Vert (\psi^{(2)}_i)^* K^i - K^\infty \Vert_{C^{[\eps_i^{-1}]} ( U^{(2)}_i)} \leq \eps_i, \]
\item \label{Thm_smooth_convergence_c} Let $x_\infty \in \RR^*$ and $x_i \in \XX^i$.
Then we have $x_i \to x_\infty$ within $\CF$ if and only if $x_i \in V_i \subset \RR^i$ for large $i$ and $\psi_i^{-1} (x_i) \to x_\infty$ in $\RR^*$.
\item \label{Thm_smooth_convergence_d} If the convergence (\ref{eq_XXimui_convergence_in_smooth_sec}) is time-wise at some time $t \in I^\infty$ for some subsequence, then for any compact subset $K \subset \RR^\infty_t$ and for the same subsequence
\[ \sup_{x \in K \cap U_i} d^Z_t (\varphi^i_t (\psi_i(x)), \varphi^\infty_t (x) )  \longrightarrow 0. \]
\item \label{Thm_smooth_convergence_e} Consider a sequence of points $x_i \in \XX^i$ such that $x_i \to x_\infty \in \XX^\infty$ within $\CF$.
Then on $\RR^*$
\[   \psi_i^* K^i (x_i; \cdot) \xrightarrow[i \to \infty]{\quad C^\infty_{\loc} \quad} K^\infty (x_\infty ; \cdot). \]
\item \label{Thm_smooth_convergence_f} Consider a sequence of conjugate heat flows $(\mu'_{i,t})_{t \in I^i, t < t_0}$ on $\XX^i$, $i \in \IN \cup \{ \infty \}$, for $t_0 \in I^\infty$ such that
\[ (\mu'_{i,t})_{t \in I^i, t < t_0} \xrightarrow[i \to \infty]{\quad \CF \quad} (\mu'_{\infty,t})_{t \in I^\infty, t < t_0}. \]
Write $d\mu'_{i,t} = v'_i \, dg^i$ on $\RR^i$ for $i \in \IN \cup \{ \infty \}$.
Then on $\RR^*$
\[     v'_i \circ \psi_i  \xrightarrow[i \to \infty]{\quad C^\infty_{\loc} \quad} v'_\infty . \]
\end{enumerate}
\end{Theorem}

\begin{Remark}
If $\RR^* = \XX^\infty$, then Theorem~\ref{Thm_smooth_convergence} combined with the compactness theory from Section~\ref{sec_compact_subsets_IF} essentially recovers Hamilton's compactness theory for Ricci flows with bounded curvature \cite{Hamilton_RF_compactness}.
\end{Remark}

The proof of Theorem~\ref{Thm_smooth_convergence} relies on the following lemma, which can be viewed as a local version.
The last statement of the following lemma will be a byproduct of the proof and will be needed in the next subsection.

\begin{Lemma} \label{Lem_smooth_conv_local}
Let $x_\infty \in \RR^*$ and consider a scale $r > 0$ and a sequence $x_i \in \XX^i$ as in Definition~\ref{Def_smooth_convergence}.
Then $x_\infty \in \RR^\infty$ and there is an open product domain $x_\infty \in U \subset \RR^\infty$ with $U \subset \RR^*$.
For large $i$ there are time-preserving and $\partial_\tf$-preserving diffeomorphisms $\psi_i : U \to V_i \subset \RR^i$ and a sequence $\eps_i \to 0$ such that Assertions~\ref{Thm_smooth_convergence_a}, \ref{Thm_smooth_convergence_c}, \ref{Thm_smooth_convergence_e}, \ref{Thm_smooth_convergence_f} of Theorem~\ref{Thm_smooth_convergence} hold for $U_i := U$ and for large $i$.
Here, the sequence of points in Assertions~\ref{Thm_smooth_convergence_c}, \ref{Thm_smooth_convergence_e} can be different from the sequence $x_i$ in this lemma, but we require that the limiting point in Assertion~\ref{Thm_smooth_convergence_c}, which is called $x_\infty$, lies in $U$.

Moreover, consider some constant $\alpha > 0$ such that $\liminf_{i \to \infty} v^i(x_i) \geq \alpha$ if $\tf(x_\infty) \neq \sup I^\infty$ and such that $r \in [\alpha, \alpha^{-1}]$, $H \leq \alpha^{-1}$ and $\liminf_{i \to \infty} |B(x_i, r)|_{g_{\tf (x_i)}} \geq \alpha$.
Then there is a universal constant $r_*=r_* ( \alpha) > 0$, which depends continuously on $\alpha$ such that the two-sided parabolic ball $P(x_\infty; r_*) \subset \RR^\infty$ is unscathed and we have $P(x_\infty; r_*) \subset U$.
\end{Lemma}

\begin{proof}
Fix $r, \alpha > 0$.
Let $r_* \in (0, r)$ be a constant whose value we will choose in the course of the proof by imposing conditions of the form $r_* \leq \ov{r}_*(\alpha)$, where the latter will always denote a generic constant depending only on $\alpha$.
By Proposition~\ref{Prop_H_dim_bound} we have $n_i := \dim x_i \leq H/4$ for large $i$.
Set $M'_i := B(\vec 0, r_*) \subset \IR^{n_i}$ and $x'_i := \vec 0 \in M'_i$ and use the exponential map to choose diffeomorphisms $\ov\phi_i : M'_i \to B(x_i, r_*) \subset \RR^i_{\tf(x_i)} \subset \XX_{\tf (x_i)}$ with $\ov\phi_i(x'_i) = x_i$ for large $i$.
Using the flow of the vector field $\partial_{\tf}$ on $\RR^i$, we can extend $\ov\phi_i$ to time-preserving and $\partial^i_{\tf}$-preserving embeddings $\phi_i : M'_i \times I' \to \RR^i$, where $I' := (\tf(x_i) - r_*^2, \tf(x_i) + r_*^2) \cap I^\infty$, such that $\phi_i (\cdot, \tf (x_i)) = \ov\phi_i$, implying $\phi_i(x'_i, \tf(x_i)) = x_i$.

Before continuing with the proof, let us first explain that, without loss of generality, we may in the following  always pass to a subsequence of the given sequence of metric flow pairs.
To see this, note first that the statements of the lemma characterizing the limiting flow, i.e. that $x_\infty \in U \subset \RR^\infty$, are not affected by such a step.
Moreover, $n := \dim (x_\infty)$ is uniquely determined by the metric flow $\XX^\infty$.
Assume for a moment that $x_\infty \in \RR^\infty$.
To see that the other statements of the lemma still follow even if we only prove them after passing to a subsequence, note that we can apply our proof to an arbitrary subsequence of the given sequence of metric flow pairs.
So if one of the convergence statements in Assertions~\ref{Thm_smooth_convergence_a}, \ref{Thm_smooth_convergence_c}, \ref{Thm_smooth_convergence_e}, \ref{Thm_smooth_convergence_f} was violated, then we could apply our proof to a subsequence with the property that any further subsequence still violates this convergence statement and obtain a contradiction.
To be more precise about this, we will now provide a brief overview of the following proof and discuss the effects of passing to a subsequence in the middle of our arguments.

In the course of the proof, we will construct a smooth time-preserving and $\partial_{\tf}$-preserving diffeomorphism onto its image $\phi_\infty : M' \times I' \to \RR^\infty$, where $M' = B(\vec 0,r_*) \subset \IR^{n}$, $I' := (\tf(x_\infty) - r_*^2, \tf(x_\infty) +  r_*^2) \cap I^\infty$, $\phi_\infty (x'_\infty, \tf (x_\infty)) = x_\infty$ and $r_* = r_* (\alpha) > 0$.
The map $\phi_\infty$ will arise as a limit of a subsequence of the maps $\phi_i : M'_i \times I' \to \RR^i \subset \XX^i$, and it will follow that $n_i = n = \dim (x_\infty)$ for large $i$.
Since we could have started our proof with an arbitrary subsequence of the given sequence of metric flow pairs and since $n$ is independent of this subsequence, we must have $n_i = n$ for large $i$ for the original sequence.
Next, we will set $U := \phi_\infty (M' \times I')$ and define $\psi_i := \phi_i \circ \phi_\infty^{-1} |_{U}$.
We will show that the maps $\psi_i : U_i := U \to V_i := \psi_i (U) \subset \RR^i$ satisfy Assertions~\ref{Thm_smooth_convergence_a}, \ref{Thm_smooth_convergence_c}, \ref{Thm_smooth_convergence_e}, \ref{Thm_smooth_convergence_f} of Theorem~\ref{Thm_smooth_convergence} for some $\eps_i \to 0$ after passing to a subsequence.
More specifically, our proof will imply that given any subsequence of these maps, we can pass to a further subsequence such that Assertions~\ref{Thm_smooth_convergence_a}, \ref{Thm_smooth_convergence_c}, \ref{Thm_smooth_convergence_e}, \ref{Thm_smooth_convergence_f} of Theorem~\ref{Thm_smooth_convergence} hold for some $\eps_i \to 0$.
Suppose now that the maps $\psi^i$ don't satisfy Assertions~\ref{Thm_smooth_convergence_a}, \ref{Thm_smooth_convergence_c}, \ref{Thm_smooth_convergence_e}, \ref{Thm_smooth_convergence_f} of Theorem~\ref{Thm_smooth_convergence} for any $\eps_i \to 0$ \emph{if we don't pass to a subsequence of the given sequence}.
For each $i \in \IN$ choose $\eps_i \in (0, \infty]$ minimal such that Assertion~\ref{Thm_smooth_convergence_a} holds.
If we didn't have $\eps_i \to 0$, then we could choose a subsequence such that $\eps_i > c'' > 0$, which would contradict the fact that Assertion~\ref{Thm_smooth_convergence_a} holds for some $\eps_i \to 0$.
Similarly, if Assertions~\ref{Thm_smooth_convergence_c}, \ref{Thm_smooth_convergence_e}, \ref{Thm_smooth_convergence_f} were violated for a sequence of points or a conjugate heat flow, then we could pass to a subsequence with the property that this violation persists for any further subsequence, in contradiction to what we will show.
Lastly, the fact that $U \subset \RR^*$ follows from Assertion~\ref{Thm_smooth_convergence_c}.
So in summary, in the following we are free to pass to a subsequence of metric flow pairs, as long as we choose $r_*$ only depending on $\alpha$.

Let us now begin with the actual proof.
By $\IF$-convergence and Lemma~\ref{Lem_pass_to_timewise} we can pass to a subsequence, such that for some set of measure zero $E_\infty \subset I^\infty$ the convergence (\ref{eq_XXimui_convergence_in_smooth_sec}) restricted to the subsequence is time-wise at any $t \in I^\infty \setminus E_\infty$.
Then by Lemma~\ref{Lem_met_flow_conv_mu_conv}
\[ (\varphi^i_t)_* \mu_{i,t} \xrightarrow[i \to \infty]{\quad W_1 \quad} (\varphi^\infty_t)_* \mu_{\infty,t} \qquad \text{for all} \quad t \in I^\infty \setminus E_\infty . \]
Moreover, we have:

\begin{Claim} \label{Cl_CHF_W1_conv}
Consider conjugate heat flows $(\td\mu_{i,t})_{t \in \td I}$ on $\XX^i$, $i \in \IN \cup \{ \infty \}$, over a right-open subinterval $\td I$ such that
\[ (\td\mu_{i,t})_{t \in \td I} \xrightarrow[i \to \infty]{\quad \CF \quad} (\td\mu_{\infty,t})_{t \in \td I}. \]
Then we have weak convergence
\[ (\varphi^i_t)_* \td\mu_{i,t} \xrightarrow[i \to \infty]{\quad  \quad} (\varphi^\infty_t)_* \td\mu_{\infty,t} \qquad \text{for all} \quad t \in \td{I} \setminus E_\infty. \]
Moreover, if $\td\mu_{\infty,t} \in \PP^1 (\XX^\infty_t)$ for all $t \in \td I$ (i.e. $\td\mu_{\infty,t}$ has finite $d_{W_1}$-distance to point masses), then the convergence holds in $W_1$.
\end{Claim}

\begin{proof}
This is a direct consequence of Theorem~\ref{Thm_tdmu_convergence_reverse}\ref{Thm_tdmu_convergence_reverse_b}. 
\end{proof}

For any $i \in \IN \cup \{ \infty \}$ let $E'_i \subset I^i$ be the set of times at which $\XX^i$ is not intrinsic and set $E' := \bigcup_{i=1}^\infty E'_i \cup E'_\infty$.
Recall that $E'_\infty$ is countable (see Theorem~\ref{Thm_alm_alw_intrinsic}) and therefore has measure zero.
Note that if $t_{\max} = \max I^\infty$ exists, then by our assumptions $t_{\max} \in I' \setminus (E_\infty \cup E')$.

Since $n_i \leq H/4$, we may pass to a subsequence, and assume in the following that $n := \dim(x_i)$ is constant;
we will write $M' := M'_i$ and $x' := x'_i$ from now on.
After passing to another subsequence, we can use \cite{Hamilton_RF_compactness} to find a Ricci flow $(M', (g'_t)_{t \in I'})$ on $M' \times I'$ such that we have local smooth convergence
\[ \phi_i^* g^i \xrightarrow[i \to \infty]{\quad C^\infty_{\loc} \quad } g'.
\]
Write $d\mu^i_t =: v^i_t \, dg^i_t$ for some $v^i \in C^\infty (\RR^i_{I^i \setminus \{ \sup I^i \}})$, $v^i \geq 0$.
By standard parabolic estimates, we may also pass to a subsequence such that for some $v' \in C^\infty_{\loc} (M' \times I')$ with $\square^* v' = 0$ such that
\begin{equation} \label{eq_vi_to_vp_Cinfty_loc}
 v^i \circ \phi_i \xrightarrow[i \to \infty]{\quad C^\infty \quad } v'.
\end{equation}

\begin{Claim} \label{Cl_inf_vp_positive}
If $r_* \leq \ov{r}_*(\alpha)$, then $v'_t$ is nowhere locally vanishing on $M'$ for all $t \in I'$.
\end{Claim}

\begin{proof}
Suppose by contradiction that $v'_{t''}$ locally vanished somewhere on $M'$ for some $t'' \in I'$.
Then by Definition~\ref{Def_smooth_convergence} we must have $t'' < \sup I$ and thus $t'' < \sup I'$.
So by the strong maximum principle we have $v' \equiv 0$ on $M' \times (I' \cap [t'', \infty))$.

Next, consider some time $t^* \in I' \setminus (E_\infty \cup E')$, $t^* \leq \tf (x_\infty)$.
Let $x^* \in \XX^\infty_{t^*}$ be an $H$-center of $x_\infty$.
For $t^*$ sufficiently close to $\tf (x_\infty)$ for any $H$-center $y^* \in \XX^\infty_{t^*}$ of any point $y \in B(x_\infty, \tfrac1{8} r_*)$ we have
\begin{multline*}
 d^\infty_{t_*} (x^*, y^*) 
\leq d^{\XX^\infty_{t^*}}_{W_1} (\delta_{x^*}, \nu^\infty_{x_\infty; t^*})
+ d^{\XX^\infty_{t^*}}_{W_1} (\nu^\infty_{x_\infty; t^*}, \nu^\infty_{y; t^*})
+ d^{\XX^\infty_{t^*}}_{W_1} (\nu^\infty_{y; t^*}, \delta_{y^*}) \\
\leq \sqrt{ \Var (\delta_{x^*}, \nu^\infty_{x_\infty; t^*})} + d^\infty_{\tf(x_\infty)} (x_\infty, y) + \sqrt{ \Var (\nu^\infty_{y; t^*}, \delta_{y^*})}
\leq 2 \sqrt{H (\tf(x_\infty) - t^*)} + \tfrac18 r_*. 
\end{multline*}
So for $t^*$ sufficiently close to $\tf (x_\infty)$ we have by Lemma~\ref{Lem_nu_BA_bound}
\begin{equation} \label{eq_muinfty_tstar12}
 \mu^\infty_{t^*} (B(x^*, \tfrac1{4} r_*)) 
\geq \int_{B(x_\infty, \frac1{8} r_*)} \nu^\infty_{y; t^*} (B(x^*, \tfrac1{4} r_*)) \, d\mu^\infty_{\tf(x_\infty)} (y)
\geq \tfrac12 \mu^\infty_{\tf(x_\infty)}(B(x_\infty, \tfrac1{8} r_*)). 
\end{equation}

By Proposition~\ref{Prop_HK_regular}\ref{Prop_HK_regular_a} and Claim~\ref{Cl_CHF_W1_conv} we have
\begin{multline*}
 \limsup_{i \to \infty} d^Z_{t^*} ( \varphi^i_{t^*}( x_i (t^*)), \varphi^\infty_{t^*} ( x^*)) \\
\leq \limsup_{i \to \infty} 
\big(
d^{\XX^i_{t^*}}_{W_1} (\delta_{x_i (t^*)}, \nu^i_{x_i; t^*})
+ d^{Z_{t^*}}_{W_1} ( (\varphi^i_{t^*})_* \nu^i_{x_i; t^*}, (\varphi^\infty_{t^*})_* \nu^\infty_{x_\infty; t^*} )
+d^{\XX^\infty_{t^*}}_{W_1} ( \nu^\infty_{x_\infty; t^*}, \delta_{x^*}) \big) \\
\leq C \sqrt{\tf(x_\infty) - t^*} + \limsup_{i \to \infty} \sqrt{\Var ( \nu^\infty_{x_\infty; t^*}, \delta_{x^*}) }
\leq (C + H^{1/2})  \sqrt{\tf(x_\infty) - t^*}. 
\end{multline*}
So for $t^*$ sufficiently close to $\tf (x_\infty)$ we have, again by Claim~\ref{Cl_CHF_W1_conv}, and (\ref{eq_muinfty_tstar12})
\begin{multline*}
 \liminf_{i \to \infty} \int_{B(x_i (t^*), t^*, \frac1{2} r_*)} (v^i_{t^*} \circ \phi_{i,t^*}) \, dg^i_{t^*}
= \liminf_{i \to \infty} \mu^i_{t^*} (B(x_i (t^*), t^*, \tfrac1{2} r_*)) \\
= \liminf_{i \to \infty} ((\varphi^i_{t^*})_* \mu^i_{t^*} ) ( B(\varphi^i_{t^*}(x_i (t^*)), t^*, \tfrac1{2} r_*))
\geq ((\varphi^\infty_{t^*})_* \mu^\infty_{t^*}) ( B(\varphi^\infty_{t^*} (x^*), \tfrac1{4} r_*)) \\
= \mu^\infty_{t^*} (B(x^*, \tfrac1{4} r_*)) 
 \geq  \tfrac12 \mu^\infty_{\tf(x_\infty)}(B(x_\infty, \tfrac1{8} r_*)). 
\end{multline*}
Letting $t^* \nearrow \tf (x_\infty)$ implies that
\[ \int_{B(x_\infty , \tf (x_\infty), \frac1{2} r_*)} v'_{\tf (x_\infty)} \, dg^\infty_{\tf (x_\infty)}\geq  \tfrac12 \mu^\infty_{\tf(x_\infty)}(B(x_\infty, \tfrac1{8} r_*)). \]
By standard parabolic derivative estimates (see Lemma~\ref{Lem_std_par_est}), this implies that $t'' > \tf(x_\infty) + c(\alpha) r^2_*$.
So the claim follows after shrinking $r_*$ appropriately.
\end{proof}

\begin{Claim} \label{Cl_construct_chi_infty_t}
If $r_* \leq \ov{r}_*(\alpha)$, then the following is true.
For any $t \in I' \setminus (E_\infty \cup E')$ and after passing to any subsequence of the current subsequence of metric flow pairs, we can pass to a further subsequence such that the maps $\varphi^i_t \circ \phi_{i,t} : M' \to Z_t$  uniformly converge to some map of the form $\varphi^\infty_t \circ \phi_{\infty,t} : M' \to Z_t$, where $\phi_{\infty, t} : (M', d_{g'_t}) \to (\XX^\infty_t, d^\infty_t)$ is a local isometry and a homeomorphism onto its image.
\end{Claim}

\begin{proof}
Fix $t$ and fix some $y \in M'$ for a moment.
We claim that the sequence $\varphi^i_t (\phi_{i,t} (y)) \in Z_t$ subsequentially converges to some point in the completion $(\ov Z, d^{\ov Z}_t)$ of $(Z, d^Z_t)$.
Suppose not.
Then we can pass to a subsequence such that for some $r' > 0$ the balls $B( \varphi^i_t (\phi_{i,t} (y)), r') \subset Z_t$ are pairwise disjoint.
Suppose that $r' > 0$ is chosen small enough such that $B_{g'_t} (y, 2r') \subset M'$.
By Claim~\ref{Cl_inf_vp_positive} we have
\[ \liminf_{i \to \infty} ((\varphi^i_t)_* \mu^i_t) (B( \varphi^i_t (\phi_{i,t} (y)), r'))
= \liminf_{i \to \infty}  \mu^i_t (B( \phi_{i,t} (y), r'))
= \int_{B_{g'_t} (y, r')} v' \, dg'_t > 0. \]
This, however, contradicts the fact that $(\varphi^i_t)_* \mu^i_t \to (\varphi^\infty_t)_* \mu^\infty_t$ in $W_1$, see Claim~\ref{Cl_CHF_W1_conv}.
So after passing to a subsequence, we have convergence $\varphi^i_t (\phi_{i,t} (y)) \to z \in \ov Z_t$ and $((\varphi^\infty_t)_* \mu^\infty_t )(B(z, r')) > 0$ for all $r' \in (0,1)$.
Thus $z \in\supp ((\varphi^\infty_t)_* \mu^\infty_t ) = \varphi^\infty_t (\XX^\infty_t) \subset Z$.

Fix some countable, dense subset $S \subset M''$.
By the previous paragraph and after passing to a diagonal subsequence, we may assume that there is a map $\phi_{\infty,t} : S \to \XX^\infty_t$ such that for all $y \in S$
\begin{equation} \label{eq_varphi_chi_y_to_inZ}
 \varphi^i_t (\phi_{i,t} (y)) \to \varphi^\infty_t (\phi_{\infty,t} (y)) \qquad \text{in} \quad Z_t 
\end{equation}
Since the maps $\varphi^i_t \circ \phi_{i,t}$ are uniformly locally Lipschitz, $\phi_{\infty,t}$ can be extended to $M'$ and (\ref{eq_varphi_chi_y_to_inZ}) holds for all $y \in M'$.
The claim follows after shrinking $r_*$.
\end{proof}

Choose a countable, dense subset $Q \subset I' \setminus (E_\infty \cup E')$ such that $t_{\max} = \max I' \in Q$ if it exists, apply Claim~\ref{Cl_construct_chi_infty_t} for each $t \in Q$ and pass to a diagonal subsequence.
In doing so, we can construct a family of local isometries and homeomorphisms onto their images $(\phi_{\infty, t} : (M', d_{g'_t}) \to (\XX^\infty_t, d^\infty_t))_{t \in Q}$ such that we have uniform convergence $\varphi^i_t \circ \phi_{i,t} \to \varphi^\infty_t \circ \phi_{\infty,t}$ in $Z_t$ for all $t \in Q$.

\begin{Claim} \label{Cl_chi_infty_basic_prop}
If $r_* \leq \ov{r}_*(\alpha)$, then the following holds for some $C^* < \infty$:
\begin{enumerate}[label=(\alph*)]
\item \label{Cl_chi_infty_basic_prop_a} Suppose that $\td I \subset I'$ is a right-open subinterval and $(\td\mu_{i,t})_{t \in \td I}$ are conjugate heat flows on $\XX^i$, $i \in \IN \cup \{ \infty \}$, such that
\[ (\td\mu_{i,t})_{t \in \td I} \xrightarrow[i \to \infty]{\quad \CF \quad} (\td\mu_{\infty,t})_{t \in \td I}. \]
Suppose that $\td\mu_{\infty,t} \in \PP^1(\XX^\infty_t)$  and write $d\td\mu_{i,t} = \td v_i \, dg^i$ on $\RR^i$ for $i \in \IN$.
Then $\td v_i \circ \phi_i \to \td v' \in C^\infty (M' \times \td I)$ in $C^\infty_{\loc}$, where $\square^* \td v' = 0$ and $d ((\phi_{\infty,t})^* \td\mu_{\infty,t}) = \td v'\, dg'_t$ for all $t \in Q \cap \td I$.
\item \label{Cl_chi_infty_basic_prop_b} 
Suppose that $\td I \subset I'$ is a left-open subinterval and $(u_{i,t})_{t \in \td I}$ are heat flows on $\XX^i$ with $|u_i| \leq C < \infty$ that converge to some heat flow $(u_{\infty,t})_{t \in \td I}$ on $\XX^\infty$ in the sense that for any sequence $y_i \in \XX^i_{\td I}$ with $y_i \to y_\infty \in \XX^\infty_{\td I}$ we have $u_i (y_i) \to u_\infty (y_\infty)$.
Then $u_i \circ \phi_i \to  u' \in C^\infty (M' \times \td I)$ in $C^\infty_{\loc}$, where $\square u^{\prime} = 0$ and $u_{\infty,t} \circ \phi_{\infty,t} =  u^{\prime}_t$ for all $t \in Q \cap \td I$.
\item \label{Cl_chi_infty_basic_prop_c} For any conjugate heat flow $(\td\mu_t)_{t \in \td I}$ on $\XX^\infty$ over a right-open subinterval $\td I \subset I'$ there is a smooth function $\td v' \in C^\infty (M' \times \td I)$ with $\square^* \td v' = 0$ such that $d ((\phi_{\infty,t})^* \td\mu_t) =  \td v' dg'_t$ for all $t \in Q \cap \td I$.
\item \label{Cl_chi_infty_basic_prop_d} For any uniformly bounded heat flow $(u_t)_{t \in \td I}$ on $\XX^\infty$ over a left-open subinterval $\td I \subset I'$ there is a smooth function $u' \in C^\infty (M' \times \td I)$ with $\square u' = 0$ such that $u_t \circ \phi_{\infty,t} =  u^{\prime}_t$ for all $t \in Q \cap \td I$.
\item \label{Cl_chi_infty_basic_prop_e} For any $y \in M'$ and any $t_1, t_2 \in Q$ with $t_1 \leq t_2$ we have
\[ d^{\XX^\infty_{t_1}}_{W_1} ( \delta_{\phi_{\infty, t_1} (y)}, \nu^\infty_{\phi_{\infty, t_2} (y); t_1} ) \leq C^* \sqrt{t_2 - t_1}. \]
\item \label{Cl_chi_infty_basic_prop_f} For any $t \in Q$, $t < \tf (x_\infty)$ we have
\[ d^{\XX^\infty_{t}}_{W_1} ( \delta_{\phi_{\infty, t} (x')}, \nu^\infty_{x_\infty; t} ) \leq C^* \sqrt{\tf (x_\infty) - t}. \]
\item \label{Cl_chi_infty_basic_prop_g} For any $t_1, t_2 \in Q$ with $t_1 \leq t_2$ the following is true.
If $y^*_1 \in \XX^\infty_{t_1}$ is an $H$-center of some point $y_2 \in \XX_{t_2}^\infty$ and $y^*_1 =  \phi_{\infty,t_1} ( y'_1 )$ for $y'_1 \in M'$ with the property that $B_{g'_{t_1}} (y'_1,  C^* \sqrt{t_2 - t_1}) \subset M'$ is relatively compact, then $y_2 \in \phi_{\infty,t_2} (M')$ and if $y_2 = \phi_{\infty, t_2} (y'_2)$, then  $d_{g'_{t_2}} (y'_1, y'_2) \leq C^* \sqrt{t_2 - t_1}$.
\end{enumerate}
\end{Claim}

\begin{proof}
In Assertions~\ref{Cl_chi_infty_basic_prop_a}, \ref{Cl_chi_infty_basic_prop_b} observe that due to standard local parabolic estimates, the functions $\td v_i \circ \phi_i$ and $u_i \circ \phi_i$ are locally uniformly bounded in any $C^m$-norm.
Moreover for any $t \in Q$ we have $(\td v_{i,t}  \circ \phi_{i,t} )d ((\phi_{i,t})^* g^i_t) =d ( (\phi_{i,t})^* \td\mu_{i,t} ) \to d ( (\phi_{\infty,t})^* \td\mu_{\infty,t} )$ in weakly (see Claims~\ref{Cl_CHF_W1_conv}, \ref{Cl_construct_chi_infty_t}) and $u^i_t  \circ \phi_{i,t} \to u^\infty_t  \circ \phi_{\infty,t}$ pointwise (see Claim~\ref{Cl_construct_chi_infty_t} and Theorem~\ref{Thm_tdmu_convergence}).
Therefore, both limits are smooth and since $Q$ is dense, we have $\td v_i \circ \phi_i \to \td v'_\infty \in C^\infty (M' \times \td I)$   and  $u_i \circ \phi_i \to  u'_\infty \in C^\infty (M' \times \td I)$ in $C^\infty_{\loc}$.
By comparing the limits for any $t \in Q$ we obtain the last statements of Assertions~\ref{Cl_chi_infty_basic_prop_a}, \ref{Cl_chi_infty_basic_prop_b}.

Assertion~\ref{Cl_chi_infty_basic_prop_c} follows from Assertion~\ref{Cl_chi_infty_basic_prop_a},
using Theorem~\ref{Thm_tdmu_convergence}\ref{Thm_tdmu_convergence_a} if $\td\mu_t \in \PP^1(\XX^\infty_t)$ for all $t \in Q$.
For the general case, we can either argue as in the proof of Theorem~\ref{Thm_tdmu_convergence_reverse} or establish Assertion~\ref{Cl_chi_infty_basic_prop_c} first for conjugate heat kernel and then use the reproduction formula combined with standard parabolic estimates.

To see Assertion~\ref{Cl_chi_infty_basic_prop_d}, fix some $t^* \in \td I \cap Q$.
Then $u_{\infty,t^*}$ is $L$-Lipschitz for some $L < \infty$.
Define $u^{*} : Z_{t^*} \to \IR$, $i \in \IN$, by
\[ u^{*}(z) := \inf_{z^\infty \in \XX^\infty_{t^*}} \big( L d^Z_{t*} ( z, \varphi^\infty_{t^*} (z^\infty) ) + u_{\infty, t^*} (z^\infty) \big). \]
Note that $u^*$ is $L$-Lipschitz and $u^* \circ \varphi^\infty_{t^*} = u_{\infty, t^*}$.
Let $(u_{i, t})_{t \in I' \cap I^i, t \geq t^*}$ be the heat flows on $\XX^i$ with initial condition $u_{i, t^*} = u^* \circ \varphi^i_{t*}$, $i \in \IN$.
We claim that $(u_{i,t})_{t \in I' \cap I^i, t \geq t^*} \to (u_{\infty, t})_{t \in I', t \geq t^*}$ in the sense of Assertion~\ref{Cl_chi_infty_basic_prop_b}.
To see this, consider some points $y_i \in \XX^i$ with $y_i \to y_\infty \in \XX^\infty$, $\tf (y_\infty) > t^*$.
Then by Claim~\ref{Cl_CHF_W1_conv} we have $(\varphi^i_{t^*})_* \nu^i_{y_i; t^*} \to (\varphi^\infty_{t^*})_* \nu^\infty_{y_\infty; t^*}$ in the $W_1$-sense and therefore
\begin{multline*}
 u_i(y_i) = \int_{\XX^i_{t^*}} (u^* \circ \varphi^i_{t*}) \, d\nu^i_{y_i;t^*}
= \int_{Z_{t^*}} u^* \, d((\varphi^i_{t*})_* \nu^i_{y_i;t^*} ) \\
\longrightarrow \int_{Z_{t^*}} u^* \, d((\varphi^\infty_{t*})_* \nu^\infty_{y_\infty;t^*}) 
= \int_{\XX^\infty_{t^*}} u_{\infty, t^*} \, d\nu^\infty_{y_\infty;t^*}
= u_\infty (y_\infty). \
\end{multline*}
So by Assertion~\ref{Cl_chi_infty_basic_prop_b}, the function $u_\infty \circ \phi_\infty$ restricted to $M' \times (Q \cap (t^*, \infty))$ can be extended to a smooth function solving the heat equation.
Since $t^* \in \td I \cap Q$ was arbitrary, this proves Assertion~\ref{Cl_chi_infty_basic_prop_d}.

Assertions~\ref{Cl_chi_infty_basic_prop_e}, \ref{Cl_chi_infty_basic_prop_f} follow by passing Proposition~\ref{Prop_HK_regular}\ref{Prop_HK_regular_a} to the limit for an appropriate constant $C^*$ and after possibly shrinking $r^*$.

To see Assertion~\ref{Cl_chi_infty_basic_prop_g} choose $y^*_{1,i} \in \XX^i_{t_1}$, $y_{2,i} \in \XX^i_{t_2}$ with $\varphi^i_{t_1} (y^*_{1, i} ) \to \varphi^\infty_{t_1} (y^*_1)$, $\varphi^i_{t_2} (y_{2, i} ) \to \varphi^\infty_{t_2} (y_2)$ in $Z_{t_1}$ and $Z_{t_2}$, respectively.
Then for large $i$ we have $B( y^*_{1,i} , \tfrac12 C^* \sqrt{t_2 - t_1}) \subset \phi_{\infty, t_1} ( M')$ and 
\[ d^{\XX^i_{t_1}}_{W_1} ( \delta_{y^*_{1,i}}, \nu^i_{y_{2,i}; t_1} ) \xrightarrow[i \to \infty]{\qquad} d^{\XX^\infty_{t_1}}_{W_1} ( \delta_{y^*_{1}}, \nu^\infty_{y_{2}; t_1} )
 \leq \sqrt{ \Var ( \delta_{y^*_{1}}, \nu^\infty_{y_{2}; t_1} )} 
 \leq \sqrt{H (t_2 -t_1)}. \]
 Choose $H$-centers $y^{**}_i \in \XX^i_{t_1}$ of $y_2$.
 Then for large $i$ we have
\begin{multline*}
 d^i_{t_1} (y^*_{1,i} , y^{**}_{i})
\leq d^{\XX^i_{t_1}}_{W_1} ( \delta_{y^*_{1,i}}, \nu^i_{y_{2,i}; t_1} )
+ d^{\XX^i_{t_1}}_{W_1} ( \nu^i_{y_{2,i}; t_1}, \delta_{y^{**}_{i}} ) \\
\leq 2\sqrt{H (t_2 -t_1)} + \sqrt{ \Var ( \nu^i_{y_{2,i}; t_1}, \delta_{y^{**}_{i}} )} 
\leq 3\sqrt{H (t_2 -t_1)} 
\end{multline*}
 and thus $B( y^{**}_i , (\tfrac12 C^* - 3 H^{1/2})\sqrt{ t_2 - t_1}) \subset \phi_{\infty,t_1} ( M')$.
 So for sufficiently large $C^*$, we obtain from Proposition~\ref{Prop_HK_regular}\ref{Prop_HK_regular_b} that $y_{2,i} \in \phi_{i,t_2} (M')$ for large $i$.
 Write $y^*_{1,i} = \phi_{i,t_1} ( y'_{1,i})$, $y_{2,i} = \phi_{i,t_2} ( y'_{2,i})$ for $y'_{1,i}, y'_{2,i} \in M'$.
 The second part of Proposition~\ref{Prop_HK_regular}\ref{Prop_HK_regular_b} implies that for large $i$ and sufficiently large $C^*$ we have
 \begin{equation} \label{eq_yp1i_yp2i}
  d_{g^i_{t_2}} (y'_{1,i}, y'_{2,i}) \leq C(H) \sqrt{t_2 - t_1} .
\end{equation}
 So, assuming $C^*$ to be large, we know that $B(y'_{2,i}, \sqrt{t_2 -t_1}) \subset \phi_{i, t_2} (M')$ for large $i$.
 So both $y'_{1,i}, y'_{2,i} \in M'$ remain in a compact subset of $M'$ and therefore by construction of $\phi_\infty$ we have $y'_{1,i} \to y'_1$, $y'_{2,i} \to y'_2$ with $y^*_1 = \phi_{\infty, t_1} (y'_1)$, $y_2 = \phi_{\infty, t_2} (y'_2)$.
 Passing (\ref{eq_yp1i_yp2i}) to the limit implies that $d_{g'_{t_2}} (y'_{1}, y'_{2}) \leq C(H) \sqrt{t_2 - t_1}$.
This proves Assertion~\ref{Cl_chi_infty_basic_prop_g}  for sufficiently large $C^*$.
\end{proof}

\begin{Claim} \label{Cl_extend_chi}
If $r_* \leq \ov{r}_*(\alpha)$, then we can extend $(\phi_{\infty,t})_{t \in Q}$ by a unique family of maps $(\phi_{\infty,t} : M' \to \XX^\infty_t)_{t \in I' \setminus Q}$ to a family $(\phi_{\infty,t})_{t \in I'}$ such that the following is true, after possibly adjusting $C^*$:
\begin{enumerate}[label=(\alph*)]
\item \label{Cl_extend_chi_a} For any $y \in M'$ and any $t_1, t_2 \in I'$ with $t_1 \leq t_2$ we have
\begin{equation} \label{eq_dW1_nu_chi_all_I}
 d^{\XX^\infty_{t_1}}_{W_1} ( \delta_{\phi_{\infty, t_1} (y)}, \nu^\infty_{\phi_{\infty, t_2} (y); t_1} ) \leq C^* \sqrt{t_2 - t_1}. 
\end{equation}
\item \label{Cl_extend_chi_b} For any $t_1, t_2 \in I'$ with $t_1 \leq t_2$ the following is true.
If $y^*_1 \in \XX^\infty_{t_1}$ is an $H$-center of some point $y_2 \in \XX_{t_2}^\infty$ and $y^*_1 =  \phi_{\infty,t_1} ( y'_1 )$ for $y'_1 \in M'$ with the property that $B_{g'_{t_1}} (y'_1, C^* \sqrt{t_2 - t_1}) \subset M'$ is relatively compact, then $y_2 \in \phi_{\infty,t_2} (M')$  and if $y_2 = \phi_{\infty, t_2} (y'_2)$, then  $d_{g'_{t_2}} (y'_1, y'_2) \leq C^* \sqrt{t_2 - t_1}$.
\item \label{Cl_extend_chi_c} $x_\infty = \phi_{\infty, \tf (x_\infty)} (x')$.
\end{enumerate}
\end{Claim}

\begin{proof}
Fix $y \in M'$ and $t \in I' \setminus Q$ for a moment and let us define $\phi_t^\infty (y)$.
Note that $t \neq \sup I'$, so there are times $t' \in Q$ with $t' > t$ such that $t' - t$ is arbitrarily small.
For any two times $t'_1, t'_2 \in Q$ with $t < t'_1 < t'_2$ we have by Claim~\ref{Cl_chi_infty_basic_prop}\ref{Cl_chi_infty_basic_prop_e}
\[ d^{\XX^\infty_{t}}_{W_1} ( \nu^\infty_{\phi_{\infty, t'_1} (y); t}, \nu^\infty_{\phi_{\infty, t'_2} (y); t} )
\leq d^{\XX^\infty_{t'_1}}_{W_1} ( \delta_{\phi_{\infty, t'_1} (y)}, \nu^\infty_{\phi_{\infty, t'_2} (y); t'_1} ) \leq C^* \sqrt{t'_2 - t'_1} \leq C^* \sqrt{t'_2 - t}. \]
So since for any $t' \in Q$ with $t' > t$ we have $\Var ( \nu^\infty_{\phi_{\infty,t'} (y); t}) \leq H (t' - t)$, we obtain
\[ \nu^\infty_{\phi_{\infty, t'} (y); t} \xrightarrow[t' \searrow t, t' \in Q]{\quad W_1 \quad} \delta_{y'} \]
for some unique $y' \in \XX^\infty_t$.
Set $\phi_{\infty,t} (y) := y'$.
By repeating this construction for all $y \in M$ and $t \in I' \setminus Q$, we can construct a family $(\phi_{\infty,t} : M' \to \XX^\infty_t)_{t \in I' \setminus Q}$ such that for any $t \in I' \setminus Q$, $t' \in Q$ with $t' > t$ and any $y \in M'$ we have
\[ d^{\XX^\infty_{t}}_{W_1} (\delta_{\phi_{\infty,t} (y)}, \nu^\infty_{\phi_{\infty, t'} (y); t})  \leq  C^* \sqrt{t' - t}. \]
Combined with Claim~\ref{Cl_chi_infty_basic_prop}\ref{Cl_chi_infty_basic_prop_e}, this shows (\ref{eq_dW1_nu_chi_all_I}) if $t_2 \in Q$.
Assume now that $t_2 \in I' \setminus Q$.
As before, we have $t_2 \neq \sup I$, so there are times $t'' \in Q$ with $t'' > t_2$ such that $t'' - t_2$ is arbitrarily small.
For any such $t''$ we have
\begin{multline*}
 d^{\XX^\infty_{t_1}}_{W_1} ( \nu^\infty_{\phi_{\infty,t_2} (y); t_1}, \delta_{\phi_{\infty,t_1} (y)} )
\leq d^{\XX^\infty_{t_1}}_{W_1} ( \nu^\infty_{\phi_{\infty,t_2} (y); t_1}, \nu^\infty_{\phi_{\infty,t''} (y); t_1} ) + d^{\XX^\infty_{t_1}}_{W_1} (\nu^\infty_{\phi_{\infty,t''} (y); t_1}, \delta_{\phi_{\infty,t_1} (y)} ) \\
\leq d^{\XX^\infty_{t_2}}_{W_1} (\delta_{\phi_{\infty,t_2} (y)}, \nu^\infty_{\phi_{\infty,t''} (y); t_2} ) + C^* \sqrt{t'' - t_1}
\leq C^* \sqrt{t'' - t_2}  + C^* \sqrt{t'' - t_1}  . 
\end{multline*}
Letting $t'' \searrow t_2$ implies (\ref{eq_dW1_nu_chi_all_I}).

Next, we prove Assertion~\ref{Cl_extend_chi_b}.
Denote by $C^{*}$ the maximum of the corresponding constants from Claim~\ref{Cl_chi_infty_basic_prop}\ref{Cl_chi_infty_basic_prop_g} and Assertion~\ref{Cl_extend_chi_a} of this claim.
We will show Assertion~\ref{Cl_extend_chi_b} for $C^*$ replaced with some constant $C^{**} \geq C^*$, which we will determine in the course of this proof.
Fix $t_1, t_2 \in I'$ with $t_1 \leq t_2$ and choose $y^*_1, y'_1, y_2$ as in Assertion~\ref{Cl_extend_chi_b}.
So $B_{g'_{t_1}} (y'_1, C^{**} \sqrt{t_2 -t_1}) \subset M'$ is relatively compact.

The case $t_1, t_2 \in Q$ is clear by Claim~\ref{Cl_chi_infty_basic_prop}\ref{Cl_chi_infty_basic_prop_g}.
Suppose next that $t_1 \in Q$ and $t_2 \in I' \setminus Q$.
Fix some sequence $t^*_j \in Q$ with $t^*_j \searrow t_2$.
By Proposition~\ref{Prop_topology_properties}\ref{Prop_topology_properties_f}, we can choose points $y^{*}_{2,j} \in \XX^\infty_{t^*_j}$ such that
\begin{equation} \label{eq_ystar_2j_convergence}
 d^{\XX^\infty_{t_2}}_{W_1} (\nu^\infty_{y^*_{2,j}; t_2}, \delta_{y_2}) \to 0. 
\end{equation}
Then
\[ d^{\XX^\infty_{t_1}}_{W_1} (\nu^\infty_{y^{*}_{2,j}; t_1}, \nu^\infty_{y_{2}; t_1}) \leq  d^{\XX^\infty_{t_2}}_{W_1} (\nu^\infty_{y^{*}_{2,j}; t_2}, \delta_{y_2})  \to 0. \]
Choose $H$-centers $y^*_{1,j} \in \XX^\infty_{t_1}$ of $y^{*}_{2,j}$.
Then for large $j$ we have
\begin{multline*}
d_{t_1}^\infty (y^*_1, y^*_{1,j})
\leq d^{\XX^\infty_{t_1}}_{W_1} (\delta_{y^*_1}, \nu^\infty_{y_2;t_1})
+ d^{\XX^\infty_{t_1}}_{W_1} (  \nu^\infty_{y_2;t_1}, \nu^\infty_{y^*_{2,j}; t_1}) 
+ d^{\XX^\infty_{t_1}}_{W_1} (  \nu^\infty_{y^*_{2,j}; t_1}, \delta_{y^*_{1,j}}) \\
\leq 2\sqrt{\Var (\delta_{y^*_1}, \nu^\infty_{y_2;t_1})} + \sqrt{\Var  (  \nu^\infty_{y^*_{2,j}; t_1}, \delta_{y^*_{1,j}})}
\leq 4 \sqrt{H (t_2 -t_1)}.
\end{multline*}
So if $C^{**} \geq \underline{C}^{**}(C^*, H)$ then by Claim~\ref{Cl_chi_infty_basic_prop}\ref{Cl_chi_infty_basic_prop_g}, for large $j$ we have $y^*_{2,j} = \phi_{\infty,t^*_j} (y'_{2,j})$ for some $y'_{2,j} \in M'$ with $d_{g'_{t^*_j}} ( y'_1, y'_{2,j} ) \leq C^* \sqrt{t^*_j - t_1}$.
If $C^{**} \geq \underline{C}^{**}(C^*)$, then this implies that the sequence $y'_{2,j}$ remains in a compact subset of $M'$.
So after passing to a subsequence, we may assume that $y'_{2,j} \to y'_2 \in M'$ and $d_{g'_{t_2}} ( y'_1, y'_{2} ) \leq C^* \sqrt{t_2 - t_1}$.
Then using (\ref{eq_ystar_2j_convergence}) and Assertion~\ref{Cl_extend_chi_a} we find that
\begin{align*}
 d^\infty_{t_2} ( y_2, \phi_{\infty,t_2} (y'_{2}))
&\leq d^{\XX^\infty_{t_2}}_{W_1} ( \delta_{y_2}, \nu^\infty_{\phi_{\infty,t^*_j} (y'_{2,j}); t_2}) 
+ d^{\XX^\infty_{t_2}}_{W_1} ( \nu^\infty_{\phi_{\infty, t^*_j} (y'_{2,j}); t_2}, \nu^\infty_{\phi_{\infty, t^*_j} (y'_{2}); t_2}) \\
&\qquad +  d^{\XX^\infty_{t_2}}_{W_1} ( \nu^\infty_{\phi_{\infty, t^*_j} (y'_{2}); t_2}, \delta_{\phi_{\infty, t_2} (y'_{2})} ) \\
&\leq d^{\XX^\infty_{t_2}}_{W_1} ( \delta_{y_2}, \nu^\infty_{y^*_{2,j}; t_2}) 
+ d^\infty_{t^*_j} ( \phi_{\infty, t^*_j} (y'_{2,j}), \phi_{\infty, t^*_j} (y'_{2})) + C^* \sqrt{ t^*_j - t_2 } \\
&= d^{\XX^\infty_{t_2}}_{W_1} ( \delta_{y_2}, \nu^\infty_{y^*_{2,j}; t_2}) + d_{g'_{t^*_j}} (y'_{2,j}, y'_{2}) + C^* \sqrt{ t^*_j - t_2 } \longrightarrow 0.
\end{align*}
So $y_2 = \phi_{\infty,t_2} (y'_{2})$, which is what we wanted to show.

Lastly, consider the case $t_1 \in I' \setminus Q$.
After adjusting $C^*$, we may assume that Assertion~\ref{Cl_extend_chi_b} holds for $C^*$ if $t_1 \in Q$.
Suppose again that $B_{g'_{t_1}} (y'_1, C^{**} \sqrt{t_2 -t_1}) \subset M'$ is relatively compact for some $C^{**} \geq C^*$, which we will determine in the course of the proof.
Note that since $I'$ is left-open, we can choose times $t' \in Q$, $t' < t_1$ with $t_1 - t'$ arbitrarily small.
If $C^{**} \geq C^* + 20 \sqrt{H}$, then for $t'$ close to $t_1$ the ball $B_{g'_{t'}} (y'_1, (C^* + 10 \sqrt{H}) \sqrt{ t_2 - t'}) \subset M'$ is relatively compact.
So for $t'$ close to $t_1$ we have, using Assertion~\ref{Cl_extend_chi_a},
\begin{multline*}
 d^{\XX^\infty_{t'}}_{W_1} ( \delta_{\phi_{\infty,t'} (y'_1)}, \nu^\infty_{y_2;t'} )
\leq d^{\XX^\infty_{t'}}_{W_1} ( \delta_{\phi_{\infty, t'} (y'_1)}, \nu^\infty_{\phi_{\infty, t_1} (y'_1);t'} ) 
+ d^{\XX^\infty_{t'}}_{W_1} ( \nu^\infty_{\phi_{\infty,t_1} (y'_1);t'}, \nu^\infty_{y_2;t'} ) \\
\leq C^* \sqrt{ t_1 - t' }
+ d^{\XX^\infty_{t_1}}_{W_1} ( \delta_{y^*_1}, \nu^\infty_{y_2;t'} ) 
\leq C^* \sqrt{t_1 - t' } + \sqrt{\Var (\delta_{y^*_1}, \nu^\infty_{y_2;t'} )} \\
\leq C^* \sqrt{t_1 - t' } + \sqrt{H (t_2 - t_1)} \leq 2\sqrt{H (t_2 - t_1)}.
\end{multline*}
So for any $H$-center $y^{**}_1 \in \XX^\infty_{t'}$ of $y_2$ we have for $t'$ close to $t_1$
\begin{multline*}
 d^\infty_{t'} (\phi_{\infty,t'} (y'_1), y^{**}_1) 
\leq 
 d^{\XX^\infty_{t'}}_{W_1} ( \delta_{\phi_{\infty,t'} (y'_1)}, \nu^\infty_{y_2;t'} ) +  d^{\XX^\infty_{t'}}_{W_1} ( \nu^\infty_{y_2;t'}, \delta_{y^{**}_1} ) \\
\leq 2\sqrt{H (t_2 - t_1)} + \sqrt{\Var ( \nu^\infty_{y_2;t'}, \delta_{y^{**}_1} )}
\leq 2\sqrt{H (t_2 - t_1)} + \sqrt{H (t_2 - t')} \leq 10 \sqrt{H (t_2 - t_1)}. 
\end{multline*}
So if $t'$ is sufficiently close to $t^*_1$, then $y^{**}_1 = \phi_{\infty,t'} (y'')$ for some $y'' \in M'$ with 
\[ B_{g'_{t'}} (y'', C^* \sqrt{t_2 - t'}) \subset B_{g'_{t'}} (y'_1, t',(C^* + 9 \sqrt{H})  \sqrt{t_2 - t'}) \subset M', \]
 which shows that $B_{g'_{t'}} (y'', C^* \sqrt{t_2 - t'}) \subset M'$ is relatively compact and therefore $y_2 \in \phi_{\infty,t_2} (M')$.
 Moreover, if $y_2 = \phi_{\infty,t_2} (y'_2)$, then for $t'$ close to $t_1$
 \begin{multline*}
 d_{g'_{t_2}} (y'_1, y'_2) 
 \leq d_{g'_{t_2}} (y'_1, y'') + d_{g'_{t_2}} (y'', y'_2)
 \leq Cd_{g'_{t'}} (y'_1, y'') + C^* \sqrt{t_2 - t'} \\
 =  C d^\infty_{t'} (\phi_{\infty,t'} (y'_1), y^{**}_1)  + C^* \sqrt{t_2 - t'}
 \leq ( C \sqrt{H} + 2C^*) \sqrt{t_2 - t_1} . 
\end{multline*}
This proves Assertion~\ref{Cl_extend_chi_b} for $C^{**} \geq \underline{C}^{**}(C^*)$.

For Assertion~\ref{Cl_extend_chi_c}, fix some $t \in Q$, $t < \tf (x_\infty)$ and let $z \in \XX^\infty_t$ be an $H$-center of $x_\infty$.
Using Claim~\ref{Cl_chi_infty_basic_prop}\ref{Cl_chi_infty_basic_prop_f}, we have
\begin{multline*}
 d^\infty_t (\phi_{\infty,t} (x'), z)
  \leq d^{\XX^\infty_t}_{W_1} (\delta_{\phi_{\infty,t} (x')}, \nu^\infty_{x_\infty; t}) + d^{\XX^\infty_t}_{W_1} ( \nu^\infty_{x_\infty; t}, \delta_z) \\
\leq C^* \sqrt{\tf (x_\infty) - t} + \sqrt{ \Var  ( \nu^\infty_{x_\infty; t}, \delta_z)}
\leq (C^* + H^{1/2}) \sqrt{\tf (x_\infty) - t}. 
\end{multline*}
So for $t$ sufficiently close to $\tf(x_\infty)$ we have $z = \phi_{\infty,t} (z')$ for some $z' \in M'$ close to $x'$ and thus by Assertion~\ref{Cl_extend_chi_b} we have $x_\infty = \phi_{\infty, \tf(x_\infty)} (x'')$ for some $x'' \in M'$ with $d_{g'_{\tf(x_\infty)}} (z', x'') \leq C^* \sqrt{ \tf(x_\infty) - t }$.
It follows that
\begin{multline*}
 d_{g'_{\tf(x_\infty)}} (x', x'')
\leq d_{g'_{\tf(x_\infty)}} (x', z') + d_{g'_{\tf(x_\infty)}} (z', x'')
\leq C d_{g'_{t}} (x', z') + C^* \sqrt{ \tf(x_\infty) - t } \\
=  C d^\infty_t (\phi_{\infty,t} (x'), z) + C^* \sqrt{ \tf(x_\infty) - t }
\leq C(C^* + H^{1/2}) \sqrt{\tf (x_\infty) - t}. 
\end{multline*}
Letting $t \nearrow \tf(x_\infty)$ implies $x' = x''$.
\end{proof}

\begin{Claim} \label{Cl_chi_infty_local_isometric_emb}
For all $t \in I'$ the map $\phi_{\infty,t} : (M', d_{g'_t}) \to (\XX^\infty_t, d^\infty_t)$ is a local isometric embedding.
Moreover, for any two $y'_1, y'_2 \in M'$ with the property that any curve between both points of time-$t$-length $\leq d^\infty_t (\phi_{\infty,t} (y'_1), \phi_{\infty,t}(y'_2))$ remains in a compact subset of $M'$ we have $d^\infty_t (\phi_{\infty,t} (y'_1), \phi_{\infty,t}(y'_2)) =  d_{g'_t} (y'_1, y'_2)$.
\end{Claim}

\begin{proof}
If $t \in Q$, then there is nothing to prove; note that in this case $\XX^\infty$ is intrinsic at time $t$.
So assume that $t \in I' \setminus Q$.
It suffices to show the last statement of the claim.
Then there are times $t' \in Q$ with $t' > t$ such that $t' - t$ is arbitrarily small.
So, using Claim~\ref{Cl_extend_chi}\ref{Cl_extend_chi_a},
\begin{multline*}
 d^\infty_t (\phi_{\infty,t} (y'_1), \phi_{\infty,t}(y'_2)) \\
\leq \liminf_{t' \searrow t, t' \in Q} \big( d^{\XX^\infty_t}_{W_1} ( \delta_{\phi_{\infty,t} (y'_1)}, \nu^\infty_{\phi_{\infty,t'} (y'_1); t}) + d^{\XX^\infty_t}_{W_1} ( \nu^\infty_{\phi_{\infty,t'} (y'_1); t},\nu^\infty_{\phi_{\infty,t'} (y'_2); t}) + d^{\XX^\infty_t}_{W_1} ( \nu^\infty_{\phi_{\infty,t'} (y'_2); t}, \delta_{\phi_{\infty,t} (y'_2)}) \big) \\
\leq \liminf_{t' \searrow t, t' \in Q} d^\infty_{t'} (\phi_{\infty, t'} (y'_1), \phi_{\infty,t'} (y'_2))
\leq  \liminf_{t' \searrow t, t' \in Q} d_{g'_{t'}} (y'_1, y'_2)
= d_{g'_t} (y'_1, y'_2) .
\end{multline*}
Similarly, we obtain
\begin{multline*}
 d^\infty_t (\phi_{\infty,t} (y'_1), \phi_{\infty,t}(y'_2))
\geq \limsup_{t' \nearrow t, t' \in Q} d^{\XX^\infty_{t'}}_{W_1} ( \nu^\infty_{\phi_{\infty,t} (y'_1); t'},\nu^\infty_{\phi_{\infty,t} (y'_2); t'}) \\
\geq \limsup_{t' \nearrow t, t' \in Q} \big( d^\infty_{t'} (\phi_{\infty,t'} (y'_1), \phi_{\infty,t'} (y'_2)) - d^{\XX^\infty_{t'}}_{W_1} ( \delta_{\phi_{\infty,t'} (y'_1)}, \nu^\infty_{\phi_{\infty,t} (y'_1); t'})  - d^{\XX^\infty_{t'}}_{W_1} ( \delta_{\phi_{\infty,t'} (y'_2)}, \nu^\infty_{\phi_{\infty,t} (y'_2); t'}) \big) \\
= \limsup_{t' \nearrow t, t' \in Q} d^\infty_{t'} (\phi_{\infty,t'} (y'_1), \phi_{\infty,t'} (y'_2))
=  \limsup_{t' \nearrow t, t' \in Q} d_{g'_{t'}} (y'_1, y'_2)
= d_{g'_t} (y'_1, y'_2)  .
\end{multline*}
The second last equality holds due to an openness argument as in the proof of Proposition~\ref{Prop_RR_intrinsic}.
This finishes the proof of the claim.
\end{proof}

\begin{Claim} \label{Cl_chi_infty_homeo}
\begin{enumerate}[label=(\alph*)]
\item \label{Cl_chi_infty_homeo_a} The family $(\phi_{\infty,t})_{t \in I'}$, when viewed as a map $\phi_\infty : M' \times I' \to \XX^\infty$, is a homeomorphism onto its image and $\phi_\infty (M' \times I')$ is open.
\item \label{Cl_chi_infty_homeo_b} For all for all $t \in I'$ the map $\phi_{\infty,t} : (M', d_{g'_t}) \to (\XX^\infty_t, d^\infty_t)$ is a local isometry.
\end{enumerate}
\end{Claim}

\begin{proof}
Fix some $(y',t) \in M' \times I'$.
We first show that for small $r' > 0$ we have $P^* (\phi_\infty(y',t); r') \subset \phi_\infty (M' \times I')$.
To see this, assume without loss of generality that $\XX^\infty$ is intrinsic at time $t - r^{\prime 2}$, consider some point $y_2 \in P^* (\phi_\infty(y',t); r')$ and let $y^*_1 \in \XX^\infty_{t - r^{\prime 2}}$ be an $H$-center of $y_2$.
Then, using Claim~\ref{Cl_extend_chi}\ref{Cl_extend_chi_a},
\begin{multline*}
 d^\infty_{t - r^{\prime 2}} (\phi_\infty(y',t - r^{\prime 2}) , y^*_1 )
\leq d^{\XX^\infty_{t - r^{\prime 2}}}_{W_1} (\delta_{\phi_\infty(y',t - r^{\prime 2})}, \nu^\infty_{\phi_\infty (y', t); t - r^{\prime 2}})
+ d^{\XX^\infty_{t - r^{\prime 2}}}_{W_1} (\nu^\infty_{\phi_\infty (y', t); t - r^{\prime 2}}, \nu^\infty_{y_2; t - r^{\prime 2}}) \\
+ d^{\XX^\infty_{t - r^{\prime 2}}}_{W_1} ( \nu^\infty_{y_2; t - r^{\prime 2}}, \delta_{y^*_1})
\leq C^* r' + r' + \sqrt{ \Var ( \nu^\infty_{y_2; t - r^{\prime 2}}, \delta_{y^*_1})}
\leq C^* r' + r' + (2H)^{1/2} r'. 
\end{multline*}
So for small enough $r'$ we have $y^*_1 = \phi_\infty (y'_1, t - r^{\prime 2})$, where $y'_1$ is close enough to $y'$ such that we can use Claim~\ref{Cl_extend_chi}\ref{Cl_extend_chi_b} to conclude that $y_2 \in \phi_\infty (M' \times I')$.
Moreover, if $y_2 = \phi_\infty (y'_2, t')$, then  $d_{g'_{t'}} (y', y'_2) \leq \sqrt{2} C^*  r'$.
This shows that $\phi_\infty$ is open.

On the other hand, set $P_{r'} := \phi_\infty^{-1} ( P^* (\phi_\infty(y',t); r') )$.
Let $( y'_1,  t_1) \in M' \times I'$ and set $t'' := \min \{ t, t_1 \}$.
If $r' > 0$ is fixed and if $(y'_1, t_1)$ is close enough to $(y',t)$, then we have, using Claim~\ref{Cl_extend_chi}\ref{Cl_extend_chi_a},
\begin{multline*}
 d^{\XX^\infty_{t- r^{\prime 2}}}_{W_1} ( \nu^\infty_{\phi_\infty (y', t); t- r^{\prime 2}}, \nu^\infty_{\phi_\infty (y'_1, t_1); t- r^{\prime 2}})
\leq d^{\XX^\infty_{t''}}_{W_1} ( \nu^\infty_{\phi_\infty (y', t); t''}, \nu^\infty_{\phi_\infty( y'_1, t_1); t''}) \\
\leq d^{\XX^\infty_{t''}}_{W_1} ( \nu^\infty_{\phi_\infty (y', t); t''}, \delta_{\phi_\infty (y', t'')}) + d^\infty_{t''} ( \phi_\infty (y', t''), \phi_\infty (y'_1, t'') ) + d^{\XX^\infty_{t''}}_{W_1} ( \delta_{\phi_\infty (y'_1, t'')}, \nu^\infty_{y'_1; t''}) \\
\leq C^* \sqrt{|t - t_1|} + d_{g'_{t''}} (y', y'_1) < r',
\end{multline*}
and thus $(y'_1, t_1) \in P_{r'}$.
This proves Assertion~\ref{Cl_chi_infty_homeo_a}.

Assertion~\ref{Cl_chi_infty_homeo_b} follows from Assertion~\ref{Cl_chi_infty_homeo_a} and Claim~\ref{Cl_chi_infty_local_isometric_emb}, because the inclusion map $\XX^\infty_t \to \XX^\infty$ is continuous (see Proposition~\ref{Prop_topology_properties}\ref{Prop_topology_properties_b}).
\end{proof}

\begin{Claim} \label{Cl_chi_chart}
\begin{enumerate}[label=(\alph*)]
\item \label{Cl_chi_chart_a} For any conjugate heat flow $(\td\mu_t)_{t \in \td I}$ on $\XX^\infty$ over a right-open subinterval $\td I \subset I'$, there is a smooth function $\td v' \in C^\infty (M' \times \td I)$ with $\square^* \td v' = 0$ such that $d ((\phi_{\infty,t})^* \td\mu_t) =  \td v'_t \, dg'_t$ for all $t \in  \td I$.
\item \label{Cl_chi_chart_b} For any uniformly bounded heat flow $(u_t)_{t \in \td I}$ on $\XX^\infty$ over a left-open subinterval $\td I \subset I'$, the function $u' := u \circ \phi_\infty$, which is defined on $M' \times  \td I$, is smooth and satisfies the heat equation $\square u' = 0$.
\end{enumerate}
\end{Claim}

\begin{proof}
Assertion~\ref{Cl_chi_chart_b} follows from Claim~\ref{Cl_chi_infty_basic_prop}\ref{Cl_chi_infty_basic_prop_d}, Claim~\ref{Cl_chi_infty_homeo}\ref{Cl_chi_infty_homeo_a} and the fact that $u$ is continuous by Proposition~\ref{Prop_topology_properties}\ref{Prop_topology_properties_f_new}.

Suppose now we are in the setting of Assertion~\ref{Cl_chi_chart_a} and let $\td v'$ be the function from Claim~\ref{Cl_chi_infty_basic_prop}\ref{Cl_chi_infty_basic_prop_c}.
So $d ((\phi_{\infty,t})^* \td\mu_t) =  \td v'_t \, dg'_t$ for all $t \in Q\cap \td I$.
We claim that the same holds for all $t \in \td I$.
For this purpose fix $t_0 \in \td I \setminus Q$ and let $u' \in C^\infty_c(M')$, $u' \geq 0$, be an arbitrary smooth function of compact support with the property that $u = u' \circ \phi_{\infty,t_0}$ for some $1$-Lipschitz function $u : \XX^\infty_{t_0} \to \IR$ with $u \equiv 0$ on $\XX^\infty_{t_0} \setminus \phi_{\infty,t_0} (M')$.
It suffices to show that
\begin{equation} \label{eq_int_Mp_up_verify}
 \int_{M'} u' \, \td v'_{t_0} \, dg'_{t_0} = \int_{\XX^\infty_{t_0}}  u \, d\td\mu_{t_0}. 
\end{equation}
To see this, consider the heat flow $(\td u_t)_{t \in I' \cap [t_0, \infty)}$ with initial condition $u$.
We first claim that
\begin{equation} \label{eq_sup_u_outside_chi_to_0}
 \sup_{\XX^\infty_{t} \setminus \phi_{\infty,t} (M')} \td u_t \xrightarrow[t \searrow t_0]{} 0. 
\end{equation}
Suppose not.
Then there is a sequence $x_j \in \XX^\infty_{t_j}  \setminus \phi_{\infty,t_j} (M')$ with $t_j \searrow t_0$ such that 
\[ \td u_{t_j} (x_j) = \int_{\XX^\infty_{t_0}} u \, d\nu^\infty_{x_j; t_0}  \geq c > 0. \]
Let $z_j \in \XX^\infty_{t_0}$ be $H$-centers of $x_j$.
Since $u$ is bounded, we must have $d^\infty_{t_0} (z_j, \supp u ) \leq C (t_j - t_0)^{1/2}$ for some $C < \infty$ by Lemma~\ref{Lem_nu_BA_bound}.
Since $\supp u$ is compact, we may pass to a subsequence such that $z_j \to z_\infty \in \supp u$ within $\XX^\infty_{t_0}$.
So for any $r' > 0$ we have
\begin{multline*}
 d^{\XX^\infty_{t_0 - r^{\prime 2}}}_{W_1} ( \nu^\infty_{z_\infty; t_0 - r^{\prime 2}}, \nu^\infty_{x_j; t_0 - r^{\prime 2}} )
\leq d^{\XX^\infty_{t_0}}_{W_1} ( \delta_{z_\infty}, \nu^\infty_{x_j; t_0} )
\leq d^\infty_{t_0} (z_\infty, z_j) + d^{\XX^\infty_{t_0}}_{W_1} ( \delta_{z_j}, \nu^\infty_{x_j; t_0} ) \\
\leq d^\infty_{t_0} (z_\infty, z_j) + \sqrt{\Var ( \delta_{z_j}, \nu^\infty_{x_j; t_0} )}
\leq d^\infty_{t_0} (z_\infty, z_j) + \sqrt{H (t_j - t_0)}
\to 0,
\end{multline*}
which implies that $x_j \in P^* (z_\infty; r')$ for large $j$.
This, however contradicts $x_j \not\in \phi_{\infty,t_j} (M')$ via Claim~\ref{Cl_chi_infty_homeo}\ref{Cl_chi_infty_homeo_a} for large $j$.
So (\ref{eq_sup_u_outside_chi_to_0}) holds.

We can now verify (\ref{eq_int_Mp_up_verify}) as follows:
\begin{multline*}
  \int_{\XX^\infty_{t_0}}   u \, d\td\mu_{t_0} 
= \lim_{t \searrow t_0, t \in Q} \int_{\XX^\infty_{t}}  \td u_{t} \, d\td\mu_{t}
= \lim_{t \searrow t_0,t \in Q} \int_{\phi_{\infty,t} (M')}  \td u_{t} \, d\td\mu_{t} \\
= \lim_{t \searrow t_0, t \in Q} \int_{M'} ( \td u \circ \phi_{\infty,t} )\, \td v'_{t} \, dg'_t 
= \int_{M'} u' \, \td v'_{t_0} \, dg'_{t_0}. 
\end{multline*}
This finishes the proof.
\end{proof}

It follows that $x_\infty \in U:= \phi_\infty (M' \times I') \subset \RR^\infty$.
Define $\psi_i  := \phi_i \circ \phi_\infty^{-1}: U \to \XX^i$.
Let us now verify Assertions~\ref{Thm_smooth_convergence_a}, \ref{Thm_smooth_convergence_c}, \ref{Thm_smooth_convergence_e}, \ref{Thm_smooth_convergence_f} of Theorem~\ref{Thm_smooth_convergence}.
The convergence of the metric in Assertion~\ref{Thm_smooth_convergence_a} is true by construction and the time vector fields converge trivially since the maps $\psi_i$ are $\partial_{\tf}$-preserving.
The last part of Assertion~\ref{Thm_smooth_convergence_a} and Assertions~\ref{Thm_smooth_convergence_e}, \ref{Thm_smooth_convergence_f}  of Theorem~\ref{Thm_smooth_convergence} follow from Claims~\ref{Cl_chi_infty_basic_prop}\ref{Cl_chi_infty_basic_prop_a}, \ref{Cl_chi_chart}\ref{Cl_chi_chart_a} and the local derivative estimates on the functions $v'_i$ after shrinking $r_*$ slightly.

Next, we verify Assertion~\ref{Thm_smooth_convergence_c}.
So fix some sequence $\td{x}_i \in \XX^i$, and a point $\td{x}_\infty \in U$.
Suppose first that $\td{x}_i \to \td{x}_\infty$ within $\CF$.
Fix some $t \in Q$, $t < \tf(\td{x}_\infty)$ close to $\tf (\td{x}_\infty)$. 
Then $(\varphi^i_t)_* \nu^i_{\td{x}_i; t} \to (\varphi^\infty_t)_* \nu^\infty_{\td{x}_\infty;t}$ in $W_1$.
Choose $H$-centers $\td x^*_i \in \XX^i_{t}$ of $\td{x}_i$ for large $i \leq \infty$.
For $t$ sufficiently close to $\tf (\td{x}_\infty)$ we have $\td x^*_\infty \in U$.
So by the construction of $\phi_\infty$ we have
\begin{align}
 \limsup_{i \to \infty} d^i_t (\td x^*_i, \psi_i(\td x^*_\infty)) 
&\leq \limsup_{i \to \infty} \big( d^{\XX^i_t}_{W_1} (\delta_{\td x^*_i}, \nu^i_{\td{x}_i; t} )
+ d^{Z_t}_{W_1} ((\varphi^i_t)_* \nu^i_{\td{x}_i; t} , (\varphi^\infty_t)_* \nu^\infty_{\td{x}_\infty; t} ) \notag \\
&\qquad\qquad\qquad +  d^{\XX^\infty_t}_{W_1} (\nu^\infty_{\td{x}_\infty; t}, \delta_{\td x^*_\infty} )
+ d^{Z}_t ( \varphi^\infty_t (\td x^*_\infty), \varphi^i_t (\psi_i(\td x^*_\infty)) ) \big) \notag \\
&\leq \limsup_{i \to \infty} \Big( \sqrt{ \Var (\delta_{\td x^*_i}, \nu^i_{\td{x}_i; t} )} + \sqrt{\Var  (\nu^\infty_{\td{x}_\infty; t}, \delta_{\td x^*_\infty} )} \Big)
\leq 2 \sqrt{ H (\tf (\td{x}_\infty) - t)}. \label{eq_ditdx_star_i_infty}
\end{align}
So by (\ref{eq_ditdx_star_i_infty}) and Assertion~\ref{Thm_smooth_convergence_a}, which we have already established, and Proposition~\ref{Prop_HK_regular}\ref{Prop_HK_regular_b} we have $\td x_i \in \RR^i$ for large $i$ and, $d^i_{\tf (\td x_i)} (\td x_i, \td x^*_i (\tf (\td x_i)) ) \leq C \sqrt{\tf(\td x_i) - t}$.
It follows that $\td{x}_i \in \psi_i (U)$ for large $i$ and by letting $t \nearrow \tf (\td{x}_\infty)$ we obtain $\psi_i^{-1}(\td{x}_i) \to \td{x}_\infty$.

Vice versa, assume that we have $\psi_i^{-1}(\td{x}_i) \to \td{x}_\infty$ and fix some times $t_1, t_2 \in Q$, $t_1 \leq t_2 < \tf(\td{x}_\infty)$. 
For $t_2$ sufficiently close to $\tf (\td{x}_\infty)$ the points $\td x^*_i := \td{x}_i (t_2)$ exist for large $i \leq \infty$ and $\td x^*_\infty \in U$.
Since the maps $\psi_i$ are $\partial_\tf$-preserving, we also have $\psi_i^{-1}(\td x^*_i) \to \td x^*_\infty$.
So by Claim~\ref{Cl_construct_chi_infty_t} and Theorem~\ref{Thm_tdmu_convergence} we have $(\nu^i_{\td x^*_i;t})_{t \in I^i, t < t_2} \to (\nu^\infty_{\td x^*_\infty;t})_{t \in I^\infty, t < t_2}$ within $\CF$ on compact time-intervals.
So, using Proposition~\ref{Prop_HK_regular}\ref{Prop_HK_regular_a},
\begin{multline*}
 \limsup_{i \to \infty} d^{Z_{t_1}}_{W_1} ( (\varphi^i_{t_1})_* \nu^i_{\td{x}_i; t_1}, (\varphi^\infty_{t_1})_* \nu^\infty_{\td{x}_\infty; t_1}) \\
\leq \limsup_{i \to \infty}  \big( 
d^{\XX^i_{t_1}}_{W_1} (  \nu^i_{\td{x}_i; t_1}, \nu^i_{\td x^*_i; t_1})
+ d^{Z_{t_1}}_{W_1} ( (\varphi^i_{t_1})_*  \nu^i_{\td x^*_i; t_1}), (\varphi^\infty_{t_1})_* \nu^\infty_{\td x^*_\infty; t_1})
+ d^{\XX^\infty_{t_1}}_{W_1} (  \nu^\infty_{\td x^*_\infty; t_1}, \nu^\infty_{\td{x}_\infty; t_1}) \big) \\
\leq \limsup_{i \to \infty}  \big( d^{\XX^i_{t_2}}_{W_1} (  \nu^i_{\td{x}_i; t_2}, \delta_{x^*_i}) + d^{\XX^\infty_{t_2}}_{W_1} ( \delta_{\td{x}^*_\infty}, \nu^i_{\td{x}_\infty; t_2}) \big)
\leq 2C\sqrt{ \tf(\td{x}_\infty) - t_2}. 
\end{multline*}
Letting $t_2 \nearrow \tf(\td{x}_\infty)$ implies that for any $t_1 \in Q$, $t_1 < \tf (\td{x}_\infty)$ we have
\[  \limsup_{i \to \infty} d^{Z_{t_1}}_{W_1} ( (\varphi^i_{t_1})_* \nu^i_{\td{x}_i; t_1}, (\varphi^\infty_{t_1})_* \nu^\infty_{\td{x}_\infty; t_1}) = 0 . \]
This implies that $(\nu^i_{\td{x}_i; t}) \to (\nu^\infty_{\td{x}_\infty; t})$ within $\CF$ on compact time-intervals by Theorem~\ref{Thm_tdmu_convergence}, as desired.

Lastly, observe that due to Assertions~\ref{Thm_smooth_convergence_a}, \ref{Thm_smooth_convergence_c} and Claim~\ref{Cl_inf_vp_positive}, we even have $U \subset \RR^*$.
\end{proof}
\bigskip

\begin{proof}[Proof of Theorem~\ref{Thm_smooth_convergence}.]
For any $x \in \RR^*$ pick a neighborhood $x \in U^{x} \subset \RR^\infty$ and a sequence of diffeomorphisms $\psi^{x}_i : U^x \to V^x_i \subset \RR^i$ according to Lemma~\ref{Lem_smooth_conv_local}.

\begin{Claim} \label{Cl_psi_x1_psi_x2_smooth_conv}
For any two points $x_1, x_2 \in \RR^*$ the sequence of maps
\[ (\psi^{x_2}_i)^{-1} \circ \psi^{x_1}_i :( \psi^{x_1}_i)^{-1} ( V^{x_1}_i \cap V^{x_2}_i ) \longrightarrow  (\psi^{x_2}_i)^{-1} ( V^{x_1}_i \cap V^{x_2}_i ) \]
converges to the identity map on $U^{x_1} \cap U^{x_2}$ in $C^\infty_{\loc}$.
\end{Claim}

\begin{proof}
Let $x \in U^{x_1} \cap U^{x_2}$.
Then by Assertion~\ref{Thm_smooth_convergence_c}, which is ensured by Lemma~\ref{Lem_smooth_conv_local}, we have $\psi^{x_1}_i (x) \to x$ within $\CF$ and therefore, again by Assertion~\ref{Thm_smooth_convergence_c}, $\psi^{x_1}_i(x) \in V^{x_2}_i$ for large $i$ and $(\psi^{x_2}_i)^{-1} ( \psi^{x_1}_i (x)) \to x$.
This shows pointwise convergence of the maps in question.
On the other hand, due to Assertion~\ref{Thm_smooth_convergence_a}, which is ensured by Lemma~\ref{Lem_smooth_conv_local}, the maps $(\psi^{x_2}_i)^{-1} \circ \psi^{x_1}_i$ are locally uniformly bounded in every $C^m$-norm.
This shows local smooth convergence.
\end{proof}

After possibly shrinking $U^x, V^x_i$, we can find a sequence of points $x_j \in \RR^*$ such that if we set $U'_j := U^{x_j}$, $V'_{j,i} := V^{x_j}_i$, $\psi'_{ j,i} := \psi^{x_j,}_i$, then the collection of subsets $\{ U'_j \}_{j=1}^\infty$ forms a locally finite cover of $\RR^*$ consisting of relatively compact subsets.

Using a center of mass construction, we can find an open neighborhood $\{ (x,x) \in (\RR^*)^2 \} \subset \Delta_2 \subset (\RR^*)^2$ and a smooth map
\[ \Sigma_2 : [0,1]^2 \times \Delta_2 \longrightarrow \RR^* \]
with the following properties for all $s_1, s_2 \in [0,1]$, $x, x_1, x_2 \in \RR^*$:
\begin{enumerate}[label=(\arabic*)]
\item $\Sigma_2 (1, 0, x_1, x_2) = x_1$, $\Sigma_2 (0,1, x_1, x_2) = x_2$.
\item $\Sigma_2 (s_1, s_2, x, x) = x$.
\item If $t = \tf(x_1)=\tf(x_2)$, then $\tf (\Sigma_2 (s_1, s_2, x_1, x_2)) = t$.
\end{enumerate}
Based on $\Delta_2, \Sigma_2$, we can inductively construct open neighborhoods 
\[ \{ (x,\ldots, x) \in (\RR^*)^N \} \subset \Delta_N \subset (\RR^*)^N \]
and smooth maps 
\[ \Sigma_N : [0,1]^N \times \Delta_N \longrightarrow \RR^* \]
such that for all $(x_1, \ldots, x_N) \in \Delta_N$, $s_1, \ldots, s_N \in [0,1]$ we have $(x_1, \ldots, x_{N-1}) \in \Delta_{N-1}$ and $(\Sigma_{N-1} ( s_1, \ldots, s_{N-1}, x_1, \ldots, x_{N-1}), x_N) \in \Delta_2$ by setting
\[ \Sigma_N (s_1, \ldots, s_N, x_1, \ldots, x_N) 
:= \Sigma_2 \big( 1-s_N, s_N, \Sigma_{N-1} ( s_1, \ldots, s_{N-1}, x_1, \ldots, x_{N-1}), x_N \big). \]
Then $\Sigma_N$ has the following properties for $s_1, \ldots, s_N \in [0,1]$, $x, x_1, \ldots, x_N \in \RR^*$:
\begin{enumerate}[label=(\arabic*)]
\item If for some fixed $j \in \{ 1, \ldots, N \}$ we have $s_{j'} = \delta_{jj'}$ for all $j' = 1, \ldots, N$, then $\Sigma_N (s_1, \lb \ldots, \lb s_N, \lb x_1,\lb  \ldots,\lb x_N) = x_j$. 
\item $\Sigma_N (s_1, \ldots, s_N, x, \ldots, x) = x$.
\item If $t = \tf(x_1)= \ldots =\tf(x_N)$, then $\tf (\Sigma_N (s_1, \ldots, s_N, x_1, \ldots, x_N)) = t$.
\item \label{prop_Sigma_N} If $s_{N'+1}  = \ldots = s_N = 0$ for some $1 \leq N' \leq N$, then $ \Sigma_N (s_1, \ldots,  s_{N}, x_1, \ldots, x_{N})= \Sigma_{N'} (s_1, \ldots,  s_{N'}, x_1, \ldots, x_{N'})$.
\end{enumerate}

Choose a partition of unity $\{ \eta_j \in C^\infty_c (U'_j) \}_{j=1}^\infty$ subordinate to the open cover $\{ U'_j \}_{j=1}^\infty$ over $\RR^*$ and set $\chi'_{j,i} := (\psi'_{ j,i})^{-1} : V'_{j,i} \to U'_j$.
For any $N$ we define
\[ \chi_{N,i} : W_{N,i} \longrightarrow \RR^\infty, \qquad y \longmapsto \Sigma_N \big( \eta_1 (\chi'_{ 1,i} (y)), \ldots, \eta_N (\chi'_{ N,i} (y)), \chi'_{ 1,i} (y), \ldots, \chi'_{ N,i} (y) \big), \]
where $W_{N,i} \subset \RR^i$ is the maximal subset on which $\chi_{N,i}$ is defined.

\begin{Claim} \label{Cl_omega_psi_p_to_id}
For any $1 \leq j \leq N$ the sequence of maps
\[ \chi_{N,i} \circ \psi'_{ j,i} :  (\psi'_{j,i})^{-1} (W_{N,i}) \cap U'_j \longrightarrow \chi_{N,i} ( W_{N,i} \cap V'_{j,i} ) \]
converges to the identity map on $U'_j$ in $C^\infty_{\loc}$.
\end{Claim}

\begin{proof}
This is a direct consequence of Claim~\ref{Cl_psi_x1_psi_x2_smooth_conv} and the definition of $\chi_{N,i}$.
\end{proof}

Set 
\[ U''_N := \bigcup_{j=1}^N U'_j \subset \RR^*, \qquad V''_{N,i} := \bigcup_{j=1}^N \psi'_{j,i} (  \{ \eta_j > 0 \}). \]

\begin{Claim} \label{Cl_sequenc_i_star}
There is a sequence $1 \leq i^*_1 < i^*_2 < \ldots$ such that the following is true:
\begin{enumerate}[label=(\alph*)]
\item \label{Cl_sequenc_i_star_a} If $i \geq i^*_N$, then $V''_{N,i} \subset \chi_{N,i}^{-1} (U''_N)$ and the map
\[ \chi_{N,i} |_{V''_{N,i}} :  V''_{N,i} \longrightarrow U''_N \]
is a diffeomorphism onto its image.
\item \label{Cl_sequenc_i_star_b} For any $j \geq 1$ the following is true for some large $N^*_j \geq j$.
If $N^*_j \leq N_1 \leq N_2$ and $i \geq i^*_{j}$, then
\[ \chi_{N_1, i} \circ \psi'_{j,i} |_{\{ \eta_j > 0 \} }= \chi_{N_2,i} \circ \psi'_{j,i} |_{\{ \eta_j > 0 \} }. \]
\end{enumerate}
\end{Claim}

\begin{proof}
For Assertion~\ref{Cl_sequenc_i_star_a} fix $N$.
If $i$ is large enough, then for all $j=1,\ldots,N$ the map $\chi_{N,i} \circ \psi'_{ j,i}$ restricted to $\{ \eta_j > 0 \} \subset U'_j$ is a diffeomorphism onto its image and $(\chi_{N,i} \circ \psi'_{ j,i})( \{ \eta_j > 0 \} ) \subset U'_j$.
So $\chi_{N,i} |_{V''_{N,i}}$ is a local diffeomorphism and
\[  \chi_{N,i} (V''_{N,i} ) = \bigcup_{j=1}^N (\chi_{N,i} \circ \psi'_{ j,i})( \{ \eta_j > 0 \})  \subset U''_N. \]
To see that $\chi_{N,i}$ is injective for large $i$, suppose that 
$\chi_{N,i} (x_{1,i}) = \chi_{N,i}(x_{2,i})$ for $x_{1,i}, x_{2,i} \in V''_{N,i}$.
So there are $x'_{1,i} \in \{ \eta_{j_{1,i}} > 0 \}$, $x'_{2,i} \in \{ \eta_{j_{2,i}} > 0 \}$, for $j_{1,i}, j_{2,i} \in \{ 1, \ldots, N \}$, such that $x_{1,i}=\psi'_{j_{1,i},i} (x'_{1,i})$, $x_{2,i}=\psi'_{j_{2,i},i} (x'_{2,i})$.
So
\[ ( \chi_{N,i} \circ \psi'_{j_{1,i},i} )(x'_{1,i}) = ( \chi_{N,i}  \circ \psi'_{j_{2,i},i} )(x'_{2,i}). \]
Using Claim~\ref{Cl_omega_psi_p_to_id}, it follows that for large $i$ the points $x'_{1,i}, x'_{2,i}$ are close enough such that we can use Claim~\ref{Cl_psi_x1_psi_x2_smooth_conv} to find a point $x''_{2,i} \in U'_{j_{1,i}}$ near $\supp \eta_{j_{1,i}}$ such that $x_{2,i}=\psi'_{j_{2,i},i} (x'_{2,i}) =\psi'_{j_{1,i},i} (x''_{2,i})$.
Therefore,
\[ ( \chi_{N,i}  \circ \psi'_{j_{1,i},i} )(x'_{1,i}) 
= ( \chi_{N,i}  \circ \psi'_{j_{2,i},i} )(x'_{2,i})
= ( \chi_{N,i}  \circ \psi'_{j_{1,i},i} )(x''_{2,i}), \]
which implies $x'_{1,i} = x''_{2,i}$ by Claim~\ref{Cl_omega_psi_p_to_id} for large $i$.
Thus $x_{1,i}=\psi'_{j_{1,i},i} (x'_{1,i}) = \psi'_{j_{1,i},i} (x''_{2,i}) = x_{2,i}$, proving Assertion~\ref{Cl_sequenc_i_star_a}.

To see Assertion~\ref{Cl_sequenc_i_star_b}, fix $j$ and choose $N^*_j \geq j$ large enough such that $U'_j$ is disjoint from $U'_{j'}$ for $j' \geq N^*_j$.
This can be achieved since we assumed that the subset $U'_j$ are relatively compact and form a locally finite cover of $\RR^*$.
Let $N^*_j \leq N_1 \leq N_2$.
If $i$ is sufficiently large, then by Claim~\ref{Cl_psi_x1_psi_x2_smooth_conv} we have $\chi'_{j',i} (\psi'_{j,i} ( \{ \eta_j > 0 \})) \cap \supp \eta_{j'} = \emptyset$ for $j' \geq N^*_j$.
Then $\chi_{N_1,j}|_{\psi'_{j,i} ( \{ \eta_j > 0 \})} = \chi_{N_2,j}|_{\psi'_{j,i} ( \{ \eta_j > 0 \})}$ by Property~\ref{prop_Sigma_N} of $\Sigma_{N_2}$.
\end{proof}

Choose $N_1 \leq N_2 \leq \ldots \to \infty$ such that $i \geq \max \{ i^*_1, \ldots, i^*_{N_i} \}$ if $N_i > 1$.
Set $U_i = V_i := \emptyset$ if $N_i = 1$ and otherwise set
\[ 
 V_i := V''_{N_i,i} , \qquad
 \chi_{i}:  =  \chi_{N_i,i} |_{V_i } : V_i \to U_i := \chi_i (V_i) . \]
 Note that Claim~\ref{Cl_sequenc_i_star} implies that $\chi_{i}$ is a diffeomorphism and that for any $j \geq 1$ and $N \geq N^*_j$ the following is true for large $i$
\begin{equation} \label{eq_omega_i_omega_i_N}
 \chi_{ i} \circ \psi'_{j,i} |_{\{ \eta_j > 0 \} }= \chi_{N,i} \circ \psi'_{j,i} |_{\{ \eta_j > 0 \} }. 
\end{equation}
Set $\psi_i := \chi_i^{-1} : U_i \to V_i$.
Then $U_1 \subset U_2 \subset \ldots \subset \RR^*$ and $\RR^* \subset \bigcup_{i=1}^\infty U_i \subset \RR^\infty$.

\begin{Claim} \label{Cl_omega_i_props}
For any $j \geq 1$ we have $\{ \eta_j > 0 \} \subset U_i$ for large $i$ and the sequence of maps
\[ \chi_{i} \circ \psi'_{j,i}|_{\{ \eta_j > 0 \} } : \{ \eta_j > 0 \} \longrightarrow ( \chi_{i} \circ \psi'_{j,i}) ( \{ \eta_j > 0 \} ) \]
converge to the identity map on $\{ \eta_j > 0 \}$ in $C^\infty_{\loc}$.
\end{Claim}

\begin{proof}
This is a direct consequence of Claim~\ref{Cl_omega_psi_p_to_id} and (\ref{eq_omega_i_omega_i_N}).
\end{proof}

It follows that $\RR^* \subset \bigcup_{i=1}^\infty U_i \subset \RR^\infty$ and after possibly restricting the maps $\psi_i$ to slightly smaller subsets, we can achieve that $U_1 \subset U_2 \subset \ldots$ without changing the statement Claim~\ref{Cl_omega_i_props}.

It now follows, using Lemma~\ref{Lem_smooth_conv_local} and Claim~\ref{Cl_omega_i_props}, that we have local smooth convergence
\begin{equation} \label{eq_loc_conv_Ass_a}
 \psi_i^* g^i \to g^\infty, \qquad \psi_i^* \partial^i_{\tf} \to \partial^\infty_{\tf}, \qquad v_i \circ \psi_i \to v_\infty. 
\end{equation}
By the same reason, Assertions~\ref{Thm_smooth_convergence_f}, \ref{Thm_smooth_convergence_e} hold.
After possibly shrinking $U_i, V_i$, the local convergences (\ref{eq_loc_conv_Ass_a}) can be turned into global convergences, as stated in Assertion~\ref{Thm_smooth_convergence_a}.

Next, we show Assertion~\ref{Thm_smooth_convergence_c}.
Consider a sequence $x_i \in \XX^i$ and a point $x_\infty \in \XX^\infty$.
Choose $j \geq 1$ such that $x_\infty \in \{ \eta_j > 0 \}$.
If $x_i \to x_\infty$ within $\CF$, then by Lemma~\ref{Lem_smooth_conv_local} we have $x_i \in V'_{j,i}$ for large $i$ and $\psi^{\prime -1}_{ j, i} (x_i) \to x_\infty$.
So by Claim~\ref{Cl_omega_i_props}, for large $i$ we have $x_i \in V_i$ and 
\[ \psi_{ i}^{-1} (x_i) = (\chi_i \circ \psi'_{ j, i})  (\psi_{j, i}^{\prime -1} (x_i) ) \to x_\infty. \]
On the other hand, if $\psi_{ i}^{-1} (x_i) \to x_\infty$, then for large $i$
\[ (\psi^{\prime -1}_{j,i}) (x_i) = (\chi_i \circ \psi'_{ j, i})^{-1} ( \psi_{ i}^{ -1} (x_i) ) \to x_\infty, \]
which implies that $x_i \to x_\infty$ within $\CF$ by Lemma~\ref{Lem_smooth_conv_local}.

To see Assertion~\ref{Thm_smooth_convergence_b}, note that Assertion~\ref{Thm_smooth_convergence_c} guarantees pointwise convergence of the pullbacks of $K^i$ to $K^\infty$.
Together with Assertion~\ref{Thm_smooth_convergence_a} and local derivative estimates, this convergence can be upgraded to local smooth convergence.
After possibly shrinking, $U_i, V_i$, we obtain the convergence statement in Assertion~\ref{Thm_smooth_convergence_b}.

Lastly, we will deduce Assertion~\ref{Thm_smooth_convergence_d} from the other assertions.
Suppose that after passing to a subsequence, the convergence (\ref{eq_XXimui_convergence_in_smooth_sec}) is time-wise $t \in I^\infty$.
Since the maps $\varphi^i_t, \varphi^\infty_t$ are isometric embeddings and the $\psi_i$ are locally uniformly Lipschitz, it suffices to show that for any $x \in \RR^*_t$ we have
\begin{equation} \label{eq_dZ_varphi_ptwise}
 d^Z_t (\varphi^i_t (\psi_i(x)), \varphi^\infty_t (x) )  \longrightarrow 0. 
\end{equation}
To see this, fix $x \in \RR^*_t$ and choose $t^* > t$ such that $x^*:= x(t^*) \in \RR^*$ exists.
Then $x^*_i := \psi_i ( x^* ) \to x^*$ within $\CF$, by Assertion~\ref{Thm_smooth_convergence_c}.
So by Theorem~\ref{Thm_tdmu_convergence_reverse}\ref{Thm_tdmu_convergence_reverse_b}
\[  (\varphi^i_t )_* \nu^i_{x^*_i; t} \xrightarrow[i \to \infty]{\quad W_1 \quad} (\varphi^\infty_t )_* \nu^\infty_{x^*; t}. \]
Let $x'_i := x^*_i( t) \in \XX^i_t$.
By Assertion~\ref{Thm_smooth_convergence_a} we have $d^i_t (x'_i, \psi_i(x)) \to 0$.
Therefore, using Proposition~\ref{Prop_HK_regular}\ref{Prop_HK_regular_a} for $t^*$ close to $t$
\begin{multline*}
  \limsup_{i \to \infty} d^Z_t (\varphi^i_t (\psi_i(x)), \varphi^\infty_t (x) ) 
= \limsup_{i \to \infty} d^Z_t (\varphi^i_t (x'_i), \varphi^\infty_t (x) ) \\
\leq  \limsup_{i \to \infty} \big( d^{\XX^i_t}_{W_1} (\delta_{x'_i}, \nu^i_{x^*_i; t} ) +
d^{Z_t}_{W_1} ( (\varphi^i_t)_* \nu^i_{x^*_i; t}, (\varphi^\infty_t)_* \nu^i_{x^*; t} ) + d^{\XX^\infty_t}_{W_1} ( \nu^i_{x^*; t}, \delta_{x} ) \big) 
\leq 2 C \sqrt{ t^* - t}. 
\end{multline*}
Letting $t^* \searrow t$ implies (\ref{eq_dZ_varphi_ptwise}).
\end{proof}
\bigskip

\subsection{Convergence of parabolic neighborhoods with bounded curvature} \label{subsec_conv_parab_nbhd_reg_part}
Suppose that we are still in the same setting as in Subsection~\ref{subsec_smooth_convergence}.
In the following, we consider a sequence of points $x_i \in \XX^i$ that converges to a point $x_\infty \in \XX^\infty$ within $\CF$.
We will assume that parabolic neighborhoods $P_i$ of uniform size around $x_i$ are unscathed and that $|{\Rm}| \leq C$ on $P_i$ for some uniform $C < \infty$.
We will see that under these conditions we obtain the existence of an unscathed parabolic neighborhood $P_\infty$ around $x_\infty$ on which the same curvature bound holds.
The radius and the backward time of this parabolic neighborhood is the same as that of the parabolic neighborhood $P_\infty$, however, the forward time may be smaller.
In this case, the density functions of any conjugate heat flow --- in particular those of the conjugate heat kernels --- converge to zero near the forward end of $P_\infty$.

In order to state the following theorem we define for any point $x \in \MM$ in a Ricci flow spacetime $\MM$ over $I$ and any $D , T^-, T^+ > 0$ the {\bf open parabolic neighborhood} as follows:
\[ P^\circ (x; D, -T^-, T^+) := P(x;  D, -T^-, T^+) \cap \MM_{(\tf (x) - T^-, \tf(x) + T^+)}. \]
We say that $P^\circ (x; D, -T^-, T^+)$ is {\bf unscathed} if for any $D' \in (0,D)$ the ball $B(x, D')$ is relatively compact and any point $y \in B(x, D)$ survives until any time in the time-interval $(\tf (x) - T^-, \tf(x) + T^+) \cap I$.
This is equivalent to the fact that $P(x; D, -T^{\prime, -}, -T^{\prime, +})$ is unscathed for any $T^{\prime, -} \in (0, T^-)$, $T^{\prime, +} \in (0, T^+)$.

\begin{Theorem} \label{Thm_conv_parab_nbhd}
Consider the setting of Subsection~\ref{subsec_smooth_convergence} and consider points $x_i \in \XX^i$ with $x_i \to x_\infty \in \XX^\infty$ within $\CF$.
Suppose that there are numbers $D , T^-, T^+ > 0$, $C < \infty$ such that for large $i$, $x_i \in \RR^i$, the open parabolic neighborhood $P_i := P^\circ (x_i; D, -T^-, T^+) \subset \RR^i$ is unscathed, that $|{\Rm}| \leq C < \infty$ on $P_i$ and $|B(x_i, D)| \geq c> 0$.
Then $x_\infty \in \RR^* \subset \RR^\infty$ and there is a $0 < T^* \leq T^+$ such that the open parabolic neighborhood  $P_\infty := P^\circ (x_\infty; D, -T^-, T^*) \subset \RR^* \subset \RR^\infty$ is unscathed and $|{\Rm}| \leq C$ on $P_\infty$.
Moreover, if $T^* < T^+$ and $\tf(x_\infty) + T^* < \sup I^\infty$, then the following is true:
\begin{enumerate}[label=(\alph*)]
\item \label{Thm_conv_parab_nbhd_a} No point in $P_\infty$ survives until or past time $\tf(x_\infty) + T^*$.
\item \label{Thm_conv_parab_nbhd_b} For any $\td x \in \XX^\infty_{>T^*}$, any conjugate heat flow $(\td\mu_t = \lb \td v_t \, dg^\infty_t)_{t \in I^\infty \cap (-\infty, t^*)}$ on $\XX^\infty$, where $t^* > T^*$, and any $D' \in (0, D)$ we have
\[ \lim_{t \nearrow \tf(x_\infty)+ T^*} \sup_{(B(x_\infty, D'))(t)} K(\td x, \cdot) 
= \lim_{t \nearrow \tf(x_\infty)+T^*} \sup_{(B(x_\infty, D'))(t)} v^\infty
= \lim_{t \nearrow \tf(x_\infty)+T^*} \sup_{(B(x_\infty, D'))(t)} \td v
= 0. \]
\end{enumerate}
\end{Theorem}

\begin{Example}
Consider a sequence of singular Ricci flows $\MM^i$ on $S^2 \times S^1$ that develop a non-degenerate neckpinch at some uniform time $T$ (see also Example~\ref{Ex_neckpinch_continuity}).
Suppose that the amount by which distances expand within $\MM^i$ is uniformly bounded.
Assume moreover that the time-$0$-slices $\MM^i_0$ contain open subsets $U_i \subset \MM^i_0$ such that the metric on $U_i$ is isometric to a standard cylinder $S^2(1) \times (0,L_i)$ of length $L_i \to \infty$ and such that the metric on the complements $\MM^i_0 \setminus U_i$ has uniformly bounded diameter and the shape of a neckpinch.
Denote by $\XX^i$ the metric flows associated to $\MM^{\prime,i} := \MM^i$ and fix points $x_i \in \XX^i_{T'}$ for some uniform $T' > T$ (see Theorem~\ref{Thm_sing_flow_to_met_flow}).
If the points $x_i$ are chosen sufficiently close to the neckpinch, then the metric flow pairs $(\XX^i, (\nu_{x_i;t}))$ subsequentially $\IF$-converge to a metric flow pair $(\XX^\infty, (\nu_{x_\infty;t}))$, where $\XX^\infty$ corresponds to one branch $\MM^{\prime, \infty} \subset \MM^\infty$ of a singular flow starting from a metric on $S^2 \times \IR$ that is isometric to $S^2(1) \times ( (-\infty, -1) \cup (1,\infty))$ on the complement of a compact subset and develops a non-degenerate neckpinch at time $T$.
In this example we can find a sequence of points $y_i \in  \RR^i \subset \XX^i$ that converges to some point $y_\infty \in \RR^\infty \subset \XX^\infty$ with $\tf (y_\infty) < T$ close to $T$, such that we have a uniform curvature bound on the two-sided, unscathed parabolic balls $P(y_i;r)$ for some $r > 0$ with $\tf (y_\infty) + r^2 > T$, but $P(y_\infty; r)$ is not unscathed, because it corresponds to a parabolic neighborhood in $\MM^\infty$ that is not fully contained in  $\MM^{\prime, \infty}$.
In this example, any conjugate heat kernel based at a point in $\XX^{ \infty}_{> T}$ converges to zero on $P(y_\infty; r)$ near $T$, as asserted in Theorem~\ref{Thm_conv_parab_nbhd}.
\end{Example}

\begin{proof}
After possibly replacing $T^+$ with $\sup I^\infty - \tf (x_\infty)$ we may assume in the following that $\tf(x_\infty) + T^+ \leq \sup I^\infty$.
Then the property $\tf(x_\infty) + T^* < \sup I^\infty$ holds whenever $T^* \leq T^+$.

Consider an arbitrary subsequence of the given sequence of metric flow pairs.
Note that $\RR^*$ may become larger for such a subsequence.
We will first show parts of the theorem, which don't involve $\RR^*$, and then return to the original sequence of metric flow pairs to show the remaining parts of the theorem.
In the following, $\RR^*$ will always denote the set of points at which the convergence of the current subsequence is smooth.

By \cite{Hamilton_RF_compactness}, we may pass to a further subsequence such that the Ricci flows $g^i$ restricted to $P_i$ and pointed at $x_i$ converge smoothly to a Ricci flow, which we will represent by a pointed Ricci flow spacetime $(\mathcal{P}', x')$ with $\tf(x') = \tf (x_\infty)$ such that $P^\circ (x'; D, -T^-, T^+) = P'$.
We can find an increasing sequence of open subsets $U'_1 \subset U'_2 \subset \ldots \subset \mathcal{P}'$, $\bigcup_{i=1}^\infty U'_i = \mathcal{P}'$ and a sequence of diffeomorphisms onto their images $\psi'_i : U'_i \to \RR^i$ that are time-preserving and $\partial_{\tf}$-preserving and such that $(\psi'_i)^* g^i \to g'$ in $C^\infty_{\loc}$.
After passing to another sequence, we may also assume that $v^i \circ \psi'_i \to v' \in C^\infty (\mathcal{P}')$ with $v' \geq 0$, $\square^* v' =0$.

\begin{Claim} \label{Cl_Pp_Pinfty}
If for some $D_1 \in (0,D]$, $T^-_1 \in (0, T^-]$, $T^+_1 \in (0, T^+]$ we have $P^\circ (x_\infty ; D_1, -T^-_1, T^+_1) \subset \RR^* \subset \RR^\infty$ and this open parabolic neighborhood is unscathed, then there is an isometry of Ricci flow spacetimes $\chi : P^\circ (x' ; D_1, -T^-_1, T^+_1) \to P^\circ (x_\infty ; D_1, -T^-_1, T^+_1)$ with $\chi(x') = x_\infty$ such that the following holds:
\begin{enumerate}[label=(\alph*)]
\item \label{Cl_Pp_Pinfty_a} $v' = v^\infty \circ \chi$ on $P^\circ (x' ; D_1, -T^-_1, T^+_1)$.
\item \label{Cl_Pp_Pinfty_b} For any $y' \in P^\circ (x' ; D_1, -T^-_1, T^+_1)$ we have $\psi'_i (y') \to \chi(y')$ within $\CF$.
\item \label{Cl_Pp_Pinfty_c} For any conjugate heat flow $(\td\mu_t)_{t \in I^\infty, t < t^*}$ on $\XX^\infty$ with $t^* > \tf(x_\infty) - T_1^-$ there is smooth function $\td v' \in C^\infty ( \PP'_{< t^*} )$ such that $\td v' = \td v \circ \chi$ on $(P^\circ (x' ; D_1, -T^-_1, T^+_1))_{< t^*}$.
\end{enumerate}
\end{Claim}

\begin{proof}
By Theorem~\ref{Thm_smooth_convergence}\ref{Thm_smooth_convergence_a} the flow $g^\infty$ restricted to $P^\circ (x_\infty ; D_1, -T^-_1, T^+_1)$ and equipped with $v^\infty$ is a geometric limit of the same sequence of flows as $P^\circ (x' ; D_1, -T^-_1, T^+_1) \subset \mathcal{P}'$.
Therefore, both flows may be identified.
Assertions~\ref{Cl_Pp_Pinfty_a}, \ref{Cl_Pp_Pinfty_b} now follow from Theorem~\ref{Thm_smooth_convergence}.
To see Assertion~\ref{Cl_Pp_Pinfty_c}, suppose first that $(\td\mu_t)_{t \in I^\infty, t < t^*}$ is a conjugate heat kernel based at some point in $\XX_{t^*}^\infty$.
Then by Theorem~\ref{Thm_pts_as_limits_within} the conjugate heat flow $(\td\mu_t)_{t \in I^\infty, t < t^*}$ can be represented as a limit of conjugate heat flows $(\td\mu_{i,t})_{t \in I^\infty, t < t^*_i}$ on $\XX^i$ within $\CF$.
Write $d\td\mu_{i,t} = \td v_i \, dg_i$ on $\RR^i_{<t^*_i}$.
By Theorem~\ref{Thm_smooth_convergence}\ref{Thm_smooth_convergence_f} and Assertion~\ref{Cl_Pp_Pinfty_b} of this claim we obtain smooth convergence of $\td v_i \circ \psi'_i \to \td v \circ \chi$ on $(P^\circ (x' ; D_1, -T^-_1, T^+_1))_{< t^*}$ and by local derivative estimates, we obtain subsequential convergence of the functions $\td v_i \circ \psi'_i$ to some $\td v' \in C^\infty ( \PP'_{< t^*} )$ on all of $\PP'_{< t^*}$.
Finally, if $(\td\mu_t)_{t \in I^\infty, t < t^*}$ is a finite convex combination of conjugate heat kernels, then Assertion~\ref{Cl_Pp_Pinfty_c} follows by linearity and the general case follows via a limit argument.
\end{proof}

Consider the conjugate flow $v'$ on $\mathcal{P}'$.
By the strong maximum principle there is some time $T^* \in ( - T^-,   T^+]$ such that $v' \equiv 0$ on $\mathcal{P}'_{[\tf(x_\infty) + T^*, \tf(x_\infty) + T^+)}$ and $v' > 0$ on $\mathcal{P}'_{(\tf(x_\infty) - T^-, \tf(x_\infty) + T^*)}$ (in the case $T^* = T^+$ we have $v' > 0$ on all of $\mathcal{P}'$).
By iterating Claim~\ref{Cl_Pp_Pinfty} and the uniformity statement of Lemma~\ref{Lem_smooth_conv_local}, we obtain that $P^\circ (x_\infty ; D, -T^-, T^*) \subset \RR^* \subset \RR^\infty$ is unscathed and if $T^* < T^+$, then for any $D' \in (0, D)$
\[  \lim_{t \nearrow \tf(x_\infty)+T^*} \sup_{(B(x_\infty, D'))(t)}  v^\infty
= 0, \]
which also implies Assertion~\ref{Thm_conv_parab_nbhd_a}.
Since the inclusion $P^\circ (x_\infty ; D, -T^-, T^*) \subset \RR^\infty$ is independent of the subsequence of the original sequence of metric flow pairs, we obtain from Assertions~\ref{Cl_Pp_Pinfty_a}, \ref{Cl_Pp_Pinfty_b} in the Claim that for \emph{any} subsequence of the original sequence of metric flow pairs we obtain the same time $T^*$ and the restrictions $\mathcal{P}'_{< T^*}$, $v' |_{\mathcal{P}'_{< T^*}}$ are always the same.
So by the same argument as above, we have $P^\circ (x_\infty ; D, -T^-, T^*) \subset \RR^*$, where $\RR^*$ now refers to the set of points at which the convergence of the original sequence of metric flow pairs is smooth.

It remains to verify Assertion~\ref{Thm_conv_parab_nbhd_b}.
So let $(\td\mu_t = \lb \td v_t \, dg^\infty_t)_{t \in I^\infty \cap (-\infty, t^*)}$ be a conjugate heat flow on $\XX^\infty$, where $t^* > T^*$.
By Assertion~\ref{Cl_Pp_Pinfty_c} of the Claim we know that $\td v \circ \chi$ can be extended to a smooth conjugate heat flow on $\PP'_{<t^*}$.
So it suffices to show that for any $D' \in (0,D)$ we have
\[ \lim_{t \nearrow \tf(x_\infty)+T^*} \int_{( (B(x_\infty, D'))(t)} \td v_t \, d \td\mu_t
=
\lim_{t \nearrow \tf(x_\infty)+T^*} \td\mu_t \big( (B(x_\infty, D') )(t) \big) = 0. \]
To see this, choose some $t^{**} \in (\tf(x_\infty) + T^*, t^*)$ and observe that for any $t < \tf(x_\infty)+T^*$
\begin{align*}
 \td\mu_t \big( (B(x_\infty, D') )(t) \big) &= \int_{\XX^\infty_{t^{**}}} \nu^\infty_{x^{**};t} \big( (B(x_\infty, D') )(t) \big) \, d \td\mu_{t^{**}} (x^{**}), \\
\mu^\infty_t \big( (B(x_\infty, D') )(t) \big) &= \int_{\XX^\infty_{t^*}} \nu^\infty_{x^{**};t} \big( (B(x_\infty, D') )(t) \big) \, d \mu^\infty_{t^{**}} (x^{**}). 
\end{align*}
Since the second integral goes to $0$ as $t \nearrow \tf(x_\infty)+T^*$, we obtain, using Definition~\ref{Def_metric_flow}\ref{Def_metric_flow_6}, that the first one has to go to $0$ as well.
\end{proof}

\bibliography{bibliography}{}
\bibliographystyle{amsalpha}

\end{document}